\documentclass[11pt,a4paper,reqno]{amsart}%
\usepackage{amsthm,amsmath,amsfonts,amssymb,xcolor,amsxtra,appendix,bookmark,dsfont,bm,mathrsfs}
\usepackage[latin1]{inputenc}
\usepackage{xcolor} 
\usepackage{amssymb}\usepackage{graphicx}
\usepackage{tikz}
\usepackage{mathtools}
\usepackage{enumitem}
\usepackage{hyperref}
\usepackage[capitalise]{cleveref}
\theoremstyle{plain}
\newtheorem{theorem}{Theorem}
\newtheorem{lemma}[theorem]{Lemma}

\newtheorem{proposition}[theorem]{Proposition}

\newtheorem{definition}{Definition}

\theoremstyle{remark}
\newtheorem{remark}[theorem]{Remark}

\allowdisplaybreaks

\title[Inhomogeneous kinetic limit of 3-wave systems]{On the wave turbulence theory for a stochastic  KdV type equation - Generalization for the inhomogeneous kinetic limit}

\author[A. Hannani]{Amirali Hannani
}
\address{Institute for Theoretical Physics, KU Leuven, 3001 Leuven, Belgium}
\email{amirali.hannani@kuleuven.be
} 
\thanks{A.H. is funded in part by the FWO grant
G098919N, the ANR grant LSD-15-CE40-0020-01, and the NSF Grant DMS-1929284.}
\author[M. Rosenzweig]{Matthew Rosenzweig
}
\address{Department of Mathematics, Massachusetts Institute of Technology, Cambridge, MA 02139, USA}
\email{ mrosenzw@mit.edu} 
\thanks{M.R. is funded in part by the NSF grant DMS-2206085 and the Simons Foundation through the Simons Collaboration on Wave Turbulence.}

\author[G. Staffilani]{Gigliola Staffilani
}
\address{Department of Mathematics, Massachusetts Institute of Technology, Cambridge, MA 02139, USA}
\email{gigliola@math.mit.edu} 
\thanks{G.S. is  funded in part by  the NSF grants DMS-1764403, DMS-2052651 and the Simons Foundation through the Simons Collaboration on Wave Turbulence.}

\author[M.-B. Tran]{Minh-Binh Tran}
\address{Department of Mathematics, Texas A\&M University College Station, TX 77843, USA}
\email{minhbinh@tamu.edu} 
\thanks{M.-B. T is  funded in part by  the  NSF Grant DMS-1854453,    Humboldt Fellowship,   NSF CAREER  DMS-2044626, and NSF Grant DMS-2204795.}

\begin{document}
\date{\today}

\begin{abstract} 
Starting from a {\it stochastic  Zakharov-Kuznetsov (ZK) equation} on a lattice, the  previous work \cite{staffilani2021wave} by the last two authors gave a derivation of the {\it homogeneous  3-wave kinetic equation} at the kinetic limit under very general assumptions: the initial condition is out of equilibrium, the dimension $d\ge 2$, the smallness of the nonlinearity $\lambda$ is allowed to be independent of the size of the lattice, the weak noise is chosen not to compete with the weak nonlinearity and  not to inject energy  into the equation. In the present work, we build on the framework of \cite{staffilani2021wave}, following the formal derivation of Spohn \cite{Spohn:TPB:2006} and inspired by previous work \cite{hannani2021stochastic} of the first author and Olla, so that the {\it inhomogeneous 3-wave kinetic equation} can also be obtained at the kinetic limit under analogous assumptions. Similar to the homogeneous case---and unlike  the cubic nonlinear Schr\"odinger equation---the {inhomogeneous} kinetic description of the deterministic lattice ZK equation is unlikely to happen due to the vanishing of the dispersion relation on a certain singular manifold on which not only $3$-wave interactions but also all $n$-wave interactions ($n\ge3$) are allowed to happen, a phenomenon first observed by Lukkarinen  \cite{lukkarinen2007asymptotics}. To the best of  our knowledge, our work provides the first rigorous derivation of a nonlinear inhomogeneous wave kinetic equation in the kinetic limit.
\end{abstract}

\maketitle

 \tableofcontents
 
\section{Introduction}\label{intro} 
Wave turbulence theory describes the statistical properties of ensembles of dispersives waves out of  thermal equilibrium with weakly nonlinear interactions using the paradigm of Boltzmann's kinetic theory for dilute gases. The theory has its origin in the works of  Peierls \cite{Peierls:1993:BRK,Peierls:1960:QTS}, Hasselmann \cite{hasselmann1962non,hasselmann1974spectral},  Benney et al. \cite{benney1969random,benney1966nonlinear},  Zakharov et al. \cite{zakharov2012kolmogorov}, and has been found in  a vast range of physical applications, a description of which can be found in the monograph \cite{Nazarenko:2011:WT}. This theory is intimately connected to the question of growth of high Sobolev norms of solutions to certain dispersive equations (in physical terms, the migration of energy to high frequencies), a line of investigation initiated by Bourgain \cite{bourgain1996growth}, which continues to be highly active to this day \cite{bourgain1999growth,
bourgain1999growtha,carles2012energy, colliander2010transfer,deng2017strichartz, hani2015modified,kuksin1997oscillations,  majda1997one,sohinger2011bounds, staffilani1997growth,staffilani2020stability,
berti2019long,guardia2015growth,haus2015growth}. 

A fundamental object in wave turbulence theory is the \emph{wave kinetic equation (WKE)}, which is the wave analogue of the Boltzmann equation for particles and describes the evolution of the system's energy density across waves of different frequencies. The precise form of this equation is deferred until later. Let $\lambda>0$ be the small parameter describing the strength of the interactions between waves in the system under consideration. The relevant time scale for the derivation of the associated wave kinetic equation in the {\it van Hove limit} or {\it the kinetic limit  } is
\begin{equation}
	\label{VanHove}
	t\backsim \lambda^{-2}.
\end{equation}

Rigorously deriving wave kinetic equations is one of the central questions in the mathematical treatment of wave turbulence theory and has become an active topic in recent years. The starting point in the derivation is a dispersive equation, which may be placed either on  a hypercubic lattices (lattice setting) or on a continuum domain (continuum setting), such as the flat torus, each of which has its own challenges (e.g., see \cite{basile2010energy,chen2005localization,chen2005localization2,lukkarinen2007asymptotics}). The pioneering work of Lukkarinen-Spohn \cite{LukkarinenSpohn:WNS:2011}  provided the first rigorous  derivation of the homogeneous 4-wave kinetic equation at the kinetic limit, starting from the  cubic nonlinear Schr\"odinger equation (NLS) on a lattice and with equilibrium initial condition. Later, work of Buckmaster et al. \cite{buckmaster2019onthe,buckmaster2019onset}, Deng-Hani \cite{deng2019derivation}, and  Collot-Germain \cite{collot2019derivation,collot2020derivation} considered the non-equilibrium problem for the continuum NLS but only for time scales that fall short of \eqref{VanHove}. We also mention that attempts to derive the homogeneous 4-wave kinetic equation from a stochastic NLS equation in the continuum setting have been carried out  by Dymov et al. in \cite{dymov2019formal,dymov2019formal2,dymov2020zakharov,dymov2021large}.

In \cite{deng2021full} and \cite{staffilani2021wave}, Deng-Hani and Staffilani-Tran respectively provided rigorous derivations out of equilibrium and at the kinetic limit \eqref{VanHove} of the homogeneous 4-wave equation from the continuum NLS and the homogeneous 3-wave equation from a stochastic variant of the lattice Zakharov-Kuznetsov (ZK), which is a higher-dimensional generalization of the Korteweg-de Vries (KdV) equation. Let us remark that \cite{staffilani2021wave} was inspired by previous work of Faou  \cite{faou2020linearized} for the KP equation. A propagation of chaos result was later obtained by Deng-Hani in \cite{deng2021propagation}, following earlier work \cite{RS2022physD} on this question by the second and third authors, and further extensions of the result of \cite{deng2021full} have been announced in \cite{deng2022rigorous}. In \cite{ampatzoglou2021derivation}, a derivation of the inhomogeneous 4-wave kinetic equation from a quadratic NLS equation in a limit short of the kinetic time has been obtained.  An extension to a nonlinear random matrix model has been done in \cite{dubach2022derivation}. Motivated by \cite{staffilani2021wave}, the recent preprint \cite{ma2022almost} proposes a  derivation of the homogeneous 3-wave kinetic equation in a limit before the kinetic timescale starting from a continuum ZK model in dimension $d \geq 3$, where diffusion has been added and the energy ($L^2$-norm) is no longer conserved.

In this paper, we substantially generalize the approach of \cite{staffilani2021wave} to rigorously derive in the kinetic limit \eqref{VanHove} the {\it inhomogeneous} 3-wave kinetic equation in the lattice setting. Our starting point is the ZK equation in dimension $d\geq 2$, which has many physical applications to drift waves in fusion plasmas  and
Rossby waves in geophysical fluids  \cite[Section 6.2]{Nazarenko:2011:WT}, as well as ionic-sonic waves in a magnetized plasma \cite{lannes2013cauchy,zakharov1974three}:
\begin{equation}\label{KleinGordonNoise}
	\begin{aligned}
		\mathrm{d}\psi(x,t) \ & = \ -\Delta\partial_{x_1} \psi(x,t) \mathrm{d}t  \ + \ \lambda\partial_{x_1}\Big(\psi^2(x,t)\Big)\mathrm{d}t,\\
		\psi(x,0) \ & = \ \psi_0(x).
	\end{aligned}
\end{equation}
Here, $1>|\lambda|>0$ is  a real-valued parameter measuring the weak interactions of the nonlinear wave system. The equation \eqref{KleinGordonNoise} is defined on a hypercubic lattice $\Lambda$ (i.e., $x\in\Lambda$), that will be precisely defined in the next section. {\it We choose to work with the ZK  equation as it  has normally served as the first example for which a 3-wave kinetic equation is derived \cite{Nazarenko:2011:WT}. In the sequel, equation \eqref{KleinGordonNoise} will be made stochastic by introducing various types of noise, which generalize the noise used in \cite{staffilani2021wave}. The exact form of the noise will be given in the next section, and the motivation for considering a stochastic equation is explained below.} 

\medskip
Let us now give an informal statement of the main theorem. A proper statement is given in \cref{TheoremMain} in \cref{Sec:Main} after all the necessary notations have been introduced. In the statement below, we refer to Assumptions (A), (B), and (C) that will be explained in the next section.

\begin{theorem}[Informal version of main result]\label{TheoremMainRough}
Suppose that $d\ge 2$. After randomization of the initial value problem \eqref{KleinGordonNoise} on a hypercubic lattice by a very weak noise, the associated Wigner distribution can be asymptotically expressed via a solution of the inhomogeneous 3-wave  kinetic equation at the kinetic time \eqref{VanHove}, under Assumptions (A), (B), and (C) on the initial data.  
\end{theorem}

\bigskip
To the best of our knowledge, our work provides the first rigorous derivation of a nonlinear inhomogeneous wave kinetic equation at the kinetic limit. Our method of proof is based on techniques originally introduced by Erd\"{o}s-Yau \cite{erdHos2000linear} (see also \cite{erdHos2002linear,erdHos2008quantum,yau1998scaling}) for the kinetic description of the Schr\"odinger with a random potential, and later developments by Lukkarinen-Spohn \cite{LukkarinenSpohn:WNS:2011} for the NLS.  Several novelties beyond these works are needed even for the homogeneous case, a detailed discussion of which may be found in \cite{staffilani2021wave}.
One of the key issues unique to the lattice ZK equation is the ``ghost manifold,'' on which the dispersion relation vanishes and which supports not just 3-wave but arbitrary n-wave interactions. This destroys the structure of 3-wave interactions; consequently, one does not  expect the emergence of either the homogeneous or inhomogeneous 3-wave kinetic equation from the ZK equation without  noise. This is in strong contrast to the nonlinear Schr\"odinger equation. {\it The noise in the homogeneous setting and its generalization to the inhomogeneous setting need to meet several criteria, as explained below:}

\begin{itemize}
	\item[(i)]\label{comm(i)}Since the 3-wave kinetic equation conserves energy and momentum, it is then important that the noise does not inject energy ($L^2$-norm) into the wave system and  that {\it the conservation laws of the dispersive equation under consideration are preserved under introduction of noise.}

	\item[(ii)]\label{comm(ii)} The introduction of the noise, which vanishes in the limit $\lambda\to0$, should only ameliorate the issue of the ghost manifold and should not affect the structure of the Feynman diagrams. In both the homogeneous and inhomogeneous cases, our noise acts only on the phases and not the amplitudes of the waves. The noise identifies those waves that  accidentally fall into the ghost manifold and pushes them out. {\it Besides the special role of controlling the singular dispersion relation  on the ghost manifold, the noise must not influence most of the---in particular, the most important---Feynman diagrams, including the leading diagrams, and the rest of the proof should hold independently of the noise.} We would like to mention the work \cite{hannani2021stochastic},  which also constructs a noise (for a different equation) that preserves all the conservation laws, and therefore  Criterion (i) is satisfied. However, the noise of \cite{hannani2021stochastic} affects both the phases and amplitudes of the waves. As a result, applying the type of noise from \cite{hannani2021stochastic} to the lattice ZK equation will change the structure of the Feynman diagrams and {\it additional terms will be introduced into the final form of the derived kinetic equation, which we would like to avoid}. Nevertheless, the noise used in our current work is inspired by the construction in \cite{hannani2021stochastic}.
	
	\item[(iii)]\label{comm(iii)} The effect of the weak noise must not compete with the effect of the weak nonlinearity.   It has shown for the homogeneous setting \cite{staffilani2021wave} that the influence of the noise almost disappears in all {\it pairing graphs}, as well as in all  {\it leading (ladder) diagrams}. As  those  diagrams are where the nonlinearity affects the most,  there is no competition between the weak nonlinearity and the weak noise. We will show in this work that a suitable generalization of this noise to the inhomogeneous setting satisfies the same constraint: {\it the weak noise does not compete with the weak nonlinearity, and consequently, the addition of the noise does not lead to any modifications in the derived kinetic equation. }

\end{itemize}

As commented above---and a point we cannot stress enough---the role of the noise is simply to handle the issue of the ghost manifold. Thus, if we replace the ``homogeneous'' noise of \cite{staffilani2021wave} by our current ``inhomogeneous'' noise, the rest of the analysis applies and we can obtain the equivalent result of \cite{staffilani2021wave}. On the other hand, the ``homogeneous noise'' used in \cite{staffilani2021wave} is intentionally designed to control the ghost manifold for homogeneous dynamics and cannot be used to derive the inhomogeneous equation, creating the need for the weaker noises constructed in this paper. Finding a suitable such weaker noise is one of the main challenges we have to overcome in this paper. Let us also remark that if one suppresses the nonlinearity while keeping the weak noise in the equation, the Wigner distribution simply follows the free transport equation in the kinetic limit.   

\medskip

We would also like to emphasize a novel difficulty in extending the analysis from the homogeneous case  to our current inhomogeneous case. In the homogeneous case, the kinetic equation reflects the dynamics of the two-point correlation function $\langle a_{k}a_{k}^*\rangle$. As a result, the $l^1$ clustering assumption used in \cite{LukkarinenSpohn:WNS:2011} would be sufficient to bound the cumulant expansions. Using  the $l^1$ clustering  estimates means that we bound the correlations in the $l^\infty$ norm in $k$. In the inhomogeneous case, one needs to investigate the dynamics of the Wigner distribution, whose Fourier transform is of the form $\lambda^{-2d}\langle a_{k_1}a_{k_2}^*\rangle$ (see \eqref{Wigner2}), which has a growth of order $\mathcal {O}(\lambda^{-2d})$ in the $l^\infty$ norm.  As a result, each $n$-layer Duhamel expansion will give a growth of order $\mathcal {O}(\lambda^{-nd})$ and simple $l^1$ clustering  estimates would lead to a divergence as $\lambda\rightarrow0$ of all leading and non-leading graphs.  In order to overcome this difficulty, several new strategies, using the projections of different norms, are then introduced in the current paper. Those techniques are based on completely generalizing the $l^2$ projection framework and the Fourier averaging crossing estimates first introduced in  \cite{staffilani2021wave}, in which the moments are projected onto the $l^2$ space  and the new crossing estimates are  formulated in an averaged sense rather than pointwise.

\medskip

The above remarks conclude our high-level introduction to this paper. In the next section, we introduce the framework for our problem, the complete statement of the main theorem, and provide some comments on the main result.

\bigskip

\bigskip

\bigskip

{\bf Acknowledgments:} The authors would like to express their gratitude to  Stefano Olla, Jani Lukkarinen, and Herbert Spohn  for several useful remarks, guidance, and instructions on the topic. This work was partially done in Spring 2022 while the first and the last authors were in residence at MIT, and they thank the institution for its hospitality. 
\medskip

\medskip

\section{Settings and main result}\label{Sec:Main} 

\subsection{The stochastic ZK equation}

We adopt the setting of \cite{staffilani2021wave} and only recall below the main points.  The ZK  equation in dimension $d\geq 2$ takes the form
\begin{equation}\label{KleinGordon}
	\begin{aligned}
		\frac{\partial\psi}{\partial t}(x,t) \ & = \ -\Delta\partial_{x_1} \psi(x,t)   \ + \ \lambda\partial_{x_1}\Big(\psi^2(x,t)\Big),\\
		\psi(x,0) \ & = \ \psi_0(x),
	\end{aligned} \qquad \forall (x,t)\in[0,1]^d\times \mathbb{R}_+,
\end{equation}
where $1>|\lambda|>0$ is  a real constant.


{For $D \in \mathbb{N}$, let $\Lambda= \Lambda(D) \coloneqq \{-D,\ldots, D \}^d$ with periodic boundary conditions.}
The dynamics for the discretized equation   read
\begin{equation}
	\label{LatticeDynamicsZ}\begin{aligned}
		\mathrm{d}\psi(x,t) \ = &  \ \ \sum_{y\in  \Lambda} O_1(x-y)\psi(y,t)\mathrm{d}t \ + \ \lambda\sum_{y\in  \Lambda} O_2(x-y)  \psi^2(y,t)\mathrm{d}t, \\
		\psi(x,0) \ & = \ \psi_0(x),\end{aligned}
		\qquad  \forall (x,t)\in\Lambda\times \mathbb{R}_+,
\end{equation}
where $O_1(x-y)$ and $O_2(x-y)$ are finite difference operators that we will express below in Fourier space. We now introduce  the  Fourier transform: 
\begin{equation}
	\label{Def:Fourier}\hat  \psi(k)\coloneqq  \sum_{x\in\Lambda} \psi(x) e^{-2\pi {\bf i} k\cdot x}, \quad k\in \Lambda_* = \Lambda_*(D) \coloneqq \left\{-\frac{D}{2D+1},\cdots,0,\cdots,\frac{D}{2D+1}\right\}^{d},
\end{equation}
where the mesh size $h$ is defined as 
\begin{equation} h\coloneqq\frac{1}{2D+1}.\end{equation}
{
We denote the elements of $\Lambda_*$ as $k=(k^1,\dots,k^d)$. Moreover we define 
\begin{equation} \label{positiveDomain}
\Lambda_*^{\pm}\coloneqq\{k =(k_1,\dots,k^d)\in \Lambda_* |  \pm k^1>0 \}.
\end{equation}
}
We set
\begin{equation}
	\label{Epsilon}\epsilon=\lambda^2.
\end{equation}
We randomize the dynamics of \eqref{KleinGordon} by introducing a  multiplicative noise term (cf. \cite[Equation 18]{staffilani2021wave}). More precisely,
\begin{equation}
	\label{LatticeDynamicsFourierZ+a}\begin{aligned}
		\mathrm{d}\hat\psi(k,t) \ & = \ {\bf i} \omega(k)\hat\psi(k,t)\mathrm{d}t \ + \ {\bf i}\sqrt{2c_r}
		\sum_{l \in \Lambda_*^+}
		\mathfrak{g}(k,l)\hat\psi(k,t)\circ\mathrm{d}W_l \\
		&\ \ \  \ +  {\bf i}\lambda \bar\omega(k)\sum_{\substack{ k_1,k_2\in\Lambda_* \\ k=k_1+k_2 \mod{\Lambda_*}}}\hat\psi(k_1,t)\hat\psi(k_2,t),\\
		\hat\psi(k,0) \ & = \ \hat\psi_0(k) , \end{aligned}
		\qquad  \forall k \in \Lambda_*^+,
\end{equation}
and {for $k' \in \Lambda_*^-$, we set $\hat{\psi}(k',t) \coloneqq \hat{\psi}^*(-k',t)$. In the above expression, the conservation of momentum is expressed modulo the lattice $\Lambda_*$, that is 
	\begin{equation}
		\label{ModeZ}
		k=k_1+k_2 \mod{\Lambda_*} \Longleftrightarrow \vec{\mathbb{V}}\in \mathbb{Z}^d \mbox{ such that } k=k_1+k_2 + |\Lambda_*|\vec{\mathbb{V}}.
	\end{equation}
In the sequel, we will abuse notation by not explicitly writing $\mod{\Lambda_*}$.
 Equivalently, the dynamics for $k' \in \Lambda_*^-$ reads
 \begin{equation}
  \label{LatticeDynamicsFourierZ-}\begin{aligned}
		\mathrm{d}\hat\psi(k',t) \ & = \ {\bf i} \omega(k')\hat\psi(k',t)\mathrm{d}t \ - \ {\bf i}\sqrt{2c_r}
		\sum_{l \in \Lambda_*^+}
		\mathfrak{g}(-k',l)\hat\psi(k',t)\circ\mathrm{d}W_l \\
		&\ \ \  \ +  {\bf i}\lambda \bar\omega(k')\sum_{\substack{k_1,k_2\in\Lambda_* \\ k'=k_1+k_2}}\hat\psi(k_1,t)\hat\psi(k_2,t),\\
		\hat\psi(k',0) \ & = \ \hat\psi_0(k') , \end{aligned}
\end{equation}}
Above, $\{ W_\ell\}_{\ell\in \Lambda^*_+}$ is a sequence of independent real standard Wiener processes on a some filtered probability space $(\tilde\Omega,{\bf{F}},({\bf{F}})_{t\ge 0},{\bf{P}})$. The product $\hat{\psi}\circ \mathrm{d}W_l$ 
should be understood in the Stratonovich sense, and 
 the values of $\mathfrak{g}(k,l)$ will be specified later.  The  dispersion relation  
 takes the  discretized form
\begin{equation}
	\label{NearestNeighbordZ}
	 \omega(k) \coloneqq  \sin(2\pi h k^1)\Big[\sin^2(2\pi h k^1) + \cdots + \sin^2(2\pi h k^d)\Big], \qquad \forall k=(k^1,\cdots,k^d) \in \Lambda_*.
\end{equation}
Moreover, we  simply set
\begin{equation}
	\label{omegaoZ}
	\bar\omega(k) \ = \sin(2\pi h k^1).
\end{equation}
We  restrict the frequency domain to
\begin{equation}\begin{aligned}
&	\Lambda^* =\Lambda^*(D) \coloneqq \Lambda_*^+ \cup \Lambda_*^-\\
 =& \ \left\{-\frac{D}{2D+1},\cdots,-\frac{1}{2D+1},\frac{1}{2D+1},\cdots,\frac{D}{2D+1}\right\}\times
\left\{-\frac{D}{2D+1},\cdots,0,\cdots,\frac{D}{2D+1}\right\}^{d-1}, \end{aligned}
\end{equation}
and the spatial domain to
\begin{equation}
	\Lambda^+  \coloneqq \{1,\cdots,D\}\times \{-D,\cdots,0,\cdots,D\}^{d-1}. 
\end{equation}

For any number $x$ in $\mathbb{R}$, we define $\lceil	 x \rceil		$ to be the integer that satisfies\footnote{Note that this definition does not coincide with the usual definition of the ceiling function.}
l\begin{equation}\label{ceilfunc}
|\lceil	 x \rceil		|\le |x| < |\lceil	 x \rceil		|+1.
\end{equation}
{
Recall the definition of $\epsilon$ \eqref{Epsilon}, and let $\complement$ to be a fixed constant. We tile $\mathbb{R}^d$ with boxes with length size $\complement \epsilon$:
\begin{equation}\begin{aligned} \label{BOXes}
 	\mathbb{R} ^d\ =  \ & \bigcup_{(n_1,\ldots,n_d)\in\mathbb{Z}^d }\underbrace{
 	[\complement \epsilon n_1,\complement \epsilon(n_1+1))\times 
 	[\complement \epsilon n_2,\complement \epsilon(n_2+1))\times\cdots\times 
 	[\complement \epsilon n_d,\complement \epsilon(n_d+1))}_{\eqqcolon\mathcal{B}_{n_1,\cdots,n_d} }.\end{aligned}
 \end{equation}
Let 
\begin{equation} \label{indiceBOX}
\mathcal{J}_{\pm}:=\{ \underline{n}  \in \mathbb{Z}^d | \mathcal{B}_{
\underline{n}} \cap\Lambda_*^{\pm} \neq \O  \}. 
\end{equation} 
Hence we tile $\Lambda_*^+$ with above disjoint boxes:
\begin{equation}
\Lambda_*^+ \subset \bigcup_{\underline{n} \in \mathcal{J}_+} \mathcal{B}_{\underline{n}}.
\end{equation}
We then define $\mathscr{E}: \Lambda^+_* \times \Lambda^+_* \to \mathbb{R}$ by
\begin{equation}\label{Matrix}
	\mathscr{E} (k,k')\coloneqq 
	\begin{cases}
 	1 \quad 	\text{if} \:\exists \underline{n} \in \mathcal{J}_+, \text{such that} \: k,k' \in \mathcal{B}_{\underline{n}}, \\
 	0 \quad \text{Otherwise}.
	\end{cases}
\end{equation}
Notice that $$\mathscr{E}(k_1,k_2) =\prod_{i=1}^d 
\delta\left([\frac{k_1^i}{\complement \epsilon}]-[\frac{k_2^i}{\complement \epsilon}]\right),$$
where $\delta$ is the Kronecker delta, and $[.]$ is the usual floor function.
It is straightforward to observe that $\mathscr{E}$ can be represented by a symmetric, 
block diagonal matrix where at each block all the entries are equal to one. In fact, each block is essentially $[D/\complement \epsilon]^d \times [D/\complement \epsilon]^d $ except for the corners. Therefore, since all the blocks are positive semi-definite, $\mathscr{E}$
is positive semi-definite.   
 }



Therefore, there exists a unique
$\mathfrak{g}: \Lambda_*^+ \times \Lambda^+_* \to \mathbb{R}$ characterized by the relation
\begin{equation}
	\sum_{l \in \Lambda_*^+} \mathfrak{g}(k,l)
	\mathfrak{g}(k',l) = \mathscr{E}(k,k'), 
\end{equation}
where $\mathscr{E}$ is the matrix corresponding to the function $\mathscr{E}(k,k')$. We define $\tilde{\mathfrak{g}}, \, \tilde{\mathscr{E}}: 
\Lambda^* \times \Lambda^* \to \mathbb{R}$
 as follows: 
 \begin{equation} \label{MatrixonLambda}
 \begin{aligned}
 	&\tilde{\mathfrak{g}}(k,l) \coloneqq \mathfrak{g}(k,l), \qquad k,l \in \Lambda^+_*,\\ \quad
 	&\tilde{\mathfrak{g}}(k,l) \coloneqq -\mathfrak{g}(-k,l), \qquad k \in \Lambda^-_*, 
 	l \in \Lambda^+_*, \quad \\
 	&\tilde{\mathfrak{g}}(k,l) \coloneqq -\mathfrak{g}(k,-l), \qquad k \in \Lambda_*^+, 
 	l \in \Lambda^-_*,\\ \quad
 	&\tilde{\mathfrak{g}}(k,l)\coloneqq \mathfrak{g}(-k,-l), \qquad k,l \in \Lambda^-_*,
 	\end{aligned}
 \end{equation}
and 
 \begin{equation} \label{Matrix3}
 \tilde{\mathscr{E}}(k,k')\coloneqq \sum_{l \in \Lambda^*} \tilde{\mathfrak{g}}(k,l)
 \tilde{\mathfrak{g}}(k',l).
 \end{equation}
	Thanks to the above definitions, we may rewrite \eqref{LatticeDynamicsFourierZ-} in the following compact form:
	\begin{equation}
	\label{LatticeDynamicsFourierZ}\begin{aligned}
		\mathrm{d}\hat\psi(k,t) \ & = \ {\bf i} \omega(k)\hat\psi(k,t)\mathrm{d}t \ + \ {\bf i}\sqrt{2c_r}
		\sum_{l \in \Lambda_*}
		\tilde{\mathfrak{g}}(k,l)\hat\psi(k,t)\circ\mathrm{d}W_l \\
		&\ \ \  \ +  {\bf i}\lambda \bar\omega(k)\sum_{\substack{k_1,k_2\in\Lambda^* \\ k=k_1+k_2}}\hat\psi(k_1,t)\hat\psi(k_2,t),\\
		\hat\psi(k,0) \ & = \ \hat\psi_0(k) , \end{aligned}
		\qquad k\in\Lambda^*,
\end{equation}
where $W_{-l}(t)\coloneqq-W_{l}(t)$ for all $l \in \Lambda_*^+$.

\medskip
We now come to Assumption (A) referred to in \cref{TheoremMainRough}.

\medskip

	{\bf Assumption (A) - Conditions on $c_r$:}{\it
	\begin{equation}\label{ConditionCr1}
		c_r = \mathfrak{C}_r \lambda^{\theta_r},
	\end{equation}
	for some universal constants $\mathfrak{C}_r>0 $ and $1\ge \theta_r>0$. Notice that $\theta_r$ is a small constant. 
}


%
%
%
%
%
\medskip
We renormalize the Fourier coefficient $\hat{\psi}(k)$ by defining
$$a_k \coloneqq \frac{\hat{\psi}(k)}{ \sqrt{|\bar\omega(k)|}},$$
	$${a}(k,1,t) \coloneqq a^*_k(t),$$ $$ {a}(k,-1,t) \coloneqq a_k(t).
$$
Out of convenience, we also sometimes write ${a}(k,\sigma,t)$  as ${a}_t(k,\sigma)$ or ${a}_{k,\sigma}$. With this notation, we rewrite the system \eqref{LatticeDynamicsFourierZ} as
\begin{equation}
	\begin{aligned}\label{StartPointZ}
			\mathrm{d}{a}(k,\sigma,t)\ & =  \ -{\bf i}\sigma\omega(k){a}(k,\sigma,t) \mathrm{d}t\ - \ {\bf i}\sigma\lambda\sum_{k_1,k_2\in\Lambda^*}\delta( k-k_1-k_2)\times\\
		& \ \  \ \ \  \ \times \mathcal{M}(k,k_1,k_2){a}(k_1,\sigma,t) {a}(k_2,\sigma,t)\mathrm{d}t\\ & - \ {\bf i}\sqrt{2c_r} \sigma\sum_{l \in \Lambda^*}
		\tilde{\mathfrak{g}}(k,l){a}(k,\sigma,t)\circ\mathrm{d}W_l,\\
		{a}(k,-1,0) \ & = \ a_0(k), \ \ \ \ \forall (k,t)\in\Lambda_*\times \mathbb{R}_+.
	\end{aligned}
\end{equation}
where
\begin{equation}
	\label{KernelZ}
	\mathcal{M}(k,k_1,k_2) \coloneqq 2\mathrm{sign}(k^1) \mathfrak{W}(k,k_1,k_2),\ \ \    \mathfrak{W}(k,k_1,k_2) \coloneqq \sqrt{{|\bar\omega(k)\bar\omega(k_1)\bar\omega(k_2)|}}.
\end{equation}
Thanks to the dispersion relation \eqref{NearestNeighbordZ}, the fact that 
$W_k(t)=-W_{-k}(t)$, and the form of the evolution equation \eqref{LatticeDynamicsFourierZ}, we observe that 
\begin{equation} \label{paritya}
a_k^*(t) = a_{-k}(t), \qquad \forall t\geq 0. 
\end{equation}

To filter out the term $-{\bf i}\sigma\omega(k){a}(k,\sigma,t)$ coming from the linear dynamics, we pass to the profile
\begin{equation}\label{HatA}
	\alpha(k,\sigma,t) \coloneqq {a}(k,\sigma,t)e^{{\bf i}\sigma \omega(k)t},
\end{equation}
which satisfies the system
\begin{equation}
	\begin{aligned}\label{StartPoint}
		\mathrm{d}\alpha_t(k,\sigma) \ & =  - \ {\bf i}\sigma\lambda \sum_{k_1,k_2\in\Lambda^*}\delta(k-k_1-k_2)\times\\
		& \ \  \ \ \  \ \times \mathcal{M}(k,k_1,k_2)\alpha_t(k_1,\sigma) \alpha_t(k_2,\sigma)e^{{\bf i}t\sigma(-\omega(k_1)-\omega(k_2)+\omega(k))}\mathrm{d}t\\
		& \ \  \ \ \  \ -\ {\bf i}\sqrt{2c_r} \sum_{l \in \Lambda^*}
		\tilde{\mathfrak{g}}(k,l)\alpha(k,\sigma,t)\circ\mathrm{d}W_l,
		\\
		\alpha_0(k,-1) \ & = \ a_0(k), \ \ \ \ \forall (k,t)\in\Lambda^*\times \mathbb{R}_+.
	\end{aligned}
\end{equation}




We now set ${a}_k=B_{1,k} + {\bf i}B_{2,k}$ and use the notation $b_{1,k}$,  $b_{2,k}$, and $\mathbf{a}_k=b_{1,k} +{\bf i}b_{2,k}$ for the values of $B_{1,k}$, $B_{2,k}$, and $a_k$, respectively. We also set  $B_1 \coloneqq (B_{1,k})_{k\in\Lambda^*}$, $B_2\coloneqq (B_{2,k})_{k\in\Lambda^*}$. By considering real and imaginary parts in \eqref{StartPoint}, we obtain the system
 
\begin{equation}\label{dtBk}
	\begin{aligned}
		\mathrm{d}B_{1,k} \ = \ & -\omega(k) B_{2,k}\mathrm{d}t  -\sqrt{2c_r}
		\sum_{l \in \Lambda^*}
		\tilde{\mathfrak{g}}(k,l)B_{2,k}\circ\mathrm{d}W_l- \lambda\frac{\partial \mathcal{H}_2(B_1,B_2)}{\partial B_{2,k}}\mathrm{d}t,\\ 
		\mathrm{d}B_{2,k} \ = \ & \omega(k) B_{1,k}\mathrm{d}t 	+\sqrt{2c_r}
		\sum_{l \in \Lambda^*}
		\tilde{\mathfrak{g}}(k,l)B_{1,k}\circ \mathrm{d}W_l + \lambda\frac{\partial \mathcal{H}_2(B_1,B_2)}{\partial B_{1,k}}\mathrm{d}t,
	\end{aligned}
\end{equation}
which is a stochastic Hamiltonian system. We will not explain here the form of $\mathcal{H}_2(B_1,B_2)$, which is the nonlinear term in the total Hamiltonian for \eqref{dtBk}; the interested reader may consult \cite[Equation (32)]{staffilani2021wave}). Instead, we will show below the full equation for the density $\rho(t,b_1,b_2)$ after a change of variable.

The initial data are chosen to be  random variables $(B_{1,k}(0,\varpi),B_{2,k}(0,\varpi))$ defined on the same probability space $(\tilde\Omega,\mathbf{F},{\mathscr{P}})$ on which the Wiener processes $\{W_k\}_{k\in \Lambda_{*}^+}$ are defined. The law of the random vector $(B_{1,k}(0,\varpi),B_{2,k}(0,\varpi))_{k\in\Lambda^*}$ is given by a probability density function $\rho(0,b_1,b_2)$, in which $(b_1,b_2)$ $=$ $(b_{1,k},b_{2,k})_{k\in\Lambda^*}$. 
The initial density function $\rho(0,b_1,b_2)$ is assumed to be smooth, non-negative, and therefore it satisfies
\begin{equation}
	\label{DensityOneZ}
	\int_{\mathbb{R}^{2|\Lambda^*|}}\mathrm{d}b_1\mathrm{d}b_2\varrho(0,b_1,b_2) \ = \ 1,
\end{equation}
with $ \mathrm{d}b_1\mathrm{d}b_2 \coloneqq\prod_{k\in \Lambda^*}\mathrm{d}b_{1,k}\mathrm{d}b_{2,k}$.

The law of the random variables $(B_1(t),B_2(t))$ is given by the probability density $\varrho(t,b_1,b_2)$, which satisfies the \emph{Liouville equation}\footnote{Also referred to as the \emph{Fokker-Planck equation} in kinetic theory or \emph{forward Kolmogorov equation} in probability.}
\begin{equation}
	\label{Liouville1Z}
	\partial_t\varrho \ =  \ c_r \mathcal{R}[\varrho] - \{\mathcal{H},\varrho\},
\end{equation}
where $\mathcal{H}$ is the Hamiltonian of the system \eqref{dtBk} and $\mathcal{R}$ is the generator of the noise in \eqref{dtBk}, the form of which will be specified momentarily. We now write in a more explicit form
\begin{equation}
	\label{LiouvilleZ}
	\begin{aligned}
		\partial_t\varrho \ =  \ c_r \mathcal{R}[\varrho]  + \sum_{k\in\Lambda^*}\omega(k)
		\left(b_{2,k}{\partial_{ b_{1,k}}}-b_{1,k}{\partial_{b_{2,k}}}\right)\varrho -\lambda\{\mathcal{H}_2,\varrho\}.\end{aligned}
\end{equation}
Changing to polar coordinates, 
\begin{equation}\label{ChangofVariable}
	b_{1,k}+{\bf i}b_{2,k}\ =\ \mathbf{a}_k \ = \ \sqrt{2c_{1,k}}e^{{\bf i}c_{2,k}},
\end{equation} with $c_{1,k}\in \mathbb{R}_+$ and $c_{2,k}\in [-\pi,\pi]$, we  transform \eqref{LiouvilleZ} to 
\begin{equation}
	\begin{aligned}\label{FokkerPlanckZ}
		\partial_t\varrho
		=\ &-\sum_{k\in\Lambda^*}\omega(k)\partial_{c_{2,k}}\varrho + c_r\mathbf{R}[\varrho]\
		+\ \sum_{k\in\Lambda^*}\lambda \mathfrak{H}_1(k)\partial_{c_{1,k}}\varrho+ \sum_{k\in\Lambda^*}\lambda\mathfrak{H}_2(k)\partial_{c_{2,k}}\varrho.
	\end{aligned}
\end{equation}
{Notice that thanks to \eqref{paritya}, we have $c_{2,k}=-c_{2,-k}$ and 
$\partial_{c_{2,k}}=-\partial_{c_{2,-k}}$. Therefore, recalling the definition \eqref{positiveDomain} of $\Lambda_*^+$, we specify $\mathbf{R}[\varrho]$, $\mathfrak{H}_1$, $\mathfrak{H}_2$ under the new variables as follows:
\begin{equation} \label{NOISEgenerator}
\begin{aligned}
\mathbf{R}[\varrho] &= 2 \sum_{k,k' \in \Lambda_*^+ } 
	\mathscr{E}(k,k')\frac{\partial^2}{\partial_{c_{2,k}} \partial_{c_{2,k'}}}\varrho\\ & =\ \frac12	\sum_{k,k' \in \Lambda_* } 
	\tilde{\mathscr{E}}(k,k')\frac{\partial^2}{\partial_{c_{2,k}} \partial_{c_{2,k'}}}\varrho,
\end{aligned}
\end{equation}}
and  
\begin{multline}
\mathfrak{H}_1({k}) \coloneqq \sum_{k_1,k_2\in\Lambda^*}\mathcal{M}(k,k_1,k_2)\sqrt{2c_{1,k_1}c_{1,k_2}c_{1,k}}\Big[\delta(k-k_1-k_2){\sin(c_{2,k_1}+c_{2,k_2}-c_{2,k})} \\
+\ 2\delta(k+k_1-k_2){\sin(-c_{2,k_1}+c_{2,k_2}-c_{2,k})}\Big],
\end{multline}
and
\begin{multline}
\mathfrak{H}_2({k})\coloneqq-\sum_{k_1,k_2\in\Lambda^*}\mathcal{M}(k,k_1,k_2)\sqrt{c_{1,k_1}c_{1,k_2}}\Big[\delta(k-k_1-k_2){\sin(c_{2,k_1}+c_{2,k_2})} \\
 + 2\delta(k+k_1-k_2){\sin(-c_{2,k_1}+c_{2,k_2})}\Big]\frac{\sin(c_{2,k})}{\sqrt{2c_{1,k}}} \\
-\sum_{k_1,k_2\in\Lambda^*}\mathcal{M}(k,k_1,k_2)\delta(k-k_1-k_2)\sqrt{c_{1,k_1}c_{1,k_2}}\Big[\delta(k-k_1-k_2){\cos(c_{2,k_1}+c_{2,k_2})} \\
 + 2\delta(k+k_1-k_2){\cos(-c_{2,k_1}+c_{2,k_2})}\Big]\frac{\cos(c_{2,k})}{\sqrt{2c_{1,k}}}.
\end{multline}
By defining the new Hamiltonian,
\begin{multline}
\mathfrak{H}(k) \coloneqq \lambda\int_{\Lambda^*}dk_1\int_{\Lambda^*}dk_2\mathcal{M}(k,k_1,k_2)\sqrt{2c_{1,k_1}c_{1,k_2}c_{1,k}}\Big[\delta(k-k_1-k_2)\cos(c_{2,k}-c_{2,k_1}-c_{2,k_2}) \\
+2\delta(k+k_1-k_2)\cos(c_{2,k}+c_{2,k_1}-c_{2,k_2})\Big] + \omega_kc_{1,k},
\end{multline}
we can rewrite the equation \eqref{FokkerPlanckZ} as
\begin{equation}
\partial_t\varrho + \sum_{k\in\Lambda^*} \left[\left[\mathfrak{H}(k),\varrho\right]\right]_k  - c_r\mathbf{R}[\varrho] =0,
\end{equation}
with Poisson bracket
\begin{equation}
\left[\left[\mathfrak{H}(k),\varrho\right]\right]_k \coloneqq \partial_{c_{2,k}}\mathfrak{H}(k)\partial_{c_{1,k}}\varrho - \partial_{c_{1,k}}\mathfrak{H}(k)\partial_{c_{2,k}}\varrho.
\end{equation}

Since the drift in \eqref{StartPointZ} is smooth, under our assumptions on $\rho(0)$ (see also \cref{averages} below), equation \eqref{LiouvilleZ}---and by implication equation \eqref{FokkerPlanckZ}---has a global classical solution. The well-posedness of the Liouville equation, which has the structure of a transport-diffusion equation, is addressed in more detail in forthcoming work of the authors \cite{RosenzweigStaffilaniTran}. 

We now define the standard Banach spaces that we use below
\begin{equation}\label{Def:Norm1}L^{p}(\mathbb{T}^d)\coloneqq \left\{F(k): \mathbb{T}^d\to \mathbb{R}\ \Big| \ \|F\|_{L^p}=\left[\int_{\mathbb{T}^d}\mathrm{d}k |F(k)|^p\right]^{\frac{1}{p}}<\infty\right\}, \ \ \ \mbox{ for } p\in[1,\infty),\end{equation}

\begin{equation}\label{Def:Norm2}L^{\infty}(\mathbb{T}^d)\coloneqq \left\{F(k): \mathbb{T}^d\to \mathbb{R}\ \Big| \ \|F\|_{L^\infty}=\mathrm{ess \ sup}_{k\in\mathbb{T}^d} |F(k)|<\infty\right\},\end{equation}
\begin{equation}\label{Def:Norm3} l^{p}(\mathbb{Z}^d) \coloneqq \left\{F(x): \mathbb{Z}^d\to \mathbb{R}\ \Big| \ \|F\|_{l^p}=\left[\sum_{x\in\mathbb{Z}^d} |F(x)|^p\right]^{\frac{1}{p}}<\infty\right\}, \ \ \ \mbox{ for } p\in[1,\infty),\end{equation}
and
\begin{equation}\label{Def:Norm4}l^{\infty}(\mathbb{Z}^d)\coloneqq \left\{F(x): \mathbb{Z}^d\to \mathbb{R}\ \Big| \ \|F\|_{l^\infty}=\sup_{x\in\mathbb{Z}^d} |F(x)|<\infty\right\}.\end{equation}
We let
\begin{equation}
	\label{Shorthand1a}
	\int_{\Lambda^*}\mathrm{d}k \coloneqq  \frac{1}{|\Lambda^*|}\sum_{k\in\Lambda^*}
\end{equation}
to denote the Lebesgue integral with respect to the uniform measure on $\Lambda^*$. For a function $F(k,k'):\mathbb{T}^d\times \mathbb{T}^d\to \mathbb{R}$, we can make a change of coordinates, $$F(k,k')=G(k'',k''')$$ with $k''=\frac{k-k'}{2}$ and  $k'''=\frac{k+k'}{2}$. We can then define the mixed norm Lebesgue space, for all $s\ge 1$
\begin{equation}\begin{aligned}\label{Def:Norm5}L^{\infty,s}(\mathbb{T}^d\times \mathbb{T}^d)\ := &\ \left\{F(k,k')=G(k'',k'''): \mathbb{T}^d\times \mathbb{T}^d\to \mathbb{R}, \ \Big|\right. \\
		& \ \ \left.\|F\|_{L^{\infty,s}}=\int_{\mathbb{T}^d}\mathrm{d}k'' \Big[\int_{\mathbb{T}^d}\mathrm{d}k''' |G(k'',k''')|^s\Big]^{\frac{1}{s}}<\infty\right\}.\end{aligned}\end{equation}
If $(a,b)$ is a bounded interval, then we define the standard Sobolev space  $H_{0}^{1}(a,b)$ consisting of continuous functions on $[a,b]$ of the form

$$f(x)=\int _{a}^{x}\mathrm {d} tf'(t)\,,\qquad x\in [a,b]$$
where the generalized derivative $f'\in L^{2}(a,b)$ and has $ 0$ integral, so that  $f(b)=f(a)=0.$

\subsubsection{The initial conditions and the Wigner distribution}

We now specify the class of initial conditions that we are interested in. Before proceeding further, we  need the following definitions.
{
\begin{definition}\label{MatrixNoise}
	Define  $\diamond(k)$ to be a one-to-one mapping from  $\Lambda^*$  to $\{1,\cdots,|\Lambda^*|\}\}$.  Define $\big[\mathbb{M}'\big]$ to be a $|\Lambda^+_*|\times |\Lambda^+_*|$ complex  Hermitian matrix whose components are written as $$\mathbb{M}'_{(\diamond(k),\diamond(k_1))}={B}(k,k_1)$$ for any $k,k_1 \in \Lambda^+_*$. We assume that $B$ has the following symmetric structure
\begin{equation}\label{Bkk1def}
{B}(k,k_1) \ = \ {B}(k_1,k) \ = \ \epsilon^{-d}\bar{B}(k_1/\epsilon,k/\epsilon)\ = \ \epsilon^{-d}\bar{B}(k/\epsilon,k_1/\epsilon), \ \ \ \ \ \forall k,k_1 \in \Lambda^+_*,
\end{equation}
in which $\bar{B}$ is some given function $\bar{B}:\mathbb{R} ^d\times \mathbb{R} ^{d}\to\mathbb{R}$ and $\bar{B}\in L^1(\mathbb{R} ^d\times \mathbb{R} ^{d})\cap  L^\infty(\mathbb{R} ^d\times \mathbb{R} ^{d})$. 
{
Recall the tiling in \eqref{BOXes}.
We assume further that $\bar{B}$ is given such that for any 
$\underline{n},\underline{m} \in \mathcal{J}_+$ with 
$\underline{n} \neq \underline{m}$,  
\begin{equation}\label{MatrixNoise:1}
 k\in \mathcal{B}_{\underline{n}}, k_1 \in \mathcal{B}_{\underline{m}} \Longrightarrow  B(k,k_1)=0.
\end{equation}}
\end{definition}}
Since $\big[\mathbb{M}'\big]$ is complex Hermitian, the spectral theorem implies that
\begin{equation}\label{MatrixNoise1}
	\big[\mathbb{M}'\big] \ = \ \big[\mathbb{M}_0\big]\big[\mathbb{M}_1 \big]\big[\mathbb{M}_0^{-1}\big],
\end{equation}
where $\big[\mathbb{M}_0\big]$ is a unitary matrix and $\big[\mathbb{M}_1\big]$ is a diagonal matrix whose nonzero entries are the real, repeated if necessary, eigenvalues $\iota_1,\cdots,\iota_{|\Lambda^*_+|}$ of $\big[\mathbb{M}'\big] $ with associated (necessarily orthonormal) eigenvectors $\mathbf{q}_1,\ldots,\mathbf{q}_{|\Lambda_+^*|}$.
 Note that the matrix $\big[\mathbb{M}_0\big]$ has the orthonormal columns created by $\mathbf{q}_1,\cdots,\mathbf{q}_{|\Lambda^*_+|}$.   We now define $\daleth(k)$ as
\begin{equation}\label{omega1}
\epsilon^{-d}\daleth(k) \coloneqq \iota_{\diamond(k)}.
\end{equation} 
The factor $\epsilon^{-d}$ is due to the scaling in \eqref{Bkk1def} and is consistent with the scaling of the Wigner transform  \eqref{Wigner1} below. {Initially, we let the system  be in a state described by a probability density $\rho(0)$, which will be specified below. Then at time $t$, the system is 	
described by the solution $\varrho(t)$  of \eqref{FokkerPlanckZ} with initial datum $\varrho(0)=\rho(0)$. Given the fact that $a_k=a_{-k}^*$ \eqref{paritya}, we define $\varrho(0), \varrho(t)$
	as probability densities over $\mathbb{R}^{2|\Lambda_*^+|}$, and we compute the average 
	of observable over $\mathbb{R}^{2|\Lambda^*|}$ by proper symmetrization. }
{
\begin{definition}[Averages and Circular Symmetric Random Variable Initial Condition] 
\label{averages}
	For any observable $F_+:\mathbb{R}^{2|\Lambda_*^+|}\to \mathbb{C}$,  and the random 
	variables $B_1=(B_{1,k})_{k\in\Lambda_*^+}$, $B_2=(B_{2,k})_{k\in\Lambda_*^+}$  on the probability space $(\tilde\Omega,\mathbf{F},{\mathscr{P}})$, we  define the average
	\begin{equation}\label{Averaget+}
		\langle F_+(B_1,B_2)\rangle^+ \ = \  \langle F_+(B_1,B_2)\rangle_{t}^+ = \Big\langle{F_+(B_1,B_2)}\Big\rangle_{\mathscr{P}}^+ \coloneqq \int_{\mathbb{R}^{2|\Lambda^
		+_*|}}
		\mathrm{d}b_1\mathrm{d}b_2 F_+(b_1,b_2)\varrho(t,b_1,b_2),
	\end{equation}	
  where $\langle\rangle_{{\mathscr{P}}}^+$ denotes the expectation with respect to $
{\mathscr{P}}$. We choose the initial probability density $\varrho(0)$ to be Gaussian with covariance matrix $\big[\mathbb{M}'\big]$, that is
\begin{equation}
	\label{InitialDensity}
	\varrho(0,b_1,b_2) \coloneqq \frac{1}{\pi^{|\Lambda_*^+|}\mathrm{det}\big[\mathbb{M}'\big]}\exp
	\Big(-\mathbf{a}^*\big[\mathbb{M}'\big]^{-1}\mathbf{a}\Big),
\end{equation}
where $\mathbf{a}$ is the vector whose   $\diamond(k)$-th  component is $\mathbf{a}_k$ (recall 
\eqref{ChangofVariable}),   $\mathbf{a}^*$ is the complex conjugate of $\mathbf{a}$, and $ 
\big[\mathbb{M}'\big]^{-1}$ is the inverse of $\big[\mathbb{M}'\big]$. Using the diagonalization \eqref{MatrixNoise1}, we can also write
\begin{equation}
	\label{InitialDensity2}
	\varrho(0,b_1,b_2) \ = \ \frac{1}{\pi^{|\Lambda^+_*|}}\prod_{i=1}^{|\Lambda^+_*|}\frac{1}{\iota_i}\exp\left(-\frac{|\mathbf{q}_i^*\mathbf{a}|^2}{\iota_i}\right).
\end{equation}
Thanks to the fact that $a_k=a_{-k}^*$, we extend the above definition to 
$\Lambda^*$ as follows:  for $B_1=(B_{1,k})_{k \in \Lambda^*}$, 
$B_2 = (B_{2,k})_{k \in \Lambda^*} $ denote
\begin{equation*}
 B_1^{\pm}\coloneqq(B_{1,k})_{k \in 
\Lambda^{\pm}_*} \quad \text{and} \quad B_2^{\pm} \coloneqq (B_{2,k})_{k \in 
\Lambda^{\pm}_*}.
\end{equation*}
Then for any observable $F: \mathbb{R}^{2 \Lambda^*} \to \mathbb{C}$,
 notice that $$F(B_1,B_2)=F(B_1^+,B_2^+,B_1^-,B_2^-),$$ and 
 we denote
\begin{equation}
		\label{Averaget}\begin{aligned}
	\langle F(B_1,B_2)\rangle_{t} \coloneqq
		\int_{\mathbb{R}^{2|\Lambda^
		+_*|}}
		\mathrm{d}b_1^+\mathrm{d}b_2^+ F(b_1^+,b_2^+,b_1^+,-b_2^+)\varrho(t,b_1,b_2).\end{aligned}
	\end{equation}
Thanks to the fact that  $\mathbf{a}^+$ is a circular, symmetric random vector, whose pseudo-covariance matrix is simply the zero matrix, as well as the above
definition, we have: 
\begin{equation}
	\label{InitialDensity1}
	\begin{aligned}
& \langle a_k(0) \rangle \ = \ 0, \: k \in \Lambda^*, \\
&\langle a_k(0) a^*_{k'}(0)\rangle \ = \ \big[\mathbb{M}\big]_{(\diamond(k),\diamond(k'))} = B(k,k'), \ \ \ \langle a_k(0) a_{k'}(0)\rangle \ = \ 0, \ \ \  \: k,k' \in \Lambda_*^+,\\
&\langle a_k(0) a^*_{k'}(0)\rangle \ = \ \big[\mathbb{M}\big]_{(\diamond(-k'),\diamond(-k))} = B(-k',-k), \ \ \ \langle a_k(0) a_{k'}(0)\rangle \ = \ 0, \ \ \  \: k,k' \in \Lambda_*^-,\\
&\langle a_k(0) a^*_{k'}(0)\rangle \ = 0, \ \ \ \langle a_k(0) a_{k'}(0)\rangle \ = \ \ \big[\mathbb{M}\big]_{(\diamond(k),\diamond(-k'))} = B(k,-k'), \ \ \  \: k \in \Lambda_*^+, k' \in \Lambda_*^-.
\end{aligned}
\end{equation}

\end{definition}}
Though $\varrho(t)$ is a probability density over $\mathbb{R}^{|\Lambda_*^+|}$, due to the above symmetrization, we may regard it as probability density over $\mathbb{R}^{|\Lambda_*|}$.

Next, we will define the Wigner distribution $\mathcal W_{\epsilon}$ (see \cite{bal2002self,basile2016thermal,basile2010energy,
bernardin2019hydrodynamic,hernandez2022quantum,hernandez2022rapidly,komorowski2006diffusion,ryzhik1996transport}) via its action on a test function $G\in\mathcal{S}(\mathbb{R}^d\times\mathbb{T}^d)$ 
\begin{equation}
	\label{Wigner1}
	\langle G,\mathcal W_{\epsilon}(t)\rangle \coloneqq \Big(\frac{\epsilon}{2}\Big)^d\sum_{z,z'\in\Lambda^*}\left\langle \tilde a(z,1)\tilde a(z',-1)\right\rangle_t\tilde{G}\Big(\epsilon\frac{z+z'}{2},z-z'\Big),
\end{equation}
where  
\begin{equation}\begin{aligned}
	\tilde a(z,\pm1)\coloneqq & \int_{\Lambda^*}\mathrm{d}k {a}(k,\pm1) e^{-2\pi {\bf i} z\cdot k},\ \ \ \\ \tilde{G}(z,z')\coloneqq & \int_{\mathbb{T}^d}\mathrm{d}k e^{-2\pi {\bf i} z'\cdot k}G (z,k), \ \ \ \  
	\forall (z,z')\in\mathbb{R}^d\times\mathbb{Z}^d.\end{aligned}
\end{equation}
The Fourier transform of the Wigner distribution is now computed as
\begin{equation}
	\label{Wigner2}
	\hat{\mathcal W}_{\epsilon}(\mho,k,t) \ = \ \Big(\frac{\epsilon}{2}\Big)^d\left\langle{a}\Big(k-\frac{\epsilon\mho}{2},-1)\Big){a}\Big(k+\frac{\epsilon\mho}{2},1)\Big)\right\rangle_t, \ \ \ (\mho,k,t)\in \Lambda^*_{2/\epsilon}\times\Lambda^*\times\mathbb{R}_+,
\end{equation}
where 
\begin{equation}
	\label{Wigner3}
	\Lambda^*_{2/\epsilon} \coloneqq  \left\{-\frac{2D}{\epsilon(2D+1)},\cdots,0,\cdots,\frac{2D}{\epsilon(2D+1)}\right\}^d.
\end{equation}
We then have by Plancherel's theorem that
\begin{equation}
	\label{Wigner4}
	\langle G,\mathcal W_{\epsilon}(t)\rangle \ = \ \int_{\mathbb{T}^d\times\mathbb{R}^d}\mathrm{d}\mho\mathrm{d}k\hat{\mathcal W}_{\epsilon}(\mho,k,t)\hat{G}^*(\mho,k),
\end{equation}
where
\begin{equation}
	 \hat{G}(\mho,k) \coloneqq \int_{\mathbb{R}^d}\mathrm{d}z e^{-2\pi {\bf i} z\cdot \mho}G (z,k), \ \ \ \  
	\forall (\mho,k)\in\mathbb{R}^d\times\mathbb{T}^d,
\end{equation}
and $\mathcal W_{\epsilon}(t)=\mathcal W_{\epsilon}(\Omega,k,t)$ with $\Omega\in (\epsilon\mathbb{Z}/2)^d$.
We now denote by $\mathscr A$ the completion of $\mathcal{S}(\mathbb{R}^d\times\mathbb{T}^d)$ with respect to the norm
\begin{equation}
	\label{Wigner5}
	\| G\|_{\mathscr A} \ = \ \int_{\mathbb{R}^d}\mathrm{d}\mho \sup_{k\in\mathbb{T}^d}|\hat{G}(\mho,k)|,
\end{equation}
and let $\mathscr A'$ denote its dual.

\medskip
We now come to Assumption (B) in the statement of \cref{TheoremMainRough}, which is motivated by similar assumptions in \cite{komorowski2020high,Spohn:TPB:2006}.

\medskip

	{\bf Assumption (B)}
		{\it We assume that the initial Wigner distribution $\hat{\mathcal W}_{\epsilon}(\mho,k,0)$    converges weakly in $\mathscr A'$ to a  function ${\bf W}_0 \in L^1(\mathbb{R}^d\times\mathbb{T}^d)\cap C(\mathbb{R}^d\times\mathbb{T}^d)$ in the limit of $h\to 0$ and $\epsilon\to 0$. 
}

\subsubsection{Properties of the matrices \eqref{Matrix}}
We record some properties of the matrices from \eqref{Matrix} used in the construction of the noise.
 
 
First, we observe some elementary differential identities for multinomials of the amplitudes $\mathbf{a}_k$. For $n,m \in \mathbb{N}$, elementary calculus yields
 \begin{equation} \label{Rak}
 	\begin{aligned}
 	&	\partial_{c_{2,k}} (\mathbf{a}_k)^n\ =\ {\bf i} n  \mathbf{a}_k^n,\ \ \ \ \ \ \ \ \ \ 
 		\quad \partial_{c_{2,k}} (\mathbf{a}_k^*)^m = -{\bf i} m 
 		(\mathbf{a}_k^*)^m, \\
 	 &	 \partial_{c_{2,k}}\left((\mathbf{a}_k)^n(\mathbf{a}_k^*)^m\right) \ =\ {\bf i} (n-m)(\mathbf{a}_k)^n(\mathbf{a}_k^*)^m, \\
 	&	\partial^2_{c_{2,k_1} c_{2,k_2}}\left((\mathbf{a}_{k_1})^{n_{k_1}} (\mathbf{a}_{k_1}^*)^{m_{k_1}}(\mathbf{a}_{k_2})^{n_{k_2}}(\mathbf{a}_{k_2}^*)^{n_{k_2}}\right) \\
 		& \ =\  
 		-(n_{k_1}-m_{k_1})(n_{k_2}-m_{k_2}) (\mathbf{a}_{k_1})^n_{k_1} 
 		(\mathbf{a}_{k_1}^*)^{m_{k_1}}
 		(\mathbf{a}_{k_2})^{n_{k_2}}(\mathbf{a}_{k_2}^*)^{n_{k_2}}.
 	\end{aligned}
 \end{equation}
 
{  
For multi-indices $\mathcal{V}_1, \mathcal{V}_2 \in \mathbb{N}^{|\Lambda^*|}$, we introduce the multinomial notation
\begin{equation} \label{defaV}
(\mathbf{a})^{\mathcal{V}_1} \coloneqq \prod_{k \in \Lambda_*} (\mathbf{a}_k)^{\mathcal{V}_{1,
\diamond(k)}}, \quad (\mathbf{a}^*)^{\mathcal{V}_2} \coloneqq \prod_{k \in \Lambda_*} (\mathbf{a}^*_k)^{\mathcal{V}_{2,
\diamond(k)}}.
\end{equation}
For $\mathcal{V} \in \mathbb{N}^{|\Lambda^*|}$ and $k \in \Lambda^*$, we denote $\mathcal{V}_k \coloneqq \mathcal{V}_{\diamond{(k)}}$ (note that in the right-hand side $k$ is a vector, while in the left-hand side $\diamond(k)$ is a positive integer; so there is no ambiguity). Then thanks to \eqref{Rak}, the definition \eqref{NOISEgenerator} of $\mathbf{R}$ , and the fact that $\mathbf{a}_k^*=\mathbf{a}_{-k}$, we have
\begin{multline}\label{Noiseaction0}
\mathbf{R}\left((\mathbf{a})^{\mathcal{V}_1} (\mathbf{a}^*)^{\mathcal{V}_2}\right) = - 2\sum_{k_1,k_2 \in \Lambda_*^+}\Big( \mathscr{E}(k_1,k_2)
 	(\mathcal{V}_{1,k_1}-\mathcal{V}_{1,-k_1}+\mathcal{V}_{2,-k_1}-\mathcal{V}_{2,k_1}) \\
 	\times 	(\mathcal{V}_{1,k_2}-\mathcal{V}_{1,-k_2}+\mathcal{V}_{2,-k_2}-\mathcal{V}_{2,k_2})
 	(\mathbf{a})^{\mathcal{V}_1} (\mathbf{a}^*)^{\mathcal{V}_2}\Big).
\end{multline}
 }

{For $\mathcal{V} \in \mathbb{R}^{|\Lambda^*|}$, we identify $\mathcal{V}$
by $\{\tilde{\mathcal{V}}_k \}_{k \in \Lambda^*}$, where $\tilde{\mathcal{V}}_k=
\tilde{\mathcal{V}}_{\diamond(k)}$. This means for $\mathcal{V} \in \mathbb{R}^{|\Lambda^*|}$ and $k \in \Lambda^*$, $\mathcal{V}_k$ should be understood in the above sense.} 
{Recall the definition of $\mathcal{J}_{\pm}$ \eqref{indiceBOX}.  For $\mathcal{V} \in \mathbb{N}^{|\Lambda^*|}$, $\underline{n} \in \mathcal{J}_{+}$ denote
 \begin{equation}
 	|\mathcal{V}|\coloneqq \sum_{k \in \Lambda^*} \mathcal{V}_k, \qquad 
 	\mathcal{V}^{\pm}\coloneqq\sum_{k \in \Lambda_*^{\pm}} \mathcal{V}_k, 
 	\qquad \mathcal{V}^{\pm}_{\underline{n}} = \sum_{k \in \mathcal{B}_{\underline{n}}}
 	\mathcal{V}_{\pm k}.
 \end{equation} }
Then we  have the following lemma 

{\begin{lemma}
	\label{Lemma:NoiseAction}
The  action of the noise $\mathbf{R}$ can be written as
	 \begin{equation} \label{genralactionR2}
	 	\begin{aligned}
	 		\mathbf{R}\left((\mathbf{a})^{\mathcal{V}_1} (\mathbf{a}^*)^{\mathcal{V}_2}\right) \ =\  C_{\mathbf{R}}(\mathcal{V}_1,
	 		\mathcal{V}_2)(\mathbf{a})^{\mathcal{V}_1} (\mathbf{a}^*)^{\mathcal{V}_2},
	 	\end{aligned}
	 \end{equation}
	 with real-valued coefficient $C_{\mathbf{R}}(\mathcal{V}_1,\mathcal{V}_2)$ given by
\begin{multline} \label{genralactionR2:a}
	 		C_{\mathbf{R}}(\mathcal{V}_1,
	 		\mathcal{V}_2)  \coloneqq 
	 		\\ -2\left(\sum_{k_1,k_2\in\Lambda_*^+} 	
	 		\mathscr{E}(k_1,k_2)
	 		(\mathcal{V}_{1,k_1}-\mathcal{V}_{1,-k_1}+\mathcal{V}_{2,-k_1}-\mathcal{V}
	 		_{2,k_1})(\mathcal{V}_{1,k_2}-\mathcal{V}_{1,-k_2}+\mathcal{V}_{2,-k_2}-
	 		\mathcal{V}
	 		_{2,k_2})\right)
	 	\\
	 	=-2\Bigg(\sum_{\underline{n},\underline{m} 
	 	\in \mathcal{J}_+}\sum_{k_1\in \mathcal{B}_{\underline{n}}, 
	 	k_2 \in \mathcal{B}_{\underline{m}} } \mathscr{E}(k_1,k_2)
	 		(\mathcal{V}_{1,k_1}-\mathcal{V}_{1,-k_1}+\mathcal{V}_{2,-k_1}-\mathcal{V}
	 		_{2,k_1})(\mathcal{V}_{1,k_2}-\mathcal{V}_{1,-k_2}+\mathcal{V}_{2,-k_2}-
	 		\mathcal{V}
	 		_{2,k_2})\Bigg) \\
	 		=-2\sum_{\underline{n} \in \mathcal{J}_+}\sum_{k_1\in \mathcal{B}_{\underline{n}}, 
	 	k_2 \in \mathcal{B}_{\underline{n}}} 
	 		(\mathcal{V}_{1,k_1}-\mathcal{V}_{1,-k_1}+\mathcal{V}_{2,-k_1}-\mathcal{V}
	 		_{2,k_1})(\mathcal{V}_{1,k_2}-\mathcal{V}_{1,-k_2}+\mathcal{V}_{2,-k_2}-
	 		\mathcal{V}
	 		_{2,k_2})\\
	 		= -2 \sum_{\underline{n} \in \mathcal{J}_+} 
	 		\bigg(\sum_{k \in \mathcal{B}_n} \left(\mathcal{V}_{1,k}-\mathcal{V}_{1,-k}+
	 		\mathcal{V}_{2,-k}-
	 		\mathcal{V}
	 		_{2,k} \right) \bigg)^2=-2\sum_{
	 		\underline{n}\in \mathcal{J}_+} \left( 
	 		\mathcal{V}_{1,\underline{n}}^+ - \mathcal{V}_{1,\underline{n}}^-
	 		-\mathcal{V}_{2,\underline{n}}^++ \mathcal{V}_{2,\underline{n}}^-\right)^2,
	 \end{multline}
	Then we have: 
	\begin{equation} \label{CR=0}
\sum_{k \in \mathcal{B}_{\underline{n}}} 
(\mathcal{V}_{1,k}-\mathcal{V}_{1,-k}+\mathcal{V}_{2,-k}-\mathcal{V}_{2,k}) = 0, \qquad \forall \underline{n} \in \mathcal{J}_+, \quad\Longrightarrow \quad C_{\mathbf{R}}=0.
\end{equation}
Otherwise, if there exists $\underline{n} \in \mathcal{J}_+$ such that
\begin{equation} \label{CR>1}
\sum_{k \in \mathcal{B}_{\underline{n}}} (\mathcal{V}_{1,k}-\mathcal{V}_{1,-k}+\mathcal{V}_{2,-k}-\mathcal{V}_{2,k}) \ \ne \ 0,
\end{equation}
then  $C_{\mathbf{R}}\le -1$.
\end{lemma}
\begin{proof}
Recall the definition of the tiling \eqref{BOXes}, and set of their 
corresponding indices $\mathcal{J}_{\pm}$ \eqref{indiceBOX}. For $\underline{n} \in 
\mathcal{J}_+$, recall the corresponding cell $\mathcal{B}_{\underline{n}}$ \eqref{BOXes}.
Thanks to these definitions, we have the first equality in \eqref{genralactionR2:a}. Then 
we can conclude \eqref{genralactionR2} thanks to the definition of $\mathscr{E}$ 
\eqref{Matrix}. Having the last line in \eqref{genralactionR2:a}, as well as the fact that $\mathcal{V}_{1,k},$ $\mathcal{V}_{2,k}$ $\in \mathbb{N}$, give \eqref{CR=0}, 
\eqref{CR>1} for free.   
\end{proof}
}

\subsubsection{Estimates on the Liouville equation}\label{proofs'} 
We now record some moment estimates for solutions of the Liouville equation \eqref{FokkerPlanckZ}. We defer the proofs until \cref{proofs}.

\begin{proposition}\label{Propo:ExpectationEqui}
	Fix $\epsilon>0$, setting  
		\begin{equation}\label{Gaussian}
		\mathbf{P} \ = \  \exp\left(c_{\mathbf{P} }\int_{\Lambda^*}
		\mathrm{d}k |k^1|c_{1,k}\right),	\end{equation}
	for any $c_{\mathbf{P}}=\epsilon^d c_{\mathbf{P}}^o\in\mathbb{R},$ and $k=(k^,\cdots,k^d)$ with any $k\in\Lambda^*$,
	we have 
	\begin{equation}\label{Propo:ExpectationEqui:1}
		\begin{aligned}
			\partial_t\int_{(\mathbb{R}_+\times[-\pi,\pi])^{|\Lambda^*|}}\mathrm{d}c_{1}\mathrm{d}c_{2}\, \varrho\, \mathbf{P}  \
			= \ & 0.
		\end{aligned}
	\end{equation}
	Setting  
	\begin{equation}\label{GaussianPol}
		\mathbf{P}_m \ = \ \left(\int_{\Lambda^*}
		\mathrm{d}k  |k^1|c_{1,k}\right)^m, \ \ \ \ \mbox{ with } m \in \mathbb{N},	\end{equation}
	we also have 
	\begin{equation}\label{Propo:ExpectationEqui:1a}
		\begin{aligned}
			\partial_t\int_{(\mathbb{R}_+\times[-\pi,\pi])^{|\Lambda^*|}}\mathrm{d}c_{1}\mathrm{d}c_{2}\, \varrho\, \mathbf{P}_m  \
			= \ & 0.
		\end{aligned}
	\end{equation}
	
	Let $n\in\mathbb{N}$, let $\{k_{i_1},\cdots,k_{i_n}\}$ be a subset of $\Lambda^*$, and let $\sigma_{i_j}\in \{\pm 1\}$ for $1\leq j\leq n$. The following bound holds true: 
	\begin{equation}\label{Propo:ExpectationEqui:2}
\int_{(\Lambda^*)^n}\prod_{j=1}^n\mathrm{d}k_{i_j}\left|\Big\langle \bar{\omega}(k_{i_1}) a_{k_{i_1},\sigma_{i_1}}\cdots\bar{\omega}(k_{i_n}) a_{k_{i_n},\sigma_{i_n}}\Big\rangle_t\right|^2  			\lesssim n! |\mathfrak{C}_{in,o}/c_{\mathbf{P}}|^{{n}} \left|\int_{(\mathbb{R}_+\times[-\pi,\pi])^{|\Lambda^*|}}\mathrm{d}c_{1}\mathrm{d}c_{2}\mathbf{P}\varrho(0)\right|^2,
	\end{equation}
	where the constant $\mathfrak{C}_{in,o}$ in the inequality \eqref{Propo:ExpectationEqui:2} is universal. 
	
	For given signs  $\{\sigma_{i_j}\}_{j=1}^n$, define the set
	$$\Lambda^*_{\{\sigma_{i_j}\}_{j=1}^n} \coloneqq \left\{\bar{k}=\sum_{j=1}^nk_{i_j}\sigma_{i_j}:  k_{i_1},\ldots,k_{i_n}\in \Lambda^*\right\}.$$
	Then for any $\epsilon k_*\in \Lambda^*_{\{\sigma_{i_j}\}_{j=1}^n} $  and $n\ge 2$, it holds that
	\begin{multline}\label{Propo:ExpectationEqui:2a}
			\left[\sum_{\epsilon k_*\in \Lambda^*_{\{\sigma_{i_j}\}_{j=1}^n}}\epsilon^d\left[\int_{(\Lambda^*)^{n-1}}\prod_{j=1}^n\mathrm{d}k_{i_j}\left|\Big\langle \bar{\omega}(k_{i_1})a_{k_{i_1},\sigma_{i_1}}\cdots \bar{\omega}(k_{i_n})a_{k_{i_n},\sigma_{i_n}}\Big\rangle_t\right|^2\right]^\frac12 \delta\left(\sum_{j=1}^nk_{i_j}\sigma_{i_j}=\epsilon k_*\right)\right]^2  \\
			\lesssim  n!	|\mathfrak{C}_{in,o}/c_{\mathbf{P}}|^{{n}} \left|\int_{(\mathbb{R}_+\times[-\pi,\pi])^{|\Lambda^*|}}\mathrm{d}c_{1}\mathrm{d}c_{2}\mathbf{P}\varrho(0)\right|^2.
		\end{multline}
We also have the bound
\begin{equation}\label{Propo:ExpectationEqui:2b}
	\begin{aligned}
		&\int_{(\Lambda^*)^n}\prod_{j=1}^n\mathrm{d}k_{i_j}\left|\Big\langle  \bar{\omega}(k_{i_1}) a_{k_{i_1},\sigma_{i_1}}\cdots  \bar{\omega}(k_{i_n}) a_{k_{i_n},\sigma_{i_n}}\Big\rangle_t\right|\ \le \ n!\mathscr{C}_{\mathbb{M}},
	\end{aligned}
\end{equation}
where $\mathscr{C}_{\mathbb{M}}$ is a universal constant independent of $t$, $\epsilon$, $h$. Moreover, for any $\Im\in(1,2)$, we have the bound
\begin{equation}\label{Propo:ExpectationEqui:2c}
	\begin{aligned}
		&\left(\int_{(\Lambda^*)^n}\prod_{j=1}^n\mathrm{d}k_{i_j}\left|\Big\langle \bar{\omega}(k_{i_1})a_{k_{i_1},\sigma_{i_1}}\cdots \bar{\omega}(k_{i_n})a_{k_{i_n},\sigma_{i_n}}\Big\rangle_t\right|^\Im\right)^\frac1\Im\ \le \ n!\mathscr{C}_{\mathbb{M}}'\epsilon^{-d(\Im-1)n/2},
	\end{aligned}
\end{equation}
 and for  $\complement>0$,
\begin{multline}\label{Propo:ExpectationEqui:2d}
		\sum_{\substack{|k_*|\le C \\ \epsilon k_*\in \Lambda^*_{\{\sigma_{i_j}\}_{j=1}^n}}}\epsilon^d\left[\int_{(\Lambda^*)^{n-1}}\prod_{j=1}^n\mathrm{d}k_{i_j}\left|\Big\langle \bar{\omega}(k_{i_1}) a_{k_{i_1},\sigma_{i_1}}\cdots \bar{\omega}(k_{i_n})a_{k_{i_n},\sigma_{i_n}}\Big\rangle_t\right|^\Im\right]^\frac1\Im \delta\left(\sum_{j=1}^nk_{i_j}\sigma_{i_j}=\epsilon k_*\right)  \\
		\lesssim (n!) \mathscr{C}_{\mathbb{M}}''\epsilon^{d(1-1/\Im)-d(\Im-1)n/2},
\end{multline}
where $\mathscr{C}_{\mathbb{M}}'$, $\mathscr{C}_{\mathbb{M}}''$ are  universal constants independent of $t$, $\epsilon$, $h$ but $\mathscr{C}_{\mathbb{M}}''$ depends on $\complement$.
\end{proposition}

\begin{remark}\label{remark:ExpectationEqui}
In the statement of \cref{Propo:ExpectationEqui}, \eqref{Propo:ExpectationEqui:1a} actually implies \eqref{Propo:ExpectationEqui:1} by Taylor's theorem. In fact, a more general statement is true. Since $\int_{\Lambda^*}dk c_{1,k}$ commutes with the Hamiltonian for \eqref{dtBk}, the average (with respect to $\varrho(t)$) of any continuous function of $\int_{\Lambda^*}dk c_{1,k}$ is conserved.
\end{remark}

\medskip
We now come to Assumption (C) from \cref{TheoremMainRough}, which imposes a growth condition on the initial ensemble average of the $\ell^2$ norm of the amplitudes.
\medskip

	{\bf Assumption (C) - $l^2$ Boundedness of the Initial Condition.}

{\it Let $n$ be an arbitrary positive natural number and $\{k_{i_1},\cdots,k_{i_n}\}$ be a subset of $\Lambda^*$.  The function $\daleth(k)$ defined in \eqref{omega1} satisfies $1\ge\daleth(k)\ge 0$ and
	\begin{equation}
		\begin{aligned}\label{AssumpBbound}
			&	\int_{(\Lambda^*)^n}\prod_{j=1}^n\mathrm{d}k_{i_j}\int_{(\mathbb{R}_+\times[-\pi,\pi])^{|\Lambda^*|}}\mathrm{d}c_{1}\mathrm{d}c_{2}\prod_{j=1}^{n}|k_{i_j}|\big|c_{1,k_{i_j}}\big|\varrho(0)\ 
			\le   \ n! |\mathfrak{C}_{in}/c_{\mathbf{P} }|^{{n}}\prod_{k\in\Lambda^*}\frac{2}{2-c_{\mathbf{P} }h^{d}\epsilon^{-d}\daleth(k )},
		\end{aligned}
	\end{equation}
	{provided $c_{\mathbf{P}}h^d\epsilon^{-d}<2$}, where the constant $\mathfrak{C}_{in}$ in the above inequality is universal and $c_{\mathbf{P}}=\epsilon^d c_{\mathbf{P}}^o$, for $c_{\mathbf{P}}^0\in\mathbb{R}$, is the constant appearing in the form of the Gaussian \eqref{Gaussian}. We also denote $k_{i_j}=(k_{i_j}^1,\cdots,k_{i_j}^d)$. }

\begin{lemma}
	\label{Lemma:BoundL1Identity}
	For $c_{\mathbf{P}}= \epsilon^{d}	c_{\mathbf{P}}^o$, with  $c_{\mathbf{P}}^o$ chosen to have sufficiently small magnitude, the last term on the right hand side of \eqref{AssumpBbound} can be bounded as		
	\begin{equation}\label{Lemma:BoundL1Identity:1}
		\begin{aligned}
			\prod_{k\in\Lambda^*}\frac{2}{2-c_{\mathbf{P}}h^{d}\epsilon^{-d}\daleth(k )} \ \le\ 	& c_{\mathbf{P}}',
		\end{aligned}
	\end{equation}
	for some universal constant $c_{\mathbf{P}}'>0$ independent of $|\Lambda^*|$.
\end{lemma}
\begin{proof}
See  \cite[Lemma 5]{staffilani2021wave}.
\end{proof}

The following proposition shows that Assumption (C) is satisfied for the particular choice of initial data \eqref{InitialDensity}.

\begin{proposition}[On the validity of Assumption (C)]\label{Propo:ExampleInitialCondition}
Take  \eqref{InitialDensity} as the initial probability density for the system. Let $n\in\mathbb{N}$, $\{k_{i_1},\cdots,k_{i_n}\}$ be a subset of $\Lambda^*$, and $\sigma_{i_1},\ldots,\sigma_{i_n}\in\{\pm 1\}$. The following estimate holds:
		\begin{equation}\label{Propo:ExampleMeasure:2}
		\begin{aligned}
			&\int_{(\Lambda^*)^n}\prod_{j=1}^n\mathrm{d}k_{i_j}\left|\Big\langle \sqrt{\bar{\omega}(k_{i_1})}a_{k_{i_1},\sigma_{i_1}}\cdots\sqrt{\bar{\omega}(k_{i_n})} a_{k_{i_n},\sigma_{i_n}}\Big\rangle_t\right|^2  
			\lesssim    n!|\mathfrak{C}_{in}/c_{\mathbf{P}}|^{{n}}\left|\prod_{k\in\Lambda^*}\frac{2}{2-c_{\mathbf{P}}h^d\epsilon^{-d}\daleth(k)}\right|,
		\end{aligned}
	\end{equation}
where  we also denote $k_{i_j}=(k_{i_j}^1,\cdots,k_{i_j}^d)$.
\end{proposition}

	\bigskip

\subsection{Prediction from wave kinetic theory and main result}
Following \cite{Nazarenko:2011:WT,Spohn:TPB:2006}, we review below the formal prediction of wave kinetic theory for the emergence of the inhomogeneous equation. We consider the Wigner distribution $\mathcal W_{\epsilon}(\Omega,k,t)$ (recall the definition \eqref{Wigner1}), where $\epsilon=\lambda^2$. In the limit of  $D\to \infty$, then $\lambda\to 0$ and $t=\lambda^{-2}\tau=\mathcal{O}(\lambda^{-2})$, the Wigner distribution $\mathcal W_{\epsilon}$ has the limit 
$$\lim_{\lambda\to 0}\lim_{D\to\infty}\mathcal W_{\epsilon}(\lfloor\Omega\rfloor_\epsilon,k,\lambda^{-2}\tau) = f^\infty(\Omega,k,\tau),$$
with $\lfloor\cdot\rfloor_\epsilon$ being modulo $\epsilon$,\footnote{By modulo $\epsilon$, we mean that $\lfloor{\Omega}\rfloor_\epsilon$ is the unique element in $\Lambda_{2/\epsilon}^*$ such that $\lfloor{\Omega}\rfloor_\epsilon\leq \Omega<\lfloor{\Omega}\rfloor_{\epsilon}+1$.}
which solves the inhomogeneous 3-wave  equation
\begin{equation}\label{Eq:WT0}
	\frac{\partial}{\partial \tau}f^\infty(\Omega,k,\tau) \ +  \ \frac{1}{2\pi} \nabla\omega(k)\cdot \nabla_\Omega f^\infty(\Omega,k,\tau)=  \mathcal{C}\big(f^\infty  \big)(\Omega,k,\tau)
\end{equation}
with the collision operator 
\begin{multline}\label{Eq:CollisionWT1}
		\mathcal{C}(f^\infty)(k_1) \coloneqq   \frac{\pi}{2^{d-2}}\int_{(\mathbb{T}^d)^{2}}
		\mathrm{d}k_2\mathrm{d}k_3|\mathcal{M}(k_1,k_2,k_3)|^2  \delta(\omega(k_3)+\omega(k_2)-\omega(k_1))\\
		\times\delta(k_2+k_3-k_1) \Big( f_2^\infty f^\infty_3-f^\infty_1f^\infty_2\mathrm{sign}(k_1^1)\mathrm{sign}(k_3^1)-f^\infty_1f^\infty_3\mathrm{sign}(k_1^1)\mathrm{sign}(k_2^1)\Big),
\end{multline}
in which $\mathbb{T}$ is the periodic torus $[-1/2,1/2]$. 
Here, we have introduced the shorthand notation $f^\infty_j=f^\infty(k_j)$, for $j=1,2,3$. We also set 
$$\mathbb{T}^d_+=\{k=(k^1,\cdots,k^d)\in\mathbb{T}^d\ | \  k^1\ge 0 \}.$$

We denote the Fourier transform of $f^\infty(\Omega,k,\tau)$ in the first argument by
$$f^o(\mho,k,\tau) \coloneqq \int_{\mathbb{R}^d}\mathrm{d}\Omega f^\infty(\Omega,k,\tau)e^{-{\bf i}\Omega\cdot \mho 2\pi},$$
which solves the equation
\begin{equation}\label{Eq:WT0:a}
	\frac{\partial}{\partial \tau}f^o(\mho,k,\tau) \ + \ {\bf i}\mho\cdot\nabla\omega(k)f^o(\mho,k,\tau)= \widehat{ \mathcal{C}}\big(f^o \big)(\mho,k,\tau)
\end{equation}
with the collision operator 
\begin{equation}
	\begin{aligned}\label{Eq:CollisionWT1:a}
		&\widehat{ \mathcal{C}}(f^o)(\mho_1,k_1) \coloneqq   \frac{\pi}{2^{d-2}}\int_{(\mathbb{T}^d)^{2}\times (\mathbb{R}^d)^{2}}
		\mathrm{d}k_2\mathrm{d}k_3\mathrm{d}\mho_2\mathrm{d}\mho_3|\mathcal{M}(k_1,k_2,k_3)|^2  \delta(\omega(k_3)+\omega(k_2)-\omega(k_1))\\
		&\times\delta(k_2+k_3-k_1)\delta(\mho_2+\mho_3-\mho_1) \Big( f^o(\mho_2,k_2) f^o(\mho_3,k_3)\\
		&\ \ \ \  -f^o(\mho_3,k_1)f^o(\mho_2,k_2)\mathrm{sign}(k_1^1)\mathrm{sign}(k_3^1)-f^o(\mho_2,k_1)f^o(\mho_3,k_3)\mathrm{sign}(k_1^1)\mathrm{sign}(k_2^1)\Big).
\end{aligned}\end{equation}
For the ZK equation, as discussed in comment (iii) of the introduction, the delta function $\delta(\omega(k_3)+\omega(k_2)-\omega(k_1))$ appearing in the collision operator \eqref{Eq:CollisionWT1} is not well-defined as a positive measure. To overcome this technical issue, we employ the convention of resonance broadening. Given $\ell>0$, we say that a function $f_\ell^\infty$ solves
the \emph{resonance-broadened 3-wave equation}  if and only if
\begin{equation}\label{Eq:WT1:a}
	\frac{\partial}{\partial \tau}f^\infty_\ell(\Omega,k,\tau) \ +  \ \frac{1}{2\pi} \nabla\omega(k)\cdot \nabla_\Omega f^\infty_\ell(\Omega,k,\tau)=  \mathcal{C}_\ell\big(f^\infty_\ell  \big)(\Omega,k,\tau),
\end{equation}
with the collision operator 
\begin{multline}\label{Eq:CollisionWT5:a}
		{ \mathcal{C}}_\ell (f^\infty_\ell)(k_1) \coloneqq   \frac{\pi}{2^{d-2}}\int_{(\mathbb{T}^d)^{2}}
		\mathrm{d}k_2\mathrm{d}k_3|\mathcal{M}(k_1,k_2,k_3)|^2  \delta_\ell (\omega(k_3)+\omega(k_2)-\omega(k_1))\\
		\times\delta(k_2+k_3-k_1) \Big( f_{\ell,2}^\infty f^\infty_{\ell,3}-f^\infty_{\ell,1}f^\infty_{\ell,2}\mathrm{sign}(k_1^1)\mathrm{sign}(k_3^1)-f^\infty_{\ell,1}f^\infty_{\ell,3}\mathrm{sign}(k_1^1)\mathrm{sign}(k_2^1)\Big),
\end{multline}
in which  $\delta_\ell$ is defined in
\begin{equation}\label{Eq:CollisionWT6}
		{\delta}_\ell (\omega(k_3)+\omega(k_2)-\omega(k_1)) \coloneqq \frac{1}{2\ell}\int_{-\ell}^\ell\mathrm{d}\mathscr{X}\int_{\mathbb{R}}\mathrm{d}se^{-{\bf i}s(\omega(k_3)+\omega(k_2)-\omega(k_1))-{\bf i}2\pi s\mathscr{X}}.
\end{equation}
%
\pagebreak
Then our main result (cf. the informal version \cref{TheoremMainRough} from above) reads as follows.

\begin{theorem}\label{TheoremMain}
	Suppose that $d\ge 2$. Let us write  $t = \tau\lambda^{-2}$,  where  $\tau>0$.  Under Assumption (A) with $\theta_r$  small but non-zero, and Assumptions (B) and (C),    there exist $1>T_*>0$, $1>\ell'>0$, independent of the initial data, and functions $f_\ell(\mho,k,\lambda^{-2}\tau)$, $f_\ell^o(\mho,k,\tau)$ defined for all $k\in\mathbb{T}^d$, $\mho\in\mathbb{R}^d$, with $0\leq\tau\leq T_*, 0<\ell<\ell'$, such that  for any constant $\mathscr{R}>1$, we have the following.
	
	\begin{itemize}[leftmargin=*]
		\item[(i)] The function $f_\ell^o(\mho,k,\tau)$ solves \eqref{Eq:WT1:a}-\eqref{Eq:CollisionWT5:a} with the initial condition  $f_\ell^o(\mho,k,0)= f^o(\mho,k,0)$.
		
		\item[(ii)] The functions $f_\ell(\mho,k,\lambda^{-2}\tau)$ are obtained from the series expansion for the Wigner function $\mathcal{W}_{\epsilon}(\lambda^{-2}\tau)$ after sending $D\rightarrow\infty$, isolating the main diagrams (which is justified because the remaining diagrams vanish in the kinetic limit), and performing resonance broadening. They converge to $f_\ell^o(\mho,k,\tau)$ as $\lambda\to 0$ in the following sense 
		\begin{multline}\label{TheoremMain:1}
		\lim_{\lambda\to 0}  \Big\|\int_{\mathbb{R}}\mathrm{d}\mathscr{X}\varphi(\mathscr{X})\widehat{f_\ell\chi_{(0,T_*)}}(\mathscr{X})\chi_{(-\lambda^{-2}\mathscr{R},\lambda^{-2}\mathscr{R})}(\mathscr{X})\\
		 -\int_{\mathbb{R}}\mathrm{d}\mathscr{X}\varphi(\mathscr{X})\widehat{f_\ell^o\chi_{(0,T_*)}}(\mathscr{X}) \chi_{(-\lambda^{-2}\mathscr{R},\lambda^{-2}\mathscr{R})}(\mathscr{X})\Big\|_{L^1(\mathbb{R}^d\times\mathbb{T}^d)}=  0,
		\end{multline}
		for any $\varphi\in C_c^\infty(\mathbb{R})$. The functions $\widehat{f_\ell\chi_{(0,T_*)}}$ and $\widehat{f_\ell^o\chi_{(0,T_*)}}$ are the Fourier transforms of  ${f_\ell(\mho,k,\lambda^{-2}\tau)\chi_{(0,T_*)}}(\tau)$ and ${f_\ell^o(\mho,k,\tau)\chi_{(0,T_*)}(\tau)}$ in the variable $\tau$, respectively. The cut-off function $\chi_{(-\mathscr{R},\mathscr{R})}(\mathscr{X})$, is the characteristic function of the interval ${(-\mathscr{R},\mathscr{R})}$.
		\item[(iii)] The difference between the functions $f_\ell(\mho,k,\lambda^{-2}\tau)$ and $\lim_{D\to\infty}{\mathcal W_{\epsilon}}^o(\mho,k,\lambda^{-2}\tau)$, the Fourier transform of $\lim_{D\rightarrow\infty}\mathcal{W}_{\epsilon}(\Omega,k,\lambda^{-2}\tau)$ with respect to $\Omega$, vanishes as $\ell\to 0$ and $\lambda\to 0$ in the following sense: for any $\theta>0$,
		\begin{equation}\label{TheoremMain:2}\lim_{\lambda\to 0}\lim_{\ell\to 0} \mathscr{W}\left(\left|\widehat{f_\ell\chi_{(0,T_*)}}(\mathscr{X})\chi_{(-\lambda^{-2}\mathscr{R},\lambda^{-2}\mathscr{R})}(\mathscr{X})-\lim_{D\rightarrow\infty}		\widehat{\mathcal W_{\epsilon}^o\chi_{(0,T_*)}}(\mathscr{X})\chi_{(-\lambda^{-2}\mathscr{R},\lambda^{-2}\mathscr{R})}(\mathscr{X})\right|>\theta \right) =  0,
		\end{equation}	
	The function $\widehat{\mathcal{W}_\epsilon^o\chi_{(0,T_*)}}$ is the Fourier transform of  $\mathcal W_{\epsilon}^o(\lambda^{-2}\tau)\chi_{(0,T_*)}(\tau)$  in the variable $\tau$. Above, the quantity $\mathscr{W}$ is the standard Lebesgue measure on $\mathbb{R}^d\times\mathbb{T}^d\times\mathbb{R}$.	
	\end{itemize}  
\end{theorem}


\subsection{Further estimates on the Liouville equation and averages}\label{proofs}
We start this subsection with the proof of \cref{Propo:ExpectationEqui} stated in \cref{proofs'} above.

\begin{proof}[Proof of \cref{Propo:ExpectationEqui}]

We only give the proof for \eqref{Propo:ExpectationEqui:1a}, as the proof for \eqref{Propo:ExpectationEqui:1} is similar (see also \cref{remark:ExpectationEqui}). 
We find from integrating by parts multiple times and using the observation that $\partial_{c_{2,k}}\hat{\mathfrak{H}}(k)=0$ when $c_{1,k}=0$,
\begin{equation} 
	\begin{aligned}
		0 \ = \ 	& \int_{(\mathbb{R}_+\times[-\pi,\pi])^{|\Lambda^*|}}\mathrm{d}c_{1}\mathrm{d}c_{2}\partial_t\varrho\mathbf{P}_m\\
		&	+\ \sum_{k\in\Lambda^*}\int_{(\mathbb{R}_+\times[-\pi,\pi])^{|\Lambda^*|}}\mathrm{d}c_{1}\mathrm{d}c_{2}\Big[\Big[\hat{\mathfrak{H}}(k),\varrho\Big]\Big]_k\mathbf{P}_m\\
		& \ - c_r\frac12	\sum_{k,k' \in \Lambda_* } 
	\tilde{\mathscr{E}}(k,k')\int_{(\mathbb{R}_+\times[-\pi,\pi])^{|\Lambda^*|}}\mathrm{d}c_{1}\mathrm{d}c_{2}\partial_{c_{2,k}} \partial_{c_{2,k'}}\varrho\mathbf{P}_m\\
		\ = \ 	& \partial_t\int_{(\mathbb{R}_+\times[-\pi,\pi])^{|\Lambda^*|}}\mathrm{d}c_{1}\mathrm{d}c_{2}\varrho\mathbf{P}_m\ 
		+\ \sum_{k\in\Lambda^*}\int_{(\mathbb{R}_+\times[-\pi,\pi])^{|\Lambda^*|}}\mathrm{d}c_{1}\mathrm{d}c_{2}\partial_{c_{2,k}}\hat{\mathfrak{H}}(k)\partial_{c_{1,k}}\varrho\mathbf{P}_m
		\\
		&	-\ \sum_{k\in\Lambda^*}\int_{(\mathbb{R}_+\times[-\pi,\pi])^{|\Lambda^*|}}\mathrm{d}c_{1}\mathrm{d}c_{2}\partial_{c_{1,k}}\hat{\mathfrak{H}}(k)\partial_{c_{2,k}}\varrho\mathbf{P}_m
		\\
		\ = \ 	& \partial_t\int_{(\mathbb{R}_+\times[-\pi,\pi])^{|\Lambda^*|}}\mathrm{d}c_{1}\mathrm{d}c_{2}\varrho\mathbf{P}_m\ 
		-\ \sum_{k\in\Lambda^*}\int_{(\mathbb{R}_+\times[-\pi,\pi])^{|\Lambda^*|}}\mathrm{d}c_{1}\mathrm{d}c_{2}\partial_{c_{2,k}}\hat{\mathfrak{H}}(k)\varrho\partial_{c_{1,k}}\mathbf{P}_m
		\\
		& + \sum_{k\in\Lambda^*}\int_{(\mathbb{R}_+)^{|\Lambda^*|-1}\times[-\pi,\pi]^{|\Lambda^*|}}\prod_{k'\in\Lambda^*\backslash\{k\}}\mathrm{d}c_{1,k'}\mathrm{d}c_{2}\partial_{c_{2,k}}\hat{\mathfrak{H}}(k)\varrho\mathbf{P}_m\Big|_{c_{1,k}=0}^{c_{1,k}=\infty}\\
		&	-\ \sum_{k\in\Lambda^*}\int_{(\mathbb{R}_+\times[-\pi,\pi])^{|\Lambda^*|}}\mathrm{d}c_{1}\mathrm{d}c_{2}\partial_{c_{1,k}c_{2,k}}\hat{\mathfrak{H}}(k)\varrho\mathbf{P}_m
		\ 	+\ \sum_{k\in\Lambda^*}\int_{(\mathbb{R}_+\times[-\pi,\pi])^{|\Lambda^*|}}\mathrm{d}c_{1}\mathrm{d}c_{2}\partial_{c_{1,k}c_{2,k}}\hat{\mathfrak{H}}(k)\varrho\mathbf{P}_m\\
		&
		\ 	-\ \sum_{k\in\Lambda^*}\int_{(\mathbb{R}_+\times[-\pi,\pi])^{|\Lambda^*|}}\mathrm{d}c_{1}\prod_{k'\in\Lambda^*\backslash\{k\}}\mathrm{d}c_{2,k'}\partial_{c_{1,k}}\hat{\mathfrak{H}}(k)\varrho\mathbf{P}_m\Big|_{c_{2,k}=-\pi}^{c_{2,k}=\pi}\\
		\ = \ 	& \partial_t\int_{(\mathbb{R}_+\times[-\pi,\pi])^{|\Lambda^*|}}\mathrm{d}c_{1}\mathrm{d}c_{2}\varrho\mathbf{P}_m\ 
		-\ \sum_{k\in\Lambda^*}\int_{(\mathbb{R}_+\times[-\pi,\pi])^{|\Lambda^*|}}\mathrm{d}c_{1}\mathrm{d}c_{2}\partial_{c_{2,k}}\hat{\mathfrak{H}}(k)\varrho\partial_{c_{1,k}}\mathbf{P}_m
		\\
		&	-\ \sum_{k\in\Lambda^*}\int_{(\mathbb{R}_+\times[-\pi,\pi])^{|\Lambda^*|}}\mathrm{d}c_{1}\mathrm{d}c_{2}\partial_{c_{1,k}c_{2,k}}\hat{\mathfrak{H}}(k)\varrho\mathbf{P}_m
		\ 	+\ \sum_{k\in\Lambda^*}\int_{(\mathbb{R}_+\times[-\pi,\pi])^{|\Lambda^*|}}\mathrm{d}c_{1}\mathrm{d}c_{2}\partial_{c_{1,k}c_{2,k}}\hat{\mathfrak{H}}(k)\varrho\mathbf{P}_m \\
		\ = \ & \partial_t\int_{(\mathbb{R}_+\times[-\pi,\pi])^{|\Lambda^*|}}\mathrm{d}c_{1}\mathrm{d}c_{2}\varrho\mathbf{P}_m\
		-\ \sum_{k\in\Lambda^*}\int_{(\mathbb{R}_+\times[-\pi,\pi])^{|\Lambda^*|}}\mathrm{d}c_{1}\mathrm{d}c_{2}\partial_{c_{2,k}}\hat{\mathfrak{H}}(k)\varrho\partial_{c_{1,k}}\mathbf{P}_m.
		\end{aligned}
\end{equation}
We next develop
\begin{equation}
	\begin{aligned}
		&	 \sum_{k\in\Lambda^*}\int_{(\mathbb{R}_+\times[-\pi,\pi])^{|\Lambda^*|}}\mathrm{d}c_{1}\mathrm{d}c_{2}\partial_{c_{2,k}}\hat{\mathfrak{H}}(k)\varrho\partial_{c_{1,k}}\mathbf{P}_m\\
		= \ &\lambda \int_{(\mathbb{R}_+\times[-\pi,\pi])^{|\Lambda^*|}}\mathrm{d}c_{1}\mathrm{d}c_{2}\varrho\left\{\sum_{k\in\Lambda^*}\partial_{c_{1,k}}\mathbf{P}_m \int_{\Lambda^*}\mathrm{d}k_1'\int_{\Lambda^*}\mathrm{d}k_2'\mathcal{M}(k,k_1',k_2')\right.\\
		&\times 2 \sqrt{2c_{1,k_1'}c_{1,k_2'}c_{1,k}}\Big[\delta(k-k_1'-k_2'){\sin(c_{2,k_1'}+c_{2,k_2'}-c_{2,k})} \Big]\Big\}\\
		= \ & \lambda \int_{(\mathbb{R}_+\times[-\pi,\pi])^{|\Lambda^*|}}\mathrm{d}c_{1}\mathrm{d}c_{2}\varrho\left\{\int_{\Lambda^*}\mathrm{d}km\mathbf{P}_{m-1} \int_{\Lambda^*}\mathrm{d}k_1'\int_{\Lambda^*}\mathrm{d}k_2'\mathcal{M}(k,k_1',k_2')\right.\\
		&\times 2 \sqrt{2c_{1,k_1'}c_{1,k_2'}c_{1,k}}\Big[\delta(k-k_1'-k_2'){\sin(c_{2,k_1'}+c_{2,k_2'}-c_{2,k})} |k^1|\Big]\Big\}\\
		= \ & \lambda \int_{(\mathbb{R}_+\times[-\pi,\pi])^{|\Lambda^*|}}\mathrm{d}c_{1}\mathrm{d}c_{2}\varrho m\mathbf{P}_{m-1}\Big\{\int_{\Lambda^*}\mathrm{d}k \int_{\Lambda^*}\mathrm{d}k_1'\int_{\Lambda^*}\mathrm{d}k_2'|\mathcal{M}(k,k_1',k_2')|\mathrm{sign}k^1\\
		&\times 2 \sqrt{2c_{1,k_1'}c_{1,k_2'}c_{1,k}}\Big[\delta(k-k_1'-k_2'){\sin(c_{2,k_1'}+c_{2,k_2'}-c_{2,k})} |k^1|\Big]\Big\},
	\end{aligned}
\end{equation}
with $k=(k^1,\cdots,k^d)$, and the indentity $\partial_{c_{1,k}}\mathbf{P}_m=h^{d}m\mathbf{P}_{m-1}$ has been used. Using the invariance of $|\mathcal{M}(k,k_1',k_2')|$ under rotations of $k,k_1',k_2'$, we find
\begin{equation}
	\begin{aligned}
		& \int_{(\Lambda^*)^3}\mathrm{d}k \mathrm{d}k_1\mathrm{d}k_2'|\mathcal{M}(k,k_1',k_2')|\mathrm{sign}k^1 2\sqrt{2c_{1,k_1'}c_{1,k_2'}c_{1,k}}\Big[\delta(k-k_1'-k_2'){\sin(c_{2,k_1'}+c_{2,k_2'}-c_{2,k})} |k^1|\Big]\\
		=\		& \int_{(\Lambda^*)^3}\mathrm{d}k \mathrm{d}k_1\mathrm{d}k_2'|\mathcal{M}(k,k_1',k_2')|k^1 2\sqrt{2c_{1,k_1'}c_{1,k_2'}c_{1,k}}\Big[\delta(k-k_1'-k_2'){\sin(c_{2,k_1'}+c_{2,k_2'}-c_{2,k})} \Big]\\
		=\  	& \frac23\int_{(\Lambda^*)^3}\mathrm{d}k \mathrm{d}k_1\mathrm{d}k_2'|\mathcal{M}(k,k_1',k_2')|(k^1-k_1'^1-k_2'^1) \sqrt{2c_{1,k_1'}c_{1,k_2'}c_{1,k}}\\
		&\times\Big[\delta(k-k_1'-k_2'){\sin(c_{2,k_1'}+c_{2,k_2'}-c_{2,k})} \Big] \\
		= \ & 0.
	\end{aligned}
\end{equation}
After a little bookkeeping, we obtain \eqref{Propo:ExpectationEqui:1a}.

We now prove \eqref{Propo:ExpectationEqui:2}. We introduce the notation $(b_{1,k_{i_j}}+{\bf i}b_{2,k_{i_j}})_{\pm 1} \coloneqq b_{1,k_{i_j}}\pm {\bf i}b_{2,k_{i_j}}$. We use the polar change of variable, Cauchy-Schwarz, and \eqref{Propo:ExpectationEqui:1} to compute
	\begin{equation}\label{Propo:ExpectationEqui:EA}
		\begin{aligned}
			&\sum_{k_{i_1},\cdots,k_{i_n}\in\Lambda^*}\left|\Big\langle \sqrt{|\bar\omega(k_{i_1})|} a_{k_{i_1},\sigma_{i_1}}\cdots  \sqrt{|\bar\omega(k_{i_n})|} a_{k_{i_n},\sigma_{i_n}}\Big\rangle_t\right|^2 \\
			= & \ \sum_{k_{i_1},\cdots,k_{i_n}\in\Lambda^*} \left|\int_{\mathbb{R}^{2n}}\prod_{j=1}^{n}\mathrm{d}b_{1,k_{i_j}}\mathrm{d}b_{2,k_{i_j}}\sqrt{|\bar\omega(k_{i_j})|}\big[(b_{1,k_{i_j}}+{\bf i}b_{2,k_{i_j}})_{\sigma_{i_j}}\big] \right.\\
			&\times\left.\int_{\mathbb{R}^{2(|\Lambda^*|-n)}}\prod_{k\in\Lambda^*\backslash\{k_{i_1},\cdots,k_{i_n}\}}\mathrm{d}b_{1,k}\mathrm{d}b_{2,k}\varrho(t)\right|^2 \\
			\lesssim &  \ \mathfrak{C}_{in,o}^n\sum_{{ k_{i_1},\cdots,k_{i_n}\in\Lambda^*}} \left|\int_{\mathbb{R}^{2n}}\prod_{j=1}^{n}\mathrm{d}b_{1,k_{i_j}}\mathrm{d}b_{2,k_{i_j}}\sqrt{|k_{i_j}|}\big|b_{1,k_{i_j}}+{\bf i}b_{2,k_{i_j}}\big| \int_{\mathbb{R}^{2(|\Lambda^*|-n)}}\prod_{k\in\Lambda^*\backslash\{k_{i_1},\cdots,k_{i_n}\}}\mathrm{d}b_{1,k}\mathrm{d}b_{2,k}\varrho(t)\right|^2 \\
			\lesssim  &\ \mathfrak{C}_{in,o}^n\sum_{{k_{i_1},\cdots,k_{i_n}\in\Lambda^*}}	\left|\int_{(\mathbb{R}_+\times[-\pi,\pi])^{|\Lambda^*|}}\mathrm{d}c_{1}\mathrm{d}c_{2}\prod_{j=1}^{n}\sqrt{|k_{i_j}|}\big|2 c_{1,k_{i_j}}\big|^
			{\frac{1}{2}}\varrho(t)\right|^2\\
			\lesssim  &\ \mathfrak{C}_{in,o}^n\sum_{{k_{i_1},\cdots,k_{i_n}\in\Lambda^*}}\left|	\int_{(\mathbb{R}_+\times[-\pi,\pi])^{|\Lambda^*|}}\mathrm{d}c_{1}\mathrm{d}c_{2}\prod_{j=1}^{n}|k_{i_j}|\big|2 c_{1,k_{i_j}}\big|\varrho(t)\right|\left|	\int_{(\mathbb{R}_+\times[-\pi,\pi])^{|\Lambda^*|}}\mathrm{d}c_{1}\mathrm{d}c_{2}\varrho(t)\right|\\
			\lesssim  &\ \mathfrak{C}_{in,o}^n\sum_{{ k_{i_1},\cdots,k_{i_n}\in\Lambda^*}}\left|	\int_{(\mathbb{R}_+\times[-\pi,\pi])^{|\Lambda^*|}}\mathrm{d}c_{1}\mathrm{d}c_{2}\prod_{j=1}^{n}\big|2 |k_{i_j}|c_{1,k_{i_j}}\big|\varrho(t)\right|  \\
			\lesssim  &\ 	\mathfrak{C}_{in,o}^n \left|\int_{(\mathbb{R}_+\times[-\pi,\pi])^{|\Lambda^*|}}\mathrm{d}c_{1}\mathrm{d}c_{2}\Big|\sum_{k\in\Lambda^*} |k_{i_j}|c_{1,k}\Big|^
			{{n}}\varrho(t)\right|\\
			\
			\lesssim   &  \  n!	|\mathfrak{C}_{in,o}/c_{\mathbf{P}}|^{{n}}h^{-dn} \left|\int_{(\mathbb{R}_+\times[-\pi,\pi])^{|\Lambda^*|}}\mathrm{d}c_{1}\mathrm{d}c_{2}\mathbf{P}\varrho(t)\right|\\
			\
			\lesssim    &  \   n!	|\mathfrak{C}_{in,o}/c_{\mathbf{P}}|^{{n}}h^{-dn} \left|\int_{(\mathbb{R}_+\times[-\pi,\pi])^{|\Lambda^*|}}\mathrm{d}c_{1}\mathrm{d}c_{2}\mathbf{P}\varrho(0)\right|,
		\end{aligned}
	\end{equation}
in which $\mathfrak{C}_{in,o}$ is a universal constant. This yields \eqref{Propo:ExpectationEqui:2}.

Next, we prove \eqref{Propo:ExpectationEqui:2a}. By the change of variable $k^*=\epsilon k_*,$ we bound, for $n\ge 2$
\begin{equation}\label{Propo:ExpectationEqui:EA:aa}
	\begin{aligned}
		&h^{d}\sum_{\epsilon k_*\in \Lambda^*_{\{\sigma_{i_j}\}_{j=1}^n}}\epsilon^d\left[\sum_{k_{i_1},\cdots,k_{i_n}\in\Lambda^*}\left|\Big\langle \sqrt{|\bar\omega(k_{i_1})|} a_{k_{i_1},\sigma_{i_1}}\cdots  \sqrt{|\bar\omega(k_{i_n})|} a_{k_{i_n},\sigma_{i_n}}\Big\rangle_t\right|^2 \right]^\frac12\\
		&\times \delta\left(\sum_{j=1}^nk_{i_j}\sigma_{i_j}=\epsilon k_*\right) \\
= &\ \sum_{\epsilon k_*\in \Lambda^*_{\{\sigma_{i_j}\}_{j=1}^n}}\left[\sum_{k_{i_1},\cdots,k_{i_n}\in\Lambda^*}\left|\Big\langle \sqrt{|\bar\omega(k_{i_1})|} a_{k_{i_1},\sigma_{i_1}}\cdots  \sqrt{|\bar\omega(k_{i_n})|} a_{k_{i_n},\sigma_{i_n}}\Big\rangle_t\right|^2 \right]^\frac12\\
&\times \delta\left(\sum_{j=1}^nk_{i_j}\sigma_{i_j}= k^*\right)\\
\lesssim &\ \left[h^{d}\sum_{ k^*\in \Lambda^*_{\{\sigma_{i_j}\}_{j=1}^n}}\sum_{\substack{k_{i_1},\cdots,k_{i_n}\in\Lambda^* \\ \sum_{j=1}^nk_{i_j}\sigma_{i_j}= k^*}}\left|\Big\langle\sqrt{|\bar\omega(k_{i_1})|} a_{k_{i_1},\sigma_{i_1}}\cdots  \sqrt{|\bar\omega(k_{i_n})|} a_{k_{i_n},\sigma_{i_n}}\Big\rangle_t\right|^2 \right]^\frac12\\
&\times \left[h^{d}\sum_{ k^*\in \Lambda^*_{\{\sigma_{i_j}\}_{j=1}^n}}1 \right]^\frac12 \\
\lesssim &\  \left[h^{d}\sum_{\substack{k^*\in \Lambda^*_{\{\sigma_{i_j}\}_{j=1}^n}}}\sum_{\substack{k_{i_1},\cdots,k_{i_n}\in\Lambda^*\\\sum_{j=1}^nk_{i_j}\sigma_{i_j}= k^*}}\left|\Big\langle \sqrt{|\bar\omega(k_{i_1})|} a_{k_{i_1},\sigma_{i_1}}\cdots  \sqrt{|\bar\omega(k_{i_n})|} a_{k_{i_n},\sigma_{i_n}}\Big\rangle_t\right|^2 \right]^\frac12.
	\end{aligned}
\end{equation}
Again using a change of variable and Cauchy-Schwarz, this implies
	\begin{equation}\label{Propo:ExpectationEqui:EA:a}
	\begin{aligned}
		&\left[h^{d}\sum_{k_*\in\Lambda^*}\epsilon^d\left[h^{d(n-1)}\sum_{k_{i_1},\cdots,k_{i_n}\in\Lambda^*}\left|\Big\langle \sqrt{|\bar\omega(k_{i_1})|} a_{k_{i_1},\sigma_{i_1}}\cdots  \sqrt{|\bar\omega(k_{i_n})|} a_{k_{i_n},\sigma_{i_n}}\Big\rangle_t\right|^2 \right]^\frac12 \right.\\
		&\times\left.\delta\left(\sum_{j=1}^nk_{i_j}\sigma_{i_j}=\epsilon k_*\right)\right]^2 \\
		\lesssim &  \ h^{dn}\mathfrak{C}_{in,o}^n\sum_{ k^*\in \Lambda^*_{\{\sigma_{i_j}\}_{j=1}^n}}\sum_{\substack{k_{i_1},\cdots,k_{i_n}\in\Lambda^* \\ \sum_{j=1}^nk_{i_j}\sigma_{i_j}= k^*}} \left|\int_{\mathbb{R}^{2n}}\prod_{j=1}^{n}\mathrm{d}b_{1,k_{i_j}}\mathrm{d}b_{2,k_{i_j}}\sqrt{|\bar\omega(k_{i_j})|}\big|b_{1,k_{i_j}}+{\bf i}b_{2,k_{i_j}}\big|\right.\\
		&\left.\times \int_{\mathbb{R}^{2(|\Lambda^*|-n)}}\prod_{k\in\Lambda^*\backslash\{k_{i_1},\cdots,k_{i_n}\}}\mathrm{d}b_{1,k}\mathrm{d}b_{2,k}\varrho(t) \right|^2 \\
		\lesssim  &\ h^{dn}\mathfrak{C}_{in,o}^n\sum_{k^*\in\Lambda^*}\sum_{\substack{k_{i_1},\cdots,k_{i_n}\in\Lambda^*\\ \sum_{j=1}^nk_{i_j}\sigma_{i_j}= k^*}} 	\left|\int_{(\mathbb{R}_+\times[-\pi,\pi])^{|\Lambda^*|}}\mathrm{d}c_{1}\mathrm{d}c_{2} \prod_{j=1}^{n}\sqrt{|\bar\omega(k_{i_j})|}\big|2c_{1,k_{i_j}}\big|^
		{\frac{1}{2}}\varrho(t) \right|^2\\
		\lesssim  &\ h^{dn}\mathfrak{C}_{in,o}^n\sum_{ k^*\in \Lambda^*_{\{\sigma_{i_j}\}_{j=1}^n}}\sum_{\substack{k_{i_1},\cdots,k_{i_n}\in\Lambda^* \\ \sum_{j=1}^nk_{i_j}\sigma_{i_j}= k^*}}	\left|\int_{(\mathbb{R}_+\times[-\pi,\pi])^{|\Lambda^*|}}\mathrm{d}c_{1}\mathrm{d}c_{2}\prod_{j=2}^{n}|\bar\omega(k_{i_j})|\big|2c_{1,k_{i_j}}\big|\varrho(t)\right| \\
		\
		\lesssim   &  \   n!	|\mathfrak{C}_{in,o}/c_{\mathbf{P}}|^{{n}} \left|\int_{(\mathbb{R}_+\times[-\pi,\pi])^{|\Lambda^*|}}\mathrm{d}c_{1}\mathrm{d}c_{2}\mathbf{P}\varrho(t)\right|^2\\
		\
		=    &  \  n! 	|\mathfrak{C}_{in,o}/c_{\mathbf{P}}|^{{n}} \left|\int_{(\mathbb{R}_+\times[-\pi,\pi])^{|\Lambda^*|}}\mathrm{d}c_{1}\mathrm{d}c_{2}\mathbf{P}\varrho(0)\right|^2,
	\end{aligned}
\end{equation}
where the final line is by \eqref{Propo:ExpectationEqui:1}. Hence, 
\begin{multline}\label{Propo:ExpectationEqui:EA:b}
		\left[\sum_{\epsilon k_*\in \Lambda^*_{\{\sigma_{i_j}\}_{j=1}^n}}\epsilon^d\left[\int_{(\Lambda^*)^{n-1}}\prod_{j=1}^n\mathrm{d}k_{i_j}\left|\Big\langle \sqrt{|\bar\omega(k_{i_1})|} a_{k_{i_1},\sigma_{i_1}}\cdots  \sqrt{|\bar\omega(k_{i_n})|} a_{k_{i_n},\sigma_{i_n}}\Big\rangle_t\right|^2\right]^\frac12 \right.\\
		\left.\times\delta\left(\sum_{j=1}^nk_{i_j}\sigma_{i_j}=\epsilon k_*\right)\right]^2   \\
		\lesssim  n!	|\mathfrak{C}_{in,o}/c_{\mathbf{P}}|^{{n}} \left|\int_{(\mathbb{R}_+\times[-\pi,\pi])^{|\Lambda^*|}}\mathrm{d}c_{1}\mathrm{d}c_{2}\mathbf{P}\varrho(0)\right|^2.
\end{multline}
The preceding inequality    yields \eqref{Propo:ExpectationEqui:2a}.

Next, we compute  
	\begin{equation}\label{Propo:ExpectationEqui:EA:BIS1}
	\begin{aligned}
		&h^{dn}\sum_{k_{i_1},\cdots,k_{i_n}\in\Lambda^*}\left|\Big\langle \sqrt{|\bar\omega(k_{i_1})|} a_{k_{i_1},\sigma_{i_1}}\cdots  \sqrt{|\bar\omega(k_{i_n})|} a_{k_{i_n},\sigma_{i_n}}\Big\rangle_t\right|\\ 		= & \int_{(\Lambda^*)^n}\mathrm{d}k_{i_1}\cdots\mathrm{d}k_{i_n}\left|\Big\langle \sqrt{|\bar\omega(k_{i_1})|} a_{k_{i_1},\sigma_{i_1}}\cdots  \sqrt{|\bar\omega(k_{i_n})|} a_{k_{i_n},\sigma_{i_n}}\Big\rangle_t\right|\\  \lesssim & \  \int_{(\Lambda^*)^n}\mathrm{d}k_{i_1}\cdots\mathrm{d}k_{i_n}\Big\langle | \bar{\omega}(k_{i_1})| |c_{1,k_{i_1}}\cdots \bar{\omega}(k_{i_n})|c_{1,k_{i_n}}\Big\rangle_t^\frac12 \\ \lesssim & \ \left[ \int_{(\Lambda^*)^n}\mathrm{d}k_{i_1}\cdots\mathrm{d}k_{i_n}\Big\langle | \bar{\omega}(k_{i_1})| c_{1,k_{i_1}}\cdots |\bar{\omega}(k_{i_n})|c_{1,k_{i_n}}\Big\rangle_t\right]^\frac12\left[ \int_{(\Lambda^*)^n}\mathrm{d}k_{i_1}\cdots\mathrm{d}k_{i_n}1\right]^\frac12\\	\lesssim & \ \left\langle h^{dn}\sum_{k_{i_1},\cdots,k_{i_n}\in\Lambda^*}|\bar{\omega}(k_{i_1})|c_{1,k_{i_1}}\cdots |\bar{\omega}(k_{i_n})|c_{1,k_{i_n}}\right\rangle_t^\frac12  \\		
		\lesssim &\   \left\langle h^{dn}\left( \sum_{k\in\Lambda^*}|\bar{\omega}(k)|c_{k}\right)^n\right\rangle_t^\frac12  \ 
		\lesssim  \  \left\langle h^{dn}\left( \sum_{k\in\Lambda^*}|\bar{\omega}(k)|c_{k}\right)^n\right\rangle_0^\frac12  \ 
		\lesssim  \   \left\langle h^{dn} \left( \sum_{k\in\Lambda^*}c_{k}\right)^n\right\rangle_0^\frac12  \\
		\lesssim  &\ \left(h^{dn}\sum_{k_{i_1},\cdots,k_{i_n}\in\Lambda^*}\Big\langle c_{1,k_{i_1}}\cdots c_{1,k_{i_n}}\Big\rangle_0\right)^\frac12 ,
	\end{aligned}
\end{equation}
where we have used \eqref{Propo:ExpectationEqui:1a}. From \eqref{Propo:ExpectationEqui:EA:BIS1}, following the standard Gaussian inequality for   products (see, for instance \cite{li2012gaussian}), we bound
	\begin{equation}\label{Propo:ExpectationEqui:EA:BIS2}
	\begin{aligned}
		&h^{dn}\sum_{k_{i_1},\cdots,k_{i_n}\in\Lambda^*}\left|\Big\langle \sqrt{|\bar\omega(k_{i_1})|} a_{k_{i_1},\sigma_{i_1}}\cdots \sqrt{|\bar\omega(k_{i_n})|} a_{k_{i_n},\sigma_{i_n}}\Big\rangle_t\right|\\ \lesssim \ & \left(h^{dn}\sum_{k_{i_1},\cdots,k_{i_n}\in\Lambda^*}\Big\langle \big|a_{k_{i_1},\sigma_{i_1}}\cdots a_{k_{i_n},\sigma_{i_n}}\big|^2\Big\rangle_0\right)^\frac12\\
		\ \lesssim \ &   \left(h^{dn}\sum_{k_{i_1},\cdots,k_{i_n}\in\Lambda^*}{\mathrm{PER}(a_{k_{i_1},\sigma_{i_1}}(0),\cdots, a_{k_{i_n},\sigma_{i_n}}(0))}\right)^\frac12\\
		\ \lesssim \ &   \left(\int_{(\Lambda^*)^n}\mathrm{d}k_{i_1}\cdots\mathrm{d}k_{i_n}{\mathrm{PER}(a_{k_{i_1},\sigma_{i_1}}(0),\cdots, a_{k_{i_n},\sigma_{i_n}}(0))}\right)^\frac12,
	\end{aligned}
\end{equation}
in which the permanent is defined by
	\begin{equation}\label{Propo:ExpectationEqui:EA:BIS3}
	\begin{aligned}
	\mathrm{PER}(a_{k_{i_1},\sigma_{i_1}}(0),\cdots, a_{k_{i_n},\sigma_{i_n}}(0)) \ \coloneqq \	& \sum_{\pi\in \mathscr{S}(n)}\prod_{j=1}^n \left\langle a_{k_{i_j},\sigma_{i_j}}a_{k_{i_{\pi(j)}},\sigma_{i_{\pi(j)}}}\right\rangle_0.
	\end{aligned}
\end{equation}
In the above definition, $\mathscr{S}$ is the symmetric group of permutations on $\{1,\cdots,n\}$. Using \cref{averages}, we bound \eqref{Propo:ExpectationEqui:EA:BIS2} by
\begin{equation}\label{Propo:ExpectationEqui:EA:BIS3}
	\begin{aligned}
		&\sum_{k_{i_1},\cdots,k_{i_n}\in\Lambda^*}h^d\left|\Big\langle \sqrt{|\bar\omega(k_{i_1})|} a_{k_{i_1},\sigma_{i_1}}\cdots \sqrt{|\bar\omega(k_{i_n})|} a_{k_{i_n},\sigma_{i_n}}\Big\rangle_t\right|\ \le \ n!\mathscr{C}_{\mathbb{M}},
	\end{aligned}
\end{equation}
in which $\mathscr{C}_{\mathbb{M}}$ is a universal constant independent of $\epsilon$ and $h$. We have proved \eqref{Propo:ExpectationEqui:2b}. Inequality \eqref{Propo:ExpectationEqui:2c} follows by a standard interpolation argument between the two inequalities \eqref{Propo:ExpectationEqui:2} and \eqref{Propo:ExpectationEqui:2b}. Inequality \eqref{Propo:ExpectationEqui:2d} can be proved by the same argument used to prove \eqref{Propo:ExpectationEqui:2a}. Thus, the proof of \cref{Propo:ExpectationEqui} is complete.

\end{proof}

Lastly, we give the proof of \cref{Propo:ExampleInitialCondition}.

\begin{proof}[Proof of \cref{Propo:ExampleInitialCondition}]

By \cref{Propo:ExpectationEqui}, we have
\begin{multline}\label{Propo:ExampleMeasure:E2}
 \int_{(\Lambda^*)^n}\prod_{j=1}^n\mathrm{d}k_{i_j}\left|\Big\langle \sqrt{|\bar\omega(k_{i_1})|} a_{k_{i_1},\sigma_{i_1}}\cdots \sqrt{|\bar\omega(k_{i_n})|} a_{k_{i_n},\sigma_{i_n}}\Big\rangle_t\right|^2 \\ 			
 \lesssim  n! |\mathfrak{C}_{in}/c_{\mathbf{P}}|^{{n}}\left|\int_{(\mathbb{R}_+\times[-\pi,\pi])^{|\Lambda^*|}}\mathrm{d}c_{1}\mathrm{d}c_{2}\mathbf{P}\varrho(0)\right|.
\end{multline}
	Substituting \eqref{InitialDensity} for $\varrho(0)$  in \eqref{Propo:ExampleMeasure:E2}, we see that
	\begin{equation}\label{Propo:ExampleMeasure:E3}
		\begin{aligned}
			&\int_{(\Lambda^*)^n}\prod_{j=1}^n\mathrm{d}k_{i_j}\left|\Big\langle \sqrt{|\bar\omega(k_{i_1})|} a_{k_{i_1},\sigma_{i_1}}\cdots \sqrt{|\bar\omega(k_{i_n})|} a_{k_{i_n},\sigma_{i_n}}\Big\rangle_t\right|^2  \\
			\lesssim  &\ \	 n!|\mathfrak{C}_{in}/c_{\mathbf{P}}|^{{n}}\left|\int_{(\mathbb{R}_+\times[-\pi,\pi])^{|\Lambda^*|}}\mathrm{d}b_{1}\mathrm{d}b_{2}\exp\left(h^dc_{\mathbf{P} }\sum_{k\in\Lambda^*}\frac{b^2_{1,k}+b^2_{2,k}}{2}\right)\frac{1}{\pi^{|\Lambda_*|}}\prod_{i=1}^{|\Lambda_*|}\frac{1}{\iota_i}\exp\left(-\frac{|\mathbf{q}_i^*\mathbf{a}|^2}{\iota_i}\right)\right|\\
			\lesssim &\ n!|\mathfrak{C}_{in}/c_{\mathbf{P}}|^{{n}} \left[(2\pi)^{|\Lambda^*|}\int_{\mathbb{R}_+^{|\Lambda^*|}}\mathrm{d}c_{1}\left(\prod_{k\in\Lambda^*}\frac{e^{ h^dc_{\mathbf{P}}c_{1,k}-\frac{2c_{1,k}}{	\epsilon^{-d}\daleth(k) }}}{\pi \epsilon^{-d}	\daleth(k) }\right)\right],
		\end{aligned}
	\end{equation}
where we have used the identity \eqref{omega1}. Recalling the elementary calculus identity 
	\begin{equation}\label{IdentityExponential}
		\int_0^\infty\mathrm{d}x x^ne^{-ax} \ = \ \frac{n!}{a^{n+1}},
	\end{equation} we compute the last term \eqref{Propo:ExampleMeasure:E3}  explicitly
	as
	\begin{equation*}\begin{aligned}
			& (2\pi)^{|\Lambda^*|}\int_{\mathbb{R}_+^{|\Lambda^*|}}\mathrm{d}c_{1}\left(\prod_{k\in\Lambda^*}\frac{e^{ h^dc_{\mathbf{P}}c_{1,k}-\frac{2c_{1,k}}{\epsilon^{-d}	\daleth(k) }}}{\pi\epsilon^{-d}	\daleth(k) }\right)\\
			= &\ \int_{\mathbb{R}_+^{|\Lambda^*|}}\prod_{k\in\Lambda^*}\mathrm{d}(2c_{1,k})\prod_{k\in\Lambda^*}\left(\frac{e^{2c_{1,k}\big(\frac{h^dc_{\mathbf{P}}}{2}-\frac{1}{	\epsilon^{-d}\daleth(k) }\big)}}{\big(-\frac{h^dc_{\mathbf{P}}}{2}+\frac{1}{\epsilon^{-d}	\daleth(k) }\big)^{-1}}\frac{1}{	\epsilon^{-d}	\daleth(k) \big(-\frac{h^dc_{\mathbf{P}}}{2}+\frac{1}{\epsilon^{-d}	\daleth(k) }\big)}\right)\\ 
			\lesssim &  \ \prod_{k\in\Lambda^*}\frac{2}{2-h^dc_{\mathbf{P}}	\epsilon^{-d}\daleth(k) },
		\end{aligned}
	\end{equation*}
which completes the proof of the proposition.
\end{proof}

\section{Duhamel expansions and moment bounds} \label{Sec:Duha}
\subsection{Kolmogorov equation} Let ${a},{a}^*$ denote the vectors $(a_k)_{k\in\Lambda^*}$, $(a_k^*)_{k\in\Lambda^*}$. Note that the vector $(a_k)_{k\in\Lambda^*}$ is comprised of all $a_k$, in which the index $k$ is fixed and runs over all of the points  of the lattice $\Lambda^*$. For two vectors $\mathcal{V}_1=(\mathcal{V}_{1,k})_{k\in\Lambda^*}\in\mathbb{N}^{|\Lambda^*|}, $ and $\mathcal{V}_2=(\mathcal{V}_{2,k})_{k\in\Lambda^*}\in\mathbb{N}^{|\Lambda^*|}$, using Definition \ref{averages}, we denote
\begin{equation}\label{AverageVectorShorthand}
  \quad \left\langle {a}^{\mathcal{V}_{1}} ({a}^*)^{\mathcal{V}_{2}}\right\rangle_t \ = \ \left\langle \prod_{k\in \Lambda^* }a_k^{\mathcal{V}_{1,k}} (a_k^*)^{\mathcal{V}_{2,k}}\right\rangle_t.
\end{equation}
We will study the evolution of the momenta $\mathbf{a}^{\mathcal{V}_1}(\mathbf{a}^*)^{\mathcal{V}_2}$, following precisely the same procedure of \cite{staffilani2021wave}, we find
$$\left\{\mathcal{H}_1,\mathbf{a}^{\mathcal{V}_1}(\mathbf{a}^*)^{\mathcal{V}_2}\right\}  =  \sum_{k\in\Lambda^*}\omega_k\left(b_{1,k}\partial_{b_{2,k}}-b_{2,k}\partial_{b_{1,k}}\right)\mathbf{a}^{\mathcal{V}_1}(\mathbf{a}^*)^{\mathcal{V}_2}$$ $$  =  {\bf i}\sum_{k\in\Lambda^*} \omega_k(\mathcal{V}_{1,k}-\mathcal{V}_{2,k})\mathbf{a}^{\mathcal{V}_1}(\mathbf{a}^*)^{\mathcal{V}_2},$$
and (recalling \eqref{NOISEgenerator}){
\begin{equation}\label{NoiseAction}\begin{aligned}
&\mathcal{R}
\left(\mathbf{a}^{\mathcal{V}_1}(\mathbf{a}^*)^{\mathcal{V}_2}\right)\\
 \ = \  & -2\Big[\sum_{k_1,k_2\in\Lambda_*^+}(\mathcal{V}_{1,k_1}+ \mathcal{V}_{2,-k_1}-\mathcal{V}_{2,k_1}-\mathcal{V}_{1,-k_1})\\
 &\times(\mathcal{V}_{1,k_2}+
 \mathcal{V}_{2,-k_2}-\mathcal{V}_{2,k_2}-\mathcal{V}_{1,-k_2})\mathscr{E}(k_1,k_2)\Big]\mathbf{a}^{\mathcal{V}_1}(\mathbf{a}^*)^{\mathcal{V}_2}.
\end{aligned}
\end{equation}}
{Notice that since $\tilde{\mathscr{E}}$ is  positive 
semi-definite then the  coefficient  in front of $\mathbf{a}^{\mathcal{V}_1}(\mathbf{a}^*)^{\mathcal{V}_2}$ is less than or equal to zero.} 
We set $\mathfrak{M}=\mathcal{M}/2$ and derive  an  equation for $\langle {a}^{\mathcal{V}_1}({a}^*)^{\mathcal{V}_2}\rangle$
{
\begin{equation}
	\label{KolmogorovEquation2}
	\begin{aligned}
		&\frac{\partial}{\partial t}\left\langle {a}^{\mathcal{V}_1}({a}^*)^{\mathcal{V}_2}\right\rangle \ = \  \Big\{{\bf i}\sum_{k\in\Lambda^*} \omega_k(\mathcal{V}_{1,k}-\mathcal{V}_{2,k})\\
		&\  -2c_r\Big[\sum_{k_1,k_2\in\Lambda_*^+}(\mathcal{V}_{1,k_1}+ \mathcal{V}_{2,-k_1}-\mathcal{V}_{2,k_1}-\mathcal{V}_{1,-k_1})\\
		&\times(\mathcal{V}_{1,k_2}+
 \mathcal{V}_{2,-k_2}-\mathcal{V}_{2,k_2}-\mathcal{V}_{1,-k_2})\mathscr{E}(k_1,k_2)\Big]\Big\}\left\langle {a}^{\mathcal{V}_1}({a}^*)^{\mathcal{V}_2}\right\rangle\\
		& + \lambda\,  \sum_{\mathcal{U}_1,\mathcal{U}_2}\mathcal{M}^o_{\mathcal{U}_1,\mathcal{U}_2}\left\langle {a}^{\mathcal{U}_1}({a}^*)^{\mathcal{U}_2}\right\rangle,
	\end{aligned}
\end{equation}}
in which  $\mathcal{M}^o_{\mathcal{U}_1,\mathcal{U}_2}$ is a coefficient depending on $\mathcal{M}$, and $\mathcal{U}_1,\mathcal{U}_2$ are  multi-indices depending on $\mathcal{V}_1$ and $\mathcal{V}_2$. 

Let us now define the new average 
{
\begin{equation}
	\label{Def:NewAverA}
	\langle\langle a^{\mathcal{V}_1}(a^*)^{\mathcal{V}_2}\rangle\rangle \ = \ e^{-{\bf i}t\sum_{k\in\Lambda^*} \omega_k(\mathcal{V}_{1,k}-\mathcal{V}_{2,k})+t c_r C_{\mathbf{R}}(\mathcal{V}_1,\mathcal{V}_2)} \langle a^{\mathcal{V}_1}(a^*)^{\mathcal{V}_2}\rangle.
\end{equation}
Equivalently, the new average for $\alpha$,  in \eqref{HatA}, takes the form 
{
\begin{equation}
	\label{Def:NewAverAlpha}
	\begin{aligned}
		& \langle \alpha(-1)^{\mathcal{V}_1} \alpha(1)^{\mathcal{V}_2}\rangle_*\ := \ \langle \alpha(-1)^{\mathcal{V}_1} \alpha(1)^{\mathcal{V}_2}\rangle_{*,t} \ := \ \langle\langle a^{\mathcal{V}_1}(a^*)^{\mathcal{V}_2}\rangle\rangle\\
		& \ = \ e^{tc_r C_{\mathbf{R}}(\mathcal{V}_1,\mathcal{V}_2)}\langle \alpha(-1)^{\mathcal{V}_1}\alpha(1)^{\mathcal{V}_2}\rangle.
	\end{aligned}
\end{equation}}

%

Let us now compute the ``inhomogeneous phase regulators'', which comes from

\begin{equation}
\begin{aligned}\label{InhomogePhaseReg1}
&\mathcal{R}\Big[\mathbf{a}_{k_{1}}\cdots \mathbf{a}_{k_{n}}\mathbf{a}_{k_{1}'}^*\cdots \mathbf{a}_{k_{{m}}'}^*\Big],
\end{aligned}
\end{equation}
for some natural numbers $n,m> 0$. { Here $k_i$s, and $k_j'$s are not necessarily distinct. 
For $k \in \Lambda^*$ define $\mathscr{U}(k)$ as follows: 
\[ \mathscr{U}(k)=
\begin{cases}
1, \text{if} \quad \quad k \in \Lambda^+_*, \\
-1, \text{if} \quad \quad k \in \Lambda^-_*.
\end{cases}
\]
Thanks to \eqref{NoiseAction}, for any $k_1,\dots, k_n,k_1',\dots, k_m'\in\Lambda^*$ we have:
\begin{equation}
\begin{aligned}\label{InhomogePhaseReg11}
\mathcal{R}\Big[\mathbf{a}_{k_{1}}\cdots \mathbf{a}_{k_{n}}\mathbf{a}_{k_{1}'}^*\cdots \mathbf{a}_{k_{{m}}'}^*\Big]\ =\  &-2 \Big( \sum_{i,j=1}^n \mathscr{E}(\mathscr{U}(k_i)k_i,
\mathscr{U}(k_j)k_j) \mathscr{U}(k_i)\mathscr{U}(k_j)\\
&   + \sum_{i,j=1}^m \mathscr{E}(\mathscr{U}(k_i')k_i',
\mathscr{U}(k_j')k_j') \mathscr{U}(k_i')\mathscr{U}(k_j') \\ &-2 \sum_{i=1}^n\sum_{j=1}^m 
\mathscr{E}(\mathscr{U}(k_i)k_i,
\mathscr{U}(k_j')k_j') \mathscr{U}(k_i)\mathscr{U}(k_j')\Big)\\
&\times \Big[\mathbf{a}_{k_{1}}\cdots \mathbf{a}_{k_{n}}\mathbf{a}_{k_{1}'}^*\cdots \mathbf{a}_{k_{{m}}'}^*\Big].
\end{aligned}
\end{equation}
Let us emphasize that in the case where all $k_i$s and $k_j$s are distinct, above 
expression is an obvious consequence of \eqref{NoiseAction}. On the other hand, in 
case where $k_i$s and $k_j$s \textit{are not} distinct the above identity is still true. 
This can be observed by a straightforward computation taking advantage of 
\eqref{NoiseAction} and identities in \eqref{Rak}. 
 }

 In general, we set
 {
\begin{equation}\begin{aligned}\label{InhomogePhaseReg3}
	\tau_{k_{1}\cdots k_{n},k_{1}'\cdots k_{m}'}\ =\ & -2 \Big( \sum_{i,j=1}^n \mathscr{E}(\mathscr{U}(k_i)k_i,
\mathscr{U}(k_j)k_j) \mathscr{U}(k_i)\mathscr{U}(k_j) \\
&  + \sum_{i,j=1}^m \mathscr{E}(\mathscr{U}(k_i')k_i',
\mathscr{U}(k_j')k_j') \mathscr{U}(k_i')\mathscr{U}(k_j') \\ &-2 \sum_{i=1}^n\sum_{j=1}^m 
\mathscr{E}(\mathscr{U}(k_i)k_i,
\mathscr{U}(k_j')k_j') \mathscr{U}(k_i)\mathscr{U}(k_j')\Big),
		\end{aligned}
\end{equation}
which will play the role of  ``inhomogeneous phase regulators'' in Definition \ref{Def:PhaseRegulator} below.}

 For vectors $k',k''\in\Lambda^*$, we also set $k'=k-\epsilon \mho/2$ and $k''=k+\epsilon \mho/2$ for $\mho\in \Lambda^*_{2/\epsilon}$, we expand, following the   Wigner distribution \eqref{Wigner1}
 
 \begin{equation}
	\begin{aligned}\label{StartPointAverageAlpha}
		& \frac{\partial}{\partial t}\langle \alpha_t(k',1),\alpha_t(k'',-1)\rangle_* \  = \    \ -{\bf i}\lambda \int_{(\Lambda^*)^2}\mathrm{d}k_1\mathrm{d}k_2\ \delta(-k'+k_1'+k_2') \mathcal{M}(k',k_1',k_2')\\
		& \times e^{{\bf i}t(-\omega(k_1')-\omega(k_2')+\omega(k'))} e^{[\tau_{k'',k'}-\tau_{k'',k_1'k_2'}]c_r t}\langle\alpha_t(k'',-1) \alpha_t(k_1',1)\alpha_t(k_2',1)\rangle_*\\
		&\ + {\bf i}\lambda\int_{(\Lambda^*)^2}\mathrm{d}k_1\mathrm{d}k_2\delta(k''-k_1''-k_2'')\mathcal{M}(k'',k_1'',k_2'')e^{[\tau_{k'',k'}-\tau_{k_1''k_2'',k'}]c_r t}\\
		& \ \times   e^{{\bf i}t(\omega(k_1'')+\omega(k_2'')-\omega(k''))}\langle\alpha_t(k_1'',-1) \alpha_t(k',1)\alpha_t(k_2'',-1)\rangle_*,
	\end{aligned}
\end{equation}

We follows \cite{staffilani2021wave} and continue to expand \eqref{StartPointAverageAlpha} using the Duhamel expansions. The expansion strategy will lead to a technical difficulty in treating the wave numbers near the singular manifold $\mathbf{S}$, which appears when performing  the estimates on the oscillatory integrals.  Therefore,    appropriate  cut-off functions that will isolate the singular manifold $\mathbf{S}$ are then needed and we have the following propositions, proved in \cite{staffilani2021wave}, about the existence of such cut-off functions. 
\subsection{Cut-off functions}

We first define the singular manifolds.
\begin{definition}\label{SingularManifold}
	We define the one-dimensional singular manifold $\mathbf{S}_{one}$ to be the subset $\{0,\pm\frac12,\pm\frac14\}$ of $\mathbb{T}$. The $d$-dimensional manifold $\mathbf{S}$ is then the union $\Big(\mathbf{S}_{one}\times\mathbb{T}^{d-1}\Big)\cup \Big(\mathbb{T}\times\mathbf{S}_{one}\times\mathbb{T}^{d-2}\Big)\cup\cdots\cup\Big(\mathbb{T}^{d-1}\times\mathbf{S}_{one}\Big)$. 
\end{definition}
\begin{proposition}[See \cite{staffilani2021wave}]\label{Propo:Phi}
	\begin{itemize}
		\item[(i)] For any $\eth>0$, there is a constant  $\lambda_1>0$ such that for any $(\sigma_1,\sigma_2 )\in\{\pm1\}^2$, $0<\lambda<\lambda_1$, there exist  cut-off functions $$\Psi^a_0(\sigma_1,k_1,\sigma_2,k_2 ), \Psi^a_1(\sigma_1,k_1,\sigma_2,k_2 ) : (\Lambda^*)^2\to [0,1],$$ $$\Psi^a_1(\sigma_1,k_1,\sigma_2,k_2 ) = 1- \Psi^a_0(\sigma_1,k_1,\sigma_2,k_2),$$ 
		and the smooth functions, which are the smooth rescaled versions of  $\Psi^a_0(\sigma_1,k_1,\sigma_2,k_2 ),$ $ \Psi^a_1(\sigma_1,k_1,\sigma_2,k_2 )$
		$$\tilde{\Psi}^a_0(\sigma_1,\tilde k_1,\sigma_2,\tilde k_2 ), \tilde{\Psi}^a_1(\sigma_1,\tilde k_1,\sigma_2,\tilde k_2 ) : \mathbb{T}^{2d}\to [0,1],$$ $$\tilde{\Psi}^a_1(\sigma_1,\tilde k_1,\sigma_2,\tilde k_2 ) = 1- \tilde{\Psi}^a_0(\sigma_1,\tilde k_1,\sigma_2,\tilde k_2 ),$$ such that $\Psi^a_0(\sigma_1,k_1,\sigma_2,k_2)= \tilde\Psi^a_0(\sigma_1,k_1,\sigma_2,k_2 ),$ $ \Psi^a_1(\sigma_1,k_1,\sigma_2,k_2 )= \tilde\Psi^a_1(\sigma_1,k_1,\sigma_2,k_2 ),$ on $(\Lambda^*)^2$ for all mesh size $h$.  Moreover,
		\begin{itemize}
			\item[(A)]$\tilde\Psi_1^a(\sigma_1,\tilde k_1,\sigma_2,\tilde k_2 ) =0$ if $\sigma_1\tilde k_1=\pm\sigma_2\tilde k_2$ or $\tilde k_i=0$, for $i\in\{1,2\}$. 
			\item[(B)] For any $i\in\{1,2\}$
			\begin{equation}\tilde\Psi^a_1(\sigma_1,\tilde k_1,\sigma_2,\tilde k_2 )\lesssim  \langle\ln\lambda\rangle^{\eth/3}d(\tilde k_i,\mathbf{S}),		\end{equation}
			and
			\begin{equation}\tilde\Psi^a_1(\sigma_1,\tilde k_1,\sigma_2,\tilde k_2)\lesssim  \langle\ln\lambda\rangle^{\eth/3}d(\sigma_1\tilde k_1\pm\sigma_2\tilde k_2,\mathbf{S}),\end{equation}			
			where $\mathbf{S}$ is defined in Definition \ref{SingularManifold} and $d(\tilde k_i,\mathbf{S})$, $d(\sigma_1\tilde k_1\pm\sigma_2\tilde k_2,\mathbf{S})$ are the distances between $\tilde k_i$, $\sigma_1\tilde k_1 \pm\sigma_2\tilde k_2$ and the set $\mathbf{S}$;
			\item[(C)] In addition, 
			\begin{equation}
				\label{Propo:Phi:1}\begin{aligned}
					&	0\le \tilde\Psi^a_0\le  \mathbf{1}\Big(d(\sigma_1\tilde k_1+\sigma_2\tilde k_2,\mathbf{S})<\langle\ln\lambda\rangle^{-\eth/3}\Big)+\mathbf{1}\Big(d(\sigma_1\tilde k_1-\sigma_2\tilde k_2,\mathbf{S})<\langle\ln\lambda\rangle^{-\eth/3}\Big)\\
					& +\sum_{i=1}^2\mathbf{1}\Big(d(\tilde k_i,\mathbf{S})<\langle\ln\lambda\rangle^{-\eth/3}\Big).\end{aligned}
			\end{equation}
		\end{itemize}
		\item[(ii)] Let $\eth'_l$, $(l=1,2)$  be either $\eth'_1$ or $\eth'_2$ to be given in Section \ref{Sec:KeyPara}. There is a constant  $\lambda_1'>0$ such that for any $(\sigma_1,\sigma_2 )\in\{\pm1\}^2$, $0<\lambda<\lambda_1'$, there exist  cut-off functions $$\Psi^b_0(\eth'_l,\sigma_1,k_1,\sigma_2,k_2),\Psi^b_1(\eth'_l,\sigma_1,k_1,\sigma_2,k_2): (\Lambda^*)^2\to [0,1],$$ $$\Psi^b_1(\eth'_l,\sigma_1,k_1,\sigma_2,k_2) = 1- \Psi^b_0(\eth'_l,\sigma_1,k_1,\sigma_2,k_2),$$ and the smooth functions, which are the smooth rescaled versions of  $\Psi^b_0(\eth'_l,\sigma_1,k_1,\sigma_2,k_2),$ $\Psi^b_1(\eth'_l,\sigma_1,k_1,\sigma_2,k_2)$ $$\tilde{\Psi}^b_0(\eth'_l,\sigma_1,\tilde k_1,\sigma_2,\tilde k_2): \mathbb{T}^{2d}\to [0,1],$$ $$\tilde{\Psi}^b_1(\eth'_l,\sigma_1,\tilde k_1,\sigma_2,\tilde k_2) = 1- \tilde{\Psi}^b_0(\eth'_l,\sigma_1,\tilde k_1,\sigma_2,\tilde k_2),$$ such that $\Psi^b_0(\eth'_l,\sigma_1,k_1,\sigma_2,k_2)= \tilde\Psi^b_0(\eth'_l,\sigma_1,k_1,\sigma_2,k_2)$, $\Psi^b_1(\eth'_l,\sigma_1,k_1,\sigma_2,k_2)= \tilde\Psi^b_1(\eth'_l,\sigma_1,k_1,\sigma_2,k_2)$  on $(\Lambda^*)^2$ for all mesh size $h$. Moreover,
		\begin{itemize}
			\item[(A')]$\tilde\Psi_1^b(\eth'_l,\sigma_1,\tilde k_1,\sigma_2,\tilde k_2) =0$ if $\sigma_i\tilde k_i=\sigma_j\tilde k_j$ or $\tilde k_i=0$, for $i,j\in\{1,2,3\}$. 
			\item[(B')] For any $i\in\{1,2\}$
			\begin{equation}\tilde\Psi^b_1(\eth'_l,\sigma_1,\tilde k_1,\sigma_2,\tilde k_2)\lesssim  \lambda^{-\eth'_l/3}d(\tilde k_i,\mathbf{S}),		\end{equation}
			and
			\begin{equation}\tilde\Psi^b_1(\eth'_l,\sigma_1,\tilde k_1,\sigma_2,\tilde k_2)\lesssim  \lambda^{-\eth'_l/3}d(\sigma_1\tilde k_1\pm\sigma_2\tilde k_2,\mathbf{S}),\end{equation}

			where $\mathbf{S}$ is defined in Definition \ref{SingularManifold};
			\item[(C')] In addition, 
			\begin{equation}
				\label{Propo:Phi:1:a}\begin{aligned}
					&	0\le \tilde\Psi^b_0(\eth'_l)\le  \mathbf{1}\Big(d(\sigma_1\tilde k_1+\sigma_2\tilde k_2,\mathbf{S})<\lambda^{\eth'_l/3}\Big)+\mathbf{1}\Big(d(\sigma_1\tilde k_1-\sigma_2\tilde k_2,\mathbf{S})<\lambda^{\eth'_l/3}\Big)\\
					& +\sum_{i=1}^2\mathbf{1}\Big(d(\tilde k_i,\mathbf{S})<\lambda^{\eth'_l/3}\Big).\end{aligned}
			\end{equation}
		\end{itemize}
	\end{itemize}
	
\end{proposition}

Let $k_1,k_2$ be the two momenta in the delta function $\delta(k_0-k_1-k_2)$ and suppose that $k_3$ is a momentum attached to a different delta function.   Now, we denote $V=2\pi(k_2-k_1)$ and $W=2\pi(k_3-k_1)$.  Using the definition of $\tilde{r}_1,\tilde{r}_2,\tilde{r}_3$, of Lemma \ref{Lemm:Angle}  and Lemma \ref{Lemm:Bessel3} of Section \ref{Sec:Techical}, for $V=2\pi(k_2-k_1)$ and $W=2\pi(k_3-k_1)$, we define, for $i=1,2,3$

\begin{equation}
	\begin{aligned}\label{Lemm:Bessel2:2bb8:1}
		\cos(\Upsilon_*(\tilde{r}_i,\aleph^*,\aleph_*))
		\ = \ & {\Big|\aleph^*\tilde{r}_i+\aleph_*+\tilde{r}_i\cos(V_1)e^{{\bf i}2V_j}+\cos(W_1)e^{{\bf i}2W_j}\Big|}\\
		&\ \times \Big[\Big|\aleph^*\tilde{r}_i+\aleph_*+\tilde{r}_i\cos(V_1)e^{{\bf i}2V_j}+\tilde{r}_i\cos(W_1)e^{{\bf i}2W_j}\Big|^2\\
		& \  +\Big|\tilde{r}_i\sin(V_1)e^{{\bf i}2V_j}+\sin(W_1)e^{{\bf i}2W_j}\Big|^2\Big]^{-\frac12},\\
		\sin(\Upsilon_*(\tilde{r}_i,\aleph^*,\aleph_*))\
		= \ & {\Big|\tilde{r}_i\sin(V_1)e^{{\bf i}2V_j}+\sin(W_1)e^{{\bf i}2W_j}\Big|}\\
		&\ \times \Big[\Big|\aleph^*\tilde{r}_i+\aleph_*+\tilde{r}_i\cos(V_1)e^{{\bf i}2V_j}+\cos(W_1)e^{{\bf i}2W_j}\Big|^2\\
		& \ +\Big|\tilde{r}_i\sin(V_1)e^{{\bf i}2V_j}+\sin(W_1)e^{{\bf i}2W_j}\Big|^2\Big]^{-\frac12},
	\end{aligned}
\end{equation}
for $\aleph^*,\aleph_*=\pm 1.$ We set \begin{equation}\label{Setstar}\mathbf{S}_*=\Big\{k=(k^1,\cdots,k^d) \Big| 2\pi k^1= \Upsilon_*(\tilde{r}_i,\aleph^*,\aleph_*)\pm \pi; i=1,2,3; \aleph^*,\aleph_*=\pm 1\Big\}.\end{equation}

We have the following proposition, whose proof is  the same as  Proposition \ref{Propo:Phi} and hence we will omit it.
\begin{proposition}\label{Propo:Phi3A} Let $k_1,k_2$ be two momenta in the delta function $\delta(k_0-k_1-k_2)$. Let $k_0'$ be a momentum belonging to a different delta function   that is written as $\delta(k_0'-k_1'-k_2')$.    We denote $V^1=2\pi(k_2-k_1)=(V^1_1,\cdots,V^1_d)$, $V^2=2\pi(k_2+k_1)=(V^2_1,\cdots,V^2_d)$ and $W^1=2\pi(k_1'-k_1)=(W^1_1,\cdots,W^1_d),$ $W^2=2\pi(k_1'+k_1)=(W^2_1,\cdots,W^2_d),$
	
	For any $\eth>0$, there is a constant $\lambda_1>0$ such that for  $0<\lambda<\lambda_1$, there exists a  cut-off function $$\Psi _2( k_0,k_1,k_2,k_0',k_1',k_2'): (\Lambda^*)^6\to [0,1],$$  
	and the smooth rescaled version   
	$$\tilde{\Psi}_2(\tilde k_0,\tilde k_1,\tilde k_2,\tilde k_0',\tilde k_1',\tilde k_2'): \mathbb{T}^{6d}\to [0,1],$$ such that $\Psi_2(k_0,k_1,k_2,k_0',k_1',k_2')= \tilde\Psi_2( k_0,k_1,k_2,k_0',k_1',k_2')$ on $ (\Lambda^*)^6$ for all mesh size $h$ and $\tilde{\Psi}_2(\tilde k_1,\tilde k_0,$ $\tilde k_3,\tilde k_0',\tilde k_1',\tilde k_2')=1$ when $	|V^i_j|, \Big|V^i_j-\frac{\pi}{2}\Big|, |V^i_j-{\pi}|,|W^i_j|, \Big|W_j^i-\frac{\pi}{2}\Big|, |V_j^i\pm W_j^{i'}|, |W_j^i-{\pi}|, |k_0-k_0'|>|\ln\lambda|^{-\eth}$ and $\Psi_2(\tilde k_1,\tilde k_0,\tilde k_3,\tilde k_0',\tilde k_1',\tilde k_2')=0$ when either $	|V^i_j|, \Big|V^i_j-\frac{\pi}{2}\Big|, |V^i_j-{\pi}|,|W^i_j|, \Big|W^i_j-\frac{\pi}{2}\Big|, |V^i_j\pm W^{i'}_j|, |k_0-k_0'|$ or $ |W^i_j-{\pi}|<|\ln\lambda|^{-\eth}/2$, for $j=1,\cdots,d$, $i,i'=1,2$.

	Moreover, there also exists a  cut-off function
	$$\Psi _3( k_0,k_1,k_2,k_0',k_1',k_2'): (\Lambda^*)^6\to [0,1],$$  
	and the smooth rescaled version     
	$$\tilde{\Psi}_3(\tilde k_0,\tilde k_1,\tilde k_2,\tilde k_0',\tilde k_1',\tilde k_2'): \mathbb{T}^{6d}\to [0,1],$$ such that $\Psi_3( k_0,k_1,k_2,k_0',k_1',k_2')= \tilde\Psi_3(\tilde k_0,\tilde k_1,\tilde k_2,\tilde k_0',\tilde k_1',\tilde k_2')$  on $ (\Lambda^*)^6$ for all mesh size $h$. Moreover, $\tilde{\Psi}_3(\tilde k_0,\tilde k_1,\tilde k_2,\tilde k_0',\tilde k_1',\tilde k_2')=1$ when $d(\tilde k^1_1,\mathbf{S}_*)>|\ln\lambda|^{-\eth}$ and $\tilde{\Psi}_3(\tilde k_0,\tilde k_1,\tilde k_2,\tilde k_0',\tilde k_1',\tilde k_2')=0$ when $d(\tilde k^1_1,\mathbf{S}_*)<|\ln\lambda|^{-\eth}/2.$
	
	We finally set \begin{equation}\begin{aligned}\bar{\Psi}_4( k_0,k_1,k_2,k_0',k_1',k_2')\ =\ &\Psi_3( k_0,k_1,k_2,k_0',k_1',k_2')\Psi_2( k_0,k_1,k_2,k_0',k_1',k_2')\\
	&\times \Psi_3( k_0,k_1,k_2,k_0',k_2',k_1')\Psi_2( k_0,k_1,k_2,k_0',k_2',k_1')\\
&\times \Psi_3( k_0,k_2,k_1,k_0',k_2',k_1')\Psi_2( k_0,k_2,k_1,k_0',k_2',k_1')\\
&\times \Psi_3( k_0,k_2,k_1,k_0',k_1',k_2')\Psi_2( k_0,k_2,k_1,k_0',k_1',k_2'),
\end{aligned}\end{equation}
and 
\begin{equation}\begin{aligned}
{\Psi}_4( k_0,k_1,k_2,k_0',k_1',k_2') \ = \ &	\bar{\Psi}_4( k_0,k_1,k_2,k_0',k_1',k_2')\mathbf{1}_{\Big\{\int_{(\Lambda^*)^6}\mathrm{d}k_0\mathrm{d}k_1\mathrm{d}k_2\mathrm{d}k_0'\mathrm{d}k_1'\mathrm{d}k_2'\bar{\Psi}_4( k_0,k_1,k_2,k_0',k_1',k_2')\ne 0\Big\}}\\
& + \mathbf{1}_{\Big\{\int_{(\Lambda^*)^6}\mathrm{d}k_0\mathrm{d}k_1\mathrm{d}k_2\mathrm{d}k_0'\mathrm{d}k_1'\mathrm{d}k_2'\bar{\Psi}_4( k_0,k_1,k_2,k_0',k_1',k_2')= 0\Big\}}.\end{aligned}
\end{equation}
The smooth rescaled versions     
$$\bar{\Psi}_4^o(\tilde k_0,\tilde k_1,\tilde k_2,\tilde k_0',\tilde k_1',\tilde k_2'): \mathbb{T}^{6d}\to [0,1],$$

$$\tilde{\Psi}_4(\tilde k_0,\tilde k_1,\tilde k_2,\tilde k_0',\tilde k_1',\tilde k_2'): \mathbb{T}^{6d}\to [0,1]$$
	of $\bar{\Psi}_4( k_0,k_1,k_2,k_0',k_1',k_2') $, ${\Psi}_4( k_0,k_1,k_2,k_0',k_1',k_2') $ are defined such that  $\Psi_4( k_0,k_1,k_2,k_0',k_1',k_2')= \tilde\Psi_4(\tilde k_0,\tilde k_1,\tilde k_2,\tilde k_0',\tilde k_1',\tilde k_2')$, $\bar{\Psi}_4( k_0,k_1,k_2,k_0',k_1',k_2')= \bar{\Psi}^o_4(\tilde k_0,\tilde k_1,\tilde k_2,\tilde k_0',\tilde k_1',\tilde k_2')$  on $ (\Lambda^*)^6$ for all mesh size $h$.

\end{proposition} 

\begin{definition}\label{def:Phi3}
	For any momenta $k_1,k_2\in\Lambda^*$ that are associated to a delta function $\delta(k_0-k_1-k_2)$ of a vertex $v_i$, we define $\mathcal{S}^*(k_0)$ to be the set of all triple $(k_0',k_1',k_2')$ that are associated to a vertex $v_j$ with $j>i$ with the delta function  $\delta(k_0'-k_1'-k_2')$. We now set a new cut-off function that combines all of the previous ones defined above as follows
	\begin{equation}\begin{aligned}		\label{def:Phi3:1}\Phi_1^a(\sigma_1, k_1,\sigma_2, k_2 ) \ = \ & \Psi_1^a(\sigma_1, k_1,\sigma_2, k_2 )\prod_{(k_0',k_1',k_2')\in\mathcal{S}^*(k_0)}\Psi_4( k_0,k_1,k_2,k_0',k_1',k_2'),\end{aligned}	
	\end{equation}
	$$\Phi^a_0(\sigma_1,k_1,\sigma_2,k_2 ) = 1- \Phi^a_1(\sigma_1,k_1,\sigma_2,k_2),$$ 
	and 
	
	\begin{equation}
		\label{def:Phi3:2}\Phi_1^b(\eth'_l,\sigma_1, k_1,\sigma_2, k_2 ) \ = \ \Psi_1^b(\eth'_l,\sigma_1, k_1,\sigma_2, k_2 ),
	\end{equation}
	$$\Phi^b_0(\eth'_l,\sigma_1,k_1,\sigma_2,k_2 ) = 1- \Phi^b_0(\eth'_l,\sigma_1,k_1,\sigma_2,k_2).$$ 
	Moreover,  we also have the smooth rescaled
	functions 
	$$\tilde{\Phi}^a_0(\sigma_1,\tilde k_1,\sigma_2,\tilde k_2 ), \tilde{\Phi}^a_1(\sigma_1,\tilde k_1,\sigma_2,\tilde k_2 ): \mathbb{T}^{2d}\to [0,1],$$ $$ \tilde{\Phi}^a_1(\sigma_1,\tilde k_1,\sigma_2,\tilde k_2 ) = 1- \tilde{\Phi}^a_0(\sigma_1,\tilde k_1,\sigma_2,\tilde k_2 ),$$ such that $\Phi^a_0(\sigma_1,k_1,\sigma_2,k_2)= \tilde\Phi^a_0(\sigma_1,k_1,\sigma_2,k_2 ),$ $\Phi^a_1(\sigma_1,k_1,\sigma_2,k_2 )= \tilde\Phi^a_1(\sigma_1,k_1,\sigma_2,k_2 )$ on $(\Lambda^*)^2$ for all mesh size $h$;
	as well as
	$$\tilde{\Phi}^b_0(\eth'_l,\sigma_1,\tilde k_1,\sigma_2,\tilde k_2 ), \tilde{\Phi}^b_1(\eth'_l,\sigma_1,\tilde k_1,\sigma_2,\tilde k_2 ) : \mathbb{T}^{2d}\to [0,1],$$ $$\tilde{\Phi}^b_1(\eth'_l,\sigma_1,\tilde k_1,\sigma_2,\tilde k_2 ) = 1- \tilde{\Phi}^b_0(\eth'_l,\sigma_1,\tilde k_1,\sigma_2,\tilde k_2 ),$$ such that $\Phi^b_0(\eth'_l,\sigma_1,k_1,\sigma_2,k_2)= \tilde\Phi^b_0(\eth'_l,\sigma_1,k_1,\sigma_2,k_2 ),$ $ \Phi^b_1(\eth'_l,\sigma_1,k_1,\sigma_2,k_2 )= \tilde\Phi^b_1(\eth'_l,\sigma_1,k_1,\sigma_2,k_2 )$ on $(\Lambda^*)^2$ for all mesh size $h$.
	
\end{definition}  

\subsection{Duhamel expansions}

Plugging the expression $1=\Phi_0^b+\Phi_1^b$ into \eqref{StartPointAverageAlpha}, we find 

\begin{equation}
	\begin{aligned}\label{StartPointAverageAlpha2}
		& \frac{\partial}{\partial t}\epsilon^d\langle \alpha_t(k',1),\alpha_t(k'',-1)\rangle_* \  = \    \ -{\bf i}\lambda \epsilon^d\int_{(\Lambda^*)^2}\mathrm{d}k_1\mathrm{d}k_2\ \delta(-k'+k_1'+k_2') \mathcal{M}(k',k_1',k_2')\\
		& \times [\Phi^b_0(\eth_1',1, k_1',1, k_2')+\Phi^b_1(\eth_1',1, k_1',1, k_2')]e^{{\bf i}t(-\omega(k_1')-\omega(k_2')+\omega(k'))}\\
		&\times e^{[\tau_{k'',k'}-\tau_{k'',k_1'k_2'}]c_r t}\langle\alpha_t(k'',-1) \alpha_t(k_1',1)\alpha_t(k_2',1)\rangle_*\\
		&\ + {\bf i}\lambda\epsilon^d\int_{(\Lambda^*)^2}\mathrm{d}k_1\mathrm{d}k_2\delta(k''-k_1''-k_2'')\mathcal{M}(k'',k_1'',k_2'')e^{[\tau_{k'',k'}-\tau_{k_1''k_2'',k'}]c_r t}\\
		& \ \times   [\Phi^b_0(\eth_1',1, k_1'',1, k_2'')+\Phi^b_1(\eth_1',1, k_1'',1, k_2'')]e^{{\bf i}t(\omega(k_1'')+\omega(k_2'')-\omega(k''))}\langle\alpha_t(k_1'',-1) \alpha_t(k',1)\alpha_t(k_2'',-1)\rangle_*.
	\end{aligned}
\end{equation}
In shortening the notations, we set
\begin{equation}
	\label{Def:Omega}\begin{aligned}
		\mathbf{X}(\sigma,k,\sigma',k_1,\sigma'',k_2) \ = & \ \sigma\omega(k) \ + \ \sigma' \omega(k_1) \ + \ \sigma'' \omega(k_2).
	\end{aligned}
\end{equation}

We then expand \eqref{StartPointAverageAlpha2}   to obtain a multilayer Duhamel expansion. In this expansion procedure, we expand  terms containing $\Phi_1^b(\eth_1')$ and then $\Phi_1^b(\eth_2)$, while terms that contain $\Phi_0^b$ are left unchanged. In other words,  the  Duhamel expansions are performed only on terms that have $\Phi_1^b$. Starting from the third iterations, we will use the same strategy, except that $\Phi_0^b,\Phi_1^b$ are replace by $\Phi_0^a,\Phi_1^a$.

Now, we follow the soft partial time integration technique.
Using \eqref{StartPoint} and the definition in \eqref{Def:Omega} we obtain
\begin{equation}
	\begin{aligned} \label{eq:duhamelproduct}
		\mathrm{d}\prod_{i=1}^2 & \alpha_t(k_i,\sigma_i) \ = \  \\
		=  &\ \varsigma\int_0^t\mathrm{d}s e^{-\varsigma (t-s)}\prod_{i=1}^2 \alpha_s(k_i,\sigma_i)\ + \ e^{-\varsigma t}\prod_{i=1}^2\alpha_0(k_i,\sigma_i)  \ -\ {\bf i}\sqrt{2c_r}\int_0^t\mathrm{d}s e^{\varsigma t}\alpha_s(k_j,\sigma_j)\circ\mathrm{d}W(t).\\
		& \ - \ {\bf i}\lambda \sum_{j=1}^2\int_0^t\mathrm{d}s\sigma_j\prod_{i=1,i\ne j}^{2} \alpha_t(k_i,\sigma_i) \Big[  \iint_{(\Lambda^*)^2} \mathrm{d}k'_1 \mathrm{d}k'_2\delta(-k_j+k_1'+k_2') \mathcal{M}(k_j,k_1',k_2') \\
		&\  \times \exp\Big[-\varsigma (t-s)-{\bf i}s\mathbf{X}(\sigma_j,k_j,\sigma_j,k_1',\sigma_j,k_2')\Big]\alpha_s(k_1',\sigma_j) \alpha_s(k_2',\sigma_j)\Big],
\end{aligned}\end{equation}
where $\varsigma>0$ is a control parameter to  be specified later. The solf partial time integration will be used below in the full Duhamel expansion.

\medskip
{\it The full Duhamel expansion.} 
By repeating the expansion process $\mathbf{N}$ times, we have a multi-layer expression  in which  the time interval $[0,t]$ is divided into $\mathbf{N}+1$ time slices $[0,s_0],$ $[s_0,s_0+s_1]$, $\dots$, $[s_0+\cdots+s_{\mathfrak{N}-1},t]$ and $t=s_0+\cdots+s_{\mathbf{N}}.$ The final result of this process can be written as follows
\begin{equation}
	\begin{aligned} \label{eq:fullDuhamel}
		& \langle a_t(k',1)a_t(k'',-1)\rangle_*\  
		=  \  \sum_{n=0}^{\mathbf{N}-1}\mathcal{F}_{n}^0(t,k',1,k'',-1,\Gamma)[\alpha_0]\\
		& \ + \ \sum_{n=0}^{\mathbf{N}-1}\varsigma_n\int_0^t\mathrm{d}s \mathcal{F}_{n}^1(s,t,k',1,k'',-1,\Gamma)[\alpha_s]\ +\ \sum_{n=1}^{\mathbf{N}}\int_0^t\mathrm{d}s \mathcal{F}_{n}^2(s,t,k',1,k'',-1,\Gamma)[\alpha_s]\\
		& \ + \ \int_0^t\mathrm{d}s \mathcal{F}_{\mathbf{N}}^3(s,t,k',1,k'',-1,\Gamma)[\alpha_s],
\end{aligned}\end{equation} 
in which   $\Gamma$ is the soft partial time integration vector 
\begin{equation}
	\label{SolfPartTimeVector}\Gamma \ = \ (\varsigma_0,\cdots,\varsigma_{\mathbf N-1}) \mbox{ in } \mathbb{R}_+^{\mathbf N}.
\end{equation}
Note that in \eqref{eq:fullDuhamel}, we expand $\langle a_t(k',1)a_t(k'',-1)\rangle_*$ instead of $\langle \alpha_t(k',1)\alpha_t(k'',-1)\rangle_*$.

We start by presenting the definition of the regulator parameters. 
\begin{definition}[Inhomogeneous Phase Regulators]\label{Def:PhaseRegulator} Let us consider the signs $\sigma_{i,j}$, with $j\in\{1,\cdots,n-i
	+2\}$. Suppose that $\sigma_{i,j_1}=\cdots= \sigma_{i,j_{n'}}=1$ and $\sigma_{i,j_{n'}}=\cdots= \sigma_{i,j_{n-i+2}}=-1$.
	We define the inhomogeneous phase regulators to be the quantities
	\begin{equation}
		\label{Def:TauGen}
		\tau_i \ = \ \tau_{k_{i,j_{n'}}\cdots k_{i,j_{n-i+2}},k_{i,j_1}\cdots k_{i,j_{n'}}}c_r \quad \mbox{ for } i\in\{0,\cdots,n-1\},\ \ \ \tau_n=\tau_{k'',k'}.
	\end{equation}
\end{definition}

The first quantity takes the following explicit form   (for $n\ge 1$)
\begin{equation}
	\begin{aligned} \label{eq:DefE0}
		&\mathcal{F}_{n}^0(t,k_{n,2},\sigma_{n,2},k_{n,1},\sigma_{n,1},\Gamma)[\alpha]\\
		\ = \  &  
		(-{\bf i}\lambda)^n\sum_{\substack{\rho_i\in\{1,\cdots,n-i+2\},\\ i\in\{0,\cdots,n-1\}}}\sum_{\substack{\bar\sigma\in \{\pm1\}^{\mathcal{I}_n},\\ \sigma_{i,\rho_i}+\sigma_{i-1,\rho_i}+\sigma_{i-1,\rho_i+1}\ne \pm3,
				\\ \sigma_{i-1,\rho_i}\sigma_{i-1,\rho_i+1}= 1}} \int_{(\Lambda^*)^{\mathcal{I}_n}}\mathrm{d}\bar{k}\Delta_{n,\rho}(\bar{k},\bar\sigma)\\
		&\ \ \ \ \ \ \ \ \ \ \ \ \ \times \prod_{i=1}^n\Big[\sigma_{i,\rho_i}\mathcal{M}( k_{i,\rho_i}, k_{i-1,\rho_i}, k_{i-1,\rho_i+1})\left\langle\prod_{i=1}^{n+2}\alpha(k_{0,i},\sigma_{0,i})\right\rangle_*\\
		&\ \ \ \ \ \ \ \ \ \ \ \ \ \times \Phi_{1,i}( \sigma_{i-1,\rho_i},k_{i-1,\rho_i}, \sigma_{i-1,\rho_i+1},k_{i-1,\rho_i+1})\Big]\\
		&\ \ \ \ \ \ \ \ \ \ \ \ \ \times \int_{(\mathbb{R}_+)^{\{0,\cdots,n\}}}\mathrm{d}\bar{s} \delta\left(t-\sum_{i=0}^ns_i\right)\prod_{i=0}^{n}e^{-s_i[\varsigma_{n-i}+\tau_i]}\\
		&\ \ \ \ \ \ \ \ \ \ \ \ \   \times\prod_{i=1}^{n}e^{-{\bf i}t_i(s)\mathbf{X}(\sigma_{i,\rho_i}, k_{i,\rho_i},\sigma_{i-1,\rho_i},k_{i-1,\rho_i},\sigma_{i-1,\rho_i},k_{i-1,\rho_i+1})},
\end{aligned}\end{equation} 
in which 
\begin{equation}
	\label{CutoffPhii}\begin{aligned}
		& \Phi_{1,1} \ = \ \Phi_{1,1}(n) \ = \  \Phi_{1}^b(\eth'_1)  \mbox{ for } n=1,\\
		& \Phi_{1,2} \ = \ \Phi_{1,2}(n) \ = \  \Phi_{1}^b(\eth'_1)  \mbox{ for } n=2,\\
		& \Phi_{1,1} \ = \ \Phi_{1,1}(n) \ = \  \Phi_{1}^b(\eth'_2)  \mbox{ for } n=2,\\
		&  \Phi_{1,n} \ = \ \Phi_{1,n}(n) \ = \ \Phi_{1}^b(\eth'_1), \mbox{ for } n>2,\\
		& \Phi_{1,n-1} \ = \ \Phi_{1,n-1}(n) \ = \ \Phi_{1}^b(\eth'_2), \mbox{ for } n>2, \\
		&  \Phi_{1,i} \ = \ \Phi_{1,i}(n) \ = \ \Phi_{1}^a \mbox{ for } 1\le i\le n-2, n>2, \\
		& \Phi_{1,n} + \Phi_{0,n} \ = \ 1.  \end{aligned}
\end{equation}
The total phases can be written as
\begin{equation}\label{GraphSec:E1}
	\begin{aligned}
		& \prod_{i=0}^{n}e^{-s_i\varsigma_{n-i}}
		\prod_{i=1}^{n}e^{-\mathbf{i}t_i(s)\mathbf{X}(\sigma_{i,\rho_i}, k_{i,\rho_i},\sigma_{i-1,\rho_i},k_{i-1,\rho_i},\sigma_{i-1,\rho_i},k_{i-1,\rho_i+1})}\\
		\ =\ &\prod_{i=0}^{n}e^{-s_i\varsigma_{n-i}}
		\prod_{i=1}^{n}e^{-\mathbf{i}\left(\sum_{j=1}^{i-1}s_j\right)\mathbf{X}(\sigma_{i,\rho_i}, k_{i,\rho_i},\sigma_{i-1,\rho_i},k_{i-1,\rho_i},\sigma_{i-1,\rho_i},k_{i-1,\rho_i+1})}
		\ =\   \prod_{i=0}^n e^{-\mathbf{i} s_i \vartheta_i},\end{aligned}
\end{equation}
in which 
\begin{equation}\label{GraphSec:E2}
	\begin{aligned}
		\vartheta_i \ := \ & \sum_{l=i+1}^n \mathbf{X}(\sigma_{l,\rho_l},k_{l,\rho_l},\sigma_{l-1,\rho_l},  k_{l-1,\rho_l},\sigma_{l-1,\rho_l+1}, k_{l-1,\rho_l+1}) - {\bf i}\varsigma_{n-i}.
	\end{aligned}
\end{equation} 
  We set $\vartheta_n=0$. 
Moreover, we have used the notations for the set of indices
\begin{equation}
	\label{IndexSet2}\mathcal{I}_{n}  \ := \ \{(j,l) \  | \ 0\le j\le n-1, 1\le l \le n-j+2\}. \end{equation}

The quantities $\tau_i$ below are the generalization of \eqref{InhomogePhaseReg3}, and we call them the ``inhomogeneous phase regulators''.

We also set
\begin{equation}
	\label{Def:Ti}t_i(s)=\sum_{j=1}^{i-1}s_j,\end{equation}
$\bar{s}$, $\bar{k}$ and $\bar\sigma$ are vectors representing all of the quantities $s_i$, $k_{i,j},\sigma_{i,j}$ appearing in the integration. The number $\rho_i$ encodes the position where the delta function
\begin{equation}
	\label{Def:Rho}\delta\left(\sigma_{i,\rho_i}k_{i,\rho_i}+\sigma_{i-1,\rho_i}k_{i-1,\rho_i}+\sigma_{i-1,\rho_i+1}k_{i-1,\rho_i+1}\right)
\end{equation}
appears.
In the above summation, we only allow $$(\sigma_{i,\rho_i},\sigma_{i-1,\rho_i},\sigma_{i-1,\rho_i+1})\ne (+1,+1,+1),(-1,-1,-1),$$
$$\mbox{ and }\sigma_{i-1,\rho_i}\sigma_{i-1,\rho_i+1}= 1.$$ Moreover, $$\mathcal{F}_{0}^0(t,k'',-1,k',1,\Gamma)[\alpha]=e^{-\varsigma_0 t}\langle\alpha(k',1)\alpha(k'',-1)\rangle_*$$
and finally
the function $\Delta_{n,\rho}$ contains all of the $\delta$-functions, which reads
\begin{equation}
	\begin{aligned} \label{eq:DefDelta}
		\Delta_{n,\rho}(\bar{k},\bar\sigma) 
		=  &\  
		\prod_{i=1}^{n}\Big\{\prod_{l=1}^{\rho_i-1}\Big[\delta(k_{i,l}-k_{i-1,l})\mathbf{1}(\sigma_{i,l}=\sigma_{i-1,l})\Big] \\
		&\times \mathbf{1}(\sigma_{i,\rho_i}=-\sigma_{i-1,\rho_i}) \delta\left(\sigma_{i,\rho_i}k_{i,\rho_i}+\sigma_{i-1,\rho_i}k_{i-1,\rho_i}+\sigma_{i-1,\rho_i+1}k_{i-1,\rho_i+1}\right)\\
		&\times\prod_{l=\rho_i+1}^{n-i+2}\Big[\delta(k_{i,l}-k_{i-1,l+1})\mathbf{1}(\sigma_{i,l}=\sigma_{i-1,l
			+1})\Big]\mathbf{1}(k_{n,1}=k_{n,2})\mathbf{1}(\sigma_{n,1}+\sigma_{n,2}=0)\Big\},
\end{aligned}\end{equation} 
where $\bar{k}$ and $\bar\sigma$ represent all of the quantities $k_{i,j},\sigma_{i,j}$ appearing in the formula.

We define the last  quantity $\mathcal{F}_n^3$  as:
\begin{equation}
	\begin{aligned} \label{eq:DefE3}
		& \mathcal{F}^3_{n}(s_0,t,k_{n,2},\sigma_{n,2},k_{n,1},\sigma_{n,1},\Gamma)[\alpha]\\ 
		= & \  
		(-{\bf i}\lambda)^n\sum_{\substack{\rho_i\in\{1,\cdots,n-i+2\},\\ i\in\{0,\cdots,n-1\}}}\sum_{\substack{\bar\sigma\in \{\pm1\}^{\mathcal{I}_n},\\ \sigma_{i,\rho_i}+\sigma_{i-1,\rho_i}+\sigma_{i-1,\rho_i+1}\ne \pm3,
				\\ \sigma_{i-1,\rho_i}\sigma_{i-1,\rho_i+1}= 1}} \int_{(\Lambda^*)^{\mathcal{I}_n}}\mathrm{d}\bar{k}\Delta_{n,\rho}(\bar{k},\bar\sigma)\\
		&\ \ \ \ \ \ \ \ \ \ \ \ \ \times \prod_{i=1}^n\Big[\sigma_{i,\rho_i}\mathcal{M}( k_{i,\rho_i}, k_{i-1,\rho_i}, k_{i-1,\rho_i+1})\Phi_{1,i}(\sigma_{i,\rho_i},k_{i,\rho_i}, \sigma_{i-1,\rho_i},k_{i-1,\rho_i}, \sigma_{i-1,\rho_i+1},k_{i-1,\rho_i+1})\Big]\\
		&\ \ \ \ \ \ \ \ \ \ \ \ \ \times \left\langle\prod_{i=1}^{n+2}\alpha(k_{0,i},\sigma_{0,i})\right\rangle_* \int_{(\mathbb{R}_+)^{\{1,\cdots,n\}}}\mathrm{d}\bar{s} \delta\left(t-\sum_{i=0}^ns_i\right)\prod_{i=1}^{n}e^{-s_i[\varsigma_{n-i}+\tau_i]}e^{-s_0\tau_0}\\
		&\ \ \ \ \ \ \ \ \ \ \ \ \   \times\prod_{i=1}^{n}e^{-{\bf i}t_i(s)\mathbf{X}(\sigma_{i,\rho_i}, k_{i,\rho_i},\sigma_{i-1,\rho_i},k_{i-1,\rho_i},\sigma_{i-1,\rho_i},k_{i-1,\rho_i+1})}.
\end{aligned}\end{equation} 
Note that in \eqref{eq:DefE0} the integration is taken over $(\mathbb{R}_+)^{\{0,\cdots,n\}}$ and in \eqref{eq:DefE3} it is over $(\mathbb{R}_+)^{\{1,\cdots,n\}}.$ 

\smallskip
The second quantity $\mathcal{F}_n^1$ is defined, using the last quantity, as follows. We set, 
$$\mathcal{F}^1_{0}(s,t,k'',-1,k',1,\Gamma)[\alpha]=e^{-\varsigma_0 (t-s)}\langle\alpha(k',1) \alpha(k'',-1)\rangle_*,$$ 
and for $n>0$, we set
\begin{equation}
	\label{eq:DefE1}
	\mathcal{F}^1_n(s,t,k_{n,2},\sigma_{n,2},k_{n,1},\sigma_{n,1},\Gamma)[\alpha] \ = \ \int_0^{t-s}\mathrm{d}r e^{-r\varsigma_n}\mathcal{F}^3_{n}(s+r,t,k_{n,1},\sigma_{n,1},\Gamma)[\alpha].
\end{equation}

And, finally,
\begin{equation}
	\begin{aligned} \label{eq:DefE2}
		&\mathcal{F}^2_{n}(s_0,t,k_{n,2},\sigma_{n,2},k_{n,1},\sigma_{n,1},\Gamma)[\alpha]\\
		=  \ &
		(-{\bf i}\lambda)^n\sum_{\substack{\rho_i\in\{1,\cdots,n-i+2\},\\ i\in\{0,\cdots,n-1\}}}\sum_{\substack{\bar\sigma\in \{\pm1\}^{\mathcal{I}_n},\\ \sigma_{i,\rho_i}+\sigma_{i-1,\rho_i}+\sigma_{i-1,\rho_i+1}\ne \pm3,
				\\ \sigma_{i-1,\rho_i}\sigma_{i-1,\rho_i+1}= 1}} \int_{(\Lambda^*)^{\mathcal{I}_n}}\mathrm{d}\bar{k}\Delta_{n,\rho}(\bar{k},\bar\sigma)\\
		&  \ \times \sigma_{1,\rho_1}\mathcal{M}( k_{1,\rho_1}, k_{0,\rho_1}, k_{0,\rho_1+1})\Phi_{0,1}(\sigma_{0,\rho_1}, k_{0,\rho_1},\sigma_{0,\rho_1+1},k_{0,\rho_1+1})\\
		& \ \times \prod_{i=2}^n\Big[\sigma_{i,\rho_i}\mathcal{M}( k_{i,\rho_i}, k_{i-1,\rho_i}, k_{i-1,\rho_i+1})\Phi_{1,i}(\sigma_{i-1,\rho_i}, k_{i-1,\rho_i},\sigma_{i-1,\rho_i+1},k_{i-1,\rho_i+1})\Big] \\
		&\ \times   \left\langle\prod_{i=1}^{n+2}\alpha(k_{0,i},\sigma_{0,i})\right\rangle_*\int_{(\mathbb{R}_+)^{\{1,\cdots,n\}}}\mathrm{d}\bar{s} \delta\left(t-\sum_{i=0}^ns_i\right)\prod_{i=1}^{n}e^{-s_i[\varsigma_{n-i}+\tau_i]}\\
		& \ \times\prod_{i=1}^{n}e^{-{\bf i}t_i(s)\mathbf{X}(\sigma_{i,\rho_i}, k_{i,\rho_i},\sigma_{i-1,\rho_i}, k_{i-1,\rho_i},\sigma_{i-1,\rho_i+1},  k_{i-1,\rho_i+1})}e^{-s_0\tau_0}.
\end{aligned}\end{equation}

The multi-layer Duhamel expansion can then be expressed via Feynman diagrams, that will be described in Section \ref{Duhamel} below.

\subsection{Wick's Theorem and Applications}\label{Sec:Cum}

\begin{definition}
	\label{Def:Truncation}
	For any finite, non-empty set $J$, we denote by $\mathcal{P}(J)$ the set of its  partition: $S\in \mathcal{P}(J)$ if and only if  each $B\in S$ is  non-empty, $\cup_{B\in S} B=J$,  and if $B, B'\in S$, $B\ne B'$ then $B'\cap B=\emptyset$. Moreover, $\mathcal{P}(\emptyset)=\emptyset$. We also set $\bowtie(S)$ to be either $1$ or $-1$ according to whether the partition is even or odd. 
	
\end{definition}
\begin{definition}
Let us recall that the permanent \cite{ryser1963combinatorial} of an $n\times n$ complex matrix ${M}= \Big[{M}_{i,j}\Big]_{i,j=1}^n$ is 
\begin{equation}
	\label{Permanent}
	\mbox{per}(M) \ = \ \sum_{\pi\in \mathscr{S}(n)}\prod_{i=1}^nM_{i,\pi(i)},
\end{equation}
where $\mathscr{S}$ is the symmetric group of permutations on $\{1,\cdots,n\}$.
\end{definition}

\begin{proposition}[Wick's Theorem]\label{Propo:Wick}
Let $(\tilde{a}_{1},\cdots, \tilde{a}_{n})$ is a circular symmetric $n$-random variable, whose pseudo-covariance matrix is simply the zero matrix, then
\begin{equation}
	\label{Propo:Wick:1}
	\Big\langle \tilde{a}_{1},\cdots, \tilde{a}_{n},\tilde{a}_{1}^*,\cdots, \tilde{a}_{n}^*\Big\rangle_0 \ = \ \mbox{per} (M_a),
\end{equation}
where $M_a=\Big[\langle\tilde{a}_{i}\tilde{a}_{j}^* \rangle\Big]_{i,j=1}.$
\end{proposition}
\begin{proof}
The proof follows standard arguments (see \cite{barvinok2007integration}) and the fact that the pseudo-covariance matrix is zero. 
\end{proof}

\begin{definition}
	\label{Def:Pairing}
	Let us consider the momenta $\{k_{0,1},\sigma_{0,1},\cdots, k_{0,n},\sigma_{0,n}\}$ of $\mathcal{F}_n^0$, $\mathcal{F}_n^1$, $\mathcal{F}_n^2$, $\mathcal{F}_n^3$. 
 We also denote $\mathcal{P}_{pair}(\{1,\cdots,n\})$ to be  the set of all partition $S$ of $\{1,\cdots,n\}$, such that if  $A=\{i,j\}\in S$ then $|A|= 2$ and  $\sigma_{0,i}+\sigma_{0,j}=0.$
	
\end{definition}

A straightforward application of the  Proposition \eqref{Propo:Wick} leads to the following propositions.

\begin{proposition}[First term]\label{Proposition:Ampl1} For  $\mathbf{N}\ge 1$, the terms associated to $\mathcal{F}_n^0$ can be written as
	\begin{equation}\label{Proposition:Ampl:1}
	Q^1 \ := \ 	Q^{1}_* 
		\ := \  \sum_{n=0}^{\mathbf{N}-1}\sum_{S\in\mathcal{P}_{pair}(\{1,\cdots,n+2\})}\mathcal{G}^{0}_{n,*}(S,t,k_{n,2},\sigma_{n,2},k_{n,1},\sigma_{n,1},\Gamma)
	\end{equation}
	where
	\begin{equation}
		\begin{aligned} \label{eq:DefFmain}
			&\mathcal{G}^{0}_{n,*}(S,t,k_{n,2},\sigma_{n,2},k_{n,1},\sigma_{n,1},\Gamma)\  
			=  \  
			(-{\bf i}\lambda)^n\sum_{\substack{\rho_i\in\{1,\cdots,n-i+2\},\\ i\in\{0,\cdots,n-1\}}}\sum_{\substack{\bar\sigma\in \{\pm1\}^{\mathcal{I}_n},\\ \sigma_{i,\rho_i}+\sigma_{i-1,\rho_i}+\sigma_{i-1,\rho_i+1}\ne \pm3,
					\\ \sigma_{i-1,\rho_i}\sigma_{i-1,\rho_i+1}= 1}}  \\
			&\   \int_{(\Lambda^*)^{\mathcal{I}_n}}\mathrm{d}\bar{k}\Delta_{n,\rho}(\bar{k},\bar\sigma)\prod_{A=\{i,j\}\in S}\left\langle  a(k_{0,i},\sigma_{0,i})a(k_{0,j},\sigma_{0,j})\right\rangle_{0} \\
			&\ \times\prod_{i=1}^n\Big[\sigma_{i,\rho_i}\mathcal{M}( \sigma_{i,\rho_i}, k_{i,\rho_i},\sigma_{i-1,\rho_i}, k_{i-1,\rho_i},\sigma_{i-1,\rho_i+1}, k_{i-1,\rho_i+1})\\
			& \ \times \Phi_{1,i}(\sigma_{i-1,\rho_i}, k_{i-1,\rho_i},\sigma_{i-1,\rho_i+1},k_{i-1,\rho_i+1})\Big]  \int_{(\mathbb{R}_+)^{\{0,\cdots,n\}}}\mathrm{d}\bar{s} \delta\left(t-\sum_{i=0}^ns_i\right)\prod_{i=0}^{n}e^{-s_i[\varsigma_{n-i}+\tau_i]}\\
			&\ \times \prod_{i=1}^{n}e^{-{\bf i}t_i(s)\mathbf{X}( \sigma_{i,\rho_i}, k_{i,\rho_i},\sigma_{i-1,\rho_i}, k_{i-1,\rho_i},\sigma_{i-1,\rho_i+1}, k_{i-1,\rho_i+1})}.
	\end{aligned}\end{equation}
We also set  $Q^1_{res}:=0$.
\end{proposition}

\begin{proposition}[Third term]\label{Proposition:Ampl2} For  $\mathbf{N}\ge 1$, the terms associated to $\mathcal{F}_n^2$ can be written as 
	\begin{equation}\label{Def:OperatorQErrors:Bis2}
		\begin{aligned}
			Q^3_*: = & \ \sum_{n=1}^{\mathbf N}\int_0^t\mathrm{d}s \mathcal{F}_{n,*}^2(s,t,k_{n,2},\sigma_{n,2},k_{n,1},\sigma_{n,1},\Gamma)[\alpha_s] \\ 
			=  &   \ \sum_{n=1}^{\mathbf N}\sum_{S\in \mathcal{P}(\{1,\cdots,n+2\})}\int_0^t\mathrm{d}s\mathcal{G}^{2}_{n,*}(S,s,t,k_{n,2},\sigma_{n,2},k_{n,1},\sigma_{n,1},\Gamma),
		\end{aligned}
	\end{equation}
	where
	\begin{equation}
		\begin{aligned} \label{eq:DefF2}
			&\mathcal{G}^{2}_{n,*}(S,s_0,t,k_{n,2},\sigma_{n,2},k_{n,1},\sigma_{n,1},\Gamma)\  
			=  \ (-{\bf i}\lambda)^n\sum_{\substack{\rho_i\in\{1,\cdots,n-i+2\},\\ i\in\{0,\cdots,n-1\}}}\sum_{\substack{\bar\sigma\in \{\pm1\}^{\mathcal{I}_n},\\ \sigma_{i,\rho_i}+\sigma_{i-1,\rho_i}+\sigma_{i-1,\rho_i+1}\ne \pm3,
					\\ \sigma_{i-1,\rho_i}\sigma_{i-1,\rho_i+1}= 1}} \\
			&\ \times \int_{(\Lambda^*)^{\mathcal{I}_n}}\mathrm{d}\bar{k}\Delta_{n,\rho}(\bar k,\bar\sigma)\sigma_{1,\rho_1}\mathcal{M}( k_{1,\rho_1}, k_{0,\rho_1}, k_{0,\rho_2}) \Phi_{0,1}(\sigma_{0,\rho_1} ,k_{0,\rho_1},\sigma_{0,\rho_2}, k_{0,\rho_2})\\
			&\ \times\prod_{i=2}^n\Big[\sigma_{i,\rho_i}\mathcal{M}( k_{i,\rho_i}, k_{i-1,\rho_i}, k_{i-1,\rho_i+1})\Phi_{1,i}(\sigma_{i-1,\rho_i} k_{i-1,\rho_i},\sigma_{i-1,\rho_i+1}k_{i-1,\rho_i+1})\Big]\\
			&\ \times e^{s_0\tau_0+{\bf i}s_0\vartheta_0} \left\langle\prod_{i=1}^{n+2}  a(k_{0,i},\sigma_{0,i})\right\rangle_{s_0}  \int_{(\mathbb{R}_+)^{\{1,\cdots,n\}}}\mathrm{d}\bar{s} \delta\left(t-\sum_{i=0}^ns_i\right)\prod_{i=1}^{n}e^{-s_i[\varsigma_{n-i}+\tau_i]}e^{-s_0\tau_0}\\
			&\  \times \prod_{i=1}^{n}e^{-{\bf i}t_i(s)\mathbf{X}( \sigma_{i,\rho_i},k_{i,\rho_i},\sigma_{i-1,\rho_i}, k_{i-1,\rho_i},\sigma_{i-1,\rho_i+1},k_{i-1,\rho_i+1})}.
	\end{aligned}\end{equation}
 Setting
	\begin{equation}\begin{aligned}
	&	\left\langle\prod_{i=1}^{n+2}  a(k_{0,i},\sigma_{0,i})\right\rangle_{s_0}\\  \ - \ & \sum_{S\in\mathcal{P}(\{1,\cdots,n+2\})}\left\langle\prod_{A=\{J_1,\dots,J_{|A|}\}\in S}   \square\left(\sum_{l=1}^{|A|}k_{0,J_l}\sigma_{0,J_l}\right)\prod_{i=1}^{n+2}  a(k_{0,i},\sigma_{0,i})\right\rangle_{s_0} 
		:= \  \mathfrak{E}^2_{n,res},\end{aligned}\end{equation} 
	in which 
	\begin{equation}\begin{aligned}\label{SquareFunction}
		\square\left(\sum_{l=1}^{|A|}k_{0,J_l}\sigma_{0,J_l}\right) \ = \ & \sum_{l=1}^{|A|}k_{0,J_l}\sigma_{0,J_l}, \mbox{ if } \sum_{l=1}^{|A|}k_{0,J_l}\sigma_{0,J_l}=\mathscr{K}_A 
		\mbox{ with } |\mathscr{K}_A |\le \lambda^{2}\complement
		 ,\\
		& \mbox{ and there exist at least one }  \sigma_{0,J_l}=1 \mbox{   and  one }  \sigma_{0,J_{l'}}=-1 \mbox{ with }l\ne l', \\
	\square\left(\sum_{l=1}^{|A|}k_{0,J_l}\sigma_{0,J_l}\right) \ = \ & 0, \mbox{ otherwise}.\end{aligned}
	\end{equation}
	We  call the quantities associated to $\mathfrak{E}^2_{n,res}$ the residual of $Q^3_*$.  The sum of all quantities that do not contain  $\mathfrak{E}^2_{n,res}$ is defined to be $Q^3$. The trivial case that concerns the  function $\square\left(\sum_{j\in A} k_{0,j}\sigma_{0,j}\right)$ with $A=\{1,\cdots,n+2\}$ is moved into $\mathfrak{C}^2_{n,res}$. The residual $Q^3_{res}$ of $Q^3_*$ is then $Q^3_*-Q^3$. 
\end{proposition}

\begin{proposition}[Last term]\label{Proposition:Ampl3}
	The component associated to $\mathcal{F}_{\mathbf{N}}^3$ can be expressed as
	\begin{equation}\label{Def:OperatorQErrors:Bis3}
		\begin{aligned}
			Q^4_*: = & \ \int_0^t\mathrm{d}s \mathcal{F}_{\mathbf{N},*}^3(S,s,t,k_{n,2},\sigma_{n,2},k_{n,1},\sigma_{n,1},\Gamma)[\alpha_s] \\
			=  & \ \sum_{S\in \mathcal{P}(\{1,\cdots,\mathbf{N}+2\})} \int_0^t\mathrm{d}s\mathcal{G}^{3}_{\mathbf{N},*}(S,s,t,k_{n,2},\sigma_{n,2},k_{n,1},\sigma_{n,1},\Gamma),
		\end{aligned}
	\end{equation}
	where
	\begin{equation}
		\begin{aligned} \label{eq:DefF3}
			&\mathcal{G}^{3}_{n,*}(S,s_0,t,k_{n,2},\sigma_{n,2},k_{n,1},\sigma_{n,1},\Gamma)\  
			=  \  (-{\bf i}\lambda)^n\sum_{\substack{\rho_i\in\{1,\cdots,n-i+2\},\\ i\in\{0,\cdots,n-1\}}}\sum_{\substack{\bar\sigma\in \{\pm1\}^{\mathcal{I}_n},\\ \sigma_{i,\rho_i}+\sigma_{i-1,\rho_i}+\sigma_{i-1,\rho_i+1}\ne \pm3,
					\\ \sigma_{i-1,\rho_i}\sigma_{i-1,\rho_i+1}= 1}} \\
			& \times \int_{(\Lambda^*)^{\mathcal{I}_n}}\mathrm{d}\bar{k}\Delta_{n,\rho}(\bar{k},\bar\sigma) e^{s_0\tau_0+{\bf i}s_0\vartheta_0}\left\langle\prod_{i=1}^{n+2}  a(k_{0,i},\sigma_{0,i})\right\rangle_{s_0}\\
			& \times  \prod_{i=1}^n\Big[\sigma_{i,\rho_i}\mathcal{M}( k_{i,\rho_i}, k_{i-1,\rho_i}, k_{i-1,\rho_i+1}) \Phi_{1,i}(\sigma_{i-1,\rho_i}, k_{i-1,\rho_i}, \sigma_{i-1,\rho_i+1},k_{i-1,\rho_i+1}) \Big] \\
			&  \times\int_{(\mathbb{R}_+)^{\{1,\cdots,n\}}}\mathrm{d}\bar{s} \delta\left(t-\sum_{i=0}^ns_i\right)\prod_{i=1}^{n}e^{-s_i[\varsigma_{n-i}+\tau_i]}e^{-s_0\tau_0}\\
			&  \times \prod_{i=1}^{n}e^{-{\bf i}t_i(s)\mathbf{X}(\sigma_{i,\rho_i},k_{i,\rho_i},\sigma_{i-1,\rho_i}, k_{i-1,\rho_i}, \sigma_{i-1,\rho_i+1}, k_{i-1,\rho_i+1})  }.
	\end{aligned}\end{equation}

	Setting
	
	\begin{equation}\begin{aligned}
		&\left\langle\prod_{i=1}^{n+2}  a(k_{0,i},\sigma_{0,i})\right\rangle_{s_0} \\ \ - \ & \sum_{S\in\mathcal{P}(\{1,\cdots,n+2\})}\left\langle\prod_{A=\{J_1,\dots,J_{|A|}\}\in S}   \square\left(\sum_{l=1}^{|A|}k_{0,J_l}\sigma_{0,J_l}\right)\prod_{i=1}^{n+2}  a(k_{0,i},\sigma_{0,i})\right\rangle_{s_0} 
		:= \  \mathfrak{E}^3_{n,res},\end{aligned}\end{equation}
	where we follow the notation of \eqref{SquareFunction}. 	We  call the quantities associated to $\mathfrak{E}^3_{n,res}$ the residual of $Q^4_*$. The sum of all quantities that do not contain  $\mathfrak{E}^3_{n,res}$ is defined to be $Q^4$. 	The trivial case that concerns the  function $\square\left(\sum_{j\in A} k_{0,j}\sigma_{0,j}\right)$ with $A=\{1,\cdots,n+2\}$ is moved into $\mathfrak{C}^3_{n,res}$. The residual $Q^4_{res}$ of $Q^4_*$ is then $Q^4_*-Q^4$.  
\end{proposition}

\begin{proposition}[Second term]\label{Proposition:Ampl4} For  $\mathbf{N}\ge 1$, the terms associated to $\mathcal{F}_n^1$ can be written as \begin{equation}\label{Def:OperatorQErrors:Bis3}
		\begin{aligned}
			Q^2_*: = & \sum_{n=0}^{\mathbf{N}-1}\int_0^t\mathrm{d}s\varsigma_n \mathcal{F}_{n}^1(s,t,k_{n,2},\sigma_{n,2},k_{n,1},\sigma_{n,1},\Gamma)[\alpha_s] \\
			= \ & \sum_{n=0}^{\mathbf{N}-1} \sum_{S\in \mathcal{P}(\{1,\cdots,n+2\})}\int_0^t\mathrm{d}s\mathcal{G}^{1}_{n}(S,s,t,k_{n,2},\sigma_{n,2},k_{n,1},\sigma_{n,1},\Gamma),
		\end{aligned}
	\end{equation}
	where
	\begin{equation}
		\begin{aligned} \label{eq:DefF3}
			&\mathcal{G}^{1}_{n,*}(S,s,t,k_{n,2},\sigma_{n,2},k_{n,1},\sigma_{n,1},\Gamma)\  
			= \varsigma_n\int_0^{t-s}\mathrm{d}r e^{-r\varsigma_n}\mathcal{G}^3_{n,*}(s+r,t,k_{n,2},\sigma_{n,2},k_{n,1},\sigma_{n,1},\Gamma)[\alpha_r].
		\end{aligned}
	\end{equation}
	The quantities associated to $\mathfrak{E}^3_{n,res}$, which above appeared in $\mathcal{G}^3_{n,*}$, are called the residual of $Q^2$. The sum of all quantities that do not contain  $\mathfrak{E}^3_{n,res}$ is defined to be $Q^2$. 	The trivial case that concerns the  function $\square\left(\sum_{j\in A} k_{0,j}\sigma_{0,j}\right)$ with $A=\{1,\cdots,n+2\}$ is moved into $\mathfrak{C}^0_{n,res}$. The residual $Q^2_{res}$ of $Q^2_*$ is then $Q^2_*-Q^2$.
\end{proposition}

\section{Technical lemmata}\label{Sec:Techical}
\subsection{Key parameters}\label{Sec:KeyPara}
Following \cite{staffilani2021wave}, we define the ``stopping rule''
\begin{equation}
	\label{Def:Para1}
	\mathbf{N}\coloneqq\max\left(1,	\left[\frac{\mathbf{N}_0|\ln\lambda|}{\ln\langle\ln\lambda\rangle}	\right]\right),
\end{equation}
in which $[x]$ is the integer part that satisfies $[x]\le x <[x]+1$ and $\mathbf{N}_0$ is any number in $\mathbb{R}_+$. The quantity $\mathbf{N}$ is the number of Duhamel expansions to be  performed to actually obtain the kinetic equation.   


The parameters $\eth'_1,\eth'_2>0$ of Lemma \ref{lemma:degree1vertex} have the following constraints
\begin{equation}
	\label{ethprime}\eth'_1+\eth'_2<\frac14,\ \ \ \  \eth'_1>\frac{3}{d},\ \ \ \  \eth'_2>\frac{\theta_r}{d},
\end{equation}
in which $\theta_r$ is defined in \eqref{ConditionCr1}. The constant $\eth>0$ of Lemma \ref{lemma:degree1vertex} is chosen arbitrary.  
The parameter that controls the partial time integration is set to be
\begin{equation}
	\label{Def:Para2}
	\varsigma' = \lambda^2\mathbf{N}^{\wp	},
\end{equation}
where $\wp$ is a positive constant. The components of the soft partial time integration vector, introduced in \eqref{SolfPartTimeVector},  are set to be
\begin{equation}
	\label{Def:Para3}
	\varsigma_n = 0 \mbox{ when } 0\le n<[\mathbf{N}/4], \mbox{ and }\varsigma_n = \varsigma'   \mbox{ when } [\mathbf{N}/4]\le n\le \mathbf{N}. 
\end{equation}

Following \cite{staffilani2021wave}, the graph bounds fall into   one of the two following types:

\begin{itemize}
	\item Type 1: 
	\begin{equation}
		\label{Def:Para4}
		\theta_0^\mathbf{N}\lambda^{\theta_1} \mathbf{N}^{\theta_2\mathbf{N}+\theta_4}((\theta_5\mathbf{N})!)^{\theta_6}\langle \ln\lambda\rangle^{\theta_7\mathbf{N}+\theta_8},
	\end{equation}
	in which $\theta_0,\theta_1,\cdots,\theta_8>0$ are parameters appearing in  the graph bounds. This quantity tends to $0$ as $\lambda$ tends to $0$ if
	\begin{equation}
		\label{Def:Para5}
		\theta_1 \ > \ (\theta_2+\theta_5\theta_6+\theta_7)\mathbf{N}_0.
	\end{equation}
	\item Type 2: 
	\begin{equation}
		\label{Def:Para4a}
		\theta_0^\mathbf{N}\lambda^{-\eth'_1-\eth'_2} \mathbf{N}^{-\wp\theta_3\mathbf{N}+\theta_2\mathbf{N}+\theta_4}((\theta_5\mathbf{N})!)^{\theta_6}\langle \ln\lambda\rangle^{\theta_7\mathbf{N}+\theta_8},
	\end{equation}
	in which $\theta_0,\cdots,\theta_8>0$ are parameters appearing in our estimates.  This quantity tends to $0$ as $\lambda$ tends to $0$ under when
	\begin{equation}
		\label{Def:Para5a}
		\wp \ > \ \frac{\eth'_1+\eth'_2+(\theta_2+\theta_5\theta_6+\theta_7)\mathbf{N}_0}{\theta_3\mathbf{N}_0}.
	\end{equation}
\end{itemize}

\subsection{Some important estimates}
Our current work inherits several   estimates from the previous work \cite{staffilani2021wave}. We will list those estimates below, without their proofs.
\subsubsection{Dispersive estimates} 
First, let us consider the dispersion relation
\begin{equation}
\label{Nearest}
\omega(k) \ = \sin(2\pi k_1)\Big[\sin^2(2\pi k^1) + \cdots + \sin^2(2\pi k^d)\Big],
\end{equation}
where $k=(k^1,\cdots,k^d)\in\mathbb{T}^d$.
 By defining $\xi=(\xi_1,\cdots,\xi_d)=2\pi k\in [-\pi,\pi]^d$, we rewrite \eqref{Nearest} as
\begin{equation}
\label{Nearest1}
\omega_\infty(\xi) \ = \sin(\xi_1)\Big[\sin^2(\xi_1) + \cdots + \sin^2(\xi_d)\Big].
\end{equation}
Under the new rescaled variables, the function $\bar{\Psi}_4^o(\tilde k_0,\tilde k_1,\tilde k_2,\tilde k_0',\tilde k_1',\tilde k_2')$ defined in Proposition \ref{Propo:Phi3A} can be identified with, by an abuse of notation, to be $\bar{\Psi}_4^o(2\pi\tilde k_0,2\pi\tilde k_1,2\pi\tilde k_2,2\pi\tilde k_0',2\pi\tilde k_1',2\pi\tilde k_2')$. From the construction of $\bar{\Psi}_4^o$, we can simplify  $\bar{\Psi}_4^o(2\pi\tilde k_0,2\pi\tilde k_1,2\pi\tilde k_2,2\pi\tilde k_0',2\pi\tilde k_1',2\pi\tilde k_2')$ to $\bar{\Psi}_4^o(2\pi\tilde k_1,2\pi\tilde k_2,2\pi\tilde k_1')$, which is a function of 3 variables $2\pi\tilde k_1,2\pi\tilde k_2,2\pi\tilde k_1'$.

We now consider the  functional
\begin{equation}
\label{extendedBesselfunctions1a}
\mathfrak{F}(m,t_0,t_1,t_2)  =  \int_{[-\pi,\pi]^d} \mathrm{d}\xi  e^{{\bf i}m\cdot \xi} e^{{\bf i}t_0\omega_\infty(\xi)+{\bf i}t_1\omega_\infty(\xi+V)+{\bf i}t_2\omega_\infty(\xi+W)},
\end{equation}
and
\begin{equation}
	\label{extendedBesselfunctions1a:1}
	\tilde{\mathfrak{F}}(m,t_0,t_1,t_2)  =  \int_{[-\pi,\pi]^d} \mathrm{d}\xi  \bar{\Psi}_4^o(\xi,\xi+V,\xi+W)e^{{\bf i}m\cdot \xi} e^{{\bf i}t_0\omega_\infty(\xi)+{\bf i}t_1\omega_\infty(\xi+V)+{\bf i}t_2\omega_\infty(\xi+W)},
\end{equation}
for $m\in\mathbb{Z}^d$. The two vectors $V=(V_1,\cdots,V_d)$, $W=(W_1,\cdots,W_d)$  are assumed to be fixed in $\mathbb{R}^d$.

The equations \eqref{extendedBesselfunctions1a}-\eqref{extendedBesselfunctions1a:1} are now rewritten as 

\begin{equation}
\label{extendedBesselfunctions1}
\begin{aligned}
&\tilde{\mathfrak{F}}(m,t_0,t_1,t_2)  =   \int_{-\pi}^\pi\mathrm{d}\xi_1 \exp\Big({{\bf i}m_1\xi_1}\Big)\\
&\ \ \times\exp\Big({\bf i}t_0\sin^3(\xi_1)+{\bf i}t_1\sin^3(\xi_1+V_1)+{\bf i}t_2\sin^3(\xi_1+W_1)\Big)\times\\
&\ \  \times\Big[\prod_{j=2}^d\int_{-\pi}^\pi\mathrm{d}\xi_j \bar{\Psi}_4^o(\xi,\xi+V,\xi+W)\exp\Big({{\bf i}m_j\xi_j}\Big) \exp\Big({\bf i}t_0\sin(\xi_1)\sin^2(\xi_j)\\
&\ \ +{\bf i}t_1\sin(\xi_1+V_1)\sin^2(\xi_j+V_j) +{\bf i}t_2\sin(\xi_1+W_1)\sin^2(\xi_j+W_j)\Big)\Big],
\end{aligned}
\end{equation}
and
\begin{equation}
	\label{extendedBesselfunctions1}
	\begin{aligned}
		&\mathfrak{F}(m,t_0,t_1,t_2)  =   \int_{-\pi}^\pi\mathrm{d}\xi_1 \exp\Big({{\bf i}m_1\xi_1}\Big)\\
		&\ \ \times\exp\Big({\bf i}t_0\sin^3(\xi_1)+{\bf i}t_1\sin^3(\xi_1+V_1)+{\bf i}t_2\sin^3(\xi_1+W_1)\Big)\times\\
		&\ \  \times\Big[\prod_{j=2}^d\int_{-\pi}^\pi\mathrm{d}\xi_j \exp\Big({{\bf i}m_j\xi_j}\Big) \exp\Big({\bf i}t_0\sin(\xi_1)\sin^2(\xi_j)\\
		&\ \ +{\bf i}t_1\sin(\xi_1+V_1)\sin^2(\xi_j+V_j) +{\bf i}t_2\sin(\xi_1+W_1)\sin^2(\xi_j+W_j)\Big)\Big].
	\end{aligned}
\end{equation}

We also need the notations 
\begin{equation}\label{Lemm:Bessel}
\begin{aligned}
& {t_0}+{t_1}\cos(V_1)e^{{\bf i}2V_j}+{t_2}\cos(W_1)e^{{\bf i}2W_j} \
= \ \Big|{t_0}+{t_1}\cos(V_1)e^{{\bf i}2V_j}+{t_2}\cos(W_1)e^{{\bf i}2W_j}\Big|e^{{\bf i}\aleph
	_1^j},\\
& \mbox{ and }\\
& {t_1}\sin(V_1)e^{{\bf i}2V_j}+{t_2}\sin(W_1)e^{{\bf i}2W_j} 
= \  \ \Big|{t_1}\sin(V_1)e^{{\bf i}2V_j}+{t_2}\sin(W_1)e^{{\bf i}2W_j}\Big|e^{{\bf i}\aleph_2^j},
\end{aligned}
\end{equation}
with $\aleph_1^j=\aleph_1^j(V,W),\aleph_2^j=\aleph_2^j(V,W)\in [-\pi,\pi]$, $j=2,\cdots,d$.

\begin{lemma}[See \cite{staffilani2021wave}]\label{Lemm:Angle}
	Suppose that \begin{equation}
		\label{Lemm:Angle:0} t_0=t_1\aleph' + t_2 \aleph''\end{equation} with $\aleph',\aleph''\in\mathbb{R}\backslash\{0\}$. 
	The following estimate then holds true
	\begin{equation}
	\label{Lemm:Angle:00}\begin{aligned}
	 \frac{1}{|1-|\cos(\aleph_1^j-\aleph_2^j)||^\frac12} 
	\ \lesssim \ & \frac{t_1^2+t_2^2}{|\mathfrak{C}_{\aleph_1^j,\aleph_2^j}^1t_1^2+\mathfrak{C}_{\aleph_1^j,\aleph_2^j}^2t_1t_2+\mathfrak{C}_{\aleph_1^j,\aleph_2^j}^3t_2^2|}+1,\end{aligned}
	\end{equation}
	where the constants on the right hand side are universal and
	\begin{equation}
	\label{Lemm:Angle:1}\begin{aligned}
	 \mathfrak{C}_{\aleph_1^j,\aleph_2^j}^1 \ = \ & -\aleph'\sin(V_1)\sin(2V_j),
	\\
	 \mathfrak{C}_{\aleph_1^j,\aleph_2^j}^3 \ = \ &  -\aleph''\sin(W_1)\cos(2W_j),\\
	  \mathfrak{C}_{\aleph_1^j,\aleph_2^j}^2 \ = \  & -\sin(2V_j-2W_j)\sin(V_1-W_1)  \ - \aleph'\sin(W_1)\sin(2W_j) \\
	& - \ \aleph''\sin(V_1)\sin(2V_j).
	\end{aligned}
	\end{equation}
	Setting $r_*=t_1/t_2$, we consider the equation
	\begin{equation}
		\label{Lemm:Angle:2}\begin{aligned}
		\mathfrak{C}_{\aleph_1^j,\aleph_2^j}^1r_*^2+\mathfrak{C}_{\aleph_1^j,\aleph_2^j}^2r_*+\mathfrak{C}_{\aleph_1^j,\aleph_2^j}^3 \ = \ 0.
		\end{aligned}
	\end{equation}
	In the case that \eqref{Lemm:Angle:2} has two real roots (or a double real root), we denote these roots by $\tilde{r}_1,\tilde{r}_2$. In the case that \eqref{Lemm:Angle:2} has two complex roots, we denote the real part of the complex root by $\tilde{r}_3$. Let $\epsilon_{r_*}\in(-1,1)$ be a small number. We consider the case when $r_*=(1+\epsilon_{r_*})\tilde r_i$, $i=1,2,3$ and define the function

		\begin{equation}
		\label{Lemm:Angle:3}\begin{aligned}
			& f_{t_1,t_2}(r_*) \ = \ \\
			 &  \frac{\Big|{r_*}\sin(V_1)e^{{\bf i}2V_j}+ \sin(W_1)e^{{\bf i}2W_j}\Big|}{\Big[\Big|r_*\Big(\cos(V_1)e^{{\bf i}2V_j}+\aleph'\Big)+ \Big(\cos(W_1)e^{{\bf i}2W_j}+\aleph'\Big)\Big|^2+\Big|{r_*}\sin(V_1)e^{{\bf i}2V_j}+ \sin(W_1)e^{{\bf i}2W_j}\Big|^2\Big]^\frac12},
		\end{aligned}
	\end{equation}
and
\begin{equation}
	\label{Lemm:Angle:3a}\begin{aligned}
		& g_{t_1,t_2}(r_*) \ = \ \\
		&  \frac{\Big|r_*\Big(\cos(V_1)e^{{\bf i}2V_j}+\aleph'\Big)+ \Big(\cos(W_1)e^{{\bf i}2W_j}+\aleph'\Big)\Big|}{\Big[\Big|r_*\Big(\cos(V_1)e^{{\bf i}2V_j}+\aleph'\Big)+ \Big(\cos(W_1)e^{{\bf i}2W_j}+\aleph'\Big)\Big|^2+\Big|{r_*}\sin(V_1)e^{{\bf i}2V_j}+ \sin(W_1)e^{{\bf i}2W_j}\Big|^2\Big]^\frac12}.
	\end{aligned}
\end{equation}
We suppose that 
\begin{equation}
	\label{Lemm:Angle:4}
	|V_j|, \Big|V_j-\frac{\pi}{2}\Big|, |V_j-{\pi}|,|W_j|, \Big|W_j-\frac{\pi}{2}\Big|, |W_j-{\pi}|\ge \langle\ln\lambda\rangle^{-c_{V,W}},
\end{equation}
and
\begin{equation}
	\label{Lemm:Angle:5}
	|V_j\pm W_j| \ge \langle\ln\lambda\rangle^{-c_{V,W}},
\end{equation}
for some constant $c_{V,W}>0$ and for all $j=1,\cdots,d$. We thus have
\begin{equation}
	\label{Lemm:Angle:6}\begin{aligned}
		&\Big|\frac{\mathrm{d}}{\mathrm{d} r_*}	|f_{t_1,t_2}(r_*)|^2\Big|	 \ \lesssim \langle\ln\lambda\rangle^{\mathfrak{C}_{\aleph_1^j,\aleph_2^j}^4},
	\end{aligned}
\end{equation} 
\begin{equation}
	\label{Lemm:Angle:7}\begin{aligned}
		&\Big|\frac{\mathrm{d}}{\mathrm{d} r_*}	|g_{t_1,t_2}(r_*)|^2\Big|	 \ \lesssim \langle\ln\lambda\rangle^{\mathfrak{C}_{\aleph_1^j,\aleph_2^j}^4},
	\end{aligned}
\end{equation} 
and
\begin{equation}
	\label{Lemm:Angle:8}\begin{aligned}
		&\Big|\frac{\mathrm{d}}{\mathrm{d} r_*}	[f_{t_1,t_2}(r_*)g_{t_1,t_2}(r_*)]\Big|	 \ \lesssim \langle\ln\lambda\rangle^{\mathfrak{C}_{\aleph_1^j,\aleph_2^j}^4},
	\end{aligned}
\end{equation} 
for some explicit constant $\mathfrak{C}_{\aleph_1^j,\aleph_2^j}^4>0$.

In the special case that $t_2=0,$ $t_1=-t_0$, we bound\begin{equation}
	\label{Lemm:Angle:9}
	[1-|\cos(\aleph_1^j-\aleph_2^j)|]^{-\frac12}\ 
	\lesssim  \ \frac{1}{|\sin(2V_j)|}+1.
\end{equation}
\end{lemma}

\begin{lemma}[See \cite{staffilani2021wave}]\label{Lemm:Bessel2} There exists a universal constant $\mathfrak{C}_{\mathfrak{F},1}$ independent of $t_0,t_1,t_2,\aleph_1,\aleph_2$, such that 
	\begin{equation}
	\label{Lemm:Bessel2:1}
	\|\mathfrak{F}(\cdot,t_0,t_1,t_2) \|_{l^2} \ = \ \left( \sum_{m\in\mathbb{Z}^d}|\mathfrak{F}(m,t_0,t_1,t_2) |^2\right)^\frac12 \ \le \ \mathfrak{C}_{\mathfrak{F},1},
	\end{equation}
and similarly
\begin{equation}
	\label{Lemm:Bessel2:1:1}
	\|\tilde{\mathfrak{F}}(\cdot,t_0,t_1,t_2) \|_{l^2}  \ \le \ \mathfrak{C}_{\mathfrak{F},1},
\end{equation}
\end{lemma}

\begin{lemma}[See \cite{staffilani2021wave}]\label{Lemm:Bessel3}
	There exists a universal constant $\mathfrak{C}_{\mathfrak{F},4}>0$ independent of $t_0,t_1,t_2,\aleph_1,\aleph_2$, such that 
	\begin{equation}
	\label{Lemm:Bessel3:1}\begin{aligned}
	&\|\mathfrak{F}(\cdot,t_0,t_1,t_2) \|_{l^4}  \le  \mathfrak{C}_{\mathfrak{F},4}\prod_{j=2}^{d}\Big\langle\min\Big\{\Big|{t_0}+{t_1}\cos(V_1)e^{{\bf i}2V_j}+{t_2}\cos(W_1)e^{{\bf i}2W_j}\Big|,\\
	& \ \Big|{t_1}\sin(V_1)e^{{\bf i}2V_j}+{t_2}\sin(W_1)e^{{\bf i}2W_j}\Big|\Big\}[1-|\cos(\aleph^j
	_1-\aleph^j
	_2)|]^{\frac{1}{2}}\Big\rangle^{-(\frac{1}{8}-)}\\
	& \ \times \min\Big\{1,\Big[\Big|{t_0}+{t_1}\cos(V_1)e^{{\bf i}2V_j}  +{t_2}\cos(W_1)e^{{\bf i}2W_j}\Big|\\
	&\ \ \ +\Big|{t_1}\sin(V_1)e^{{\bf i}2V_j}+{t_2}\sin(W_1)e^{{\bf i}2W_j}\Big|\Big]\\
	&\ \ \ \times\Big[\Big|{t_0}+{t_1}e^{{\bf i}3V_1}+{t_2}e^{{\bf i}3W_1}\Big|+(2d+1)\Big|{t_0}+{t_1}e^{{\bf i}V_1}+{t_2}e^{{\bf i}W_1}\Big|\Big]^{-1}\Big\}^\frac14.\end{aligned}
	\end{equation}
In addition \eqref{Lemm:Bessel3:1} also holds true for $\tilde{\mathfrak{F}}$.
Suppose further that \eqref{Lemm:Angle:0},\eqref{Lemm:Angle:4} and \eqref{Lemm:Angle:5} hold true, and 	\begin{equation}
	\label{Lemm:Bessel3:1:a}r_*=t_1/t_2=(1+\epsilon_{r_*})\tilde{r}_l\end{equation} for $l=1,2,3$ in which $\tilde{r}_l$ are defined in Lemma \ref{Lemm:Angle}. When \begin{equation}
	\label{Lemm:Bessel3:1:b} |\epsilon_{r_*}|=|\epsilon_{r_*}'|\langle
\ln\lambda\rangle^{-c}\end{equation}  for an explicit constant $c>0$ depending only on the cut-off functions, and $\epsilon_{r_*}'$ is sufficiently small but independent of $\lambda$ and the cut-off functions,  then we have the estimate \begin{equation}
	\label{Lemm:Bessel3:2}\begin{aligned}
		&\|\tilde{\mathfrak{F}}(\cdot,t_0,t_1,t_2) \|_{l^4}  \le  \mathfrak{C}_{\mathfrak{F},4}\langle\ln\lambda\rangle^{\mathfrak{C}'_{\mathfrak{F},4}}\prod_{j=2}^{d}\Big\langle\Big\{\Big|{t_0}+{t_1}\cos(V_1)e^{{\bf i}2V_j}+{t_2}\cos(W_1)e^{{\bf i}2W_j}\Big|+\\
		& + \ \Big|{t_1}\sin(V_1)e^{{\bf i}2V_j}+{t_2}\sin(W_1)e^{{\bf i}2W_j}\Big|\Big\}|\cos(\aleph^j
		_1-\aleph^j
		_2)|^{\frac{1}{2}}\Big\rangle^{-(\frac18-)},\end{aligned}
\end{equation}
for universal constants $\mathfrak{C}_{\mathfrak{F},4},\mathfrak{C}_{\mathfrak{F},4}'>0.$
\end{lemma}

\begin{lemma}[See \cite{staffilani2021wave}]\label{Lemm:Bessel4} There exists a universal constant $\mathfrak{C}_{\mathfrak{F},3}>0$ independent of $t_0,t_1,t_2,\aleph_1,\aleph_2$, such that 
	\begin{equation}
	\label{Lemm:Bessel4:1}\begin{aligned}
&	\|\mathfrak{F}(\cdot,t_0,t_1,t_2) \|_{l^3} \ \le  \mathfrak{C}_{\mathfrak{F},3}\prod_{j=2}^{d}\Big\langle\min\Big\{\Big|{t_0}+{t_1}\cos(V_1)e^{{\bf i}2V_j}+{t_2}\cos(W_1)e^{{\bf i}2W_j}\Big|,\\
& \ \Big|{t_1}\sin(V_1)e^{{\bf i}2V_j}+{t_2}\sin(W_1)e^{{\bf i}2W_j}\Big|\Big\}[1-|\cos(\aleph^j
_1-\aleph^j
_2)|]^{\frac{1}{2}}\Big\rangle^{-(\frac{1}{12}-)}\\
& \ \times \min\Big\{1,\Big[\Big|{t_0}+{t_1}\cos(V_1)e^{{\bf i}2V_j}  +{t_2}\cos(W_1)e^{{\bf i}2W_j}\Big|\\
&\ \ \ +\Big|{t_1}\sin(V_1)e^{{\bf i}2V_j}+{t_2}\sin(W_1)e^{{\bf i}2W_j}\Big|\Big]\\
&\ \ \ \times\Big[\Big|{t_0}+{t_1}e^{{\bf i}3V_1}+{t_2}e^{{\bf i}3W_1}\Big|+(2d+1)\Big|{t_0}+{t_1}e^{{\bf i}V_1}+{t_2}e^{{\bf i}W_1}\Big|\Big]^{-1}\Big\}^\frac16.\end{aligned}
	\end{equation}
In addition, \eqref{Lemm:Bessel4:1} also holds true for $\tilde{\mathfrak{F}}$.
Suppose further that \eqref{Lemm:Angle:0}-\eqref{Lemm:Angle:4}-\eqref{Lemm:Angle:5} hold true, and 	\begin{equation}
	\label{Lemm:Bessel4:1:a}r_*=t_1/t_2=(1+\epsilon_{r_*})\tilde{r}_l\end{equation} for $l=1,2,3$ in which $\tilde{r}_l$ are defined in Lemma \ref{Lemm:Angle}. When \begin{equation}
	\label{Lemm:Bessel4:1:b} |\epsilon_{r_*}|=|\epsilon_{r_*}'|\langle
	\ln\lambda\rangle^{-c}\end{equation}  for an explicit constant $c>0$ depending only on the cut-off functions, and $\epsilon_{r_*}'$ is sufficiently small but independent of $\lambda$ and the cut-off functions,  then we have the estimate 
	\begin{equation}
	\label{Lemm:Bessel4:2}\begin{aligned}
		&	\|\tilde{\mathfrak{F}}(\cdot,t_0,t_1,t_2) \|_{l^3} \ \le  \mathfrak{C}_{\mathfrak{F},3}\langle \ln\lambda\rangle^{\mathfrak{C}_{\mathfrak{F},3}'}\prod_{j=2}^{d}\Big\langle\Big\{\Big|{t_0}+{t_1}\cos(V_1)e^{{\bf i}2V_j}+{t_2}\cos(W_1)e^{{\bf i}2W_j}\Big|+\\
		&+ \ \Big|{t_1}\sin(V_1)e^{{\bf i}2V_j}+{t_2}\sin(W_1)e^{{\bf i}2W_j}\Big|\Big\}|\cos(\aleph^j
		_1-\aleph^j
		_2)|^{\frac{1}{2}}\Big\rangle^{-(\frac{1}{12}-)},\end{aligned}
\end{equation}
for universal constants $\mathfrak{C}_{\mathfrak{F},3},\mathfrak{C}_{\mathfrak{F},3}'>0.$
\end{lemma}

\begin{proof}
	Observe that $\frac{1}{3} =\frac{\theta}{4} +\frac{1-\theta}{2}$ with $\theta=\frac{2}{3}$. By interpolating between $l_4$ and  $l_2$,  with the  use  of $\theta$,  we  get  \eqref{Lemm:Bessel4:1}.

\end{proof}

Next, we define a similar function as in  \eqref{extendedBesselfunctions1}, where $t_1=t_2=0$, but with the addition of the cut-off function $\tilde{\Phi}_1$, which is   $\tilde{\Psi}_1^a=1-\tilde{\Psi}_0^a$, defined in Proposition \ref{Propo:Phi} 
\begin{equation}
\label{Improved:extendedBesselfunctions1}
\mathfrak{F}^{O}(m,t_0) \ = \ \int_{[-\pi,\pi]^d} \mathrm{d}\xi (\tilde\Phi_1)^N(\sigma_1,k_1,\sigma_2 k_2)e^{{\bf i}m\cdot \xi} e^{{\bf i}t_0\omega_\infty(\xi)}\mathfrak{K}(\xi),
\end{equation}
for $m=(m_1,\cdots,m_d)\in\mathbb{Z}^d$, $N>0$ is an arbitrary positive power. The kernel $\mathfrak{K}(\xi)$ can be either $1$ or $|\sin(\xi_1)|$, with $\xi=(\xi_1,\cdots,\xi_d)\in\mathbb{R}^d$. Note that $\tilde{\Psi}_1^a$ is a function of $2$ variables $k_1,k_2$. As a result, we can set    $\xi$ to be equal to one of the variables $\xi=2\pi k_1$ or $\xi=2\pi k_2$. 
For the sake of simplicity, we denote  $(\tilde\Phi_1)^N(\sigma_1,k_1,\sigma_2, k_2)$ by $ \check{\Phi}(\xi)$ and re-express \eqref{Improved:extendedBesselfunctions1} in the following form

\begin{equation}
\label{Improved:extendedBesselfunctions1a}
\begin{aligned}
&\mathfrak{F}^{O}(m,t_0)  =   \int_{-\pi}^\pi\mathrm{d}\xi_1 \mathfrak{K}(\xi)\check{\Phi}(\xi) \exp\Big({{\bf i}m_1\xi_1}\Big)\exp\Big({\bf i}t_0\sin^3(\xi_1)\Big)\\
&\ \  \times\Big[\prod_{j=2}^d\int_{-\pi}^\pi\mathrm{d}\xi_j \exp\Big({{\bf i}m_j\xi_j}\Big) \exp\Big({\bf i}t_0\sin(\xi_1)\sin^2(\xi_j)\Big)\Big].
\end{aligned}
\end{equation}
\begin{lemma}[See \cite{staffilani2021wave}]\label{Lemm:Improved:Bessel2} There exists a universal constant $\mathfrak{C}_{\mathfrak{F}^{cut},2}$ independent of $t_0$ and $\lambda$, such that 
	\begin{equation}
	\label{Lemm:Improved:Bessel2:1}
	\|\mathfrak{F}^{O}(\cdot,t_0) \|_{l^2} \ = \ \left( \sum_{m\in\mathbb{Z}^d}|\mathfrak{F}^{O}(m,t_0) |^2\right)^\frac12 \ \le \ \mathfrak{C}_{\mathfrak{F}^{cut},2}.
	\end{equation}
\end{lemma}
 
\begin{lemma}[See \cite{staffilani2021wave}]\label{Lemm:Improved:Bessel3}
	There exist  universal constants $\mathfrak{C}_{\mathfrak{F}^{O},4},\mathfrak{C}_{\mathfrak{F}^{O},4'}>0$ independent of $t_0$ and $\lambda$, such that
	\begin{equation}
	\label{Lemm:Improved:Bessel3:1}\begin{aligned}
	&\|\mathfrak{F}^{O}(\cdot,t_0) \|_{l^4}  \le |\ln|\lambda||^{\mathfrak{C}_{\mathfrak{F}^{O},4'}}\mathfrak{C}_{\mathfrak{F}^{O},4}\langle t_0\rangle^{-{\big(\frac{
				d-1}{8}-\big)}}.\end{aligned}
	\end{equation}
\end{lemma}
   
\begin{lemma}[See \cite{staffilani2021wave}]\label{Lemm:Improved:Bessel4} There exist universal constants $\mathfrak{C}_{\mathfrak{F}^{O},3}, \mathfrak{C}_{\mathfrak{F}^{O},3'}>0$ independent of $t_0$ and $\lambda$, such that
	\begin{equation}
	\label{Lemm:Improved:Bessel4:1}\begin{aligned}
	&	\|\mathfrak{F}^{O}(\cdot,t_0) \|_{l^3} \ \le  |\ln|\lambda||^{\mathfrak{C}_{\mathfrak{F}^{O},3'}}\mathfrak{C}_{\mathfrak{F}^{O},3}\langle{t_0}\rangle^{-{\big(\frac{
				d-1}{12}-\big)}}.\end{aligned}
	\end{equation}
\end{lemma}

Now, we define a similar function as   \eqref{Improved:extendedBesselfunctions1}, where the cut-off function $\tilde{\Phi}_1$ is replaced by the kernel $|\sin(\xi_1)|^n$, $n\ge 0, 2n+5<d, n\in\mathbb{N}$, with $\xi=(\xi_1,\cdots,\xi_d)\in\mathbb{R}^d$
\begin{equation}
	\label{Ker:extendedBesselfunctions1}
	\mathfrak{F}^{Ker}(m,t_0) \ = \ \int_{[-\pi,\pi]^d} \mathrm{d}\xi |\sin(\xi_1)|^ne^{{\bf i}m\cdot \xi} e^{{\bf i}t_0\omega_\infty(\xi)},
\end{equation}
for $m=(m_1,\cdots,m_d)\in\mathbb{Z}^d$.  We recall that in the whole paper, our assumption on the dimension is $d\ge 2$.
\begin{lemma}[See \cite{staffilani2021wave}]\label{Lemm:Ker:Bessel2} There exists a universal constant $\mathfrak{C}_{\mathfrak{F}^{Ker},2}$ independent of $t_0$ and $\lambda$, such that 
	\begin{equation}
		\label{Lemm:Ker:Bessel2:1}
		\|\mathfrak{F}^{Ker}(\cdot,t_0) \|_{l^2} \ = \ \left( \sum_{m\in\mathbb{Z}^d}|\mathfrak{F}^{Ker}(m,t_0) |^2\right)^\frac12 \ \le \ \mathfrak{C}_{\mathfrak{F}^{Ker},2}.
	\end{equation}
\end{lemma}

\begin{lemma}[See \cite{staffilani2021wave}]\label{Lemm:Ker:Bessel3}
	There exists a  universal constant $\mathfrak{C}_{\mathfrak{F}^{Ker},4}>0$ independent of $t_0$ and $\lambda$, such that
	\begin{equation}
		\label{Lemm:Ker:Bessel3:1}\begin{aligned}
			&\|\mathfrak{F}^{Ker}(\cdot,t_0) \|_{l^4}  \le \mathfrak{C}_{\mathfrak{F}^{Ker},4}\langle t_0\rangle^{-{\frac{2n}{7}-\frac{13}{30}}}.\end{aligned}
	\end{equation}
\end{lemma}

\begin{lemma}[See \cite{staffilani2021wave}]\label{Lemm:Ker:Bessel4} There exists a universal constant $\mathfrak{C}_{\mathfrak{F}^{Ker},3}>0$ independent of $t_0$ and $\lambda$, such that
\begin{equation}
	\label{Lemm:Ker:Bessel4:1}\begin{aligned}
		&	\|\mathfrak{F}^{Ker}(\cdot,t_0) \|_{l^3} \ \le \mathfrak{C}_{\mathfrak{F}^{Ker},3}\langle{t_0}\rangle^{{-{\frac{4n}{21}-\frac{13}{45}}}}.\end{aligned}
\end{equation}
\end{lemma}
Finally, we provide an estimate of the following quantity
\begin{equation}
	\label{AnotherKernel:1}	\begin{aligned} 
		{\mathfrak{F}}^{Kern}(x,s) \ = \ & \int_{\mathbb{T}^d} \mathrm{d}k |\sin(2\pi k_0^1)|^2  |\sin(2\pi k^1)|^2|\sin(2\pi (k_0^1+k^1)|^2 e^{{\bf i}2\pi x\cdot k}e^{{\bf i}t\omega(k) +{\bf i}t\omega(-k_0-k)}, 	\end{aligned}
\end{equation}
which can be rewritten as
\begin{equation}
	\label{AnotherKernel:2}	\begin{aligned} 
		{\mathfrak{F}}^{Kern}(x,s) \ = \ & \int_{[-\pi,\pi]^d} \mathrm{d}\xi |\sin(\xi^*_1)|^2  |\sin(\xi_1)|^2|\sin(\xi^*_1+\xi_1)|^2 e^{{\bf i}x\cdot \xi}e^{{\bf i}t\omega(\xi) +{\bf i}t\omega(-\xi^*-\xi)}. 	\end{aligned}
\end{equation}
We then obtain the estimate.
\begin{lemma}[See \cite{staffilani2021wave}]\label{Lemm:AnotherKernel2} There exists a universal constant $\mathfrak{C}_{\tilde{\mathfrak{F}}^{Kern},2}>0$ independent of $t$, such that 
	\begin{equation}
		\label{Lemm:AnotherKernel2:1}
		\|{\mathfrak{F}}^{Kern}(\cdot,t) \|_{l^4}  \ \le \ \mathfrak{C}_{\tilde{\mathfrak{F}}^{Kern},2}  \frac{|\sin(\xi^*_1)|^2}{\langle |t||\sin(\xi_1^*)|\rangle^{\frac12-}}.
	\end{equation}
\end{lemma}

\subsubsection{Resolvent-identity, integrals over free momenta and collision operator estimates} 
 
\begin{lemma}[See \cite{staffilani2021wave}]\label{Lemma:Kidentity} Let $I$ be a non-empty finite index set and $s>0$. Suppose that $A$ is a non-empty subset of $I,$ $\vartheta_j\in \ss$, $j\in I$, with $\ss$ being a compact subset of $\mathbb{C}$. We choose an additional time index label $J$, i.e. suppose $J\notin I$, and denote $A^c=I\backslash A$, $A'=A^c\cup\{J\}$. Then for any path $\Gamma_{\ss}$ going once counterclockwise around $\ss$, the following holds true
	\begin{equation}\label{Lemma:PhaseResolve:1}
		\begin{aligned}
			&\int_{(\mathbb{R}_+)^I}\prod_{j\in I}\mathrm{d}s_j\delta\left(s-\sum_{j\in I}s_j\right)\prod_{j\in I} e^{- {\bf i}\vartheta_j s_j}\\
			&\indent = \ -\oint_{\Gamma_{\ss}}\frac{\mathrm{d}z}{2\pi}\int_{(\mathbb{R}_+)^{A'}}\prod_{j\in A'}\mathrm{d}s_j\delta\left(s-\sum_{j\in A'}s_j\right)\prod_{j\in A'} e^{- {\bf i}\vartheta_j s_j}\Big|_{\vartheta_J=z}\prod_{j\in A}\frac{{\bf i}}{z-\vartheta_j}.\end{aligned}\end{equation}
\end{lemma}

\begin{lemma}[Degree-one vertex estimate]\label{lemma:degree1vertex} Let $H(k)$ be any function  in $L^2(\mathbb{T}^d)$. For any $k^*\in\mathbb{T}^d$, $d\ge 2$, $\alpha,
	\lambda\in\mathbb{R}$, $\mathscr{R}>0$ and $\sigma,\sigma',\sigma''\in\{\pm 1\}$, we have 
	\begin{equation}
		\label{eq:degree1vertex}
		\left\|\int_{\mathbb{T}^d}\mathrm{d}k\frac{\tilde\Psi_1^a(\sigma', k,\sigma'',k^*)|H(k)|}{|\omega (k)+\sigma \omega (k^*+k) -2\pi\mathscr{X}- \alpha + {\bf i}\lambda^2|}\right\|_{L^2(-\mathscr{R},\mathscr{R})}\ \lesssim \ \langle \ln|\lambda| \rangle^{2+\eth}\|H\|_{L^\Im(\mathbb{T}^d)},
	\end{equation}
	and
	\begin{equation}
		\label{eq:degree1vertex:a}
		\left\|\int_{\mathbb{T}^d}\mathrm{d}k\frac{\tilde\Psi_1^b(\eth'_l,\sigma', k,\sigma'',k^*)|H(k)|}{|\omega (k)+\sigma \omega (k^*+k) -2\pi\mathscr{X}- \alpha + {\bf i}\lambda^2|}\right\|_{L^2(-\mathscr{R},\mathscr{R})}\ \lesssim \ \langle \ln|\lambda| \rangle^{2}\lambda^{-\eth'_l}\|H\|_{L^\Im(\mathbb{T}^d)},
	\end{equation}
	in which the constant on the right hand side is independent of $k,k^*,\sigma,\sigma',\alpha$ and $\lambda$. Moreover $\Im=\frac{2\Im_o}{2\Im_o-1}$ with $\Im_o\in\mathbb{N}$. The norm $L^2(-\mathscr{R},\mathscr{R})$ is with respect to the variable $\mathscr{X}$. The two cut-off functions $\tilde\Psi_1^a$ and $\tilde\Psi_1^b$ are are the components of $\tilde\Phi_1^a$ and $\tilde\Phi_1^b$ in \eqref{def:Phi3:1}-\eqref{def:Phi3:2} (see Proposition \ref{Propo:Phi},  Proposition \ref{Propo:Phi3A}, Definition \ref{def:Phi3}).
	\begin{proof} We follow the proof of Lemma 45 of \cite{staffilani2021wave}.
			We first set $\tilde\Psi_1^a(\sigma', k,\sigma'',k^*)$ or $\tilde\Psi_1^b(\eth'_l,\sigma' ,k,\sigma'',k^*)$ to be  $\mu_1(k)$ and $\langle\ln|\lambda| \rangle^{\eth}$ or $\lambda^{-\eth'_l}$ by $d(\mathfrak{S})^3$. 
		We will prove the inequality
		\begin{equation}
			\label{eq:degree1vertex:E1}
			\left\|	\int_{\mathbb{T}^d}\mathrm{d}k\frac{\mu_1|H|}{|\omega(k)+\sigma \omega(k^*+k)  +2\pi\mathscr{X}- \alpha + {\bf i}\lambda^2|}\right\|_{L^2(-\mathscr{Q},\mathscr{Q})}\ \lesssim \ \frac{\langle\mathrm{ln}\lambda\rangle^2}{d(\mathfrak{S})^3}\|H\|_{L^2(\mathbb{T}^d)},
		\end{equation}
		which eventually leads to both \eqref{eq:degree1vertex} and \eqref{eq:degree1vertex:a}.

		We first define $K_1(s)$ to be the modified Bessel function of the second kind 
			\begin{equation}
			\label{eq:degree1vertex:E34} K_1(\lambda^2 \nu)=\int _{0}^{\infty }\mathrm{d}\zeta e^{-\lambda^2 |\nu|\cosh \zeta},
		\end{equation}
	and obtain  the identity (see \cite{watson1995treatise}) 
		\begin{equation}
			\label{eq:degree1vertex:E3}\begin{aligned}
				\frac{1}{
					|\eta+{\bf i}\lambda^2|}= & \frac{1}{\sqrt {\eta^{2}+\lambda^4}}
				=  2\int _{-\infty }^{\infty }\mathrm{d}\nu {e^{-{\bf i}2\pi \eta \nu}} \,K_{1}(2\pi\lambda^2 \nu)
				=  \frac{1}{\pi}\int _{-\infty }^{\infty }\mathrm{d}\nu {e^{-{\bf i}\eta \nu}} \,K_{1}(\lambda^2 \nu),\end{aligned}
		\end{equation}
		in which we have used the change of variables $2\pi \nu\to \nu$ in the last step. This leads to
		\begin{equation}
			\label{eq:degree1vertex:E5}\begin{aligned}
				\frac{1}{
					|\eta+{\bf i}\lambda^2|}
				= &\ \frac{1}{\pi}\int _{-\infty }^{\infty }\mathrm{d}\nu {e^{-{\bf i} \eta \nu}} \,\int _{0}^{\infty }\mathrm{d}\zeta e^{-\lambda^2 |\nu|\cosh \zeta}.\end{aligned}
		\end{equation}
		Following \cite{watson1995treatise}, we have the standard estimate
		\begin{equation}
			\label{eq:degree1vertex:E6}
			\tilde{K}_1(\nu) := \frac{1}{\pi}\int _{0}^{\infty }\mathrm{d}\zeta e^{-\lambda^2 |\nu|\cosh \zeta} \lesssim \langle \mathrm{ln}\lambda^2\rangle\left( e^{\lambda^2|\nu|} \ + \ \mathbf{1}_{|\nu|\le 1}|\mathrm{ln}|\nu||^{-1}\right).
		\end{equation}
		Next,  replacing $\eta$ in \eqref{eq:degree1vertex:E5} by $\omega(k)+\sigma \omega(k^*+k) - \alpha $, we find
		\begin{equation}
			\label{eq:degree1vertex:E7}\begin{aligned}
				&		\int_{\mathbb{T}^d}\mathrm{d}k\frac{\mu_1H}{|\omega(k)+\sigma \omega(k^*+k)-2\pi\mathscr{X} - \alpha + {\bf i}\lambda^2|}\\
				=  \ & \int _{-\infty }^{\infty }\mathrm{d}\nu \int_{\mathbb{T}^d}\mathrm{d}k \mu_1H{\tilde{K}_1(\nu)} e^{{\bf i}\nu(\omega(k)+\sigma\omega(k^*+k)-2\pi\nu{\bf i}\mathscr{X}}.\end{aligned}
		\end{equation}
		Now, define $\chi_{(-\mathscr{R},\mathscr{R})}(\nu)$ to be the cut-off function of the variable $\nu$ on $(-\mathscr{R},\mathscr{R})$, for a  small constant $\mathscr{R}>0$, we bound
		\begin{equation}
			\begin{aligned}\label{eq:degree1vertex:E7:a}
				&\left(\int_{(-\mathscr{Q},\mathscr{Q})}\mathrm{d}\mathscr{X}\left|\int _{-\infty }^{\infty }\mathrm{d}\nu \int_{\mathbb{T}^d}\mathrm{d}k \mu_1H{\tilde{K}_1(\nu)} e^{{\bf i}\nu(\omega(k)+\sigma\omega(k^*+k))-2\pi\nu{\bf i}\mathscr{X}}\right|^2\right)^\frac12 \\
				\lesssim \ & \left(\int_{(-\mathscr{Q},\mathscr{Q})}\mathrm{d}\mathscr{X}\left|\int _{-\infty }^{\infty }\mathrm{d}\nu \chi_{(-\mathscr{R},\mathscr{R})}(\nu) \int_{\mathbb{T}^d}\mathrm{d}k \mu_1H{\tilde{K}_1(\nu)} e^{{\bf i}\nu(\omega(k)+\sigma\omega(k^*+k))-2\pi\nu{\bf i}\mathscr{X}}\right|^2\right)^\frac12\\
				& + \left(\int_{(-\mathscr{Q},\mathscr{Q})}\mathrm{d}\mathscr{X}\left|\int _{-\infty }^{\infty }\mathrm{d}\nu [1-\chi_{(-\mathscr{R},\mathscr{R})}(\nu)] \int_{\mathbb{T}^d}\mathrm{d}k \mu_1H{\tilde{K}_1(\nu)} e^{{\bf i}\nu(\omega(k)+\sigma\omega(k^*+k))-2\pi\nu{\bf i}\mathscr{X}}\right|^2\right)^\frac12\\
				\lesssim \ & \left|\int _{-\mathscr{R} }^{\mathscr{R} }\mathrm{d}\nu  \left| \int_{\mathbb{T}^d}\mathrm{d}k \mu_1H{\tilde{K}_1(\nu)} e^{{\bf i}\nu(\omega(k)+\sigma\omega(k^*+k))}\right|^2\right|^\frac12\\
				& + \left(\int_{(-\mathscr{Q},\mathscr{Q})}\mathrm{d}\mathscr{X}\left|\int _{-\infty }^{\infty }\mathrm{d}\nu [1-\chi_{(-\mathscr{R},\mathscr{R})}(\nu)] \int_{\mathbb{T}^d}\mathrm{d}k \mu_1H{\tilde{K}_1(\nu)} e^{{\bf i}\nu(\omega(k)+\sigma\omega(k^*+k)-2\pi\nu{\bf i}\mathscr{X}}\right|^2\right)^\frac12\\
				\lesssim \ & \left|\int _{-\infty }^{\infty}\mathrm{d}\nu  \left| \int_{\mathbb{T}^d}\mathrm{d}k \mu_1H{\tilde{K}_1(\nu)} e^{{\bf i}\nu(\omega(k)+\sigma\omega(k^*+k))}\right|^q\right|^\frac1q\\
				& + \left(\int_{(-\mathscr{Q},\mathscr{Q})}\mathrm{d}\mathscr{X}\left|\int _{-\infty }^{\infty }\mathrm{d}\nu [1-\chi_{(-\mathscr{R},\mathscr{R})}(\nu)] \int_{\mathbb{T}^d}\mathrm{d}k \mu_1H{\tilde{K}_1(\nu)} e^{{\bf i}\nu(\omega(k)+\sigma\omega(k^*+k)-2\pi\nu{\bf i}\mathscr{X}}\right|^2\right)^\frac12
			\end{aligned}
		\end{equation}
		where $q>2$ is a constant to be specified later.
		We now set
		\begin{equation}
			\begin{aligned}\label{eq:degree1vertex:E7:b}
				& \mathscr{B}_{0,\mathscr{Q}}(\mathscr{X})\ :=\ \int _{-\infty }^{\infty }\mathrm{d}\nu [1-\chi_{(-\mathscr{R},\mathscr{R})}(\nu)] \int_{\mathbb{T}^d}\mathrm{d}k \mu_1H{\tilde{K}_1(\nu)} e^{{\bf i}\nu(\omega(k)+\sigma\omega(k^*+k)-2\pi\nu{\bf i}\mathscr{X}},\\
				& \mathscr{B}_{-1,\mathscr{Q}}(\mathscr{X})\ :=\ \int _{-\infty }^{\infty }\mathrm{d}\nu [1-\chi_{(-\mathscr{R},\mathscr{R})}(\nu)] \int_{\mathbb{T}^d}\mathrm{d}k \mu_1H{\tilde{K}_1(\nu)} e^{{\bf i}\nu(\omega(k)+\sigma\omega(k^*+k)-2\pi\nu{\bf i}\mathscr{X}}[-{\bf i}2\pi \nu]^{-1},\\
				& \mathscr{B}_{1,\mathscr{Q}}(\mathscr{X})\ :=\ \int _{-\infty }^{\infty }\mathrm{d}\nu [1-\chi_{(-\mathscr{R},\mathscr{R})}(\nu)] \int_{\mathbb{T}^d}\mathrm{d}k \mu_1H{\tilde{K}_1(\nu)} e^{{\bf i}\nu(\omega(k)+\sigma\omega(k^*+k)-2\pi\nu{\bf i}\mathscr{X}}[-{\bf i}2\pi \nu].
			\end{aligned}
		\end{equation}
		We observe that $\frac{\mathrm{d}^2}{\mathrm{d}\mathscr{X}^2 }\mathscr{B}_{-1,\mathscr{Q}}=\frac{\mathrm{d}}{\mathrm{d}\mathscr{X} }\mathscr{B}_{0,\mathscr{Q}}=\mathscr{B}_{1,\mathscr{Q}}$. 	We pick $c		>0$ to be a small constant and $\mathscr{Q}'>>\mathscr{Q}$ be a sufficiently large constant. There exists a smooth cut-off function $\mathscr{A}(\mathscr{X})$ that satisfies  $\mathscr{A}(\mathscr{X})=1$ for $-\mathscr{Q}'+2c	<\mathscr{X}<\mathscr{Q}'-2c	$, $\mathscr{A}(\mathscr{X})=0$ for $\mathscr{X}<-\mathscr{Q}'+c	$ or $\mathscr{X}>\mathscr{Q}'-c	$ and $0\le \mathscr{A}(\mathscr{X})\le 1$ elsewhere. We define
		\begin{equation}
			\begin{aligned}\label{eq:degree1vertex:E7:c}
				&  \overline{\mathscr{B}_{-1,\mathscr{Q}}}(\mathscr{X})\ :=\  \mathscr{A}(\mathscr{X})\int _{-\mathscr{R}' }^{\mathscr{R}' }\mathrm{d}\nu [1-\chi_{(-\mathscr{R},\mathscr{R})}(\nu)] \int_{\mathbb{T}^d}\mathrm{d}k \mu_1H{\tilde{K}_1(\nu)} e^{{\bf i}\nu(\omega(k)+\sigma\omega(k^*+k)-2\pi\nu{\bf i}\mathscr{X}}[-{\bf i}2\pi \nu]^{-1}, \\
				& \ \ \frac{\mathrm{d}^2}{\mathrm{d}\mathscr{X}^2 }\overline{\mathscr{B}_{-1,\mathscr{Q}}}=\frac{\mathrm{d}}{\mathrm{d}\mathscr{X} }\overline{\mathscr{B}_{0,\mathscr{Q}}}=\overline{\mathscr{B}_{1,\mathscr{Q}}},
			\end{aligned}
		\end{equation}
		for some sufficiently large constant $\mathscr{R}'>0.$
		Let  $\mathscr{E}(\mathscr{X})$  be the solution of the Laplace equation $\frac{\mathrm{d}^2}{\mathrm{d}\mathscr{X}^2 }\mathscr{E}(\mathscr{X})=\widetilde{\mathscr{B}_{1,\mathscr{Q}}}$ in $L^2(-\mathscr{Q}',\mathscr{Q}')$, and $\mathscr{E}(-\mathscr{Q}')=\mathscr{E}(\mathscr{Q}')=0.$ We take $\phi$ to be an arbitrary function in $L^2(-\mathscr{Q},\mathscr{Q})$. Let $\bar\phi$ be a solution of the Laplace equation $\bar\phi''=\phi$ in $(-\mathscr{Q},\mathscr{Q})$ and $\bar\phi(-\mathscr{Q})=\bar\phi(\mathscr{Q})=0.$ Since
		$$\int_{-\mathscr{Q}}^{\mathscr{Q}}\mathrm{d}\mathscr{X}\mathscr{E}'\phi=\int_{-\mathscr{Q}}^{\mathscr{Q}}\mathrm{d}\mathscr{X}\overline{\mathscr{B}_{0,\mathscr{Q}}}\phi,\ \ \ \int_{-\mathscr{Q}}^{\mathscr{Q}}\mathrm{d}\mathscr{X}\mathscr{E}''\bar\phi'=\int_{-\mathscr{Q}}^{\mathscr{Q}}\mathrm{d}\mathscr{X}\overline{\mathscr{B}_{1,\mathscr{Q}}}\bar\phi',$$
		we infer $\|\mathscr{E}'\|_{L^2(-\mathscr{Q},\mathscr{Q})}=\|\overline{\mathscr{B}_{0,\mathscr{Q}}}\|_{L^2(-\mathscr{Q},\mathscr{Q})}$. As $\overline{\mathscr{B}_{1,\mathscr{Q}}}\in H_0^1(-\mathscr{Q}',\mathscr{Q}')$, we obtain \begin{equation}
			\label{eq:degree1vertex:E7:d}\|\mathscr{E}'\|_{L^2(-\mathscr{Q},\mathscr{Q})}\le \|\mathscr{E}\|_{L^2(-\mathscr{Q}',\mathscr{Q}')}.\end{equation}
		For any $\varphi\in H_0^1(-\mathscr{Q}',\mathscr{Q}')$, we have
		$\int_{-\mathscr{Q}'}^{\mathscr{Q}'}\mathrm{d}\mathscr{X}\mathscr{E}'\varphi=\int_{-\mathscr{Q}'}^{\mathscr{Q}'}\mathrm{d}\mathscr{X}\overline{\mathscr{B}_{0,\mathscr{Q}}}\varphi,$
		yielding
		$$\left|\int_{-\mathscr{Q}'}^{\mathscr{Q}'}\mathrm{d}\mathscr{X}\mathscr{E}\phi\right|=\left|\int_{-\mathscr{Q}'}^{\mathscr{Q}'}\mathrm{d}\mathscr{X}\overline{\mathscr{B}_{-1,\mathscr{Q}}}\phi\right|\le \|\overline{\mathscr{B}_{-1,\mathscr{Q}}}\|_{L^2(-\mathscr{Q}',\mathscr{Q}')}\|\varphi\|_{H_0^1(-\mathscr{Q}',\mathscr{Q}')}.$$
		We now deduce, \begin{equation}
			\label{eq:degree1vertex:E7:e} \|\mathscr{D}\|_{L^2(-\mathscr{R},\mathscr{R})}=\|\mathscr{D}\|_{H^{-1}(-\mathscr{R},\mathscr{R})} \le \|\overline{\mathscr{U}_{-1,\mathscr{R}}}\|_{L^2(-\mathscr{R},\mathscr{R})} \ \lesssim  \left|\int _{-\infty }^{\infty}\mathrm{d}\nu  \left| \int_{\mathbb{T}^d}\mathrm{d}k \mu_1H{\tilde{K}_1(\nu)} e^{{\bf i}\nu(\omega(k)+\sigma\omega(k^*+k)}\right|^q\right|^\frac1q .\end{equation}
		Combining \eqref{eq:degree1vertex:E7:a}, \eqref{eq:degree1vertex:E7:d}, \eqref{eq:degree1vertex:E7:e}, we obtain
		\begin{equation}
			\label{eq:degree1vertex:E7:a:1}\begin{aligned}
				&\left(\int_{(-\mathscr{Q},\mathscr{Q})}\mathrm{d}\mathscr{X}\left|\int _{-\mathscr{R}' }^{\mathscr{R}' }\mathrm{d}\nu \int_{\mathbb{T}^d}\mathrm{d}k \mu_1H{\tilde{K}_1(\nu)} e^{{\bf i}\nu(\omega(k)+\sigma\omega(k^*+k)-2\pi\nu{\bf i}\mathscr{X}}\right|^2\right)^\frac12 \\
				\lesssim\ &  \left[\int _{-\infty }^{\infty }\mathrm{d}\nu \left|\int_{\mathbb{T}^d}\mathrm{d}k \mu_1H{\tilde{K}_1(\nu)} e^{{\bf i}\nu(\omega(k)+\sigma\omega(k^*+k))}\right|^q\right]^\frac1q.\end{aligned}
		\end{equation}
		As this inequality holds for all $\mathscr{R}'>0$ and thus, we can let $ \mathscr{R}'\to \infty$  while \eqref{eq:degree1vertex:E7:a:1} still holds true. 
		We next provide an  estimate of the integral  on the right hand side of \eqref{eq:degree1vertex:E7:a:1}. Let ${G}(\nu)$  be a test function in $L^p(\mathbb{R})$ with $\frac1p+\frac1q=1$, we will try to develop
		
		\begin{equation}
			\label{eq:degree1vertex:E7:a:2}\begin{aligned}
				&\  \Big|\int_{\mathbb{R}} {\mathrm{d}\nu}	 \int_{\mathbb{T}^d}\mathrm{d}k \mu_1H(k){\tilde{K}_1(\nu)} e^{{\bf i}\nu(\omega(k)+\sigma\omega(k^*+k))}{G}(\nu)\Big|.
			\end{aligned}
		\end{equation}
	As $H\in L^{\Im}(\mathbb{T}^d)$, we   study the dual norm $L^{\frac{\Im}{\Im-1}}$ with $\frac{\Im}{\Im-1}=2{\Im_o}$
		\begin{equation}
			\label{eq:degree1vertex:E7:a:3}\begin{aligned}
				& \Big\|\int_{\mathbb{R}} {\mathrm{d}\nu}	 \mu_1{\tilde{K}_1(\nu)} e^{{\bf i}\nu(\omega(k)+\sigma\omega(k^*+k))}{G}(\nu)\Big\|_{L^{2{\Im_o}}}^{2{\Im_o}},
			\end{aligned}
		\end{equation}
		which we will expand and bound as follows
		\begin{equation}
			\label{eq:degree1vertex:E7:a:4}\begin{aligned}
				& \left|\int_{\mathbb{R}^{\Im_o}} \prod_{j=1}^{\Im_o}{\mathrm{d}\nu_j}\int_{\mathbb{R}^{\Im_o}} \prod_{j=1}^{\Im_o}{\mathrm{d}\nu_j'}	\int_{\mathbb{T}^d}\mathrm{d}k\mu_1^{2\Im_o}\prod_{j=1}^{\Im_o}{\tilde{K}_1(\nu_j)} e^{{\bf i}\big(\sum_{j=1}^{\Im_o}\nu_j\big)(\omega(k)+\sigma\omega(k^*+k))}\prod_{j=1}^{\Im_o}{G}(\nu_j)\right.\\
				&\left.\times \prod_{j=1}^{\Im_o}{\tilde{K}_1(\nu'_j)} e^{-{\bf i}\big(\sum_{j=1}^{\Im_o}\nu_j'\big)(\omega(k)+\sigma\omega(k^*+k))}\prod_{j=1}^{\Im_o}\overline{{G}(\nu'_j)}\right|\\
				\lesssim\	& \left[\int_{\mathbb{R}} {\mathrm{d}\nu}|{G}(\nu)|^p\right]^\frac{2\Im_o}{p} \left[\int_{\mathbb{R}} {\mathrm{d}\nu}e^{-q\lambda^2 |\nu|/(2\Im_o)}|\mathfrak{F}_\omega(\nu)|^\frac{q}{2\Im_o}\right]^\frac{2\Im_o}{q},
			\end{aligned}
		\end{equation}
		where
		\begin{equation}
			\label{eq:degree1vertex:E7:a:5}\begin{aligned}
				\mathfrak{F}_\omega(\nu) \ = \ & \int_{\mathbb{T}^d}\mathrm{d}k \mu_1^2(k)\  e^{{\bf i}\nu(\omega(k)+\sigma\omega(k^*+k))}.
			\end{aligned}
		\end{equation}
		
	Now, the same estimates used in the proof of Lemma 45 of \cite{staffilani2021wave} can be repeated, we finally obtain

\begin{equation}
	\label{eq:de1}
	\begin{aligned}
		\left\|	\int_{\mathbb{T}^d}\mathrm{d}k\frac{\mu_1|H|}{|\omega(k)+\sigma \omega(k^*+k)  +2\pi\mathscr{X}- \alpha + {\bf i}\lambda^2|}\right\|_{L^2(-\mathscr{Q},\mathscr{Q})}\
		\lesssim \ 	& \frac{{\langle\mathrm{ln}\lambda\rangle}^2\|H\|_{L^\Im}}{|d(\mathfrak{S})|^3}.
	\end{aligned}
\end{equation}
	\end{proof}
\end{lemma}

	In the  definition of $Q_1^{colli}$ below, $F_1 $ is a function of $k_1$, while in the definition of $Q_2^{colli}$, $F_1 $ is a function of $k_0$. 

\begin{lemma}
	\label{Lemma:Resonance1}
	Let $1>\upsilon	>0$ be a positive constant, suppose that $d\ge 2$, we define the following operators, for all $k_0,k_1,k_2,k_0',k_1',k_2'\in \mathbb{T}^d$, $\sigma_0,\sigma_1,\sigma_2,\sigma_0',\sigma_1',\sigma_2'\in\{\pm1\}$,  and for any constant $T_o>0$
	\begin{equation}\label{Lemma:Resonance1:a}
		\begin{aligned} 
			& 
			F_1,F_2\in L^2(\mathbb{T}^{2d}) \longrightarrow \mathcal{Q}^{colli}_1[F_1,F_2,\sigma_0,\sigma_1,\sigma_2,\sigma_0',\sigma_1',\sigma_2',s](k_0,k_0'):=\\\
			& e^{{\bf i}s\sigma_0 \omega(k_0)}\iiiint_{(\mathbb{T}^d)^4}\!\! \mathrm{d} k_1  \mathrm{d} k_2\mathrm{d} k_1'  \mathrm{d} k_2' \,
			\delta(\sigma_0k_0+\sigma_1k_1+\sigma_2k_2) \delta(\sigma_0'k_0'+\sigma_1'k_1'+\sigma_2'k_2') 
			e^{{\bf i}s\sigma_1 \omega(k_1)}\\
			&\times  {F}_1(k_1,k_1') F_2(k_2,k_2')e^{{\bf i} s \sigma_2 \omega(k_2)}\sqrt{ |\sin(2\pi k_1^1)||\sin(2\pi k_2^1)||\sin(2\pi k_1'^1)||\sin(2\pi k_2'^1)|},
		\end{aligned}
	\end{equation}
	\begin{equation}\label{Lemma:Resonance1:b}
		\begin{aligned} 
			& 
			F_1\in L^\infty(\mathbb{T}^{2d}),F_2\in L^2(\mathbb{T}^{2d}) \longrightarrow \mathcal{Q}^{colli}_2[F_1,F_2,\sigma_0,\sigma_1,\sigma_2,\sigma_0',\sigma_1',\sigma_2',s](k_0,k_0'):=\\\
			& e^{{\bf i}s\sigma_0 \omega(k_0)}\iiiint_{(\mathbb{T}^d)^4}\!\! \mathrm{d} k_1  \mathrm{d} k_2\mathrm{d} k_1'  \mathrm{d} k_2' 
			\delta(\sigma_0k_0+\sigma_1k_1+\sigma_2k_2) \delta(\sigma_0k_0'+\sigma_1k_1'+\sigma_2k_2') 
			e^{{\bf i}s\sigma_1 \omega(k_1)}|\sin(2\pi k_1^1)|\\
			&\times \sqrt{|\sin(2\pi k_0^1)||\sin(2\pi k_0'^1)|}  {F}_1(k_0,k_1') F_2(k_2,k_2')e^{{\bf i} s \sigma_2 \omega(k_2)} \sqrt{|\sin(2\pi k_2^1)||\sin(2\pi k_2'^1)|},
		\end{aligned}
	\end{equation}
	\begin{equation}\label{Lemma:Resonance1:b}
		\begin{aligned} 
			& 
			F_1\in L^2(\mathbb{T}^{2d}) \longrightarrow \mathcal{Q}^{colli}_3[F_1,\sigma_0,\sigma_1,\sigma_2,\sigma_0',\sigma_1',\sigma_2',s](k_0,k_0'):=\\\
			& e^{{\bf i}s\sigma_0 \omega(k_0)}\iiiint_{(\mathbb{T}^d)^4}\!\! \mathrm{d} k_1  \mathrm{d} k_2\mathrm{d} k_1'  \mathrm{d} k_2' 
			\delta(\sigma_0k_0+\sigma_1k_1+\sigma_2k_2) \delta(\sigma_0k_0'+\sigma_1k_1'+\sigma_2k_2') \\
			&\times
			e^{{\bf i}s\sigma_1 \omega(k_1)}\sqrt{|\sin(2\pi k_0^1)||\sin(2\pi k_1^1)| |\sin(2\pi k_0'^1)||\sin(2\pi k_1'^1)|}\\
			&\times  F_1(k_1,k_1')e^{{\bf i} s \sigma_2 \omega(k_2)} \sqrt{|\sin(2\pi k_2^1)||\sin(2\pi k_2'^1)|}.
		\end{aligned}
	\end{equation}
	Moreover, we also define
	
	\begin{equation}
		\label{Lemma:Resonance1:1}\begin{aligned}
			&	F_1,F_2\in L^2(\mathbb{T}^{2d}), \phi\in C^\infty_c(\mathbb{T}^{2d}\times(-\infty,\infty))\longrightarrow  \ \mathbb{O}^1_{\upsilon, \lambda}[F_1,F_2][\phi] :=\\
			&  \int_{-T_o\lambda^{-2}}^{T_o\lambda^{-2}} \frac{\mathrm{d}s}{\pi}e^{-\upsilon |s|}\iiint_{(\mathbb{T}^d)^3}\!\! \mathrm{d} k_0  \mathrm{d} k_1  \mathrm{d} k_2\iiint_{(\mathbb{T}^d)^3}\!\! \mathrm{d} k_0'  \mathrm{d} k_1'  \mathrm{d} k_2'\\
			&\times \Big[\delta(\sigma_0k_0+\sigma_1k_1+\sigma_2k_2)\delta(\sigma_0k_0'+\sigma_1k_1'+\sigma_2k_2')
			e^{{\bf i}s[\sigma_0 \omega(k_0)+\sigma_1 \omega(k_1)+\sigma_2 \omega(k_2)]}\\
			&\times \ \ \sqrt{|\sin(2\pi k_1^1)||\sin(2\pi k_2^1)||\sin(2\pi k_1'^1)||\sin(2\pi k_2'^1)|}	{F}_1(k_1,k_1') F_2(k_2,k_2')\phi(k_0,k_0',s)\Big],
		\end{aligned}	
	\end{equation}
	and
	\begin{equation}
		\label{Lemma:Resonance1:1}\begin{aligned}
			&	F_2\in L^2(\mathbb{T}^{2d}), F_1\in  L^\infty(\mathbb{T}^{2d}), \phi\in C^\infty_c(\mathbb{T}^{2d}\times(-\infty,\infty))\longrightarrow  \ \mathbb{O}^2_{\upsilon, \lambda}[F_1,F_2][\phi] :=\\
			& \int_{-T_o\lambda^{-2}}^{T_o\lambda^{-2}} \frac{\mathrm{d}s}{\pi}e^{-\upsilon |s|} \iiint_{(\mathbb{T}^d)^3}\!\! \mathrm{d} k_0  \mathrm{d} k_1  \mathrm{d} k_2\iiint_{(\mathbb{T}^d)^3}\!\! \mathrm{d} k_0' \mathrm{d} k_1'  \mathrm{d} k_2' \,\\
			&\times \Big[\delta(\sigma_0k_0+\sigma_1k_1+\sigma_2k_2)\delta(\sigma_0k_0'+\sigma_1k_1'+\sigma_2k_2')
			e^{{\bf i}s[\sigma_0 \omega(k_0)+\sigma_1 \omega(k_1)+\sigma_2 \omega(k_2)]}\\
			&\times \sqrt{|\sin(2\pi k_0^1)||\sin(2\pi k_1^1)||\sin(2\pi k_2^1)||\sin(2\pi k_0'^1)||\sin(2\pi k_1'^1)||\sin(2\pi k_2'^1)|}\\
			&\times	{F}_1(k_0,k_1') F_2(k_2,k_2')\phi(k_0,k_0',s)\Big],
		\end{aligned}	
	\end{equation}
	where $k_0=(k_0^1,\cdots,k_0^d),k_1=(k_1^1,\cdots,k_1^d),k_2=(k_2^1,\cdots,k_2^d),$ $k_0'=(k_0'^1,\cdots,k_0'^d),$ $k_1'=(k_1'^1,\cdots,k_1'^d),$ $k_2'=(k_2'^1,\cdots,k_2'^d)\in\mathbb{T}^d.$

	The following claims then hold true.
	\begin{itemize}
		\item[(a)]  
		There exists a constant $\mathscr{M}>1$ such that we  have
		\begin{equation}
			\label{Lemma:Resonance1:2:bis}\begin{aligned}
				& \lim_{\upsilon\to 0} \int_{-T_o\lambda^{-2}}^{T_o\lambda^{-2}} \frac{\mathrm{d}s}{\pi}e^{-\upsilon |s|}\Big\|\mathcal {Q}^{colli}_1[F_1,F_2,\sigma_0,\sigma_1,\sigma_2,s]\Big\|_{L^2}^\mathscr{M}\\
				\ = \ & \int_{-T_o\lambda^{-2}}^{T_o\lambda^{-2}} \frac{\mathrm{d}s}{\pi}\Big\|\mathcal {Q}^{colli}_1[F_1,F_2,\sigma_0,\sigma_1,\sigma_2,s]\Big\|_{L^2}^\mathscr{M},\end{aligned}	
		\end{equation}	
		\begin{equation}
			\label{Lemma:Resonance1:2:bis1}\begin{aligned}
				&\ \int_{-T_o\lambda^{-2}}^{T_o\lambda^{-2}} {\mathrm{d}s}\Big\|\mathcal {Q}^{colli}_1[F_1,F_2,\sigma_0,\sigma_1,\sigma_2,s]\Big\|_{L^2}^\mathscr{M} \lesssim \|F_1\|_{L^2}^\mathscr{M}  \|F_2\|_{L^2}^\mathscr{M}. 
			\end{aligned}
		\end{equation}
		and 	\begin{equation}
			\label{Lemma:Resonance1:2:bis1:a}\begin{aligned}
				\lim_{\upsilon\to 0}\mathbb{O}^1_{\upsilon, \lambda}[F_1,F_2][\phi]
				\ =  &\ \mathbb{O}^1_{0, \lambda}[F_1,F_2][\phi].
			\end{aligned}
		\end{equation}	
		\item[(b)] 
		Moreover, we also have
		
		\begin{equation}
			\label{Lemma:Resonance1:2}\begin{aligned}
				& \lim_{\upsilon\to 0} \int_{-T_o\lambda^{-2}}^{T_o\lambda^{-2}} \frac{\mathrm{d}s}{\pi}e^{-\upsilon |s|}\Big\|\mathcal {Q}^{colli}_2[F_1,F_2,\sigma_0,\sigma_1,\sigma_2,s]\Big\|_{L^2}^\mathscr{M}\\
				\ = \ & \int_{-T_o\lambda^{-2}}^{T_o\lambda^{-2}} \frac{\mathrm{d}s}{\pi}\Big\|\mathcal {Q}^{colli}_2[F_1,F_2,\sigma_0,\sigma_1,\sigma_2,s]\Big\|_{L^2}^\mathscr{M},\end{aligned}	
		\end{equation}	
		\begin{equation}
			\label{Lemma:Resonance1:2:bis2}\begin{aligned}
				&\ \int_{-T_o\lambda^{-2}}^{T_o\lambda^{-2}} {\mathrm{d}s}\Big\|\mathcal {Q}^{colli}_2[F_1,F_2,\sigma_0,\sigma_1,\sigma_2,s]\Big\|_{L^2}^\mathscr{M} \\
				\lesssim & \|\sqrt{ |\sin(2\pi k^1)\sin(2\pi k'^1)}|F_1(k,k')\|_{L^\infty\times L^2}^\mathscr{M}\|F_2\|_{L^2}^\mathscr{M}, 
			\end{aligned}
		\end{equation}
		with $k=(k^1,\cdots,k^d),k'=(k'^1,\cdots,k'^d)\in\mathbb{T}^d,$
		and 	\begin{equation}
			\label{Lemma:Resonance1:2:bis2:a}\begin{aligned}
				\lim_{\upsilon\to 0}\mathbb{O}^2_{\upsilon, \lambda}[F_1,F_2][\phi]
				\ =  &\ \mathbb{O}^2_{0, \lambda}[F_1,F_2][\phi].
			\end{aligned}
		\end{equation}	
		\item[(c)] The following bounds hold true for the collision operators $\mathcal{Q}^{colli}_1$  \begin{equation}\label{Lemma:Resonance1:7}
			\begin{aligned} 
				& 
				\Big\|\mathcal {Q}^{colli}_1[F_1,F_2,\sigma_0,\sigma_1,\sigma_2,s]\Big\|_{L^\infty}\\
				\ \lesssim\	&	\Big\|\sqrt{|\sin(2\pi k^1)||\sin(2\pi k'^1)|} F_1(k,k')\Big\|_{L^2}\Big\|\sqrt{|\sin(2\pi k^1)||\sin(2\pi k'^1)|}  F_2(k,k')\Big\|_{L^2},
			\end{aligned}
		\end{equation}
		and $\mathcal{Q}^{colli}_2$
		
		\begin{equation}\label{Lemma:Resonance1:8}
			\begin{aligned} 
			&\Big\|\mathcal {Q}^{colli}_2[F_1,F_2,\sigma_0,\sigma_1,\sigma_2,s]\Big\|_{L^\infty}\\
				\lesssim\			&	\Big\|\sqrt{|\sin(2\pi k^1)||\sin(2\pi k'^1)|} F_1(k,k')\Big\|_{L^\infty}\Big\|\sqrt{|\sin(2\pi k^1)||\sin(2\pi k'^1)|} F_2(k,k') \Big\|_{L^2\times L^\infty}.
			\end{aligned}
		\end{equation}
		\item[(d)] 	For any $\infty\ge p> 2$
		\begin{equation}
			\label{Lemma:Resonance1:9}\begin{aligned}
				&\ \Big[\int_{-T_o\lambda^{-2}}^{T_o\lambda^{-2}} {\mathrm{d}s}\Big\|\mathcal {Q}^{colli}_3[F_1,\sigma_0,\sigma_1,\sigma_2,s]\Big\|_{L^p}^\mathscr{M}\Big]^\frac{1}{\mathscr{M}} \lesssim \|F_1\|_{L^2}. 
			\end{aligned}
		\end{equation}
		
	\end{itemize}
\end{lemma}
\begin{proof}
	The proof is quite similar as that of Lemma 50 of \cite{staffilani2021wave} and  is hence omitted. 
\end{proof}
\begin{remark}\label{Remark:Resonance} 	
	Let us remark that the problem of analyzing resonance manifolds, as well as the near-resonance consideration, is an interesting research topic by itself and have been studied in  \cite{CraciunBinh,CraciunSmithBoldyrevBinh,GambaSmithBinh,germain2020optimal,nguyen2017quantum,ToanBinh,PomeauBinh,rumpf2021wave,Binh1}. 
\end{remark}
\section{Feynman diagrams}
\label{Duhamel}

\begin{figure}
	\centering
	\includegraphics[width=.99\linewidth]{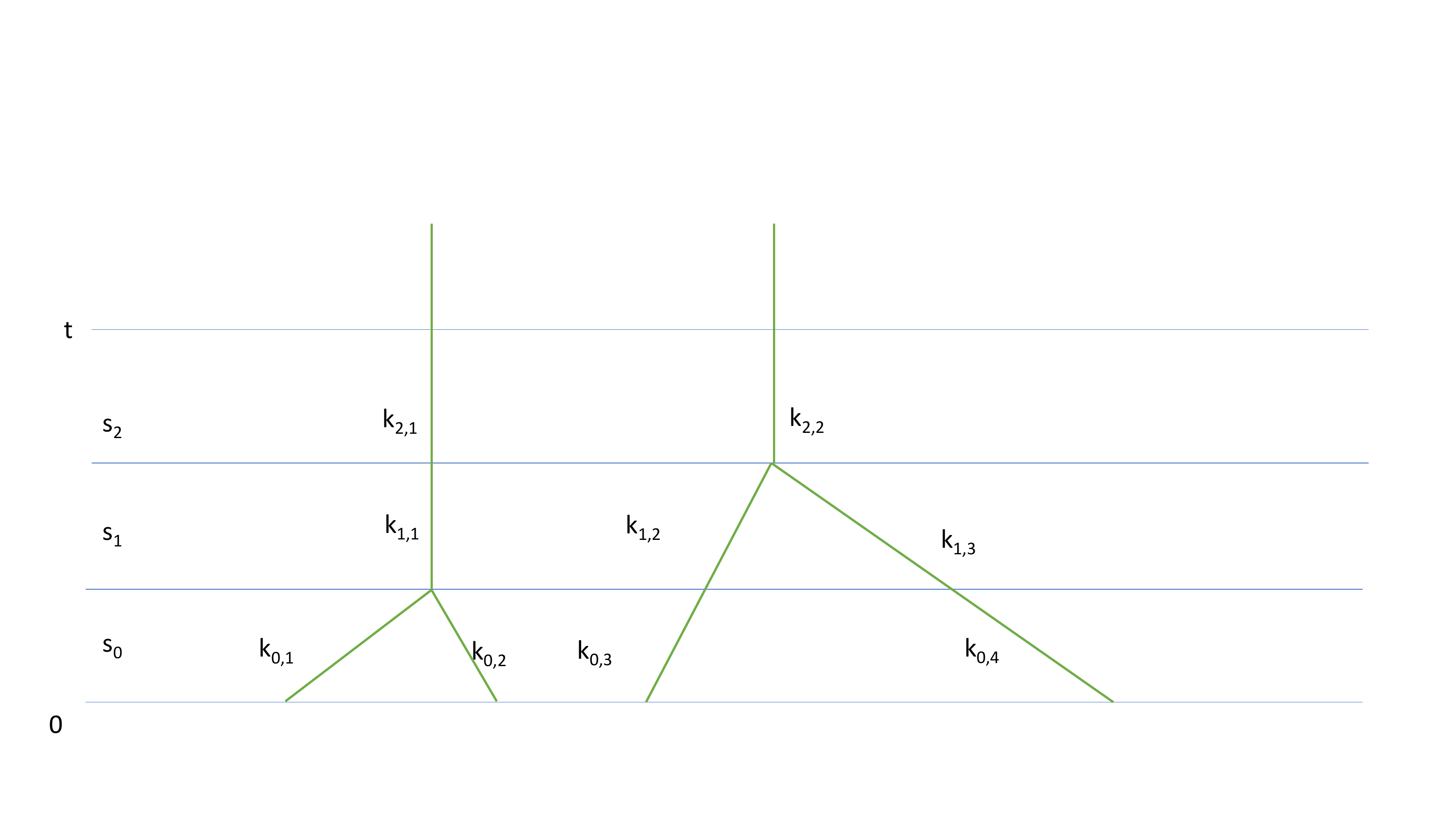}
	\caption{An example of a Feynman diagram. At time slice $s_0$, the edges are $k_{0,1},k_{0,2},k_{0,3},k_{0,4}$, with the signs $-,-,+,+$. At time slice $s_1$, the edges are $k_{1,1},k_{1,2},k_{1,3}$, with the signs $-,+,+$.  At time slice $s_2$, the edges are $k_{2,1},k_{2,2}$, with the signs $-,+$. We have $-k_{1,1}+k_{0,1}+k_{0,2}=0$ and $k_{2,2}-k_{1,2}-k_{1,3}=0$.}
	\label{Fig1}
\end{figure}

In this section, we will follow the construction in \cite{staffilani2021wave} for our  Feynman diagrams. The construction in \cite{staffilani2021wave} is as follows.
\begin{itemize}
	\item The time interval $[0,t]$ is divided into $n+1$ time slices $s_0$, $s_1$, $\dots$, $s_n$. We represent the time slices $s_0$, $s_1$, $\dots$, $s_n$ from the bottom to the top of the diagram.
	\item At each time slice $s_i$, only one Duhamel expansion is allowed, with the term containing $\Phi_{1,i}$.  At time slice $s_i$, we have a combination of one couple of the segments of time slice $s_{i-1}$ into one segment of time slice $s_i$. Those segments of time slice $s_{i-1}$ are denoted  by $k_{i-1,\rho_i},k_{i-1,\rho_i+1}$ and the one at time slice $s_i$ is denoted by $k_{i,\rho_i}$. 
	\item  If  $i$ is the index of time slice $s_i$, the  combination happens at the segment  $\rho_i$.
	\item Moreover, the signs for $k_{i-1,\rho_i},k_{i-1,\rho_i+1}$  are denoted by $\sigma_{i-1,\rho_i},\sigma_{i-1,\rho_i+1}$. We then have $\sigma_{i,\rho_i}k_{i,\rho_i}+\sigma_{i-1,\rho_i}k_{i-1,\rho_i}+\sigma_{i-1,\rho_i+1}k_{i-1,\rho_i+1}=0.$
	\item The number of segments at time slice $s_{i}$ is $2+n-i$. They are indexed by $k_{i,1},\cdots k_{i,2+n-i}$.  
	\item We index the segments in two consecutive time slices $s_{i-1}$ and $s_i$ as follows: $k_{i,l}=k_{i-1,l}, \sigma_{i,l}=\sigma_{i-1,l}, $ for  $l\in\{1,\cdots,\rho_i-1\}$ and $k_{i,l}=k_{i-1,l+1}, $ $ \sigma_{i,l}=\sigma_{i-1,l+1}, $ for  $l\in\{\rho_i+1,\cdots,2+{n-i}\}$.  
	\item In the sequel, we will identify an edge with its attached momentum.
\end{itemize}

\subsection{General diagrams}\label{Sec:FirstDiagrams} Below we construct the diagrams corresponding to the expressions   $\mathcal{G}^0_n$, $\mathcal{G}^1_n$, $\mathcal{G}^2_n$, $\mathcal{G}^3_n$ following \cite{staffilani2021wave} (see Section \ref{Sec:Cum}).
\begin{figure}
	\centering
	\includegraphics[width=.99\linewidth]{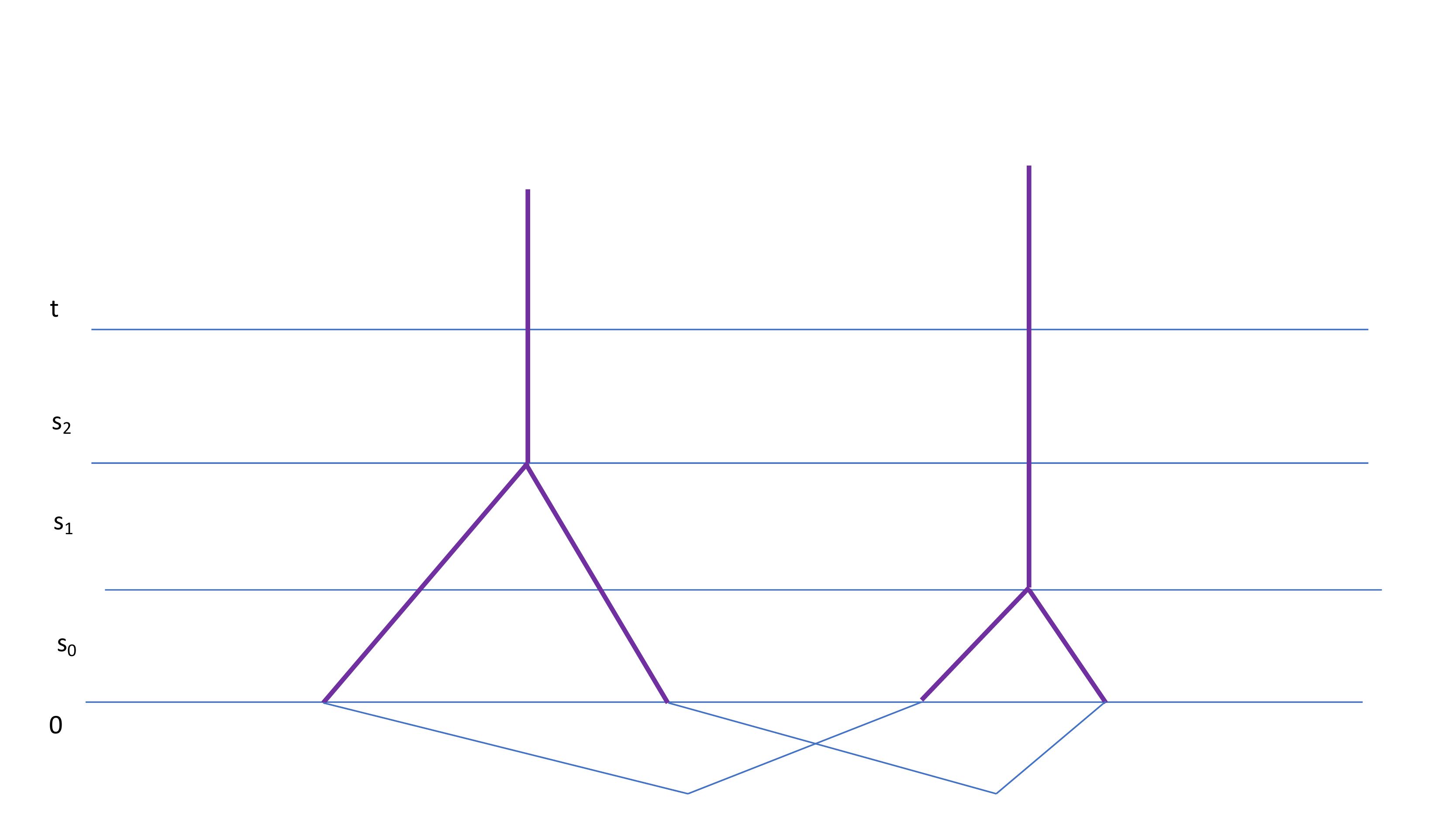}
	\caption{An example of a Feynman diagram with clusters. One new cluster vertex is added at the bottom of the diagram. The arrows correspond to the {\it first assigned orientation.}}
	\label{Fig2}
\end{figure}
\begin{itemize}
	\item For each  function $\square\Big(\sum_{j\in K} k_{0,j}\sigma_{0,j}=\mathscr{K}_K\Big)$ in $\mathcal{G}^0_n$, $\mathcal{G}^1_n$, $\mathcal{G}^2_n$, $\mathcal{G}^3_n$, we assume that  \begin{equation}
		\label{ClusterMomentaEps}\sum_{j\in K} k_{0,j}\sigma_{0,j}=\mathscr{K}_K \mbox{ with } |\mathscr{K}_K |\le \lambda^{2}\complement,
	\end{equation}   where $\mathscr{K}_K$ are   new momenta. We present this on the diagram by using an extra ``cluster vertex'' at the bottom of the graph, for each $K$. The role of the cluster vertex is to connect all of the vertices $ k_{0,j}$, $j\in K$, and we assume that each cluster vertex also carries a new ``cluster momenta'' $$\mathscr{K}_K=\sum_{j\in K} k_{0,j}\sigma_{0,j}$$  (see Figure \ref{Fig2}). For each pair $K=\{i,j\}\in S\in \mathcal{P}_{pair}(\{1,\cdots,n+2\}) $ appeared in the presentation of $Q_*^1$, we also make a cluster vertex that connect $k_{0,i}$ and $k_{0,j}$ and also set
\begin{equation}
	\label{ClusterMomentaEps}\sum_{l\in K=\{i,j\}} k_{0,l}\sigma_{0,l}=\mathscr{K}_K \mbox{ with } |\mathscr{K}_K |\le \lambda^{2}\complement.
\end{equation}  
 Those cluster momenta will be integrated at the end of our process.
	Graphs that have clusters consisting of only one element, which connects all of the vertices at the bottom, are defined to be  ``residual graphs''.

	\item The two identities   $\mathbf{1}(\sigma_{n,2}=1)$ and $\mathbf{1}(\sigma_{n,1}=-1)$ can be represented  by  setting $\sigma_{n,1}=-1$ and $\sigma_{n,2}=1$ for the signs of the two topmost segments, where the value of $k_{n,1}$, the momentum of the topmost segment on the left, is the same with $k_{n,2}$, the momentum of the topmost segment on the right.    The remaining signs  are assigned to the diagrams  using the same rule, from the top to the bottom. 
	\item The total phase of the diagram are written as  $\prod_{i=0}^n e^{-\mathrm{i} s_i \vartheta_i}$, where
	\begin{equation}\label{GraphSec:E3}\begin{aligned}
			\mathrm{Re}\vartheta_i\ =\  &\sum_{l=1}^{n-i+2}\sigma_{i,l}\omega(k_{i,l})-2\omega(k_{n,1}),\end{aligned}\end{equation}

	\item  In each diagram, there are two components. The  one on the left is  associated to the momentum $k_{n,1}$ and  the one on the right is associated to the momentum $k_{n,2}$. The left component is called  the ``minus diagram'' and the right one is called the ``plus diagram''. 
	We assume that the first splitting always happens on the plus diagram.

\end{itemize}

Next, we will introduce the way to integrate the diagrams, by adding an extra orientation to these diagrams, that indicates which edges are integrated first, which edges are integrated next. Diagrams that can be  integrated are called integrated graphs.
An integrated graph is denoted by $\mathfrak{G}=(\mathfrak{V},\mathfrak{E})$, in which $\mathfrak{V}$ and $\mathfrak{E}$ are the sets of vertices and edges.  The construction of an integrated graph is as follows.

\begin{itemize}
	\item We suppose that $\mathfrak{V}^0$ contains two initial vertices $v_{n+2}$ and $v_{n+1}$ and start with $\mathfrak{G}^0=(\mathfrak{V}^0,\mathfrak{E}^0)$. We set that $v_{n+1}$ to be in  the minus tree. At the first iteration, a new edge $e_1$ is attached to $v_{n+2}$. This edge belongs to the plus tree and contains two vertices, one of its vertices is $v_{n+2}$ and the other  vertex is denoted by $v_{n}$.   In the next iteration, a new edge $e_2$ is attached to $v_{n+1}$. This edge is in the minus tree and  its vertices are $v_{n+1}$ and $v_{n-1}$.  In our construction, the first combination is on the plus diagram, thus, the edge $e_1$ will be decomposed into another two edges. This leads to our construction for the next step: we attach  two edges to  $v_{n}$, from left to right.   The parameter $\rho_{n-1}$ contains the location of the next  vertex where the expansion happens. This vertex  is either $v_{n-1}$ or $v_{n-2}$. We keep attaching two new edges to either $v_{n-1}$ or $v_{n-2}$ in the next step. The above procedure can be repeated and all of the     vertices   can be labeled.  Those vertices are  $v_{n}, v_{n-1}, \cdots v_1$ from the top to the bottom of the diagrams. We denote $\mathfrak{V}_H=\{v_1,\cdots,v_{n+2}\}$. The set $\mathfrak{V}_H$ is called ``higher time  vertex set''. The vertices from $v_1$ to $v_n$ are called ``interacting vertices'' and the set of all of them is called ``interacting vertex set'' $\mathfrak{V}_I=\mathfrak{V}_H\backslash \mathfrak{V}_T=\mathfrak{V}_H\backslash \{v_{n+1},v_{n+2}\}.$ The set of the other vertices is called ``zero time vertex set'' and denoted by $\mathfrak{V}_0$. For each cluster $K$ in the definition of  $\mathcal{G}^0_n$, $\mathcal{G}^1_n$, $\mathcal{G}^2_n$, $\mathcal{G}^3_n$, we call the associated vertex by ``cluster vertex''. The  set $\mathfrak{V}_C$ of all cluster vertices is set to be the  ``cluster vertex set''. The vertex set of the final graph $\mathfrak{G}$ is then $\mathfrak{V}=\mathfrak{V}_T\cup\mathfrak{V}_H\cup\mathfrak{V}_0\cup\mathfrak{V}_C.$ 
	\item For any  pair of vertices $(v,v')$, if there exists a set of vertices $\{v_{i_j}\}_{j=1}^m$ such that $v_{i_1}=v$, $v_{i_m}=v'$, and $v_{i_l}$ is connected to $v_{i_{l+1}}$ by an edge, we define $\{v_{i_j}\}_{j=1}^m$ to be a ``path'' connecting $v$ and $v'$. For any $v\in\mathfrak{V}$,  the set of all edges attached to $v$ is denoted by $\mathfrak{E}(v) = \{e\in\mathfrak{E} \  | \ v\in e\}$.  For each vertex $v\in\mathfrak{V}_I$, we define two new sets $\mathfrak{E}_+(v)$ and $\mathfrak{E}_-(v)$, that satisfy $\mathfrak{E}(v)=\mathfrak{E}_+(v)\cup \mathfrak{E}_-(v)$, in which $\mathfrak{E}_+(v)=\{e\}$, where $e$ is the first edge attached to $v$ in the construction  and $\mathfrak{E}_-(v)=\mathfrak{E}(v)\backslash\{e\}$. Next, we define the cluster $\mathfrak{V}_c^{(j)}$ as a cluster  of  edges intersecting with the time slice $j$. The construction of $\mathfrak{V}_c^{(j)}$ is interatively done as follows. We suppose that the additional vertex $v_j$ combines the two edges in $\mathfrak{E}_-(v_j)$ into the new one in $\mathfrak{E}_+(v_j)$ and the two edges of $\mathfrak{E}_-(v_j)$ belong the cluster constructed in the previous iteration $\mathfrak{V}_c^{(j-1)}$. The cluster $\mathfrak{V}_c^{(j)}$ is constructed by joining all clusters of $\mathfrak{V}_c^{(j-1)}$ and replacing the three edges by the one in $\mathfrak{E}_+(v_j)$. The bottom one   $\mathfrak{V}_c^{(0)}$ is precisely $\mathfrak{V}_c$. We also define the function $\mathcal{T}:\mathfrak{V}\to [0,n+2]$, in which $\mathcal{T}(v)=j$ if $v=v_j$ for $j\in\{1,\cdots,n+2\}$,  and $\mathcal{T}(v)=0$ in the other cases. 
\item For an interacting vertex $v_i$, denoting by $k_0,k_1,k_2$ and $\sigma_0,\sigma_1,\sigma_2$ the  momenta associated to $v_i$ and the corresponding signs,  the total phase of $v_i$ is defined to be
$
\label{Def:Xvi}
\mathbf{X}_i \ = \ \mathbf{X}(v_i) \ = \ \sigma_0\omega(k_0)+\sigma_1\omega(k_1)+
\sigma_2\omega(k_2).
$ We denote with $\mathfrak{E}_{in}(v)$ by the set of all the edges $e$ in $\mathfrak{E}(v)$ such that the sign of their momenta $k_e$ is always $+1$, and  $\mathfrak{E}_{out}(v)$  the set of all the edges $e$ in $\mathfrak{E}(v)$ such that the sign of their momenta $k_e$ is always $-1$.
\end{itemize} 
The above process gives us an unoriented graph $(\mathfrak{V},\mathfrak{E}).$   Next, we define an orientation for our graphs, called the {\it first assigned orientation.}

\begin{itemize}
	\item For a segment $e$ associated to an interacting vertex $v$ and belongs to $\mathfrak{E}_-(v)$, the orientation of  this edge is going toward the vertex $v$ if the sign $\sigma_e$ of the associated  momentum $k_e$ is $+1$, otherwise, the orientation is going outward the vertex. Similarly, if  $e$ belongs to $\mathfrak{E}_+(v)$,  this edge has the orientation of going toward the vertex $v$ if the sign $\sigma_e$ of the associated  momentum $k_e$ is $-1$, otherwise,  the orientation is going outward the vertex.  The edges of the cluster vertices are also assigned directions, as follows. Let us consider  $e$ being a segment associated  to some cluster vertex. We suppose that  the vertex of $e$  belonging to $\mathfrak{V}_0$  is $v'$, which should belong to another segment $e'$ connecting to another interacting vertex $v''\in\mathfrak{V}_I$. As  $e'$ has already an orientation, if  the orientation  is going inward the vertex $v'$, then the orientation of $e$ is going outward the vertex $v'$ and if the orientation $e'$ is going outward the vertex $v'$, then the orientation of $e$ is going inward the vertex $v'$. 
	\item For any $v\in\mathfrak{V}_I$, let $e$ be an edge in $\mathfrak{E}_-(v)$,  $v'$ be the other vertex of the edge $e$ then   $\sigma_e$ is the sign associated to the edge $e$, with respect to the vertex $v$ and the sign associated to the edge $e$, with respect to the vertex $v'$ is the opposite of $\sigma_e$, which is $-\sigma_e$. We then define   $\sigma_v \ : \ \mathfrak{E}(v) \to \{-1,1\}$ as follow
	\begin{equation}\label{GraphSec:E5}
		\begin{aligned}
			\sigma_v (e) \  = & \ \sigma_e,\
			\sigma_{v'} (e) \  =  \ -\sigma_e.
		\end{aligned}
	\end{equation}
For any $v_K\in\mathfrak{V}_C$, let $e$ be an edge attached to   $v_K$  then   $\sigma_e$ is the sign associated to the edge $e$. Suppose that $v'$ is the other vertex in $\mathfrak{V}_0$ of $e$. We then define   $\sigma_v \ : \ \mathfrak{E}(v) \to \{-1,1\}$ as follow
\begin{equation}\label{GraphSec:E5bis}
	\begin{aligned}
		\sigma_{v_K}(e) \  = & \ \sigma_e,\
		\sigma_{v'} (e) \  =  \ -\sigma_e.
	\end{aligned}
\end{equation}
\end{itemize}

Our aim is to integrate out all the delta functions.

\begin{definition}[Cycles \cite{staffilani2021wave}] Let us consider a graph of $\mathfrak{N}$ interacting vertices and a set of vertices $\{v_{i}\}_{i\in\mathfrak{I}}$ of $\mathfrak{V}_H\cup\mathfrak{V}_0$. We consider two cases.
	
	{\it Case 1: The set of vertices $\{v_{i}\}_{i\in\mathfrak{I}}$  does not contain both of the two top vertices $v_{n+1}$ and $v_{n+2}$.} If there exists a set of cluster vertices $\{u_K\}_{K\in S'\subset S}$ such that we could go from one vertex $v_{i_1}$ of the set $\{v_{i}\}_{i\in\mathfrak{I}}$  to all of the vertices of $\{v_{i}\}_{i\in\mathfrak{I}} \cup \{u_K\}_{K\in S'\subset S}$ and back to $v_{i_1}$ via the edges  of the graph, we call it a cycle. Suppose that $v_l$ is the top most interacting vertex in the cycle, we say that this is a cycle of the vertex $v_l$.
	
	{\it Case 2: The set of vertices $\{v_{i}\}_{i\in\mathfrak{I}}$   contains both of the two top vertices $v_{n+1}$ and $v_{n+2}$.}
	We ignore $v_{n+1}$ and $v_{n+2}$ and introduce a ``virtual vertex'' $v_*$, that connects both of the two  vertices $v_n$ and $v_{n-1}$.  If there exists a set of cluster vertices $\{u_K\}_{K\in S'\subset S}$ such that we could go from $v_{*}$ to all of the vertices of $\{v_*\}\cup \{v_{i}\}_{i\in\mathfrak{I}}\cup \{u_K\}_{K\in S'\subset S}$ and back to $v_{*}$ via the edges  of the graph, we call it a cycle. Suppose that $v_l$ is the top most  vertex in the cycle, we say that this is a cycle of the vertex $v_l$. Thus, $v_l$ can be either the virtual vertex $v_*$ or an interacting vertex. 
\end{definition}

Using the definition of cycles,  we have the following scheme to to obtain free edges below. 
\begin{itemize}
	\item We start with the vertices in $\mathfrak{V}_0$ and the cluster vertices in $\mathfrak{V}_C$ and all of the edges that are connecting them. Those edges are non-free edges or integrated edges. The momenta $\mathscr{K}_K$ are left until the end of our process and they are not free and not integrated either.  For  the first vertex $v_{1}$, among the two edges attached to it, the right edge is set to be  an integrated edge. If the other one forms a cycle with $v_{1}$ and the vertices  in $\mathfrak{V}_C$, we set it a free-edge. Otherwise, we set it  an integrated edge.    We continue the procedure in a recursive manner as follows. At step $i$, we consider the two edges associated to $v_{i}$ that connect $v_{i}$ with the vertices of $\mathfrak{V}_0$ and $\{v_j\}_{j=1}^{i-1}$. If the right one forms a cycle with $\mathfrak{V}_0$ and $\{v_j\}_{j=1}^{i}$, we set it free, otherwise we set it non-free, which then becomes an integrated edge. Then we continue the same procedure with the left one. The procedure is carried on until $i$ reaches $i=n$.   The momenta associated to the free edges are called ``free momenta'' and we will use the terminologies ``free momenta'' and ``free edges'' for the same purpose. 
	\item  	In this construction, there could exist a free edge attached to the virtual vertex $v_*$. If this free edge is on the minus diagram, we call it the ``virtually free edge attached to $v_{n+1}$'' or the ``virtually free edge attached to $v_{n-1}$''. The associated momentum is called the ``virtually free momentum attached to $v_{n+1}$'' and, equivalently, $v_{n-1}$.  If this free edge is on the minus diagram, we call it the ``virtually free edge attached to $v_{n+2}$''  and equivalently,  $v_{n}$.   The associated momentum is called the ``virtually free momentum attached to $v_{n+2}$ and, equivalently,  $v_{n}$. 
	\item {\it The second assigned orientation of the diagram: } For any integrated edge $e=\{v',v\}$, the orientation of the edge is from $v'$ to $v$ if $v$ belongs to the path which does not contain any free edge going from $v'$ to the virtual vertex $v_*$. We the assign the  order $v'\succ v$ and $v' 	\prec v$. The graph $\mathfrak{G}$  with the partial order $\succ $ is now defined to be  $\mathfrak{G}^\succ$. 
	We also define $
		\mathfrak{P}(v)=\{v' \in\mathfrak{V}_I | v' 	\succ v, v'\ne v_*,v_{n+1},v_{n+2}\}. 
$ We also extend the  order to the cluster vertices as follows. Suppose  $v_K$ is a cluster vertex, if $v_K$ belongs to the path $v'\prec v$, where $v'$ and $v$ are two interacting vertices, then we assign $v'  \prec v_K  \prec v$. It is then clear that this extension is unique. We also define $
\mathfrak{P}^{clus}(v)=\{v' \in\mathfrak{V}_C| v' 	\succ v\}. 
$ The union of $
\mathfrak{P}(v)$ and $\mathfrak{P}^{clus}(v)$ is denoted by $\mathfrak{P}^{uni}(v)$. Let $\mathfrak{F}$ denote the set of free edges and $\mathfrak{E}'$ denote the set of integrated edges.  The second orientation of the diagram allows us to  integrate all the delta functions, based on the  integrated edges in $\mathfrak{E}'$.  
\end{itemize}

We also recall the following definition.

\begin{definition}[Degree of a Vertex \cite{staffilani2021wave}]
	For any $v\in\mathfrak{V}_I$, let $\mathfrak{F}(v)$ denote the set of free edges attached to $v$, then $\mathfrak{F}(v)=\mathfrak{E}(v)\cap \mathfrak{F}$. We call the number of free edges in $\mathfrak{F}(v)$ the degree of the interacting vertex $v$ and denote it by $\mathrm{deg}v$.
\end{definition}

Let $k_{i_1}\in\mathfrak{E}_-(v_i)$ be a momenta of an interacting vertex $v_i$, $k_{i_2}$ be the momentum in $ \mathfrak{E}_+(v_i)$ and $v_j$ be the vertex on the other end of $k_{i_2}$. We have $k_{i_2}\in \mathfrak{E}_-(v_j)$ and we suppose that $k_{i_3}$ is the other momentum in $\mathfrak{E}_-(v_j)$. The  momentum $k_{i_3}$ is defined to be related to $k_{i_1}$.

\subsection{Properties of integrated graphs}\label{Sec:PropMometum}

\begin{lemma}[see \cite{staffilani2021wave}] \label{Lemma:ProductSignFirstDirection} The following properties hold true.
	
	\begin{itemize}
		\item 	If $e=\{v,v'\}\in \mathfrak{E}$ does not intersect $\mathfrak{V}_T$, then $\sigma_v(e)\sigma_{v'}(e)=-1.$
		\item 	Let $e=\{v,v'\}$ be a free edge and suppose that $\mathcal{T}(v)>\mathcal{T}(v')$. Let $v''$ be a vertex belonging to the cycle of the vertex $v$, then either $v\succ v''$ or $v'\succ v''$. 	Moreover, let $e'=\{v,v''\}$, $e''=\{v,v'''\}$ be two integrated edges attached to an arbitrary vertex $v$. Then it cannot happen that $v\succ v''$ and $v\succ v'''$. 
		\item 	A Feynman diagram with $n$ interacting vertices has totally $n+2-|S|$ free momenta. 
		\item 	Let $v_j$ be an interacting vertex , $1\le j \le n$. Let  $n_l(j)$ with $l=0,1$  be the number of interacting vertices of degree $l$ below and including $v_j$.   The following relation between $n_l(j)$ and  $\mathfrak{V}_c^{(j)}$, which is defined above to be the clustering of the edges intersecting with the time slice $j$, holds true
		\begin{equation}\label{lemma:n2jn0j:1}
			|\mathfrak{V}_c^{(j)}| \ =  \ |S| \ - \ j \ + \ n_1(j),
		\end{equation}
		and
		\begin{equation}\label{lemma:n2jn0j:2} |\mathfrak{V}_c^{(j)}| \ = \ |S| \ + \ n_0(j).\end{equation}
		Moreover, we also have
		\begin{equation}\label{lemma:n2jn0j:3} n_1(j) \le n+2-2|S| + n_0(j),\end{equation}
		and 
		\begin{equation}\label{lemma:n2jn0j:4} \frac{j-(n+2-2|S|)}{2} \le n_0(j).\end{equation}
	\end{itemize}

\end{lemma}

\begin{lemma}\label{Lemma:SumFreeEdges}
	Let $k_e$ be an integrated momentum with $e$ being an edge in $\mathfrak{G}^\succ$, which is the graph $\mathfrak{G}$ associated with the orientation $\succ$ defined above.   Let $v_1$, $v_2$ be the two vertices of $k_e$ and the second orientation of $k_e$ is from $v_1$ to $v_2$, that means $v_1\succ v_2$.   Let  $\mathfrak{P}(v_1), \mathfrak{P}^{clus}(v_1)$ be the sets defined in the previous section.  The following identity then holds true
	\begin{equation}
	\label{Lemma:SumFreeEdges:1}
	k_e \ =  \ \sum_{v\in\mathfrak{P}(v_1)}\sum_{e'\in\mathfrak{F}(v)}\left(-\sigma_{v_1}(e)\sigma_v(e')\right)k_{e'} + \sum_{v_K\in\mathfrak{P}^{clus}(v_1)}\sigma_{v_1}(e)\mathscr{K}_K,
	\end{equation}
	in which $k_{e'}$ is the momentum of $e'$ and $\mathscr{K}_K$ is the momentum attached to the cluster vertex $v_K$.
	
	In the  formula \eqref{Lemma:SumFreeEdges:1}, the edge $e$ is said to ``depend'' on the edges $e'$ and $k_e$ is said to ``depend'' on $k_{e'}$. If the edge $e$ does not depend on $e'$, we say that $e$ is ``independent'' of $e'$.
\end{lemma}
\begin{proof}
 We will perform an  induction proof with respect to the number $\mathscr{N}=|\mathfrak{P}^{uni}(v_1)|=|\mathfrak{P}(v_1)\cup \mathfrak{P}^{clus}(v_1)|$. We first consider the case  that $\mathscr{N}=1$, then $\mathfrak{P}(v_1)=
	\{v_1\}$ and $\mathfrak{P}^{clus}(v_1)=\emptyset$. Hence, $\mathcal{F}(v_1)=\mathfrak{E}(v_1)\backslash\{e\}$. Therefore, there is one  delta function associated to $v_1$, and $\mathcal{F}(v_1)=\mathfrak{E}(v_1)\backslash\{e\}$, leading to
	$\sum_{e'\in \mathfrak{E}(v_1)} \sigma_{v_1}(e')k_{e'} \ = \ 0,$ that yields
	$k_e \ = \ -\sigma_{v_1}(e)\sum_{e'\in\mathfrak{E}({v_1})\backslash\{e\}}\sigma_{v_1}(e')k_{e'} \ = \ \sum_{e'\in\mathfrak{E}({v_1})\backslash\{e\}}(-\sigma_{v_1}(e)\sigma_{v_1}(e'))k_{e'}.$
	Thus \eqref{Lemma:SumFreeEdges:1} holds for $\mathscr{N}=1$. Suppose by induction that \eqref{Lemma:SumFreeEdges:1} holds for any $\mathscr{N}$ up to $\mathscr{N}_0\ge 1$, we next consider the edge $e=\{v_1,v_2\}$ with $|\mathfrak{P}^{uni}(v_1)|=\mathscr{N}+1$ and $v_1\succ v_2$. As $v_1\notin\mathfrak{V}_T$, we deduce
	$k_e \ = \ \sum_{e'\in\mathfrak{E}(v_1)\backslash\{e\}}(-\sigma_{v_1}(e)\sigma_{v_1}(e'))k_{e'}.$
	In this sum, for any edge $e'=\{v_e,v_1\}$,  in $\mathfrak{E}({v_1})\backslash\{e\}$, there are two possibilities. If the edge $e'$ is free, we are done. If $e'$ is not free, we also have two possibilities, either $v_e\succ v_1$ or $v_1\succ v_e$.   If $v_e\succ v_1$, then the set $\mathfrak{P}^{uni}(v_e)$ contains at most $\mathscr{N}$ elements due to the fact that $v_1$ is on the path from $v_e$ to the virtual vertex $v_*$.  The induction hypothesis can now be applied to  $e'$. In the second scenario,  $v_1\succ v_e$, recalling that we also have $v_1\succ v_2$, this yields to a contradiction with the conclusion of Lemma \ref{Lemma:ProductSignFirstDirection}. 
	By induction,  the identity \eqref{Lemma:SumFreeEdges:1} also holds for $\mathscr{N}+1$. This completes our induction proof. 
\end{proof}

\begin{lemma}\label{Lemma:SumFreeEdges2}
(a)	For any integrated edge $e=(v_1,v_2)\in\mathfrak{E}'$, $v_1\succ v_2$, the following identities hold true
	\begin{equation}
	\begin{aligned}\label{Lemma:SumFreeEdges2:1}
	k_e \ = & \ \sum_{v\in\mathfrak{P}(v_1)}\sum_{e'=(v,v_{e'})\in \mathfrak{F}(v)}\mathbf{1}_{v_{e'}\notin\mathfrak{P}(v_1)}\left(-\sigma_{v_1}(e)\sigma_v(e') \right)k_{e'} + \sum_{v_K\in\mathfrak{P}^{clus}(v_1)}\sigma_{v_1}(e)\mathscr{K}_K\\
	\ = & \ -\sigma_{v_1}(e)\sum_{e'\in\mathfrak{F}}\mathbf{1}(\exists v\in e' \cap \mathfrak{P}(v_1) \mbox{ and } e' \cap \mathfrak{P}(v_1)^c \neq \emptyset)\sigma_v(e')k_{e'} + \sum_{v_K\in\mathfrak{P}^{clus}(v_1)}\sigma_{v_1}(e)\mathscr{K}_K,
	\end{aligned}\end{equation}
	in which $k_{e},k_{e'}$ are the momenta of $e$, $e'$. 
	Moreover, if $e'=(v,v')$, such that $e'\neq e$, $v\in\mathfrak{P}(v_1)$ and $v'\notin \mathfrak{P}(v_1)$, then $e'$ is free. 
	
	For any edge $e=(v_1,v_2)\in\mathfrak{E}$. If $e$ is integrated,  we suppose $v_1\succ v_2$ and define  
	\begin{equation}
	\label{Lemma:SumFreeEdges2:0}
	\mathfrak{F}_e \ = \ \{e'\in\mathfrak{F}|\exists v\in e' \cap \mathfrak{P}(v_1) \mbox{ and } e' \cap \mathfrak{P}(v_1)^c \neq \emptyset)\},
	\end{equation}
	then formula \eqref{Lemma:SumFreeEdges2:1} can be expressed under the following form, in which  $\sigma_{e,e'}\in\{\pm 1\}$,
	\begin{equation}
	\label{Lemma:SumFreeEdges2:2}
	k_e \ = \ \sum_{e'\in \mathfrak{F}_e}\sigma_{e,e'}k_{e'}+ \sum_{v_K\in\mathfrak{P}^{clus}(v_1)}\sigma_{v_1}(e)\mathscr{K}_K. 
	\end{equation}
	We also denote 
	\begin{equation}
	\label{Lemma:SumFreeEdges2:0:1}
	\mathfrak{F}_{k_e} \ = \ \{k_{e'}\ | \ e'\in\mathfrak{F}, \ \exists v\in e' \cap \mathfrak{P}(v_1) \mbox{ and } e' \cap \mathfrak{P}(v_1)^c \neq \emptyset)\}.
	\end{equation}
	If $e$ is free, we set $\mathfrak{F}_e=\{e\}$ and $\sigma_{e,e}=1$. We have the following version of formula \eqref{Lemma:SumFreeEdges2:2}
	\begin{equation}
	\label{Lemma:SumFreeEdges2:3}
	k_e \ = \ \sum_{e\in \mathfrak{F}_e}\sigma_{e,e}k_{e}+ \sum_{v_K\in\mathfrak{P}^{clus}(v_1)}\sigma_{v_1}(e)\mathscr{K}_K. 
	\end{equation}

(b)	Let $v$ be any interacting vertex. Then $\mathrm{deg}(v)\in\{0,1\}$. 
If $v\in\mathfrak{V}_I$ and $\mathrm{deg}(v) = 1$, then $\mathfrak{E}_-(v)=\{e,e'\}$, where $e$ is a free edge. Denote by $k_e$ and $k_{e'}$ the momenta associated to $e$ and $e'$. We have $k_{e'} = \pm k_{e} + q$, where  $q$ is independent of $k_e$ and can depend on  the cluster momenta $\mathscr{K}_K$. Moreover, the  dependence on $k_e$ of the other integrated edges is given below.

\begin{itemize}
	\item[(i)] Consider the integrated edge $e'=(v_1,v_2)\in\mathfrak{E}'$ and denote by $k_{e'}$ its momentum. Then one of the following $3$ possibilities should happen
	$k_{e'} = \pm k_e + \tilde{k}_{e'}$,  or
	$k_{e'}  =   \tilde{k}_{e'}$,
	where $ \tilde{k}_{e'} $ is independent of $k_e$ and can depend on  the cluster  momenta $\mathscr{K}_K$ in all 3 cases. Moreover, the component that depends on the  cluster  momenta is written as $\sum_{K\in S}\sigma_{\tilde{k}_{e'},K}\mathscr{K}_K$, with $\sigma_{\tilde{k}_{e'},K}\in\{0,\pm1\}$.
	\item[(ii)] If $v_1\in\mathfrak{V}_I$, and suppose $e',e''\in\mathfrak{E}(v_1)$, $e'\neq e''$, and denote by $k_{e'},k_{e''}$ their momenta,  then  one of the following  $5$ possibilities should happen
	$k_{e'} + k_{e''}   =   \pm k_e + \tilde{k}_{e',e''}$, or $k_{e'} + k_{e''}  =  \pm 2k_e\ + \tilde{k}_{e',e''},$ or $k_{e'} + k_{e''}  =  \tilde{k}_{e',e''},$
	where $ \tilde{k}_{e',e''}$ is independent of $k_e$ and  $k_{e'}$ and can depend on  the cluster momenta $\mathscr{K}_K$ in all 5 cases.  Moreover, the component that depends on the  cluster  momenta is written as $\sum_{K\in S}\sigma_{\tilde{k}_{e',e''},K}\mathscr{K}_K$, with $\sigma_{\tilde{k}_{e',e''},K}\in\{0,\pm1\}$.
\end{itemize}
\end{lemma}
\begin{proof}
	We only provide  the first identity, since the second one is an easy a straightforward corollary.
	By \eqref{Lemma:SumFreeEdges}, we have
$
	k_e \ =  \ \sum_{v\in\mathfrak{P}(v_1)}\sum_{e'\in\mathfrak{F}(v)}\left(-\sigma_{v_1}(e)\sigma_v(e')\right)k_{e'} + \sum_{v_K\in\mathfrak{P}^{clus}(v_1)}\sigma_{v_1}(e)\mathscr{K}_K.
$
	For any $v,v'\in\mathfrak{P}(v_1)$, such that $e'=(v,v')$ is free, we have  that $-\sigma_{v_1}(e)\sigma_{v}(e')-\sigma_{v_1}(e)\sigma_{v'}(e')=0$. Thus, these pair of edges cancel each other and summing over edges in $\mathfrak{F}(v)$ and $\mathfrak{F}(v')$ gives the  identity.
	For any edge $e'=(v,v')$, such that $e'\neq e$, $v\in\mathfrak{P}(v_1)$ and $v'\notin \mathfrak{P}(v_1)$, then $e'$ is free. The two formulas \eqref{Lemma:SumFreeEdges2:2} and \eqref{Lemma:SumFreeEdges2:3} then follow in a straightforward manner. The proof of part (b) is   a modification of  the homogeneous case discussed in \cite{staffilani2021wave} and is omitted.
\end{proof}

\begin{definition}[Singular Graph]\label{Def:SingGraph} A graph  $\mathfrak{G}$ is said to be singular if there  is  an edge $e$ such that its momentum $k_e$ is either zero or depends only on the cluster momenta.  The edge $e$ is called the singular edge and the momentum $k_e$ is the singular momentum.
\end{definition}
We also need the following definition of  1-Separation and 2-Separation by Tutte  \cite{tutte2019connectivity,tutte1984graph}. 
\begin{definition}[1-Separation and 2-Separation]\label{Def:1Sep}
	A graph $\mathfrak{G}$ is said to have a 1-separation if $\mathfrak{G}$ is the union   of two components  $\mathfrak{G}_1\cup\mathfrak{G}_2=\mathfrak{G}$ such that the intersection $\mathfrak{G}_1\cap\mathfrak{G}_2$ contains only one edge. 
		A graph $\mathfrak{G}$ is said to have a 2-separation if $\mathfrak{G}$ is the union   of two components $\mathfrak{G}_1$ and $\mathfrak{G}_2$: $\mathfrak{G}_1\cup\mathfrak{G}_2=\mathfrak{G}$ such that the intersection $\mathfrak{G}_1\cap\mathfrak{G}_2$ contains exactly two  edges. 
\end{definition}
\begin{lemma}\label{Lemma:ZeroMomentum}
	For any edge $e\in\mathfrak{E}$, $k_e$ is either $0$ or depends only on the cluster momenta if and only if $\mathfrak{F}_e=\emptyset$, or equivalently the graph is singular with $e$ being the singular edge, which is also equivalent with the fact that the graph  has a 1-separation: it can be decomposed into two components $\mathfrak{G}_1\cup\mathfrak{G}_2=\mathfrak{G}$ such that they are connected only via the edge $e$. Moreover, when $k_e$ depends only on the cluster momenta, it is written as $$k_e=\sum_{K\in S}\sigma_{k_e,K}\mathscr{K}_K,$$ with $\sigma_{k_e,K}\in\{0,\pm1\}$. As a consequence, consider a cycle of a vertex $v_i$, if one of the vertices of this cycle belongs to $\mathfrak{G}_j$, $j$ can be either $1$ or $2$, then all of the vertices of the cycle belong to $\mathfrak{G}_j$.
	
		Suppose that $e,e'\in\mathfrak{E}$, $e\neq e'$, we then have:
	\begin{itemize}
		\item[(i)] Suppose that $\mathfrak{F}_e=\mathfrak{F}_{e'}\ne \emptyset,$ then  if one removes both $e$ and $e'$, the graph $\mathfrak{G}$ is split into  disconnected components. As a result, the graph has a 2-separation. 
		\item[(ii)] $\mathfrak{F}_e=\mathfrak{F}_{e'}$ if and only if there is $\sigma\in\{\pm 1\}$ such that $k_e=\sigma k_{e'}+q$, where $k_e,k_{e'}$ are the momenta associated to $e,e'$ and $q$ is either zero or depends only on the cluster momenta.  Moreover, the component that depends on the  cluster  momenta is written as $\sum_{K\in S}\sigma_{q,K}\mathscr{K}_K$, with $\sigma_{q,K}\in\{0,\pm1\}$.
	\end{itemize}

\end{lemma}
\begin{proof}
	
	The proof  is a modification of  the homogeneous case discussed in \cite{staffilani2021wave} and is omitted.
\end{proof}

\subsection{Pairing graphs}

\begin{definition}[Pairing and Non-Pairing Graphs] A graph  $\mathfrak{G}$ is called ``pairing'' if for every $K\in S$, then $|K|=2$. Otherwise, it is ``non-pairing''.
\end{definition}


 We  have the following classical definition of the size of a cycle, following  Estrada  (cf. \cite{estrada2012structure}).
\begin{definition}[Size of a Cycle and Interacting Size of a Cycle]\label{Def:EstradaSizecycle}
	Let $v_{i}$ be a degree-one vertex. Define $\{v_{\mathscr{M}}\}_{\mathscr{M}\in\mathfrak{I}}$ to be the set of all of the vertices in the cycle of $v_i$, including $v_i$. Then the number of elements $\mathscr{N}=|\mathfrak{I}|$ of $\mathfrak{I}$ is defined to be the size of the cycle of $v_i$. Those cycles are  called $C_{\mathscr{N}}$ cycles.

	Let $v_{i}$ be a degree-one vertex. Define $\{v_{\mathscr{M}}\}_{\mathscr{M}\in\mathfrak{I}}$ to be the set of all of the interacting vertices in the cycle of $v_i$, including $v_i$. Then the number of elements $\mathscr{N}=|\mathfrak{I}|$ of $\mathfrak{I}$ is defined to be the interacting size of the cycle of $v_i$. The cycles are then called $\mathrm{iC}_{\mathscr{N}}$ cycles.
\end{definition}

We have the following lemmas. 
\begin{lemma}\label{Lemma:Size1cycle}
	Let $v_{i}$ be a degree-one vertex and  $k\in\mathfrak{E}_+(v_{i})$.  Suppose that the interacting size of the cycle of $v_i$ is $1$, then $k=\pm\mathscr{K}_K$, where $\mathscr{K}_K$ is one of the cluster momenta, and the graph containing the cycle of $v_{i}$ has a 1-separation and is singular. Therefore, any cycle in a non-singular graph has an interacting size bigger than $1$. 
\end{lemma}
\begin{proof}
	We assume  that $k_1,k_2\in\mathfrak{E}_-(v_{i})$. Recalling that the cycle containing $v_{i}$ has only one interacting vertex, we deduce the two edges $k_1,k_2$ are directly connected. Thus, if we remove the edge associated to $k$, the cycle is split from the graph. Using Lemma \ref{Lemma:ZeroMomentum}, we infer  $k=\pm\mathscr{K}_K$. Moreover, the graph has a 1-separation and is singular.
\end{proof}

The following lemma illustrates the structure of vertices inside a cycle.

\begin{lemma}\label{Lemma:VerticeOrderInAcycle}
	Let $v_\mathscr{N}$ be a degree-one vertex and define $\{v'_{\mathscr{M}}\}_{\mathscr{M}\in\mathfrak{I}}$ to be the set of all of the interacting vertices in the cycle of $v_\mathscr{N}$, including $v_\mathscr{N}$. Then for all $\mathscr{M}\in\mathfrak{I}$, we always have $\mathcal{T}(v'_\mathscr{M})\le \mathscr{N}$.
\end{lemma}
\begin{proof}

	The proof  is a modification of  the homogeneous case discussed in \cite{staffilani2021wave} and is omitted.
\end{proof}

Below, we present a classification of cycles for a non-singular graph.

\begin{definition}[Classification of Cycles in Non-singular Graphs]\label{Def:Classificationcycles}
	In a non-singular graph $\mathfrak{G}$, let $v_i$ be a degree-one vertex and $\{v_j\}_{j\in\mathfrak{I}}$ be the sets of all the interacting vertices in its cycle. The cycle of $v_i$ is then an $\mathrm{iC}_{|\mathfrak{I}|}$ cycle.
	\begin{itemize}
		\item  If the size $|\mathfrak{I}|$ of this $\mathrm{iC}_{|\mathfrak{I}|}$ cycle is bigger than or equal to  $3$, the cycle is called a ``long collision''. 
		\item  If the size $|\mathfrak{I}|$ of this $\mathrm{iC}_{|\mathfrak{I}|}$ cycle is exactly  $2$, the cycle is called a ``short collision''.
	\end{itemize}
\end{definition}

\begin{definition}[Classification of Short Collisions in Non-singular Graphs]\label{Def:ClassificationShortCollision}
	In a non-singular graph $\mathfrak{G}$, let $v_\mathscr{N}$ be a degree-one vertex, whose cycle is a short collision and contains two interacting vertices $v_\mathscr{N},v_\mathscr{M}$ with $\mathscr{M}<\mathscr{N}$.
	\begin{itemize}
		\item If $\mathscr{M}=\mathscr{N}-1$, the cycle is called a ``recollision'' (see Figures \ref{Fig10}-\ref{Fig11}). There are two types of recollision. If a recollision uses only one cluster vertex, we call it a ``single-cluster recollision''. A single-cluster recollision is an $\mathrm{iC}_2$ and a $C_5$ cycle. If a recollision uses two cluster vertices, we call it a ``double-cluster recollision''. A double-cluster recollision is an $\mathrm{iC}_2$ and a $C_8$ cycle. A recollision is called an ``$\mathrm{iC}^r_2$ cycle.''
		\item If $\mathscr{M}\ne \mathscr{N}-1$, the cycle is called a ``short delayed recollision''. By Lemma \ref{Lemma:VerticeOrderInAcycle}, it follows that $\mathscr{M}<\mathscr{N}-1$. A short delayed recollision  is called an ``$\mathrm{iC}^d_2$ cycle.''
		\item For both recollisions and short delayed recollisions, each cycle has the same number of interacting vertices (see Figures \ref{Fig10},\ref{Fig11}). 
	\end{itemize}
\end{definition}
\begin{figure}
	\centering
	\includegraphics[height=.70\linewidth]{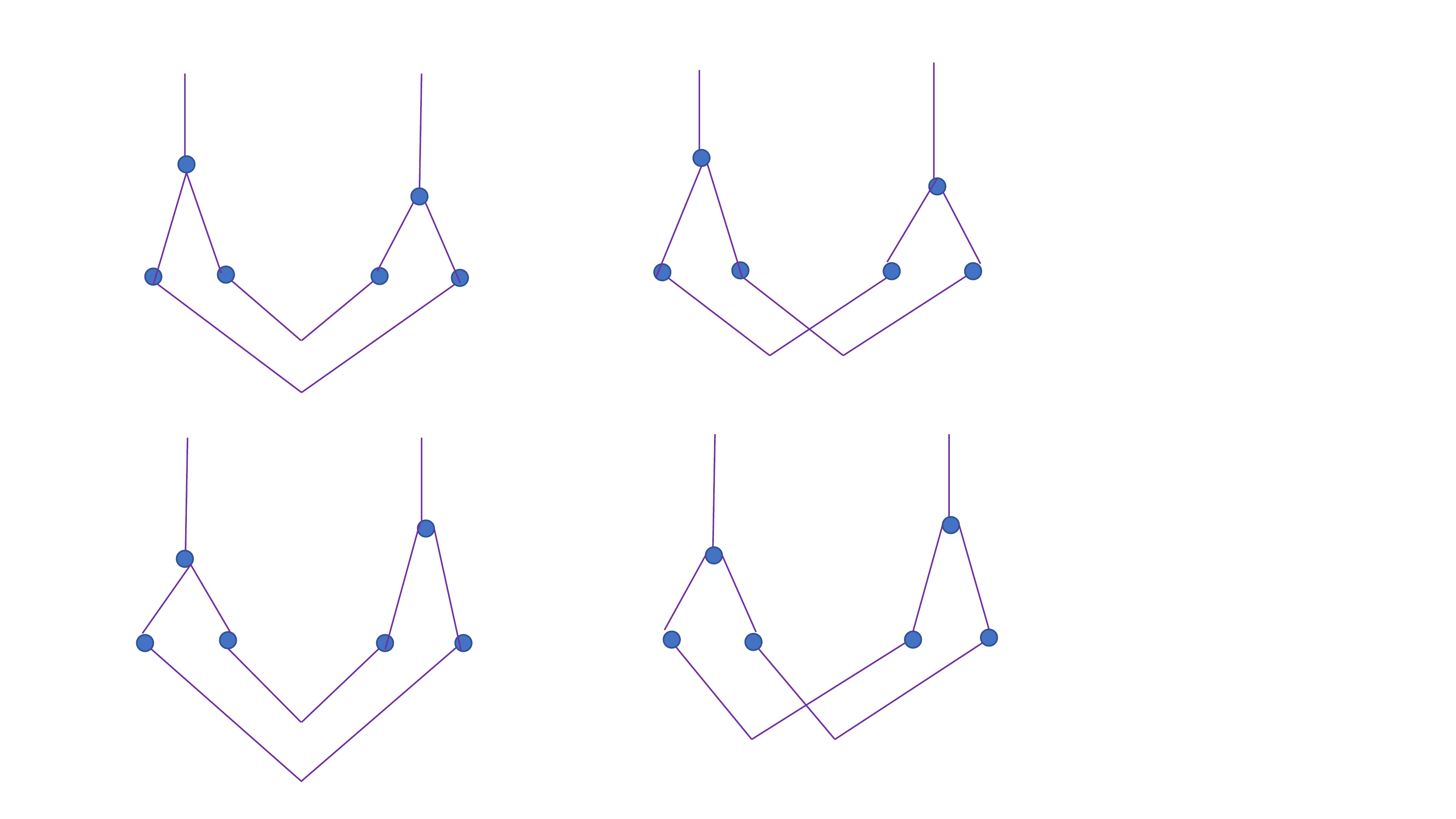}
	\caption{The 4 double-cluster recollisions. Each of them uses two cluster vertices.   They are $\mathrm{iC}_2^r$ and $C_8$ cycles, the skeletons of  $\mathrm{iCL}_2$ ladder graphs.}
	\label{Fig10}
\end{figure}
\begin{figure}
	\centering
	\includegraphics[height=.65\linewidth]{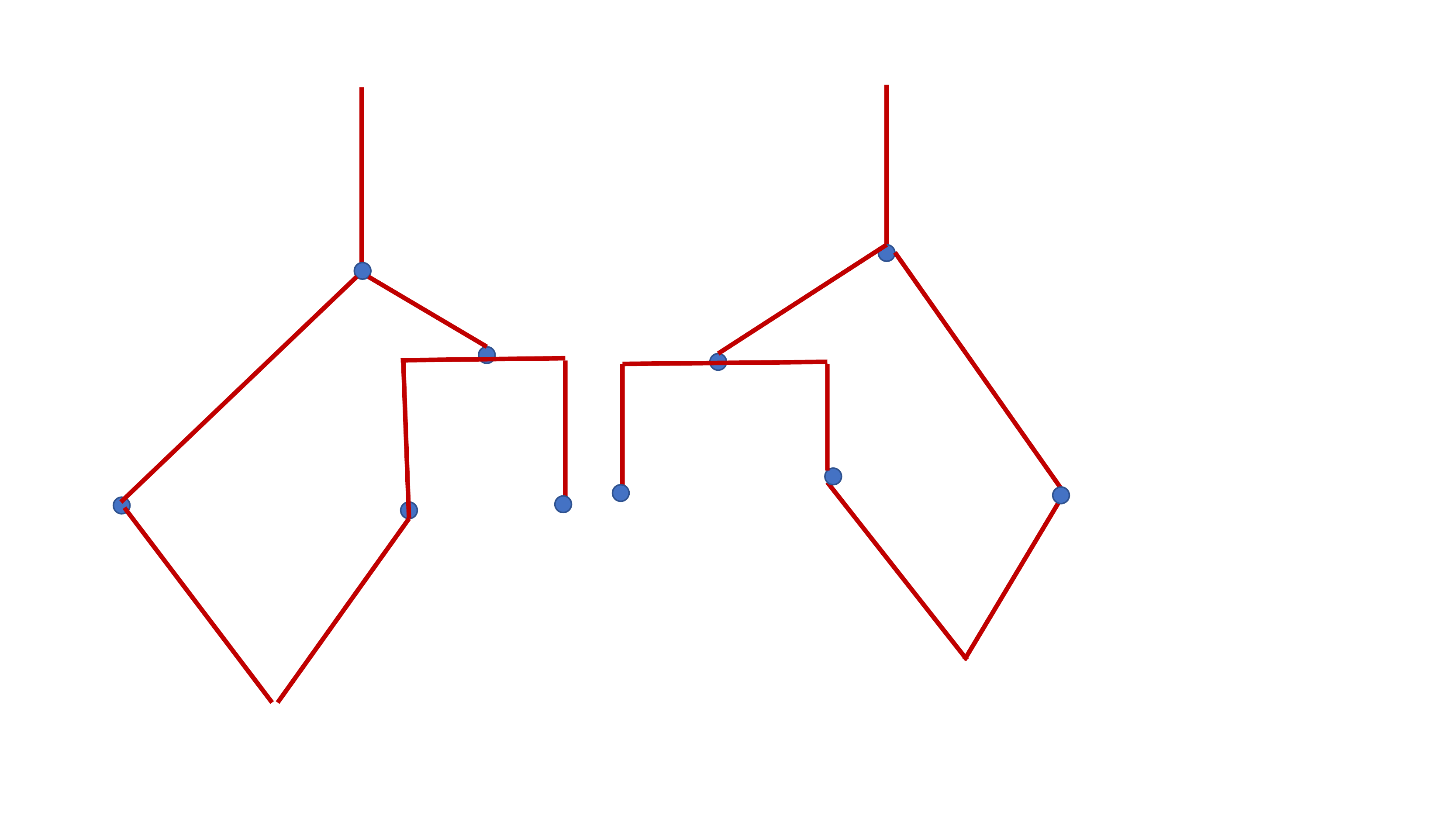}
	\caption{The 2 single-cluster recollisions. Each of them uses only one cluster vertex. They are $\mathrm{iC}_2^r$ and $C_5$ cycles, the skeletons of  $\mathrm{iCL}_2$ ladder graphs.}
	\label{Fig11}
\end{figure}

We will need to further classify long collisions. To this end, we  prove the following Lemma.

\begin{lemma}\label{Lemma:VerticesLongCollisions}
	Let $v_{\mathscr{M}}$ be a degree-one vertex in a non-singular graph. Denote by $k_1,k_2$ its edges in $\mathfrak{E}_-(v_{\mathscr{M}})$ and $k_0$ the edge in $\mathfrak{E}_+(v_{\mathscr{M}})$. Denote their signs by $\sigma_{k_0},\sigma_{k_1},\sigma_{k_2}$. Suppose that $k_1$ is the momentum of the free edge attached to $v_{\mathscr{M}}$. Assume that there exists a vertex $v_{l_0}$ in the cycle of $v_{\mathscr{M}}$ such that $\mathbf{X}(v_{\mathscr{N}})+\mathbf{X}({v_{\mathscr{M}}})$ is not a function of $k_1$. Denote the three edges of $v_{l_0}$ by $k_0',k_1',k_2'$, in which $k_0'\in\mathfrak{E}_+(v_{\mathscr{N}})$ and $k_1',k_2'\in\mathfrak{E}_-(v_{\mathscr{N}})$.  The signs of those edges are denoted by $\sigma_{k_0'},\sigma_{k_1'},\sigma_{k_2'}$.  Then, the followings hold true.
	\begin{itemize}
		\item[(i)] One of the following identities is satisfied
		\begin{equation}
		\label{Lemma:VerticesLongCollisions:1}
		\sigma_{k_2}k_2+\sigma_{k_2'}k_2'=\sum_{K\in S}\sigma_{k_2,K}\mathscr{K}_K, \mbox{ and }\sigma_{k_1}k_1+\sigma_{k_1'}k_1'=\sum_{K\in S}\sigma_{k_1,K}\mathscr{K}_K,
		\end{equation} 
		\begin{equation}
		\label{Lemma:VerticesLongCollisions:2}
		\sigma_{k_2}k_2+\sigma_{k_1'}k_1'=\sum_{K\in S}\sigma_{k_2,K}\mathscr{K}_K, \mbox{ and }\sigma_{k_1}k_1+\sigma_{k_2'}k_2'=\sum_{K\in S}\sigma_{k_1,K}\mathscr{K}_K,
		\end{equation} 
		\begin{equation}
		\label{Lemma:VerticesLongCollisions:3}
		\sigma_{k_1}k_1+\sigma_{k_0'}k_0'=\sum_{K\in S}\sigma_{k_1,K}\mathscr{K}_K, \mbox{ and } \sigma_{k_2}k_2+\sigma_{k_2'}k_2'=\sum_{K\in S}\sigma_{k_2,K}\mathscr{K}_K,
		\end{equation}
		and
		\begin{equation}
		\label{Lemma:VerticesLongCollisions:4}
		\sigma_{k_1}k_1+\sigma_{k_0'}k_0'=\sum_{K\in S}\sigma_{k_1,K}\mathscr{K}_K, \mbox{ and }\sigma_{k_2}k_2+\sigma_{k_1'}k_1'=\sum_{K\in S}\sigma_{k_2,K}\mathscr{K}_K.
		\end{equation}
		\item[(ii)] There does not exist another vertex $v_j$ in the cycle of $v_{\mathscr{M}}$ such that $\mathbf{X}(v_{j})+\mathbf{X}({v_{\mathscr{M}}})$  is not a function of $k_1$.  Therefore, $v_{\mathscr{N}}$ is the only vertex with this property.
		
	\end{itemize}

\end{lemma}
\begin{proof}  	The proof  is a modification of  the homogeneous case discussed in \cite{staffilani2021wave} and is omitted.
\end{proof}
By Lemma \ref{Lemma:VerticesLongCollisions}, we then have a classification for long collisions. 
\begin{definition}[Classification of Cycles in Non-singular Graphs, see \cite{staffilani2021wave}]\label{Def:ClassificationLongCollisions}
	Let $v_{\mathscr{M}}$ be a degree-one vertex in a non-singular graph and suppose that its $\mathrm{iC}_l$ cycle is long $l>2$. Denote by $k_1$ its free momentum. 
	
	\begin{itemize}
		\item[(i)] If there exists a vertex $v_{\mathscr{N}}$ in the cycle of $v_{\mathscr{M}}$ such that $\mathbf{X}(v_{\mathscr{N}})+\mathbf{X}({v_{\mathscr{M}}})$ is not a function of $k_1$, and $\mathscr{M}-\mathscr{N}>1$, the cycle is called a ``long delayed recollision''  and denoted by $\mathrm{iC}_l^d$.
		\item[(ii)] If there does not exist a vertex $v_{\mathscr{N}}$ in the cycle of $v_{\mathscr{M}}$ such that $\mathbf{X}(v_{\mathscr{N}})+\mathbf{X}({v_{\mathscr{M}}})$ is not a function of $k_1$, the cycle is called an ``irreducible long collision''and denoted by $\mathrm{iC}_l^i$.
	\end{itemize}
	Let $v_{\mathscr{M}}$ be a degree-one vertex in a non-singular graph. Denote by $k_1$ its free edges. 
	If there exists a vertex $v_{\mathscr{N}}$ in the $\mathrm{iC}_l$ cycle of $v_{\mathscr{M}}$ such that $\mathbf{X}(v_{\mathscr{N}})+\mathbf{X}({v_{\mathscr{M}}})$ is not a function of $k_1$, and if the $\mathrm{iC}_l$ cycle is not a recollision, it is called a ``delayed recollision''. In this case, there are two possibilities.
	\begin{itemize}
		\item[(iii)] If the cycle is long $l>2$, it is a ``long delayed recollision''
		\item[(iv)] If the cycle is short $l=2$, it is a ``short delayed recollision''.
	\end{itemize}
\end{definition}
We have the following observation for a recollision.
\begin{lemma}\label{Lemma:RecollisionXiXi-1}
	Let $v_{i}$ be a degree-one vertex in a non-singular graph and suppose that it is either a  recollision or a delayed recollision. Denote by $k_1$ its free edges. Let $v_{j}$ be the another   vertex in the cycle. Then $\mathbf{X}(v_{i})+\mathbf{X}({v_{j}})$ is not a function of $k_1$.
	
\end{lemma}
\begin{proof}
	The proof follows straightforwardly from the forms of $\mathbf{X}(v_{i})$ and $\mathbf{X}({v_{j}})$.
\end{proof}

Below, we give a definition for different types of graphs.
\begin{definition}[Different types of graphs, see \cite{staffilani2021wave}]
	For a non-singular, pairing graph, we consider all of the degree-one vertices from the bottom to the top of the graph.
	
	 If, from the bottom to the top, the first degree-one vertex,  which does not  correspond to an $\mathrm{iC}_2^r$ recollision or to a cycle formed by iteratively applying the recollisions), is associated to an $\mathrm{iC}_m^i$  long irreducible  collision $(m> 2)$, we call the graph ``long irreducible''. 
	
	If, from the bottom to the top, the first degree-one vertex, which does not  correspond to a recollision or to a cycle formed by iteratively applying the recollisions, is associated to an $\mathrm{iC}_m^d$ delayed recollision $(m\ge 2)$, we call the graph ``delayed recollisional''. 
	
		A graph that can be obtained by iteratively adding $M$  recollisions ($\mathrm{iC}_2^r$ cycles) is an ``$\mathrm{iCL}_2$ recollision ladder graph'' . We call the $\mathrm{iC}_2^r$ cycles the ``skeletons'' of the $\mathrm{iCL}_2$ ladder graph. 
	
\end{definition}

We then have the lemma. 

\begin{lemma}\label{Lemma:Recollisions}
	Given a non-singular pairing graph, suppose that all of its cycles are $\mathrm{iC}_2^r$ recollisions. Then the graph is a $\mathrm{iCL}_2$ ladder. 
\end{lemma}

\begin{proof}
	The proof  is a modification of  the homogeneous case discussed in \cite{staffilani2021wave} and is omitted.

\end{proof}

\section{Graph estimates}
\label{Sec:GeneralGraph}
This section is devoted to the estimates of all the terms defined
Propositions \ref{Proposition:Ampl1}, \ref{Proposition:Ampl2}, \ref{Proposition:Ampl3}, including  $Q^1_*$, $Q^2_*$, $Q^3$, $Q^4_*$, $ Q^1$, $Q^2$, $Q^3$, $Q^4$. Section \ref{Sec:FirstDiagram} is devoted to the estimates of $ Q^1$. The section contains several propositions treating different types of graphs appearing in $Q^1$. Section \ref{Sec:SecondDiagram} is devoted to the estimates of various terms appearing in  $ Q^3$. Section \ref{Sec:ThirdDiagram} is devoted to the estimates  of various terms appearing in $ Q^4$. And the estimates of terms in $Q^2$ are done in Section \ref{Sec:FourthDiagram}. In the  two Sections  \ref{Sec:FifthDiagram} and \ref{Sec:SixthDiagram}, we provide estimates for two special types of pairing graphs: long irreducible and delayed recollisional. The estimates of all the terms in $Q^1_*$, $Q^2_*$, $Q^3$, $Q^4_*$ are done in Section \ref{Sec:Res}. The dominance of ladder diagrams are proved in Section \ref{Sec:SecondLadderDiagram}.

\subsection{First type of diagram estimates}\label{Sec:FirstDiagram} In this subsection, we will provide estimates on $Q^1$, which is defined in Proposition \ref{Proposition:Ampl1}. 

\begin{definition}\label{Def:Q1pair}
	We write
	\begin{equation}
	\label{lemma:Q1FinalEstimate:2}
	Q^1 \ = \ {Q}^{1,pair}
	\end{equation}
	where
	\begin{equation}\label{lemma:Q1FinalEstimate:3}
	{Q}^{1,pair} 
	\ = \ \sum_{n=0}^{\mathfrak{N}-1}\sum_{S\in\mathcal{P}_{pair}^o(\{1,\cdots n+2\})}\mathcal{G}^{0}_{n}(S,t,k'',-1,k',1,\Gamma),
	\end{equation}
	in which $\mathcal{P}_{pair}^o(\{1,\cdots n+2\}) $ denotes the subset of $\mathcal{P}_{pair} (\{1,\cdots n+2\})$ such that the graph is non-singular.
\end{definition}
We will present below five different estimates for $Q^1$.

\begin{proposition}[The first estimate on   ${Q}^1$]\label{lemma:BasicGEstimate1} Let $S$ be an arbitrary partition of $\{1,\cdots n+2\}$. If the corresponding graph is   non-singular, there is a constant $\mathfrak{C}_{Q_{1,1}}>0$ such that for $1\le n\le \mathbf{N}$ and $t=\tau\lambda^{-2}>0$, we consider one generalization of an element of $Q^1$, which is a function of $k_{n,1}$ and $k_{n,2}$
	\begin{equation}
		\begin{aligned} \label{eq:BasicGEstimate1:0}
			&\tilde{Q}_1^1(\tau):= 
			\lambda^{n}\mathbf{1}({\sigma_{n,1}=-1})\mathbf{1}({\sigma_{n,2}=1}) \sum_{\substack{\bar\sigma\in \{\pm1\}^{\mathcal{I}_n},\\ \sigma_{i,\rho_i}+\sigma_{i-1,\rho_i}+\sigma_{i-1,\rho_i+1}\ne \pm3,
					\\ \sigma_{i-1,\rho_i}\sigma_{i-1,\rho_i+1}= 1}} \int_{(\Lambda^*)^{\mathcal{I}_n}}  \mathrm{d}\bar{k}\Delta_{n,\rho}(\bar{k},\bar\sigma) \mathscr{Z}  \\
			&\times \prod_{i=1}^n\Big[\sigma_{i,\rho_i}\mathcal{M}( k_{i,\rho_i}, k_{i-1,\rho_i}, k_{i-1,\rho_i+1})  \Phi_{1,i}(\sigma_{i-1,\rho_i}, k_{i-1,\rho_i},\sigma_{i-1,\rho_i+1}, k_{i-1,\rho_i+1})\Big] \\
			&\times\int_{(\mathbb{R}_+)^{\{0,\cdots,n\}}}\mathrm{d}\bar{s} \delta\left(t-\sum_{i=0}^ns_i\right)\prod_{i=0}^{n}e^{-s_i[\varsigma_{n-i}+\tau_i]} \prod_{i=1}^{n} e^{-{\bf i}t_i(s)\mathbf{X}_i},
	\end{aligned}\end{equation}
	then, for any constants $1>T_*>0$ and $\mathscr{R}>0$,

	\begin{equation}
		\begin{aligned} \label{eq:BasicGEstimate1}
			& \Big\|\limsup_{D\to\infty} \mathscr{F}_{1/n,\mathscr{R}}\Big[\widehat{\tilde{Q}_1^1\chi_{[0,T_*)}}(\mathscr{X}{\lambda^{-2}})\Big](k_{n,1},k_{n,2})\Big\|_{L^{\infty,\Im}(\mathbb{T}^{2d})}
			\\
			\le & 4^n e^{T_*}\frac{T_*^{1+{n}_0(n)-{n}_0(n-[\mathfrak{N}/4])}}{({n}_0(n)-{n}_0(n-[\mathbf{N}/4]))!}\lambda^{\bowtie_\Im+n-2{n}_0(n)-\eth'_1-\eth'_2}\mathbf{N}^{-\wp	n_0(n-[\mathbf{N}/4])}\\
			&\times \mathfrak{C}_{Q_{1,1}}^{1+{n}_1(n)}\langle \ln n\rangle \langle\ln\lambda\rangle^{1+(2+\eth)({n}_1(n)-2)},
	\end{aligned}\end{equation}
	where $\mathbf{N}$ is the number of Duhamel expansions we need,  ${n}_j(n)$ is the number of interacting vertices of degree $j$ and $n_0(n-[\mathbf{N}/4])$ is the number of degree $0$ interacting vertices $v_l$ with $0<l\le n-[\mathbf{N}/4]$,  $1<\Im=\frac{2\Im_o}{2\Im_o-1}<2$ for some   integer $\Im_o>0$
	$$\bowtie_\Im \ = \ d(\Im-1)\left[\frac{2(n+2-n_0)}{\Im}-n-2\right]-c_{\mathbf{N}_0},$$ and $\widehat{\tilde{Q}_1^1\chi_{[0,T_*)}}$ is the Fourier transform of $\tilde{Q}_1^1(\tau)\chi_{[0,T_*)}(\tau)$ in the variable $\tau$. The quantity $\mathscr{Z}, \mathscr{F}_{1/n,\mathscr{R}}$ are defined to be 
	\begin{equation}
		\label{eq:BasicGEstimate1:A:1}
		\mathscr{Z} \  = \ \mathscr{U}(0), \mbox{ or } \mathscr{W}_1(0), \mbox{ or } \mathscr{W}_2(0).
	\end{equation}
	The function $\mathscr{W}(s)$, $\mathscr{U}(s)$ are defined for all $s\ge 0$ as 
	  \begin{equation}
		\begin{aligned}\label{eq:BasicGEstimate1:A}\mbox{ } \mathscr{U}(s)\ := \ &  \prod_{A=\{i,j\}\in S}\left\langle  a(k_{0,i},\sigma_{0,i})a(k_{0,j},\sigma_{0,j})\right\rangle_{s}, \mbox{  when  } S\in \mathcal{P}_{pair}^o (\{1,\cdots n+2\}),\\ 
\mbox{ } \mathscr{W}_1(s)\ := \ 	&	\left\langle\prod_{A=\{J_1,\dots,J_{|A|}\}\in S}   \square\left(\sum_{l=1}^{|A|}k_{0,J_l}\sigma_{0,J_l}\right)\prod_{i=1}^{n+2}  a(k_{0,i},\sigma_{0,i})\right\rangle_{s},\\
\mbox{or } \mathscr{W}_2(s)\ := \ & 	\left\langle\prod_{i=1}^{n+2}  a(k_{0,i},\sigma_{0,i})\right\rangle_{s}, 
\\
			\ \mathscr{F}_{1/n,\mathscr{R}}[F] \ := \ &	\Big|\int_{(-\mathscr{R},\mathscr{R})}\mathrm{d}\mathscr{X}|F(\mathscr{X})|^{21/n}\Big|^{n},	\end{aligned}\end{equation}
	where $\mathscr{R}>0$ is an arbitrary constant and $F$ is any function that makes the integral finite. The time integration is defined to be $\int_{(\mathbb{R}_+)^{\{0,\cdots,n\}}}\mathrm{d}\bar{s} \delta\left(t-\sum_{i=0}^ns_i\right)= \int_{(\mathbb{R}_+)^{\{0,\cdots,n\}}}\mathrm{d}{s}_0\cdots \mathrm{d}{s}_n \delta\left(t-\sum_{i=0}^ns_i\right).$ The parameters $\varsigma_{n-i}$ are defined in \eqref{Def:Para3} and $\eth_1',\eth_2'$ are  defined in Section \ref{Sec:KeyPara}. We have used the notation
$		\mathbf{X}_i = \mathbf{X}(v_i)   =  \mathbf{X}(\sigma_{i,\rho_i},  k_{i,\rho_i},\sigma_{i-1,\rho_i}, k_{i-1,\rho_i},\sigma_{i-1,\rho_i+1}, k_{i-1,\rho_i+1}).
$	The three notations $\bar{k}$, $\bar\sigma$ and $\bar{s}$ always stand for the vectors that are composed of all of the possible $k_{i,j}$, $\sigma_{i,j}$ and $s_i$. 
\end{proposition}
\begin{proof} We suppose that $$\mathscr{Z} =  	\left\langle\prod_{A=\{J_1,\dots,J_{|A|}\}\in S}   \square\left(\sum_{l=1}^{|A|}k_{0,J_l}\sigma_{0,J_l}\right)\prod_{i=1}^{n+2}  a(k_{0,i},\sigma_{0,i})\right\rangle_{0}.$$ The other cases can be estimated by the same method. 
	By the definition, there
	are in total $n$ interacting vertices and   $n$ time slices. We now divide  the set $\{0,1,\cdots,n\}$ into two sets. One set contains    time slice indices $0\le i<n$ satisfying $\mathrm{deg}(v_{i+1})=0$. The set therefore contains the first time slice $i=0$ since $v_1$ is always a degree-zero vertex. This set is then denoted by $I''$. The second set  contains all time slice indices $0\le i<n$ satisfying $\mathrm{deg}(v_{i+1})=1$ and the new index $n$. This set is labeled by $I$.  We next define the set $\mathscr{Y}=\{i\in \{0,1,\cdots,n\} ~~|~~ l\le n-[\mathbf{N}/4], \ \ \  \mathrm{deg}(v_{i+1})=0\}.$ Following \eqref{GraphSec:E1}, the phase factor can be rewritten as $$\prod_{i=0}^{n}e^{-s_i\varsigma_{n-i}} \prod_{i=1}^{n} e^{-{\bf i}t_i(s)\mathbf{X}_i}=\prod_{i=0}^{n} e^{-{\bf i}s_i\vartheta_i}.$$
	
	We define the new time slice $s_{n+1}=\bar{s}_{n+1}\lambda^{-2}$, and rewrite the time integration as
	
	\begin{equation}
		\begin{aligned}
			\label{eq:Aestimate0:1}
			&	\int_{{(\mathbb{R}_+)^{{\{0,\cdots,n\}}}}}\mathrm{d}\bar{s}\delta\left(t-\sum_{i=0}^{n}s_i\right) \ = \ \int_{(\mathbb{R}_+)^{I'}}\mathrm{d}\vec{s'}\delta\left(t-\sum_{i\in I'}s_i\right)\int_{(\mathbb{R}_+)^{I}}\mathrm{d}\vec{s''}\delta\left(s_{n+1}-\sum_{i\in I}s_i\right),
		\end{aligned}
	\end{equation}
	in which $I'=\{n+1\}\cup I''$,  $\vec{s'}$ is the vector whose components are $s_i$ with $i\in I'$,  $\vec{s''}$ is the vector whose components are $s_i$ with $i\in I$. As  a consequence, 
	 $\tilde{Q}_1^1(\tau)$ can be rewritten as $\int_0^{\tau\lambda^{-2}}\mathrm{d}{s}_{n+1} $ $\mathscr{Q}(\tau-\bar{s}_{n+1}) \mathscr{G}(s_{n+1})$, where $\mathscr{Q}(\tau-\bar{s}_{n+1})$ contains  $\delta\left(t-\sum_{i\in I'}s_i\right)$ and $\mathscr{G}(s_{n+1})$ contains $\delta\left(s_{n+1}-\sum_{i\in I}s_i\right),$ with

	\begin{equation}
		\begin{aligned}\label{eq:Aestimate0:1:A}
			\mathscr{Q}(\tau-\bar{s}_{n+1}) \ = \ 	\int_{(\mathbb{R}_+)^{\bar{I}''}}\mathrm{d}\vec{s'''}\delta\left(t-\sum_{i\in \bar{I}''}s_i-s_{n+1}\right)\prod_{i\in I''}e^{-{\bf i}s_i\vartheta_i},
	\end{aligned}\end{equation}
	where  $\vec{s'''}$ is the vector whose components are $s_i$ with $i\in \bar{I}''=I''\cup\{n+2\}$. Note that on the right hand side, we have a function of $s_{n+1}$ and on the left hand side, we have a function of $\bar{s}_{n+1}$ but $s_{n+1}=\bar{s}_{n+1}\lambda^{-2}$. Observe that for $\mathbf{N}/4\le n-i$, we have, following \eqref{Def:Para2}
	\begin{equation}\label{eq:Aestimate1:1}
		\mathrm{Im}(- \vartheta_i) \ge \varsigma_{n-i}=\varsigma'=\lambda^2 \mathbf{N}^{\wp}\ge 0.
	\end{equation}
	Using
	$
		\int_{\mathbb{R}_+^m}\mathrm{d}\bar{s}\delta\left(t-\sum_{i=1}^m s_i\right) \ = \ \frac{t^{m-1}}{(m-1)!},$
	we estimate
	\begin{equation}
		\begin{aligned}\label{eq:Aestimate2a}
		&\left|\int_{(\mathbb{R}_+)^{\bar{I}''}}\mathrm{d}\vec{s'''}\delta\left(t-\sum_{i\in \bar{I}''}s_i-s_{n+1}\right)\prod_{i\in I''}e^{-{\bf i}s_i\vartheta_i}\right|\\
 \le\ 			& \int_{(\mathbb{R}_+)^{\bar{I}''}}\mathrm{d}\vec{s'''}\delta\left(t-s_{n+1}-\sum_{i\in \bar{I}''}s_i\right)\prod_{i\in I''}e^{-\varsigma_{n-i}s_i}\\
			\quad \le\ &  \int_{(\mathbb{R}_+)^{\mathscr{Y}}}\mathrm{d}\bar{s}'''' \prod_{i\in \mathscr{Y}}e^{-\varsigma' s_i}\int_{(\mathbb{R}_+)^{\bar{I}''\backslash \mathscr{Y}}}\mathrm{d}\bar{s}'''\delta\left(t-s_{n+1}-\sum_{i\in \mathscr{Y}}s_i-\sum_{i\in \bar{I}''\backslash \mathscr{Y}}s_i\right)
			\\
			\quad \le\ & (\varsigma')^{-|\mathscr{Y}|}\frac{t^{{n}_0(n)-{n}_0(n-[\mathbf{N}/4])}}{({n}_0(n)-{n}_0(n-[\mathbf{N}/4]))!}, 
	\end{aligned}\end{equation}
	in which $\bar{s}'''$ is the vectors containing $s_i$ with $i\in \bar{I}''\backslash \mathscr{Y}$, $\bar{s}''''$ is the vectors containing $s_i$ with $i\in \mathscr{Y}$.  In the above computations, we split the vector $\vec{s'''}$ into two components $\bar{s}'''$ and $\bar{s}''''$.

	Next, we estimate 
	\begin{equation}
		\begin{aligned} \label{eq:Aestimate0:2}
			&\mathscr{F}_{1/n,\mathscr{R}}\Big[\widehat{\tilde{Q}_1^1\chi_{[0,T_*)}}(\mathscr{X}{\lambda^{-2}})\Big]\\
			&	\ = \   \mathscr{F}_{1/n,\mathscr{R}}\left[\int_{\mathbb{R}}\mathrm{d}\tau \chi_{[0,T_*]}(\tau)	\int_{\mathbb{R}}\mathrm{d}\bar{s}_{n+1}\lambda^{-2} \mathscr{G}(\lambda^{-2}\bar{s}_{n+1})\right.\\
			&\left.\ \ \ \ \ \ \ \times\chi_{[0,\infty)}(\bar{s}_{n+1}) \mathscr{Q}(\tau-\bar{s}_{n+1})\chi_{[0,\infty)}(\tau-\bar{s}_{n+1}) e^{-{\bf i}2\pi\tau\lambda^{-2}\mathscr{X}}\right]
			\\
			&	\ = \  \mathscr{F}_{1/n,\mathscr{R}}\left[\int_{0}^\infty\mathrm{d}\tau''	\int_{0}^\infty\mathrm{d}\bar{s}_{n+1}\lambda^{-2} \mathscr{G}(\lambda^{-2}\bar{s}_{n+1})e^{-{\bf i}2\pi\bar{s}_{n+1}\lambda^{-2}\mathscr{X}} \mathscr{Q}(\tau'')e^{-{\bf i}2\pi\tau''\mathscr{X}}\mathbf{1}\big(\bar{s}_{n+1}+\tau''\le T_*\big)\right]\\
			&	\ = \   \mathscr{F}_{1/n,\mathscr{R}}\left[\int_{0}^{T_*}\mathrm{d}\tau''	\int_{0}^{T_*-\tau''}\mathrm{d}\bar{s}_{n+1}\lambda^{-2} \mathscr{G}(\lambda^{-2}\bar{s}_{n+1})e^{-{\bf i}2\pi\bar{s}_{n+1}\lambda^{-2}\mathscr{X}} \mathscr{Q}(\tau'')e^{-{\bf i}2\pi\tau''\lambda^{-2}\mathscr{X}}\right]
			\\
			&	\ = \  \mathscr{F}_{1/n,\mathscr{R}}\left[\int_{0}^{T_*}\mathrm{d}\tau''	\int_{0}^{(T_*-\tau'')\lambda^{-2}}\mathrm{d}{s}_{n+1} \mathscr{G}({s}_{n+1})e^{-{\bf i}2\pi{s}_{n+1}\mathscr{X}} \mathscr{Q}(\tau'')e^{-{\bf i}2\pi\tau''\lambda^{-2}\mathscr{X}}\right]\\
			&	\ \lesssim \   \left[\int_{0}^{T_*}\mathrm{d}\tau \|\mathscr{Q}\|_{L^\infty}\right]\sup_{\tau''\in[0,T_*]} 	\mathscr{F}_{1/n,\mathscr{R}}\left[\int_{0}^{(T_*-\tau'')\lambda^{-2}}\mathrm{d}{s}_{n+1} \mathscr{G}({s}_{n+1})e^{-{\bf i}2\pi{s}_{n+1}\mathscr{X}}\right],
		\end{aligned}
	\end{equation}
	where we applied the change of variables $\bar{s}_{n+1}\lambda^{-2}=s_{n+1}$ and $\tau''=T_*-\bar{s}_{n+1}$.
		Defining    the contour 
	\begin{equation}
		\label{DepictContour}\begin{aligned}
			\Gamma_{n}\ :=\ &\left\{\xi\in\mathbb{C}\ \ \ \ \Big|\ \ \ \  \mathrm{Re}\xi=\pm\Big(1+2(n+2)\|\omega\|_{L^\infty}\Big), \ \ \ \  \mathrm{Im}\xi\in\Big[-1-2\varsigma',\lambda^2\Big]\right\}\\
			\ &\bigcup \left\{\xi\in\mathbb{C}\ \ \ \ \Big|\ \ \ \  \mathrm{Re}\xi\in\Big[-\Big(1+2(n+2)\|\omega\|_{L^\infty}\Big),1+2(n+2)\|\omega\|_{L^\infty}\Big],\right.\\
			&\ \ \ \  \ \ \left.\mathrm{Im}\xi\in\Big\{-1-2\varsigma',\lambda^2\Big\}\right\},
		\end{aligned}
	\end{equation}
	 using Lemma \ref{Lemma:Kidentity}, we find
	\begin{equation}
		\begin{aligned}\label{eq:Aestimate0}
			& \int_{(\mathbb{R}_+)^{I}}\mathrm{d}\vec{s''}\delta\left(s_{n+1}-\sum_{i\in I}s_i\right)\prod_{i\in I}e^{-{\bf i}s_i \mathrm{Re}\vartheta_i}e^{-{\bf i}2\pi{s}_{n+1}\mathscr{X}}\\
			= \ & {\bf i} \oint_{\Gamma_{n}}\frac{\mathrm{d}\xi}{2\pi}e^{-{\bf i}s_{n+1}\xi}\prod_{i\in I}\frac{1}{\xi-\mathrm{Re}\vartheta_i-2\pi\mathscr{X}},
		\end{aligned}
	\end{equation}
	which implies
	\begin{equation}
		\begin{aligned}\label{eq:Aestimate1}
			& \left|\int_{(\mathbb{R}_+)^{I}}\mathrm{d}\vec{s''}\delta\left(s_{n+1}-\sum_{i\in I}s_i\right)\prod_{i\in I}e^{-{\bf i}s_i \vartheta_i}e^{-{\bf i}2\pi{s}_{n+1}\mathscr{X}}\mathscr{Z} \prod_{i\in I}e^{-s_i\tau_{i}}  \right|\\
			= & \left|\int_{(\mathbb{R}_+)^{I}}\mathrm{d}\vec{s''}\delta\left(s_{n+1}-\sum_{i\in I}s_i\right)e^{-{\bf i}2\pi{s}_{n+1}\mathscr{X}}\prod_{i\in I}e^{-{\bf i}s_i \mathrm{Re}\vartheta_i}\prod_{i\in I}^{n}e^{s_i \mathrm{Im}\vartheta_i}\mathscr{Z} \prod_{i\in I}e^{-s_i\tau_{i}} \right|\\
			\le &  \oint_{\Gamma_{n}}\frac{|\mathrm{d}\xi|}{2\pi}\left|e^{-{\bf i}s_{n+1}\xi}\right|\prod_{i\in I}\left|e^{-{\bf i}s_i\mathrm{Re}\vartheta_i}\right|  \Big|\prod_{i\in I}e^{s_i \mathrm{Im}\vartheta_i}\Big|\prod_{i\in I}\frac{1}{|\xi-(\mathrm{Re}\vartheta_i+2\pi\mathscr{X}-{\bf i} \lambda^2)|}|\mathscr{Z}|.
		\end{aligned}
	\end{equation}

	We next bound the left hand side of \eqref{eq:BasicGEstimate1} by
	\begin{equation}
		\begin{aligned} \label{eq:Aestimate4}
			&
			\left\|\limsup_{D\to\infty}  \mathscr{F}_{1/n,\mathscr{R}}\left[\mathbf{1}({\sigma_{n,1}=-1})\mathbf{1}({\sigma_{n,2}=1})\sum_{\substack{\bar\sigma\in \{\pm1\}^{\mathcal{I}_n},\\ \sigma_{i,\rho_i}+\sigma_{i-1,\rho_i}+\sigma_{i-1,\rho_i+1}\ne \pm3,					\\ \sigma_{i-1,\rho_i}\sigma_{i-1,\rho_i+1}= 1}}\right.\right.  \\
			&\times  \int_{(\Lambda^*)^{\mathcal{I}_n}}  \mathrm{d}\bar{k}\Delta_{n,\rho}(\bar{k},\bar\sigma) \prod_{i=1}^n\Big[|\mathcal{M}( k_{i,\rho_i}, k_{i-1,\rho_i}, k_{i-1,\rho_i+1})|\\
			&\times \Phi_{1,i}(\sigma_{i-1,\rho_i}, k_{i-1,\rho_i},\sigma_{i-1,\rho_i+1}, k_{i-1,\rho_i+1})\Big]||\mathscr{Z}|\lambda^{n-2n_0(n)}\mathbf{N}^{-\wp n_0(n-[\mathbf{N}/4])} \\
			&\left.\left.\times \frac{T_*^{1+{n}_0(n)-n_0(n-[\mathfrak{N}/4])}}{({n}_0(n)-n_0(n-[\mathbf{N}/4]))!}
			\oint_{\Gamma_{n}}\frac{|\mathrm{d}\xi|}{2\pi}\frac{e^{T_*}}{|\xi|}\prod_{i\in I\backslash\{n\}}\frac{1}{|\xi-(\mathrm{Re}\vartheta_i+2\pi\mathscr{X}-{\bf i} \lambda^2)|}\right]\right\|_{L^{\infty,\Im}(\mathbb{T}^{2d})},
	\end{aligned}\end{equation}
where the norm is taken with respect to the two variables $k_{n,1}$ and $k_{n,2}$. 
	By H\"older's inequality applied to at most $\mathbf{N}/2$ terms that are associated to the $\mathscr{Z}$-variable  $\frac{1}{|\xi-(\mathrm{Re}\vartheta_i+2\pi\mathscr{X}-{\bf i} \lambda^2)|}$, we bound \eqref{eq:Aestimate4} by
	\begin{equation}
		\begin{aligned} \label{eq:Aestimate4:a}
			&
			\left\|\limsup_{D\to\infty}  \mathbf{1}({\sigma_{n,1}=-1})\mathbf{1}({\sigma_{n,2}=1})\sum_{\substack{\bar\sigma\in \{\pm1\}^{\mathcal{I}_n},\\ \sigma_{i,\rho_i}+\sigma_{i-1,\rho_i}+\sigma_{i-1,\rho_i+1}\ne \pm3,					\\ \sigma_{i-1,\rho_i}\sigma_{i-1,\rho_i+1}= 1}}\right. \\
			&\times  \int_{(\Lambda^*)^{\mathcal{I}_n}}  \mathrm{d}\bar{k}\Delta_{n,\rho}(\bar{k},\bar\sigma)\prod_{i=1}^n\Big[|\mathcal{M}( k_{i,\rho_i}, k_{i-1,\rho_i}, k_{i-1,\rho_i+1})|\Phi_{1,i}(\sigma_{i-1,\rho_i}, k_{i-1,\rho_i},\sigma_{i-1,\rho_i+1}, k_{i-1,\rho_i+1})\Big]\\
			&\times|\mathscr{Z}|  \lambda^{n-2n_0(n)}\mathbf{N}^{-\wp n_0(n-[\mathbf{N}/4])}\frac{T_*^{1+{n}_0(n)-n_0(n-[\mathbf{N}/4])}}{({n}_0(n)-n_0(n-[\mathbf{N}/4]))!}
			\oint_{\Gamma_{n}}\frac{|\mathrm{d}\xi|}{2\pi}\frac{e^{T_*}}{|\xi|}\\
			&\times\left.\prod_{i\in I\backslash\{n\}}\left[\int_{-\mathscr{R}}^{\mathscr{R}}\mathrm{d}\mathscr{X}\frac{1}{|\xi-(\mathrm{Re}\vartheta_i+2\pi\mathscr{X}-{\bf i} \lambda^2)|^2}\right]^\frac12\right\|_{L^{\infty,\Im}(\mathbb{T}^{2d})}\\
			 \ =: \ & \|\tilde{Q}_1^*(\tau,k_{n,2},\sigma_{n,2}, k_{n,1},\sigma_{n,1})\|_{L^{\infty,\Im}(\mathbb{T}^{2d})},
	\end{aligned}\end{equation}
in which $1<\Im=\frac{2\Im_o}{2\Im_o-1}<2$ for some  integer $\Im_o>0$.
First, we observe that for each partition $S$ of the  momenta at the bottom of the graph, each cluster momentum $\mathscr{K}_A$ in the structure of $\tilde{Q}_1^*(\tau,k_{n,2},\sigma_{n,2}, k_{n,1},\sigma_{n,1})$ is associated to a set of momenta $k_{0,j}$ with $j\in A$, that statisfies
$$\mathscr{K}_A=\sum_{j\in A} k_{0,j}\sigma_{0,j}.$$
We denote the vector whose components are $k_{0,j}$ with $j\in A$ by $\big(k_{0,j}\big)_{j\in A}.$
We can write the quantity $\tilde{Q}_1^*(\tau,k_{n,2},\sigma_{n,2}, k_{n,1},\sigma_{n,1})$  as the product of independent components $\tilde{Q}_{1,A}$, in which each $\tilde{Q}_{1,A}$ depends only on $\mathscr{K}_A$ and $k_{0,j}$ with $j\in A$
$$\tilde{Q}_1^*(\tau,k_{n,2},\sigma_{n,2}, k_{n,1},\sigma_{n,1}) \ = \ \prod_{A\in S} \tilde{Q}_{1,A}^*\left(\mathscr{K}_A,\big(k_{0,j}\big)_{j\in A}\right).$$
Noticing that $\sum_{A\in S}\mathscr{K}_A=k_{n,2}\sigma_{n,2}+k_{n,1}\sigma_{n,1}$,

\begin{equation}\label{eq:Aestimate4:aa}\begin{aligned} & \tilde{Q}_1^*(\tau,k_{n,2},\sigma_{n,2}, k_{n,1},\sigma_{n,1})\\
	= \ & \prod_{A\in S}\int_{-\epsilon  \complement}^{\epsilon  \complement}\mathrm{d}\mathscr{K}_A \delta\left(\mathscr{K}_A-\sum_{j\in A} k_{0,j}\sigma_{0,j}\right) \prod_{j\in A}\int_{\Lambda^*}\mathrm{d}k_{0,j} \tilde{Q}_{1,A}^*\Big(\mathscr{K}_A,\big(k_{0,j}\big)_{j\in A}\Big).\end{aligned}\end{equation}
We   bound  by the standard convolution  inequality 
\begin{equation}\label{eq:Aestimate4:aaa}\begin{aligned} &\|\tilde{Q}_1^*(\tau,k_{n,2},\sigma_{n,2}, k_{n,1},\sigma_{n,1})\|_{L^{\infty,\Im}(\mathbb{T}^{2d})}\\
	\lesssim \ & \epsilon^{|S|d(1-1/\Im)}\prod_{A\in S}\left[\prod_{j\in A}\int_{\Lambda^*}\mathrm{d}k_{0,j}\Big|\tilde{Q}_{1,A}^*\Big(\mathscr{K}_A,\big(k_{0,j}\big)_{j\in A}\Big)\Big|^\Im\right]^\frac1\Im.\end{aligned}\end{equation}
	 We can  pass to the limit $D\to \infty$ and   the integral $\int_{\Lambda^*}\mathrm{d}k$  after passing to the limit $D\to\infty$ by the dominated convergence theorem is be replaced   by $\int_{\mathbb{T}^d}\mathrm{d}k$.  In the same fashion as in   Proposition \ref{Propo:Phi},  Proposition \ref{Propo:Phi3A}, Definition \ref{def:Phi3},  the cut-off functions are replaced as $
		\Phi_{1,i}\to \tilde{\Phi}_{1,i}.
$ The component $\Psi_4$ of  $\tilde{\Phi}_{1,i}$ is always bounded by $1$. 
	The number of possibilities to assign signs to $\bar{\sigma}$ is always at most $4^{n+1}$.
Next, we will use \eqref{Propo:ExpectationEqui:2a} via a process of integrating the momenta to bound the components $\left[\prod_{j\in A}\int_{\Lambda^*}\mathrm{d}k_{0,j}\Big|\tilde{Q}_{1,A}^*\Big(\mathscr{K}_A,\big(k_{0,j}\big)_{j\in A}\Big)\Big|^\Im\right]^\frac1\Im$.	The process of integrating the momenta to bound the components $\left[\prod_{j\in A}\int_{\Lambda^*}\mathrm{d}k_{0,j}\Big|\tilde{Q}_{1,A}^*\Big(\mathscr{K}_A,\big(k_{0,j}\big)_{j\in A}\Big)\Big|^\Im\right]^\frac1\Im$ is done as follows.  All of the free momenta  are integrated from the bottom to the top, starting from time slice $0$ while the cluster momenta are left until the end of the process due to \eqref{eq:Aestimate4:aa}.  The factor $(n!)^{(\Im-1)}\le (\mathbf{N}!)^{(\Im-1)}$ in \eqref{Propo:ExpectationEqui:2d}   can be absorbed by $\lambda^{-c_{\mathbf{N}_0}}$ in which $c_{\mathbf{N}_0}$ can be chosen as small as we need by balancing $\mathbf{N}_0$.	At the bottom of the graph, the moment functions   are $L^\Im$ functions due to  the   bounds in Proposition \ref{Propo:ExpectationEqui}. 
   In the integrating process, we will introduce a way to estimate the quantities that are associated to the degree-one vertices. Let   $v_i$ be a degree one vertex and denote by $k,k'$ the two edges in $\mathfrak{E}_-(v_i)$, in which   $k$ is the free edge. We will need to estimate $$
		\int_{\mathbb{T}^d}\mathrm{d}k\frac{\tilde\Phi_{1,i}(k)\mathscr{H}_0(k)\mathscr{H}_1(k')}{|\xi-(\mathrm{Re}\vartheta_i+2\pi\mathscr{X}-{\bf i} \lambda^2)|},
$$
	where $\mathscr{H}_0$ and $\mathscr{H}_1$ come from the previous iterations and belong to $L^\Im(\mathbb{T}^d)$. As the momentum $k'$ depends on $k$, we have $k'=\sigma k + k''$, where $\sigma\in\{\pm1\}$. We suppose $k''$ depends on a free edge $k'''$ attached to a vertex $v_j$ which is above $k$ and $v_i$. Assuming that $\mathscr{H}_1(k')$ does not contain $\frac{\tilde\Phi_{1,j}(k)}{|\xi-(\mathrm{Re}\vartheta_j+2\pi\mathscr{X}-{\bf i} \lambda^2)|}$, we bound
	\begin{equation}
		\label{eq:Aestimate5}\begin{aligned}
			&  \int_{\mathbb{T}^d}\mathrm{d}k'''\Big|	\int_{\mathbb{T}^d}\mathrm{d}k\frac{\tilde\Phi_{1,i}(k)|\mathscr{H}_0(k)|\mathscr{H}_1(\sigma k +k'')}{|\xi-(\mathrm{Re}\vartheta_i+2\pi\mathscr{X}-{\bf i} \lambda^2)|}\Big|^\Im 
\ = \ \int_{\mathbb{T}^d}\mathrm{d}k''\Big|	\int_{\mathbb{T}^d}\mathrm{d}k\frac{\tilde\Phi_{1,i}(k)|\mathscr{H}_0(k)|\mathscr{H}_1(\sigma k +k'')}{|\xi-(\mathrm{Re}\vartheta_i+2\pi\mathscr{X}-{\bf i} \lambda^2)|}\Big|^\Im 			\\
			\ \le \ & \Big[\int_{\mathbb{T}^d}\mathrm{d}k'	\int_{\mathbb{T}^d}\mathrm{d}k\frac{\tilde\Phi_{1,i}(k)|\mathscr{H}_0(k)||\mathscr{H}_1(k')|^2}{|\xi-(\mathrm{Re}\vartheta_i+2\pi\mathscr{X}-{\bf i} \lambda^2)|}\Big]\int_{\mathbb{T}^d}\mathrm{d}k\frac{\tilde\Phi_{1,i}(k)|\mathscr{H}_0(k)|}{|\xi-(\mathrm{Re}\vartheta_i+2\pi\mathscr{X}-{\bf i} \lambda^2)|}\\\
			\ \le  \ & \Big\|\int_{\mathbb{T}^d}\mathrm{d}k\frac{\tilde\Phi_{1,i}(k)|\mathscr{H}_0(k)|}{|\xi-(\mathrm{Re}\vartheta_i+2\pi\mathscr{X}-{\bf i} \lambda^2)|}\Big\|_{L^\infty}^\Im\|\mathscr{H}_1\|^\Im_{L^\Im}.\end{aligned}
	\end{equation}
As	  $\mathscr{H}_1(k')$ does not contain $\frac{\tilde\Phi_{1,j}(k)}{|\xi-(\mathrm{Re}\vartheta_j+2\pi\mathscr{X}-{\bf i} \lambda^2)|}$, in the next iterations, the integrations $\Big\|\int_{\mathbb{T}^d}\mathrm{d}k''$ $\frac{\tilde\Phi_{1,j}(k'')|\tilde{H}(k'')|}{|\xi-(\mathrm{Re}\vartheta_j+2\pi\mathscr{X}-{\bf i} \lambda^2)|}\Big\|_\infty$, for some $\tilde{H}$ will be estimated.  In those estimates  $|\tilde{H}|$ already belongs to $L^\Im$ due to the estimates  \eqref{eq:Aestimate5} obtained in the previous iterations. 
	By  Lemma \ref{lemma:degree1vertex}, we obtain a factor of $\langle\ln\lambda\rangle^{2+\eth}$ (see Section \ref{Sec:KeyPara} for the definitions of the parameters). If one of the degree-one vertices is associated to the two cut-off functions $\Phi_{1}^b(\eth'_1), \Phi_{1}^b(\eth'_2),$ then by Lemma \ref{lemma:degree1vertex}, we obtain  $\lambda^{-\eth'_1-\eth'_2}$. Therefore, the  factor of $\lambda$ coming from Lemma \ref{lemma:degree1vertex} is bounded by $\lambda^{n-2{n}_0(n)-\eth'_1-\eth'_2}\langle\ln\lambda\rangle^{(2+\eth)(n_1(n)-2)}.$  Inequality \eqref{Propo:ExpectationEqui:2d} gives the $\lambda$ factor of $$\lambda^{2|S|d(1-1/\Im)-d(\Im-1)(n+2)} \ =\ \lambda^{d(\Im-1)\big[2|S|/\Im-n-2\big]} \ =\ \lambda^{d(\Im-1)\big[\frac{2(n+2-n_0)}{\Im}-n-2\big]}.$$
	The total $\lambda$ factor is now $$\lambda^{d(\Im-1)\big[\frac{2(n+2-n_0)}{\Im}-n-2\big]+n-2{n}_0(n)-\eth'_1-\eth'_2}.$$
	The final integral  is
	$
	\oint_{\Gamma_{n}}\frac{|\mathrm{d}\xi|}{2\pi}\frac{1}{|\xi|},
	$		which can be bounded  by $C\langle\ln\lambda^2\rangle.$  As a consequence, the result of the Proposition follows.

\end{proof}

\begin{proposition}[The second   estimate on ${Q}^1$]\label{lemma:BasicGEstimate1a} Let $S$ be an arbitrary partition of $\{1,\cdots n+2\}$. If the corresponding graph is non-pairing and non-singular,  there is a constant $\mathfrak{C}_{Q_{1,2}}>0$ such that for $1\le n\le \mathbf{N}$ and $t=\tau\lambda^{-2}>0$, we set 
	\begin{equation}
		\begin{aligned} \label{eq:BasicGEstimate1a:1}
			\tilde{Q}^1_2(\tau):=\	& 
			\lambda^{n}\mathbf{1}({\sigma_{n,1}=-1})\mathbf{1}({\sigma_{n,2}=1})\Big[ \sum_{\substack{\bar\sigma\in \{\pm1\}^{\mathcal{I}_n}, \sigma_{i,\rho_i}+\sigma_{i-1,\rho_i}+\sigma_{i-1,\rho_i+1}\ne \pm3,
					\\ \sigma_{i-1,\rho_i}\sigma_{i-1,\rho_i+1}= 1}}\\
			&\times\int_{(\Lambda^*)^{\mathcal{I}_n}}  \mathrm{d}\bar{k}\Delta_{n,\rho}(\bar{k},\bar\sigma) \mathscr{Z} \prod_{i=1}^n\Big[\sigma_{i,\rho_i}\mathcal{M}( k_{i,\rho_i}, k_{i-1,\rho_i}, k_{i-1,\rho_i+1})\prod_{i=0}^{n}e^{-s_i[\varsigma_{n-i}+\tau_i]}\prod_{i=1}^{n} e^{-{\bf i}t_i(s)\mathbf{X}_i}\\
			&\times \Phi_{1,i}(\sigma_{i-1,\rho_i}  k_{i-1,\rho_i},\sigma_{i-1,\rho_i+1}, k_{i-1,\rho_i+1})\Big]\int_{(\mathbb{R}_+)^{\{0,\cdots,n\}}}\mathrm{d}\bar{s} \delta\left(t-\sum_{i=0}^ns_i\right)\Big]	\end{aligned}\end{equation}
	then for any constants $1>T_*>0$ and $\mathscr{R}>0$,
	\begin{equation}
		\begin{aligned} \label{eq:BasicGEstimate1a}
			&\Big\|\limsup_{D\to\infty}\mathscr{F}_{1/n,\mathscr{R}}\Big[\widehat{\tilde{Q}^1_2\chi_{[0,T_*)}}(\mathscr{X}{\lambda^{-2}})\Big](k_{n,1},k_{n,2})\Big\|_{L^{\infty,\Im}(\mathbb{T}^{2d})}
			\\
			\le& 4^n e^{T_*}\frac{T_*^{1+{n}_0(n)-{n}_0(n-[\mathbf{N}/4])}}{({n}_0(n)-{n}_0(n-[\mathbf{N}/4]))!}\lambda^{\bowtie_\Im+1-\eth'_1-\eth'_2}\mathbf{N}^{-\wp n_0(n-[\mathbf{N}/4])}\mathfrak{C}_{Q_{1,2}}^{1+{n}_1(n)}\langle \ln n\rangle \langle\ln\lambda\rangle^{1+(2+\eth)({n}_1(n)-2)},
	\end{aligned}\end{equation}

	where the same notations as in  Proposition \ref{lemma:BasicGEstimate1} have been used.
	
\end{proposition}
\begin{proof} 
	Let us consider the power  $\lambda^{n-2n_0(n)}$,  obtained in Proposition  \ref{lemma:BasicGEstimate1}. As $n=n_0(n)+n_1(n)$, the
	quantity $n-2n_0(n)=n_1(n)-n_0(n)$.  We have $n_0(n)=|S|-2$ or $|S|-1$ and $n_1(n)= n+2-|S|$ or $n+1-|S|$, respectively. 
	Thus
	$n_1(n)-n_0(n)=(n+2-|S|)-|S|+2=n+4-2|S|,$
	or
	$n_1(n)-n_0(n)=(n+1-|S|)-|S|+1=n+2-2|S|.$ Since the graph is not pairing, 
	$n+2-2|S|\ge 1,$
yielding
	$n_1(n)-n_0(n)\ge 1.$
	Therefore, the power $\lambda^{n-2{n}_0(n)}$ in Proposition \ref{lemma:BasicGEstimate1} can be replaced by $\lambda$ as $0<\lambda<1$. 
\end{proof}

\begin{proposition}[The third   estimate on ${Q}^1$]\label{lemma:BasicGEstimate2}  Let $S$ be an arbitrary partition of $\{1,\cdots n+2\}$.  If the corresponding graph is non-singular, there is a constant $\mathfrak{C}_{Q_{1,3}}>0$ such that for $1\le n\le \mathbf{N}$ and $\tau=t\lambda^{-2}>0$, we set
	\begin{equation}
		\begin{aligned} \label{eq:BasicGEstimate2:1}
			\tilde{Q}^1_3(\tau):=	&  
			\lambda^{n}\mathbf{1}({\sigma_{n,1}=-1})\mathbf{1}({\sigma_{n,2}=1})e^{\tau_nt}\Big[ \sum_{\substack{\bar\sigma\in \{\pm1\}^{\mathcal{I}_n},\\ \sigma_{i,\rho_i}+\sigma_{i-1,\rho_i}+\sigma_{i-1,\rho_i+1}\ne \pm3,
					\\ \sigma_{i-1,\rho_i}\sigma_{i-1,\rho_i+1}= 1}} \int_{(\Lambda^*)^{\mathcal{I}_n}}  \mathrm{d}\bar{k}\Delta_{n,\rho}(\bar{k},\bar\sigma)\mathscr{Z}\\
			&\times \  \prod_{i=1}^n\Big[\sigma_{i,\rho_i}\mathcal{M}( k_{i,\rho_i}, k_{i-1,\rho_i}, k_{i-1,\rho_i+1}) \Phi_{1,i}(\sigma_{i-1,\rho_i}, k_{i-1,\rho_i}, \sigma_{i-1,\rho_i+1}, k_{i-1,\rho_i+1})\Big] \\
			&\times\int_{(\mathbb{R}_+)^{\{0,\cdots,n\}}}\mathrm{d}\bar{s} \delta\left(t-\sum_{i=0}^ns_i\right)\prod_{i=0}^{n}e^{-s_i[\varsigma_{n-i}+\tau_i]}\prod_{i=1}^{n} e^{-{\bf i}t_i(s)\mathbf{X}_i}\Big],
	\end{aligned}\end{equation}
	then for any constants $1>T_*>0$ and $\mathscr{R}>0$,
	\begin{equation}
		\begin{aligned} \label{eq:BasicGEstimate2}
			& \Big\|\limsup_{D\to\infty}\mathscr{F}_{1/n,\mathscr{R}}\Big[\widehat{\tilde{Q}_3^1\chi_{[0,T_*)}}(\mathscr{X}{\lambda^{-2}})\Big](k_{n,1},k_{n,2})\Big\|_{L^{\infty,\Im}(\mathbb{T}^{2d})}
			\\
			\le& 4^n e^{T_*}\frac{T_*^{1+{n}_0(n)}}{{n}_0(n)!}\lambda^{\bowtie_\Im+n-2{n}_0(n)-\eth'_1-\eth'_2}\mathfrak{C}_{Q_{1,3}}^{1+{n}_1(n)}\langle \ln n\rangle \langle\ln\lambda\rangle^{1+(2+\eth)({n}_1(n)-2)},
	\end{aligned}\end{equation}
	where the same notations as in  Proposition  \ref{lemma:BasicGEstimate1} have been used.
	
	Moreover, if the graph is non-pairing, the right hand side of \eqref{eq:BasicGEstimate2}  can also be replaced  by
	\begin{equation}
		\label{eq:BasicGEstimate1aa}
		4^n\lambda^{1-\eth'_1-\eth'_2} e^{T_*}\frac{T_*^{1+{n}_0(n)}}{({n}_0(n))!}\mathfrak{C}_{Q_{1,3}}^{1+{n}_1(n)}\langle \ln n\rangle \langle\ln\lambda\rangle^{1+(2+\eth)({n}_1(n)-2)}.
	\end{equation}
\end{proposition}
\begin{proof}
	The proof follows exactly the same lines of arguments  as for the proof    of Proposition  \ref{lemma:BasicGEstimate1}   and Proposition \ref{lemma:BasicGEstimate1a}. 
\end{proof}
\begin{proposition}[The fourth   estimate on ${Q}^1$]\label{lemma:BasicGEstimate4}  Let $S$ be an arbitrary partition of $\{1,\cdots n+2\}$. If the corresponding graph is singular,   for $1\le n$ and $\tau>0$, the quantity
	\begin{equation}
		\begin{aligned} \label{eq:BasicGEstimate4:1}
			&\mathbf{1}({\sigma_{n,1}=-1})\mathbf{1}({\sigma_{n,2}=1}) \sum_{\substack{\bar\sigma\in \{\pm1\}^{\mathcal{I}_n},\\ \sigma_{i,\rho_i}+\sigma_{i-1,\rho_i}+\sigma_{i-1,\rho_i+1}\ne \pm3,
					\\ \sigma_{i-1,\rho_i}\sigma_{i-1,\rho_i+1}= 1}}\int_{(\Lambda^*)^{\mathcal{I}_n}}  \mathrm{d}\bar{k}\Delta_{n,\rho}(\bar{k},\bar\sigma) \\
			&\times \mathscr{Z} \prod_{i=1}^n\Big[\sigma_{i,\rho_i}\mathcal{M}( k_{i,\rho_i}, k_{i-1,\rho_i}, k_{i-1,\rho_i+1}) \Phi_{1,i}(\sigma_{i-1,\rho_i}  k_{i-1,\rho_i},\sigma_{i-1,\rho_i+1}, k_{i-1,\rho_i+1})\Big] \\
			&\times\int_{(\mathbb{R}_+)^{\{0,\cdots,n\}}}\mathrm{d}\bar{s} \delta\left(t-\sum_{i=0}^ns_i\right)\prod_{i=0}^{n}e^{-s_i[\varsigma_{n-i}+\tau_i]}\prod_{i=1}^{n} e^{-{\bf i}t_i(s)\mathbf{X}_i},
	\end{aligned}\end{equation}
	vanishes  when $\lambda$ is sufficiently small, where the same notations as in Proposition \ref{lemma:BasicGEstimate1} have been used.
	
\end{proposition}
\begin{proof} We first  consider the case when  $\mathscr{Z}=\mathscr{U}(0)$.	By our Definition \ref{Def:SingGraph},  the graph is singular when there is an edge $k_{j,\rho_j}$ that depends only on the cluster momenta, which is smaller than $\epsilon\complement$, this enforces $$\Phi_{1,i}(\sigma_{j-1,\rho_j}  k_{j-1,\rho_j},\sigma_{j-1,\rho_j+1}, k_{j-1,\rho_j+1})\to0 \mbox{ and  } \mathcal{M}( k_{i,\rho_i}, k_{i-1,\rho_i}, k_{i-1,\rho_i+1})\to0$$  when $\epsilon\to 0$.  As a result, the whole quantity \eqref{eq:BasicGEstimate4:1} becomes  zero when $\lambda$ is small. 

Next, we consider the case when $\mathscr{Z}\ne\mathscr{U}(0)$.	As the graph is singular, we can split the graph as the graph as a  1-Separation. The singular edge is therefore associated to a function $\Phi_{1,j}$ and a kernel $\mathcal{M}$. As the momentum is of order $\epsilon/|\ln|\lambda||$, it also leads 
 to the vanishing of the graph.
\end{proof}

%
%
%
%

\subsection{Second type of diagram estimates}\label{Sec:SecondDiagram}
\begin{definition}\label{Def:Q3pair}
	We decompose
	\begin{equation}
		\label{lemma:Q3FinalEstimate:2}
		Q^3 \ = \ Q^{3,pair} \ + \ Q^{3,nonpair} 
	\end{equation}
	where
	\begin{equation}\label{lemma:Q1FinalEstimate:3}
		Q^{3,pair} 
		\ = \ \sum_{n=1}^{\mathbf{N}}\sum_{S\in\mathcal{P}^o_{pair}(\{1,\cdots n+2\})}\int_0^t\mathrm{d}s\mathcal{G}^{2}_{n}(S,s,t,k'',-1,k',1,\Gamma).
	\end{equation}
\end{definition}
\begin{proposition}[The first   estimate on ${Q}^3$]\label{lemma:Q3FirstEstimate}
	Let $S$ be an arbitrary partition of $\{1,\cdots n+2\}$.  If the corresponding graph is  non-singular  there is a constant $\mathfrak{C}_{3,1}>0$ such that for $1\le n\le \mathbf{N}$ and $t=\tau\lambda^{-2}>0$, we set
	\begin{equation}
		\begin{aligned} \label{eq:BasicGEstimate3:1}
			\tilde{Q}^3:= & \	\lambda^{n}\mathbf{1}({\sigma_{n,1}=-1})\mathbf{1}({\sigma_{n,2}=1})\Big[  \\&\times \sum_{\substack{\bar\sigma\in \{\pm1\}^{\mathcal{I}_n},\\ \sigma_{i,\rho_i}+\sigma_{i-1,\rho_i}+\sigma_{i-1,\rho_i+1}\ne \pm3,
					\\ \sigma_{i-1,\rho_i}\sigma_{i-1,\rho_i+1}= 1}}\int_0^t\mathrm{d}s_0\int_{(\Lambda^*)^{\mathcal{I}_n}}  \mathrm{d}\bar{k}\Delta_{n,\rho}(\bar{k},\bar\sigma)\mathscr{W}(s_0) \\
			&\times  e^{{\bf i}s_0\vartheta_0} \Phi_{0,1}(\sigma_{0,\rho_1},  k_{0,\rho_1},\sigma_{0,\rho_2},  k_{0,\rho_2})\sigma_{i,\rho_i}\mathcal{M}( k_{1,\rho_1}, k_{0,\rho_1}, k_{0,\rho_1+1})\\
			&\times \prod_{i=2}^n\Big[\sigma_{i,\rho_i}\mathcal{M}( k_{i,\rho_i}, k_{i-1,\rho_i}, k_{i-1,\rho_i+1})\Phi_{1,i}(\sigma_{i-1,\rho_i},  k_{i-1,\rho_i},\sigma_{i-1,\rho_i+1},  k_{i-1,\rho_i+1})\Big] \\
			&\times \int_{(\mathbb{R}_+)^{\{1,\cdots,n\}}}\mathrm{d}\bar{s} \delta\left(t-\sum_{i=0}^ns_i\right)\prod_{i=1}^{n}e^{-s_i[\varsigma_{n-i}+\tau_i]} e^{-{\bf i}t_i(s)\mathbf{X}_i}\Big],
	\end{aligned}\end{equation}
	then for any constants $1>T_*>0$ and $\mathscr{R}>0$,
	
	\begin{equation}
		\begin{aligned} \label{eq:BasicGEstimate3}
			&\Big\|\limsup_{D\to\infty}\mathscr{F}_{1/n,\mathscr{R}}\Big[\widehat{\tilde{Q}^3\chi_{[0,T_*)}}(\mathscr{X}{\lambda^{-2}})\Big](k_{n,1},k_{n,2})\Big\|_{L^{\infty,\Im}(\mathbb{T}^{2d})}
			\\
			\le & 4^n e^{T_*}\frac{T_*^{1+{n}_0(n)-{n}_0(n-[\mathbf{N}/4])}}{({n}_0(n)-{n}_0(n-[\mathbf{N}/4])!}\lambda^{n-2{n}_0(n)-\eth'_1-\eth'_2}\mathbf{N}^{-\wp	n_0(n-[\mathbf{N}/4]})\mathfrak{C}_{3,1}^{1+{n}_1(n)}\langle \ln n\rangle  \langle\ln\lambda\rangle^{1+(2+\eth)({n}_1(n)-2)},
	\end{aligned}\end{equation}
	where $\mathscr{W}$ is either $\mathscr{W}_1$ or $\mathscr{W}_2$ and the same notations as in  Proposition  \ref{lemma:BasicGEstimate1} have been used. The time integration is defined to be $$\int_{(\mathbb{R}_+)^{\{1,\cdots,n\}}}\mathrm{d}\bar{s} \delta\left(t-\sum_{i=0}^ns_i\right)= \int_{(\mathbb{R}_+)^{\{1,\cdots,n\}}}\mathrm{d}{s}_1\cdots \mathrm{d}{s}_n \delta\left(t-\sum_{i=0}^ns_i\right).$$	If the graph is non-pairing then the right hand side of \eqref{eq:BasicGEstimate3} can be replaced by 
	\begin{equation}\begin{aligned}
			\label{eq:BasicGEstimate3:1a}
			& 4^ne^{T_*}\frac{T_*^{1+{n}_0(n)-{n}_0(n-[\mathbf{N}/4])}}{({n}_0(n)-{n}_0(n-[\mathbf{N}/4])!}\lambda^{\bowtie_\Im+1-\eth'_1-\eth'_2}\mathbf{N}^{-\wp	n_0(n-[\mathbf{N}/4])} \mathfrak{C}_{3,1}^{1+{n}_1(n)}\langle \ln n\rangle  \langle\ln\lambda\rangle^{1+(2+\eth)({n}_1(n)-2)},\end{aligned}
	\end{equation}
	and the inequality holds true for all $\mathbf{N}\ge n\ge1$.

	If the graph is singular then  the right hand side of \eqref{eq:BasicGEstimate3} then tends to  $0$ in the limit $\lambda\to 0$.
\end{proposition}
\begin{proof}
	The proof is similar as the homogeneous case \cite{staffilani2021wave} and is omitted. 
	
\end{proof}

\begin{proposition}[The second estimate on   ${Q}^3$]\label{lemma:Q3SecondEstimate} There is a constant $\mathfrak{C}_{Q_{3,2}}>0$ such that for $t=\tau\lambda^{-2}>0$, we have, for any constants $1>T_*>0$ and $\mathscr{R}>0$,

	\begin{equation}
		\begin{aligned} \label{lemma:Q3SecondEstimate:2}
			&
			\Big\|\limsup_{D\to\infty} \mathscr{F}_{1/n,\mathscr{R}}\Big[\widehat{{Q}^{3,nonpair}\chi_{[0,T_*)}}(\mathscr{X}{\lambda^{-2}})\Big](k_{n,1},k_{n,2})\Big\|_{L^{\infty,\Im}(\mathbb{T}^{2d})}\\
			\le\ & \lambda^{\bowtie_\Im+1-\eth'_1-\eth'_2} \sum_{n=1}^{\mathbf{N}} n^n \mathfrak{C}_{Q_{3,2}}^n n! 
			e^{T_*}\frac{T_*^{1+{n}_0(n)}}{{n}_0(n)!} \langle \ln n\rangle  \langle\ln\lambda\rangle^{-1+(2+\eth)(n_1(n)-2)-\eth d}
	\end{aligned}\end{equation}
	 	where the same notations and in  Proposition   \ref{lemma:BasicGEstimate1} have been used. 
For a sufficient choice of $\Im$, the quantity on the left hand side of \eqref{lemma:Q3SecondEstimate:2} goes to $0$.
\end{proposition}
\begin{proof}
	
	Following the proof of Proposition \ref{lemma:BasicGEstimate1}, we can suppose that $\Lambda^*$  is replaced by $\mathbb{T}^d$ after passing to the limit $D\to\infty$.  Now, for  each graph,   the places where the splittings happen are encoded in the vectors  $\{\rho_i\}_{i=1}^n$. It is therefore necessary to 
	 compute  the number of choices for $\{\rho_i\}_{i=1}^n$. For each $\rho_i\in\{1,\cdots,n-i+2\}$, there are at most $(n-i+2)!$ choices. Therefore, the number of choices of $\{\rho_i\}_{i=1}^n$ is  at most 
	$
		\prod_{i=0}^{n-1}(n-i+2)!\le c_{\rho}n^n,
$
	for some universal constant $c_\rho$.  
	
	Now, we will make a control over the sum of all of the partitions $S$. As thus, the following quantity needs to be estimated
$
	\sum_{S\in \mathcal{P}(\{1,\cdots n+2\})\backslash\mathcal{P}^o_{pair}(\{1,\cdots n+2\})}\prod_{A\in S}\left(|A|!c^{|A|}\right),
$
	for some constant $c>1$, obtained by using the $L^2$ moment estimates used in the proof of Propositions \eqref{lemma:BasicGEstimate1} and \ref{lemma:Q3FirstEstimate}. This quantity can be bounded as follows
	\begin{equation}
		\label{lemma:Q1FinalEstimate:E2}
		\begin{aligned}
			\sum_{S\in \mathcal{P}(\{1,\cdots n+2\})\backslash\mathcal{P}^o_{pair}(\{1,\cdots n+2\})}\prod_{A\in S}\left(|A|!c^{|A|}\right)\ \le & \  \sum_{l=1}^n\sum_{{S\in \mathcal{P}(\{1,\cdots n+2\})\backslash\mathcal{P}^o_{pair}(\{1,\cdots n+2\})}, |S|=l}\prod_{A\in S}|A|!c^{|A|}\\
			\ \le & \ c^{n}\sum_{l=1}^n\sum_{{S\in \mathcal{P}(\{1,\cdots n+2\})\backslash\mathcal{P}^o_{pair}(\{1,\cdots n+2\})}, |S|=l}\prod_{A\in S}|A|!.
		\end{aligned}
	\end{equation}
We now estimate
	$$\sum_{l=1}^n\sum_{{S\in \mathcal{P}(\{1,\cdots n+2\})\backslash\mathcal{P}^o_{pair}(\{1,\cdots n+2\})}, |S|=l}\prod_{A\in S}|A|! = \frac{n!}{l!}\sum_{n=(n_1,\cdots,n_l)}\mathbf{1}\left(\sum_{j=1}^l n_j=n\right)\le \frac{n!(n-1)^{l-1}}{l!},$$
yielding  \eqref{lemma:Q3SecondEstimate:2}.  The factor $\lambda^{1-\eth'_1-\eth'_2}$ is due to the fact the graph is non-pairing. For  $\Im$ being sufficiently closed to $1$, we have $1-\eth'_1-\eth'_2>-\bowtie_\Im$, the quantity on the left hand side of \eqref{lemma:Q3SecondEstimate:2} goes to $0$.
\end{proof}
\subsection{Third type of diagram estimates}\label{Sec:ThirdDiagram}
\begin{proposition}[The first estimate on   ${Q}^4$]\label{lemma:Q4FirstEstimate}  Let $S$ be an arbitrary partition of $\{1,\cdots n+2\}$. There is a constant $\mathfrak{C}_{Q_{4,1}}>0$ such that for $1\le n\le \mathbf{N}$ and $t=\tau\lambda^{-2}>0$, we define
	\begin{equation}
		\begin{aligned} \label{lemma:Q4FirstEstimate:1}
			\tilde{Q}^4(\tau) :=\	& 
			\lambda^{n}\mathbf{1}({\sigma_{n,1}=-1})\mathbf{1}({\sigma_{n,2}=1})\Big[  \\&  \sum_{\substack{\bar\sigma\in \{\pm1\}^{\mathcal{I}_n},\\ \sigma_{i,\rho_i}+\sigma_{i-1,\rho_i}+\sigma_{i-1,\rho_i+1}\ne \pm3,
					\\ \sigma_{i-1,\rho_i}\sigma_{i-1,\rho_i+1}= 1}}\Big[\int_0^t\mathrm{d}s_0\int_{(\Lambda^*)^{\mathcal{I}_n}}  \mathrm{d}\bar{k} \Delta_{n,\rho}(\bar{k},\bar\sigma)\mathscr{W}(s_0)e^{{\bf i}s_0\vartheta_0}\\
			&\times   \prod_{i=1}^n\Big[\sigma_{i,\rho_i}\mathcal{M}( k_{i,\rho_i}, k_{i-1,\rho_i}, k_{i-1,\rho_i+1}) \prod_{i=1}^{n}e^{-s_i[\varsigma_{n-i}+\tau_i]} \prod_{i=1}^{n} e^{-{\bf i}t_i(s)\mathbf{X}_i} \\
			& \times\Phi_{1,i}(\sigma_{i-1,\rho_i}, k_{i-1,\rho_i},\sigma_{i-1,\rho_i+1}, k_{i-1,\rho_i+1})\Big] \int_{(\mathbb{R}_+)^{\{1,\cdots,n\}}}\mathrm{d}\bar{s} \delta\left(t-\sum_{i=0}^ns_i\right)\Big]\Big],	\end{aligned}\end{equation} then for any constants $1>T_*>0$ and $\mathscr{R}>0$,
	\begin{equation}
		\begin{aligned} \label{lemma:Q4FirstEstimate:2}
			&\Big\|\limsup_{D\to\infty} \mathscr{F}_{1/n,\mathscr{R}}\Big[\widehat{\tilde{Q}^4\chi_{[0,T_*)}}(\mathscr{X}{\lambda^{-2}})\Big](k_{n,1},k_{n,2})\Big\|_{L^{\infty,\Im}(\mathbb{T}^{2d})}
			\\
			\le\ & 4^ne^{T_*}\frac{T_*^{1+{n}_0(n)-{n}_0(n-[\mathbf{N}/4])}}{({n}_0(n)-{n}_0(n-[\mathbf{N}/4])!}\lambda^{\bowtie_\Im+n-2{n}_0(n)-\eth'_1-\eth'_2}\mathbf{N}^{-\wp	n_0(n-[\mathfrak{N}/4])} \mathfrak{C}_{Q_{4,1}}^{1+{n}_1(n)}\langle \ln n\rangle\\
			&\times \langle\ln\lambda\rangle^{-1+(2+\eth)(n_1(n)-2)-\eth d},
	\end{aligned}\end{equation}
	where the  time integration is defined to be $$\int_{(\mathbb{R}_+)^{\{1,\cdots,n\}}}\mathrm{d}\bar{s} \delta\left(t-\sum_{i=0}^ns_i\right)= \int_{(\mathbb{R}_+)^{\{1,\cdots,n\}}}\mathrm{d}{s}_1\cdots \mathrm{d}{s}_n \delta\left(t-\sum_{i=0}^ns_i\right),$$
	$\mathscr{W}$ is either $\mathscr{W}_1$ or $\mathscr{W}_2$ and the same notations as in  Proposition  \ref{lemma:BasicGEstimate1} have been used.
\end{proposition}
\begin{proof}
	The proof of the Proposition is the same as  that of Propositions \ref{lemma:BasicGEstimate1}  and \ref{lemma:Q3FirstEstimate}. 
\end{proof}

\begin{proposition}[The main estimate on   ${Q}^4$]\label{lemma:Q4SecondEstimate}
	There is a constant $\mathfrak{C}_{Q_{4,2}}>0$ such that for $\tau=t\lambda^{-2}>0$, we have, for any constants $1>T_*>0$ and $\mathscr{R}>0$,
	
	\begin{equation}
		\begin{aligned} \label{lemma:Q4SecondEstimate:1}
			&
			\Big\|\limsup_{D\to\infty} \mathscr{F}_{1/n,\mathscr{R}}\Big[\widehat{{Q}^4\chi_{[0,T_*)}}(\mathscr{X}{\lambda^{-2}})\Big](k_{n,1},k_{n,2})\Big\|_{L^{\infty,\Im}(\mathbb{T}^{2d})}
			\\				\le\ & e^{T_*}\frac{T_*^{1+{n}_0(\mathbf{N})-{n}_0(\mathbf{N}-[\mathbf{N}/4])}}{({n}_0(\mathbf{N})-{n}_0(\mathbf{N}-[\mathbf{N}/4]))!} \mathfrak{C}_{Q_{4,2}}^\mathbf{N} \mathbf{N}! \lambda^{\bowtie_\Im-\eth'_1-\eth'_2}\mathbf{N}^{-\frac{\wp\mathbf{N}}{4}	}\mathbf{N}^\mathbf{N} \mathfrak{C}_{Q_{4,2}}^{1+{n}_1(\mathbf{N})}\\
			&\times\langle \ln (\mathbf{N})\rangle \langle\ln\lambda\rangle^{-1+(2+\eth)(n_1(n)-2)-\eth d},
	\end{aligned}\end{equation}
	where the same notations with Lemma \ref{lemma:BasicGEstimate1} have been used. This quantity goes to $0$ as $\lambda$ tends to $0$, thanks to the choice of $\wp$ in \eqref{Def:Para5a}
	\begin{equation}
		\begin{aligned} \label{lemma:Q4SecondEstimate:2}
			& \lim_{\lambda\to0} \Big\|\limsup_{D\to\infty}\mathscr{F}_{1/n,\mathscr{R}}\Big[\widehat{{Q}^4\chi_{[0,T_*)}}(\mathscr{X}{\lambda^{-2}})\Big](k_{n,1},k_{n,2})\Big\|_{L^{\infty,\Im}(\mathbb{T}^{2d})} \ = \ 0.
	\end{aligned}\end{equation} 
\end{proposition}
\begin{proof} The proposition is a consequence of Proposition \ref{lemma:Q4FirstEstimate}. Observing that 
	$$\frac{\mathbf{N}-[\mathbf{N}/4]-(\mathbf{N}+2-2|S|)}{2} \le n_0(\mathbf{N}-[\mathbf{N}/4]),$$
	which yields
	$$\lambda^{\mathbf{N}-2{n}_0(\mathbf{N})-2\eth'}\mathbf{N}^{-\wp	n_0(\mathbf{N}-[\mathbf{N}/4])}  \le\ \lambda^{\mathbf{N}-2{n}_0(\mathfrak{N})-2\eth'}\mathbf{N}^{-\frac{\wp\mathbf{N}}{4}	}\mathbf{N}^{\wp\frac{\mathbf{N}-2{n}_0(\mathbf{N})}{2}}\
			 \lesssim \ \lambda^{-\eth'_1-\eth'_2}\mathbf{N}^{-\frac{\wp\mathbf{N}}{4}	}.$$
\end{proof}
\subsection{Fourth  type of diagram estimates}\label{Sec:FourthDiagram}

\begin{definition}\label{Def:Q2pair}
	We split
	\begin{equation}
		\label{Def:Q2pair:2}
		Q^2 \ = \ Q^{2,pair} \ + \ Q^{2,nonpair} 
	\end{equation}
	where
	\begin{equation}\label{Def:Q2pair:3}
		Q^{2,nonpair} 
		\ = \ \sum_{n=0}^{\mathfrak{N}-1}\sum_{S\in\mathcal{P}_{pair}(\{1,\cdots,n+2\})}\int_0^t\mathrm{d}s\mathcal{G}^{1}_{n}(S,s,t,k'',-1,k',1,\Gamma).
	\end{equation}
\end{definition}

\begin{proposition}[The first estimate on   ${Q}^2$]\label{lemma:BasicQ2Estimate1} Let $S$ be an arbitrary partition of $\{1,\cdots n+2\}$. If the corresponding graph is   non-singular, there is a constant $\mathfrak{C}_{Q_{2,1}}>0$ such that for $[\mathbf{N}/4]\le n\le \mathbf{N}$ and $t=\tau\lambda^{-2}>0$, we define
	\begin{equation}
		\begin{aligned} \label{eq:BasicQ2Estimate1:1}
			\tilde{Q}^2(\tau):=\	&\varsigma_{n}	
			\lambda^{n}\mathbf{1}({\sigma_{n,1}=-1})\mathbf{1}({\sigma_{n,2}=1})\Big[  \\
			& \sum_{\substack{\bar\sigma\in \{\pm1\}^{\mathcal{I}_n},\\ \sigma_{i,\rho_i}+\sigma_{i-1,\rho_i}+\sigma_{i-1,\rho_i+1}\ne \pm3,
					\\ \sigma_{i-1,\rho_i}\sigma_{i-1,\rho_i+1}= 1}}\Big[\int_0^{s_0}e^{-(s_0-s_0')\varsigma_n}\mathrm{d}s_0'\int_{(\Lambda^*)^{\mathcal{I}_n}}  \mathrm{d}\bar{k} \Delta_{n,\rho}(\bar{k},\bar\sigma)\\
			&\times  e^{s_0'\tau_0+{\bf i}s_0'\vartheta_0}\mathscr{W}(s_0')\prod_{i=1}^n\Big[\sigma_{i,\rho_i}\mathcal{M}( k_{i,\rho_i}, k_{i-1,\rho_i}, k_{i-1,\rho_i+1})\Phi_{1,i}(\sigma_{i-1,\rho_i}, k_{i-1,\rho_i},\sigma_{i-1,\rho_i+1}, k_{i-1,\rho_i+1})\Big] \\
			&\times\int_{(\mathbb{R}_+)^{\{0,1,\cdots,n\}}}\mathrm{d}\bar{s} \delta\left(t-\sum_{i=0}^ns_i\right)\prod_{i=0}^{n}e^{-s_i[\varsigma_{n-i}+\tau_i]}  \prod_{i=0}^{n} e^{-{\bf i}t_i(s)\mathbf{X}_i}\Big]\Big]\end{aligned}\end{equation}
	then for any constants $1>T_*>0$ and $\mathscr{Q}>0$,
	
	\begin{equation}
		\begin{aligned} \label{eq:BasicQ2Estimate1}
			&	\Big\|\limsup_{D\to\infty} \mathscr{F}_{1/n,\mathscr{R}}\Big[\widehat{\tilde{Q}^2\chi_{[0,T_*)}}(\mathscr{X}{\lambda^{-2}})\Big](k_{n,1},k_{n,2})\Big\|_{L^{\infty,\Im}(\mathbb{T}^{2d})}				
			\\
			\le\ &4^n e^{T_*}\frac{T_*^{1+{n}_0(n)-{n}_0(n-[\mathbf{N}/4])}}{({n}_0(n)-{n}_0(n-[\mathbf{N}/4]))!}\lambda^{\bowtie_\Im+n-2{n}_0(n)-\eth'_1-\eth'_2}\\
			&\times\mathbf{N}^{-\wp	n_0(n-[\mathbf{N}/4])}\mathfrak{C}_{Q_{2,1}}^{1+{n}_1(n)}\langle \ln n\rangle \langle\ln\lambda\rangle^{1+(2+\eth)({n}_1(n)-2)},
	\end{aligned}\end{equation}
	where the  time integration is defined to be $$\int_{(\mathbb{R}_+)^{\{0,\cdots,n\}}}\mathrm{d}\bar{s} \delta\left(t-\sum_{i=0}^ns_i\right)= \int_{(\mathbb{R}_+)^{\{0,\cdots,n\}}}\mathrm{d}{s}_1\cdots \mathrm{d}{s}_n \delta\left(t-\sum_{i=0}^ns_i\right),$$  $\mathscr{W}$ is either $\mathscr{W}_1$ or $\mathscr{W}_2$ and	and the same notations as in  Lemma \ref{lemma:BasicGEstimate1} have been used. 
\end{proposition}
\begin{proof}
	The proof of the Proposition is the same as  that of Propositions \ref{lemma:BasicGEstimate1}  and \ref{lemma:Q3FirstEstimate}. 
\end{proof}
\begin{proposition}[The main estimate on   ${Q}^2$]\label{lemma:Q2FirstEstimate}
	There is a constant $\mathfrak{C}_{Q_{2,2}}>0$ such that for $t=\tau\lambda^{-2}>0$, we have, for any constants $1>T_*>0$ and $\mathscr{R}>0$,
	
	\begin{equation}
		\begin{aligned} \label{lemma:Q2FirstEstimate:2}
			&
			\Big\|\limsup_{D\to\infty} \mathscr{F}_{1/n,\mathscr{R}}\Big[\widehat{{Q}^{2,nonpair}\chi_{[0,T_*)}}(\mathscr{X}{\lambda^{-2}})\Big](k_{n,1},k_{n,2})\Big\|_{L^{\infty,\Im}(\mathbb{T}^{2d})}
			\\
			\le\ & \sum_{n=[\mathbf{N}/4]}^{{\mathbf{N}}-1}n^n 4^n e^{T_*}\frac{T_*^{1+{n}_0(n)-{n}_0(n-[\mathbf{N}/4])-2}}{({n}_0(n)-{n}_0(n-[\mathbf{N}/4]))!}\lambda^{\bowtie_\Im+1-\eth'_1-\eth'_2} \mathbf{N}^{-\wp	n_0(n-[\mathbf{N}/4])}\mathfrak{C}_{Q_{2,2}}^{1+{n}_1(n)}\\
			&\times\langle \ln n\rangle \langle\ln\lambda\rangle^{1+(2+\eth)({n}_1(n)-2)},
	\end{aligned}\end{equation}
	which eventually tends to $0$ due to the choice of the parameters in Section \ref{Sec:KeyPara}. Note that  we have used the same notations with Proposition   \ref{lemma:BasicGEstimate1}.
\end{proposition}

\begin{proof} 
	The proof follows the same argument used in the homogeneous case.

\end{proof}

\subsection{Fifth  type of diagram estimates}\label{Sec:FifthDiagram}

\begin{proposition} \label{Propo:CrossingGraphs} Let $S$ be an arbitrary partition of $\mathcal{P}^o_{pair}(\{1,\cdots n+2\})$.  Suppose that the corresponding graph is long irreducible. There are constants $\mathfrak{C}_{Q_{LIC}},\mathfrak{C}'_{Q_{LIC}},\mathfrak{C}''_{Q_{LIC}},\mathfrak{C}'''_{Q_{LIC}}>0$ such that for $1\le n\le \mathbf{N}$ and $t=\tau\lambda^{-2}>0$, we set
	\begin{equation}
		\begin{aligned} \label{Propo:CrossingGraphs:1:0}
			{Q}_{LIC}(\tau):=	& \left[
			\lambda^{n}\left[\sum_{\bar\sigma\in \{\pm1\}^{\mathcal{I}_n}}\int_{(\Lambda^*)^{\mathcal{I}_n}}  \mathrm{d}\bar{k}\Delta_{n,\rho}(\bar{k},\bar\sigma) \right. \right.\\
			&\times   \mathbf{1}({\sigma_{n,1}=-1})\mathbf{1}({\sigma_{n,2}'=1})\mathscr{U}(0)  \int_{(\Lambda^*)^{\mathcal{I}_n'}}\mathrm{d}\bar{k}\Delta_{n,\rho}(\bar k,\bar\sigma)\tilde{\Phi}_0(\sigma_{0,\rho_1},  k_{0,\rho_1},\sigma_{0,\rho_2},  k_{0,\rho_2})\\
			&\times \sigma_{1,\rho_1}\mathcal{M}( k_{1,\rho_1}, k_{0,\rho_1}, k_{0,\rho_1+1})\prod_{i=2}^n\Big[ \sigma_{i,\rho_i}\mathcal{M}( k_{i,\rho_i}, k_{i-1,\rho_i}, k_{i-1,\rho_i+1})\\
			&\times \Phi_{1,i}(\sigma_{i-1,\rho_i},  k_{i-1,\rho_i},\sigma_{i-1,\rho_i+1},  k_{i-1,\rho_i+1})\Big] \\
			&\left.\left.\times\int_{(\mathbb{R}_+)^{\{0,\cdots,n\}}}\mathrm{d}\bar{s} \delta\left(t-\sum_{i=0}^ns_i\right)\prod_{i=1}^{n}e^{-s_i[\tau_{n-i}+\tau_i]}\prod_{i=1}^{n} e^{-{\bf i}t_i(s)\mathbf{X}_i}\right]\right]\end{aligned}\end{equation}
	then for any constants $1>T_*>0$ and $\mathscr{R}>0$,
	
	\begin{equation}
		\begin{aligned} \label{Propo:CrossingGraphs:1}
			&
			\Big\|\limsup_{D\to\infty} \mathscr{F}_{1/n,\mathscr{R}}\Big[\widehat{{Q}_{LIC}\chi_{[0,T_*)}}(\mathscr{X}{\lambda^{-2}})\Big](k_{n,1},k_{n,2})\Big\|_{L^{\infty,2}(\mathbb{T}^{2d})}
			\\
			\le\ & \mathfrak{C}'''_{Q_{LIC}} e^{T_*}\frac{T_*^{1+{n}_0(n)}}{({n}_0(n))!}\lambda^{\mathfrak{C}'_{Q_{LIC}}}\mathfrak{C}_{Q_{LIC}}^{1+{n}_1(n)}\langle \ln n\rangle \langle\ln\lambda\rangle^{2+2{n}_1(n)+\mathfrak{C}''_{Q_{LIC}}},
	\end{aligned}\end{equation}
	where we have used the same notation  as in Proposition \ref{lemma:BasicGEstimate1}.  The cut-off function $\tilde{\Phi}_0$ can be either $\Phi_{0,1}$ or $\Phi_{1,1}$.  If we replace $\mathscr{U}(0)$ by $\mathscr{W}_1(s_0)$ and  $\varsigma_n\int_0^{s_0}e^{-(s_0-s_0')\varsigma_n}\mathrm{d}s_0'\mathscr{W}_1(s_0')$, the same estimate holds true. 
\end{proposition}

\begin{proof} The computation is similar to that of the homogeneous case \cite{staffilani2021wave}. We therefore only highlight the key steps and the main differences. 
	As $\tilde{\Phi}_0$ is associated to the first interacting vertex $v_1$, which cannot be a degree-one vertex, otherwise, the edge in $\mathfrak{E}_+(v_1)$  is singular.   Therefore, we can   bound $\tilde{\Phi}_0$  by a constant in our estimates. By Lemma \ref{Lemma:VerticesLongCollisions}, as the graph contains at least one long irreducible cycle,  the set of all degree-one vertices at the top of those long irreducible cycles is denoted  by $\{v_{i}\}_{i\in I}$. We denote $i_1=\min\{i\in I\}$, thus,  in the  long irreducible $\mathrm{iC}_m^i$ cycle of $v_{i_1}$, there  exists at least one vertex  $v_{i}$ such that $\mathfrak{X}_{i}+\mathfrak{X}_{i_1}$ is a function of the free edge  of $v_{i_1}$.  The largest among those indices is then denoted by $i_0$.  For the sake of simplicity, we denote by $\tilde{\Phi}_{1,i}$   the cut-off function associated to the vertex $v_i$. 
	We denote   the momenta of the edges in $\mathfrak{E}(v_{i_1})$ as $k_0,k_1,k_2$, in which $k_1,k_2\in\mathfrak{E}_-(v_{i_1})$  and $k_1$ is the free momenta. The associated  signs are denoted by $\sigma_{k_0},\sigma_{k_1},\sigma_{k_2}$. We then have
$$
		\mathrm{Re}\vartheta(i_0-1)  =  \mathbf{X}_{i_1}  +  \mathbf{X}_{i_0}  +  \sum_{i>i_1}\mathbf{X}_i,
	$$
	for $i_1=i_0+1$. When  $i_1>i_0+1$, we obtain	
$$
		\mathrm{Re}\vartheta(i_0-1) \ = \ \mathbf{X}_{i_1} \ + \ \mathbf{X}_{i_0} \ + \  \sum_{i=i_0+1}^{i_1-1}\mathbf{X}_i \ + \ \sum_{i>i_1}\mathbf{X}_i.$$ 
	The three edges associated to $v_{i_0}$ are denoted by $k_0'\in\mathfrak{E}_+(v_{i_0})$ and $k_1',k_2'\in\mathfrak{E}_-(v_{i_0})$. They are equipped with the signs $\sigma_{k_0'},\sigma_{k_1'},\sigma_{k_2'}$. We define $\mathfrak{I}=\Big\{i\in \{i_0+1,\cdots,i_1-1\} \Big| \mathrm{deg}v_i=1\Big\},$
	and $\mathfrak{I}'=\Big\{i\in \{i_0+1,\cdots,i_1-1\}, i+1\notin \mathfrak{I} \Big| \mathrm{deg}v_i=0\Big\},$
	then 
	$$\vartheta_*=\sum_{i\in \mathbf{I}}[\mathbf{X}_i+\mathbf{X}_{i-1}]+\sum_{i\in \mathfrak{I}'}\mathbf{X}_i,$$
	and
	$$\vartheta_{**}=\sum_{i>i_1}\mathfrak{X}_i,$$
	are both independent of $k_1$.  
	As a result,  
$$
		\mathrm{Re}\vartheta(i_0-1)  = \  \mathfrak{X}_{i_1} + \mathbf{X}_{i_0}+\vartheta_* + \vartheta_{**},\ \  \mathrm{Re}\vartheta(i_1-1)  =   \mathbf{X}_{i_1} + \vartheta_{**}.
$$ We can now rewrite  
 $$
		\mathbf{X}_{i_0} =  \alpha_1\omega(k_1+ \tilde\varkappa_1)+\alpha_2\omega(k_1+ \tilde\varkappa_2) + \omega',
$$
where $\alpha_1,\alpha_2\in\{\pm1\}$,  $\tilde\varkappa_1,\tilde\varkappa_2$, $\omega'$ are independent of $k_1$. In addition,
$$
			\mathbf{X}_{i_1} =  \sigma_{k_{0}}\omega(k_{0})+\sigma_{k_{1}}\omega(k_1)+\sigma_{k_{2}}\omega(k_2)\
			 =  \alpha_3\omega(k_1+ \tilde\varkappa_3)+\alpha_4\omega(k_1+ \tilde\varkappa_4) + \omega'',
	$$
	where $\alpha_3,\alpha_4\in\{\pm1\}$,  $\tilde\varkappa_3,\tilde\varkappa_4$, $\omega''$ are independent of $k_1$. As one of the two momenta $\tilde\varkappa_3,\tilde\varkappa_4$ should be $0$, $k_1+ \tilde\varkappa_1$ or $k_1+ \tilde\varkappa_2$ should  coincide or pair with either the momentum $k_1+ \tilde\varkappa_3$ or the momentum $k_1+ \tilde\varkappa_4$ of the vertex $v_{l_1}$, we suppose  that  $k_1+ \tilde\varkappa_1$  coincides or pairs with $k_1+ \tilde\varkappa_4$. If $k_1+ \tilde\varkappa_1$  coincides with $k_1+ \tilde\varkappa_4$, we have $\tilde\varkappa_4=\tilde\varkappa_1$. Otherwise, we have $\tilde\varkappa_4-\tilde\varkappa_1=\tilde{k}\lambda^2/|\ln|\lambda||$, in which the momentum $\tilde{k}$ is fixed (see Figure \ref{Fig42}).  In both cases, we can always write $\tilde\varkappa_4-\tilde\varkappa_1=\tilde{k}\lambda^2/|\ln|\lambda||$ as the former case corresponds to $\tilde{k}=0$. We now have  $$
			\mathbf{X}_{i_1}
			 = \alpha_3\omega(k_1+ \tilde\varkappa_3)+\alpha_4\omega(k_1+ \tilde\varkappa_1 +\tilde{k}\lambda^2/|\ln|\lambda||) + \omega''.
		$$

	\begin{figure}
	\centering
	\includegraphics[width=.99\linewidth]{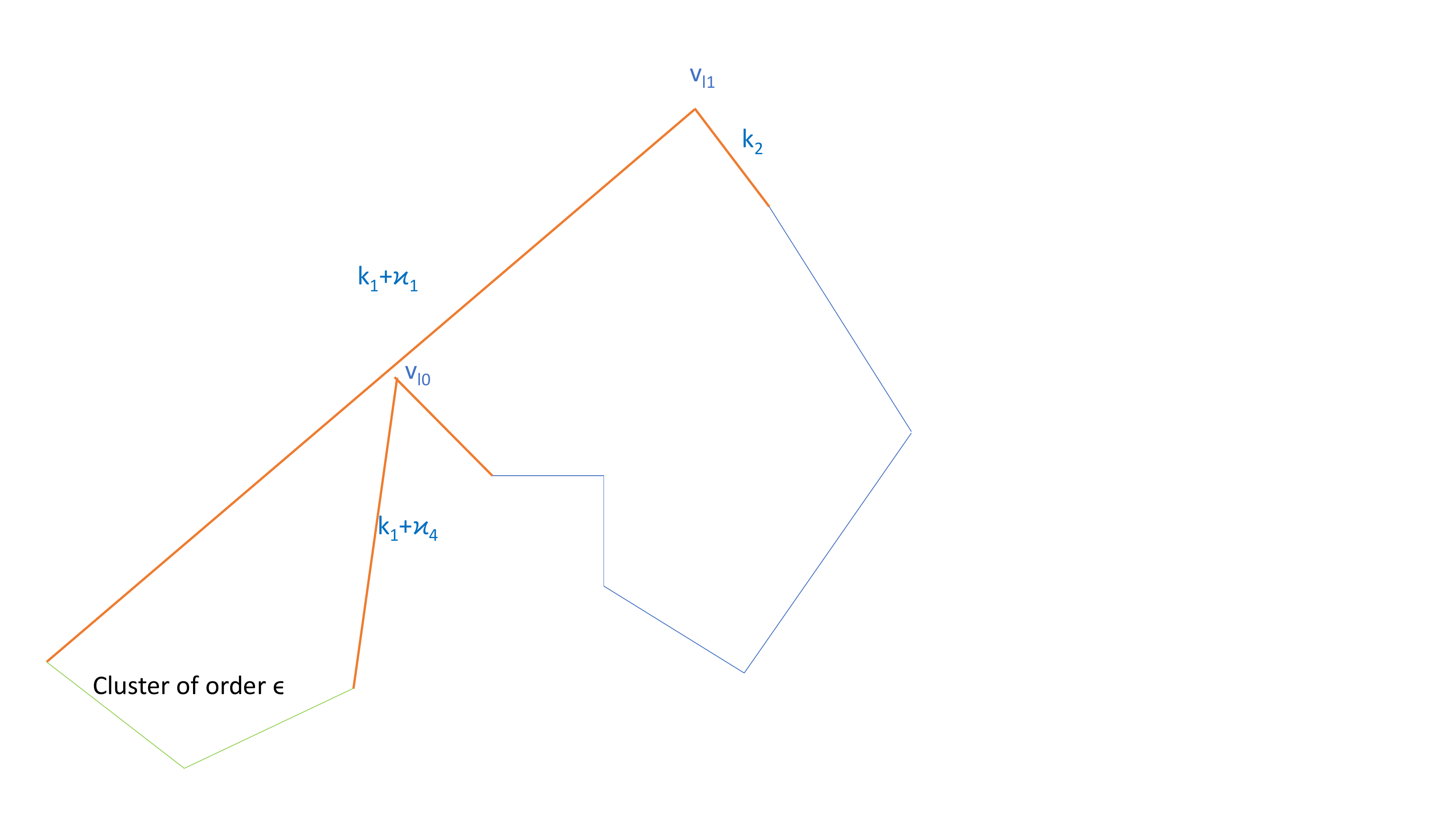}
	\caption{Illustration: $k_1+ \tilde\varkappa_1$ should pair with   $k_1+ \tilde\varkappa_4$.}
	\label{Fig42}
\end{figure}
	
	Now, we follow the strategy of Proposition \ref{lemma:BasicGEstimate1} by splitting the set $\{0,1,\cdots,n\}$ into smaller subsets. One set contains time slice indices $0\le i<n$ such that $\mathrm{deg}(v_{i+1})=0$, excluding $i_0-1$. This set is denoted by $I''$. The other set involves all of  time slice indices $0\le i<n$ such that $\mathrm{deg}(v_{i+1})=1$ and the extra index $n$. This set is called $I=\{l_0-1\}\cup I'$. We  obtain, following the same lines of computations as \eqref{eq:Aestimate1}
	\begin{equation}
		\begin{aligned}\label{eq:crossingraph4}
			& \left|\int_{(\mathbb{R}_+)^{I}}\mathrm{d}\vec{s''}\delta\left(s_{n+1}-\sum_{i\in I}s_i\right)\prod_{i\in I}e^{-{\bf i}s_i \vartheta_i}e^{-{\bf i} s_{n+1}2\pi\mathscr{X}}   \right|\\
			=	&  \oint_{\Gamma_{n}}\frac{|\mathrm{d}\xi|}{2\pi}\left|e^{-{\bf i}s_{n+1}(\xi+{\bf i}\lambda^2)}\right|\frac{1}{\sqrt{\xi^2+\lambda^4}}\prod_{i\in\{l_0-1\}\cup I'\backslash\{n\}}\frac{\tilde\Phi_{1,i+1}}{|\xi-(\mathrm{Re}\vartheta_i+2\pi\mathscr{X}-{\bf i} \lambda^2)|}\\
			\quad \le\ &   e^{\tau}  \oint_{\Gamma_{n}}\frac{|\mathrm{d}\xi|}{2\pi}\frac{1}{|\xi|}\prod_{i\in\{l_0-1\}\cup I'\backslash\{n\}}\frac{\tilde\Phi_{1,i+1}}{|\xi-(\mathrm{Re}\vartheta_i+2\pi \mathscr{X}-{\bf i} \lambda^2)|}. 
	\end{aligned}\end{equation}
	We also have  an inequality similar to \eqref{eq:Aestimate2a}, with $\mathscr{Y}=\emptyset$ 
	\begin{equation}
		\begin{aligned}
			& \int_{(\mathbb{R}_+)^{I''\cup\{n+1\}}}\mathrm{d}\bar{s}'\delta\left(t-\sum_{i\in I''\cup\{n+1\}}s_i\right)\prod_{i\in I''}e^{-\varsigma_{n-i}s_i}\\
			\quad \le\ &   \int_{(\mathbb{R}_+)^{I''\cup\{n+1\}}}\mathrm{d}\bar{s}'\delta\left(t-\sum_{i\in I''\cup\{n+1\}}s_i\right)\ \le\   \frac{t^{\frac{n}{2}-1}}{(\frac{n}{2}-1)!} \ = \ \lambda^{-2(\frac{n}{2}-1)} \frac{\tau^{\frac{n}{2}-1}}{(\frac{n}{2}-1)!}, 
	\end{aligned}\end{equation} 
	where the notation $	\bar{s}'$ denotes the vector whose components are $s_i$ with $i\in I''\cup\{n+1\}$. Thus, 
	 the same strategy as in  the proof of Proposition \ref{lemma:BasicGEstimate1}, can be reused with a modification: the terms that are associated to $v_{i_0},v_{i_1}$  are grouped and estimated separately.  
	We now estimate the quantities that are associated to $v_{i_0}$ and $v_{i_1}$, which are those that contain $\vartheta_i$ with $i=i_0-1, i_1-1$ respectively
	\begin{equation}
		\begin{aligned}\label{eq:crossingraph6}
			&\int_{\Lambda^{*}}\mathrm{d}k_1
			\frac{\tilde{\Phi}_{1,i_0}(v_{i_0})|H(k_1)|}{\Big|\xi-\mathrm{Re}\vartheta_{i_0-1}+2\pi \mathscr{X}-{\bf i} \lambda^2\Big|\Big|\xi-\mathrm{Re}\vartheta_{i_1-1}+2\pi\mathscr{X}-{\bf i} \lambda^2\Big|}\\
			\
			= & \ \int_{\Lambda^{*}}\mathrm{d}k_1
			\frac{\tilde{\Phi}_{1,i_0}(v_{l_0})|H(k_1)|}{\Big|\xi-\mathbf{X}_{i_0}-\mathbf{X}_{i_1}-\vartheta_*-\vartheta_{**}+2\pi \mathscr{X}-{\bf i} \lambda^2\Big|\Big|\xi-\mathbf{X}_{i_1}-\vartheta_{**}+2\pi \mathscr{X}-{\bf i} \lambda^2\Big|},
		\end{aligned}
	\end{equation}
	where $H\in L^2(\mathbb{T}^d)$.
	Similarly as  the proof of Proposition \ref{lemma:BasicGEstimate1},  $\Lambda^{*}$ can be replaced by $\mathbb{T}^d$ in the limit $D\to\infty$ and the cut-off functions on $\mathbb{T}^d$ are then $\tilde{\Phi}_{1,i_0}(v_{i_0})$. 
		We now follow the framework of the homogeneous case and  bound (see \cite{staffilani2021wave}[Proposition 89 ])
	\begin{equation}
		\begin{aligned}\label{eq:crossingraph8:1}
			&\left\|\int_{\mathbb{T}^{d}}\mathrm{d}k_1
			\frac{\tilde{\Phi}_{1,i_0}(v_{i_0})|H(k_1)|}{\Big|\xi-\mathbf{X}_{i_0}-\mathbf{X}_{i_1}-\vartheta_*-\vartheta_{**}+2\pi \mathscr{X}-{\bf i} \lambda^2\Big|\Big|\xi-\mathbf{X}_{i_1}-\vartheta_{**}+2\pi \mathscr{X}-{\bf i} \lambda^2\Big|}\right\|_{L^2(-\mathscr{R},\mathscr{R})}\\
			\lesssim\	&  \|H\|_{L^2}\lambda^{-\frac2p}\left[\iint_{\mathbb{R}^2}\mathrm{d}r_1\mathrm{d}r_2e^{-q\lambda^2|r_1|^2/2}e^{-q\lambda^2|r_2|^2/2}\right.\\
			&\times\left. \left|\int_{\mathbb{T}^{d}}\mathrm{d}k_1 |\tilde{\Phi}_{1,i_0}(v_{i_0})|^2e^{-{\bf i}r_1(\mathbf{X}_{i_1}+\vartheta_{**}) -{\bf i}r_2(\mathbf{X}_{i_0}+\mathbf{X}_{i_1}+\vartheta_{*}+\vartheta_{**})}\right|^\frac{q}{2}\right]^\frac1q,	\end{aligned}
	\end{equation}
	where, $q>1$ is a constant to be determined  and satisfies $\frac1p+\frac1q=1$. Performing the change of variable $r_1+r_2\to r_1$, we bound
	\begin{equation}
		\begin{aligned}\label{eq:crossingraph8}
			&\left\|\int_{\mathbb{T}^{d}}\mathrm{d}k_1
			\frac{\tilde{\Phi}_{1,i_0}(v_{i_0})|H(k_1)|}{\Big|\xi-\mathbf{X}_{i_0}-\mathbf{X}_{i_1}-\vartheta_*-\vartheta_{**}+2\pi \mathscr{X}-{\bf i} \lambda^2\Big|\Big|\xi-\mathbf{X}_{i_1}-\vartheta_{**}+2\pi \mathscr{X}-{\bf i} \lambda^2\Big|}\right\|_{L^2(-\mathscr{R},\mathscr{R})}\\
	 \lesssim  \		&\|H\|_{L^2}\lambda^{-\frac2p} \left[\iint_{\mathbb{R}^2}\mathrm{d}r_1\mathrm{d}r_2e^{-q\lambda^2|r_1-r_2|^2/2}e^{-q\lambda^2|r_2|^2/2}\right.\\
			&\times\left.\left|\int_{\mathbb{T}^{d}}\mathrm{d}k_1| \tilde{\Phi}_{1,i_0}(v_{i_0})|^2e^{-{\bf i}r_1\mathbf{X}_{i_1} -{\bf i}r_2\mathbf{X}_{i_0}-{\bf i}(r_2\vartheta_*+r_1\vartheta_{**})}\right|^\frac{q}{2}\right]^\frac1q.
		\end{aligned}
	\end{equation}
	
We now  bound all the terms that are independent of $k_1$, including $\vartheta^*$, simply by $1$, we obtain 
	\begin{equation}
		\begin{aligned}\label{eq:crossingraph11}
			&\left|\int_{\mathbb{T}^{d}}\mathrm{d}k_1 |\tilde{\Phi}_{1,i_0}(v_{i_0})|^2e^{-{\bf i}r_1\mathbf{X}_{i_1}-{\bf i}r_2\mathbf{X}_{i_0}-{\bf i}\vartheta^*}\right|\\
			= \  & \left|\int_{\mathbb{T}^{d}}\mathrm{d}k_1 |\tilde{\Phi}_{1,i_0}(v_{i_0})|^2e^{-{\bf i}r_1 (\alpha_3\omega(k_1+ \tilde\varkappa_3)+\alpha_4\omega(k_1+ \tilde\varkappa_1+\tilde{k}\lambda^2/|\ln|\lambda||) + \omega'')} e^{-{\bf i}r_2(\alpha_1\omega(k_1+ \tilde\varkappa_1)+\alpha_2\omega(k_1+ \tilde\varkappa_2) + \omega')}e^{-{\bf i}\vartheta^*}\right|\\
			\lesssim \  & \left|\int_{\mathbb{T}^{d}}\mathrm{d}k_1 |\tilde{\Phi}_{1,i_0}(v_{i_0})|^2e^{-{\bf i}r_1 (\alpha_3\omega(k_1+ \tilde\varkappa_3)+\alpha_4\omega(k_1+ \tilde\varkappa_1+\tilde{k}\lambda^2/|\ln|\lambda||))}e^{-{\bf i}r_2(\alpha_1\omega(k_1+ \tilde\varkappa_1)+\alpha_2\omega(k_1+ \tilde\varkappa_2))}\right|\\
			\lesssim \  & \left|\int_{\mathbb{T}^{d}}\mathrm{d}k_1 |\tilde{\Phi}_{1,i_0}(v_{i_0})|^2e^{-{\bf i}r_1 \alpha_3\omega(k_1+ \tilde\varkappa_3)}e^{-{\bf i}(r_1 \alpha_4+r_2\alpha_1)\omega(k_1+ \tilde\varkappa_1)} \right.\\
			&\ \ \ \ \ \ \  \times \left. e^{-{\bf i}r_2\alpha_2\omega(k_1+ \tilde\varkappa_2)}e^{{\bf i}r_1 \alpha_4[\omega(k_1+ \tilde\varkappa_1)-\omega(k_1+ \tilde\varkappa_1+\tilde{k}\lambda^2/|\ln|\lambda||)]}\right|.
		\end{aligned}
	\end{equation}
	The right hand side of \eqref{eq:crossingraph11} can be expressed as
	\begin{equation}
		\begin{aligned}\label{eq:crossingraph12}
			& \Big|\int_{\mathbb{T}^{d}}\mathrm{d}k_1 e^{-{\bf i}r_1 \alpha_3\omega(k_1+ \tilde\varkappa_3)-{\bf i}(r_1 \alpha_4+r_2\alpha_1)\omega(k_1+ \tilde\varkappa_1)-{\bf i}r_2\alpha_2\omega(k_1+ \tilde\varkappa_2)}e^{{\bf i}r_1 \alpha_4[\omega(k_1+ \tilde\varkappa_1)-\omega(k_1+ \tilde\varkappa_1+\tilde{k}\lambda^2/|\ln|\lambda||)]}\\
			&\ \times |\tilde{\Phi}_{1}^a(k_1+ \tilde\varkappa_2)|^2\Big|\\
			= \ & \left|\int_{\mathbb{T}^{d}}\mathrm{d}k_1 e^{-{\bf i}r_1 \alpha_3\omega(k_1+ \tilde\varkappa_3)-{\bf i}(r_1 \alpha_4+r_2\alpha_1)\omega(k_1+ \tilde\varkappa_1)-{\bf i}r_2\alpha_2\omega(k_1+ \tilde\varkappa_2)}e^{{\bf i}r_1 \alpha_4[\omega(k_1+ \tilde\varkappa_1)-\omega(k_1+ \tilde\varkappa_1+\tilde{k}\lambda^2/|\ln|\lambda||)]}\right.
			\\
			&\times \left.\left[\sum_{m_1,m_2\in\mathbb{Z}^d}g(m_1,m_2)e^{{\bf i}2\pi [m_1\cdot k_1+m_2\cdot(k_1+ \tilde\varkappa_2)]}\right]|\tilde{\Phi}_{1}^a(k_1+ \tilde\varkappa_2)|^2\right|,
		\end{aligned}
	\end{equation}
	in which $g$ is the  function $\mathbf{1}_{m_1=0}\mathbf{1}_{m_2=0}$. 
	We now distribute $e^{{\bf i}2\pi [m_1\cdot k_1+m_2\cdot(k_1+ \tilde\varkappa_2)]}$ into the terms inside the integral of $k_1$ and obtain
	\begin{equation}
		\begin{aligned}\label{eq:crossingraph13}
			& \Big|\int_{\mathbb{T}^{d}}\mathrm{d}k_1 e^{-{\bf i}r_1 \alpha_3\omega(k_1+ \tilde\varkappa_3)-{\bf i}(r_1 \alpha_4+r_2\alpha_1)\omega(k_1+ \tilde\varkappa_1)-{\bf i}r_2\alpha_2\omega(k_1+ \tilde\varkappa_2)}|\tilde{\Phi}_{1}^a(k_1+ \tilde\varkappa_2)|^2
			\\
			&\ \ \ \ \ \ \  \times \left. e^{{\bf i}r_1 \alpha_4[\omega(k_1+ \tilde\varkappa_1)-\omega(k_1+ \tilde\varkappa_1+\tilde{k}\lambda^2/|\ln|\lambda||)]}\right|\\
			= \ & \Big|
			\int_{\mathbb{T}^{d}}\mathrm{d}k_1 \sum_{m_1,m_2\in\mathbb{Z}^d}g(m_1,m_2)e^{-{\bf i}r_1 \alpha_3\omega(k_1+ \tilde\varkappa_3)-{\bf i}(r_1 \alpha_4+r_2\alpha_1)\omega(k_1+ \tilde\varkappa_1)-{\bf i}r_2\alpha_2\omega(k_1+ \tilde\varkappa_2)+{\bf i}2\pi m_1\cdot k_1}\\
			&\times |\tilde{\Phi}_{1}^a(k_1+ \tilde\varkappa_2)|^2e^{{\bf i}2\pi m_2\cdot(k_1+ \tilde\varkappa_2)}e^{{\bf i}r_1 \alpha_4[\omega(k_1+ \tilde\varkappa_1)-\omega(k_1+ \tilde\varkappa_1+\tilde{k}\lambda^2/|\ln|\lambda||)]}\Big|\\
			= \ & \Big|
			\int_{\mathbb{T}^{2d}}\mathrm{d}k_1\mathrm{d}k_*\delta(k_*-k-\tilde{\varkappa}_2) \sum_{m_1,m_2\in\mathbb{Z}^d}g(m_1,m_2){\tilde{\Phi}_{1}^a}(k+\tilde\varkappa_2){\tilde{\Phi}_{1}^a}(k_*)e^{{\bf i}2\pi m_2 \cdot k_*}\\
			&\times e^{{\bf i}r_1 \alpha_4[\omega(k_*+ \tilde\varkappa_1-\tilde\varkappa_2)-\omega(k_*+ \tilde\varkappa_1+\tilde{k}\lambda^2/|\ln|\lambda||-\tilde\varkappa_2)]}\\
			&\times  e^{-{\bf i}r_1 \alpha_3\omega(k_1+ \tilde\varkappa_3)-{\bf i}(r_1 \alpha_4+r_2\alpha_1)\omega(k_1+ \tilde\varkappa_1)-{\bf i}r_2\alpha_2\omega(k_1+ \tilde\varkappa_2)+{\bf i}2\pi m_1\cdot k_1}\Big|
			.
		\end{aligned}
	\end{equation}
	Using the identity
	$\sum_{y\in\mathbb{Z}^d}e^{{\bf i}2\pi ({k}_{*}-k_1-\tilde{\varkappa}_2)\cdot y}=\delta({k}_{*}-k_1-\tilde{\varkappa}_2),$
	we rewrite  \eqref{eq:crossingraph13} as 
	\begin{equation}
		\begin{aligned}\label{eq:crossingraph13:1}
			& \Big|\int_{\mathbb{T}^{d}}\mathrm{d}k_1 e^{-{\bf i}r_1 \alpha_3\omega(k_1+ \tilde\varkappa_3)-{\bf i}(r_1 \alpha_4+r_2\alpha_1)\omega(k_1+ \tilde\varkappa_1)-{\bf i}r_2\alpha_2\omega(k_1+ \tilde\varkappa_2)}|\tilde{\Phi}_{1}^a(k_1+ \tilde\varkappa_2)|^2\Big|\\
			\\
			&\ \ \ \ \ \ \  \times \left. e^{{\bf i}r_1 \alpha_4[\omega(k_1+ \tilde\varkappa_1)-\omega(k_1+ \tilde\varkappa_1+\tilde{k}\lambda^2/|\ln|\lambda||)]}\right|\\
			= \ & \Big|
			\int_{\mathbb{T}^{2d}}\mathrm{d}k_1\mathrm{d}k_*\sum_{y\in\mathbb{Z}^d}e^{{\bf i}2\pi ({k}_{*}-k_1-\tilde{\varkappa}_2)\cdot y} \sum_{m_1,m_2\in\mathbb{Z}^d}g(m_1,m_2)\tilde{\Phi}_{1}^a(k_1+ \tilde\varkappa_2)\\
			\\
			&\times e^{{\bf i}r_1 \alpha_4[\omega(k_*+ \tilde\varkappa_1-\tilde\varkappa_2)-\omega(k_*+ \tilde\varkappa_1+\tilde{k}\lambda^2/|\ln|\lambda||-\tilde\varkappa_2)]}\\
			&\times  e^{-{\bf i}r_1 \alpha_3\omega(k_1+ \tilde\varkappa_3)-{\bf i}(r_1 \alpha_4+r_2\alpha_1)\omega(k_1+ \tilde\varkappa_1)-{\bf i}r_2\alpha_2\omega(k_1+ \tilde\varkappa_2)+{\bf i}2\pi m_1\cdot k_1}\tilde{\Phi}_{1}^a(k_*)e^{{\bf i}2\pi m_2\cdot k_*}\Big|\\
			= \ & \Big|
			\int_{\mathbb{T}^{2d}}\mathrm{d}k_1\mathrm{d}k_*\sum_{y\in\mathbb{Z}^d}e^{-{\bf i}2\pi \tilde{\varkappa}_2\cdot y} \sum_{m_1,m_2\in\mathbb{Z}^d}g(m_1,m_2)\tilde{\Phi}_{1}^a(k_1+ \tilde\varkappa_2)\\
			\\
			&\times e^{{\bf i}r_1 \alpha_4[\omega(k_*+ \tilde\varkappa_1-\tilde\varkappa_2)-\omega(k_*+ \tilde\varkappa_1+\tilde{k}\lambda^2/|\ln|\lambda||-\tilde\varkappa_2)]}\\
			&\times  e^{-{\bf i}r_1 \alpha_3\omega(k_1+ \tilde\varkappa_3)-{\bf i}(r_1 \alpha_4+r_2\alpha_1)\omega(k_1+ \tilde\varkappa_1)-{\bf i}r_2\alpha_2\omega(k_1+ \tilde\varkappa_2)+{\bf i}2\pi (m_1-y)\cdot k_1}\tilde{\Phi}_{1}^a(k_*)e^{{\bf i}2\pi (m_2+y)\cdot k_*}\Big|\\
			= \ & \left|
			\sum_{m_1,m_2\in\mathbb{Z}^d}g(m_1,m_2)\sum_{y\in\mathbb{Z}^d}e^{-{\bf i}2\pi \tilde{\varkappa}_2\cdot y}\mathfrak{H}_1(m_1-y)\mathfrak{H}_2(m_2+y)\right|
			,
		\end{aligned}
	\end{equation}	 
	
	in which
	\begin{equation}
		\begin{aligned}\label{eq:crossingraph14}
			\mathfrak{H}_1(m)= \  & \int_{\mathbb{T}^{d}}\mathrm{d}k_1e^{-{\bf i}r_1 \alpha_3\omega(k_1+ \tilde\varkappa_3)-{\bf i}(r_1 \alpha_4+r_2\alpha_1)\omega(k_1+ \tilde\varkappa_1)-{\bf i}r_2\alpha_2\omega(k_1+ \tilde\varkappa_2)+{\bf i}2\pi m\cdot k_1}{\tilde{\Phi}_{1}^a}(k_1+ \tilde\varkappa_2),\\
			\mathfrak{H}_2(m)= \  & \int_{\mathbb{T}^{d}}\mathrm{d}{k}_{*}{\tilde{\Phi}_{1}^a}(k_*)e^{{\bf i}2\pi m\cdot k_*}  e^{{\bf i}r_1 \alpha_4[\omega(k_*+ \tilde\varkappa_1-\tilde\varkappa_2)-\omega(k_*+ \tilde\varkappa_1+\tilde{k}\lambda^2/|\ln|\lambda||-\tilde\varkappa_2)]}.
		\end{aligned}
	\end{equation}
	By H\"oder's inequality, applied to the right hand side of \eqref{eq:crossingraph13}, we bound
	\begin{equation}
		\begin{aligned}\label{eq:crossingraph15}&\Big|\int_{\mathbb{T}^{d}}\mathrm{d}k_1 e^{-{\bf i}r_1 \alpha_3\omega(k_1+ \tilde\varkappa_3)-{\bf i}(r_1 \alpha_4+r_2\alpha_1)\omega(k_1+ \tilde\varkappa_1)-{\bf i}r_2\alpha_2\omega(k_1+ \tilde\varkappa_2)}|\tilde{\Phi}_{1}^a(k_1+ \tilde\varkappa_2)|^2
			\\
			&\ \ \ \ \ \ \  \times \left. e^{{\bf i}r_1 \alpha_4[\omega(k_1+ \tilde\varkappa_1)-\omega(k_1+ \tilde\varkappa_1+\tilde{k}\lambda^2/|\ln|\lambda||)]}\right|\\
			\lesssim \  & \left|\sum_{m_1,m_2\in\mathbb{Z}^d}g(m_1,m_2)\right|\big\|\mathfrak{H}_1\big\|_4\big\|\mathfrak{H}_2\big\|_\frac43\lesssim \   \big\|\mathfrak{H}_1\big\|_4\big\|\mathfrak{H}_2\big\|_\frac43.
		\end{aligned}
	\end{equation}
We  now bound $\|\mathfrak{H}_2\big\|_\frac43$ as follows
\begin{equation}
	\begin{aligned}\label{eq:crossingraph14:1}
&		\|\mathfrak{H}_2\big\|_\frac43\ = \  \\
\ = \  &\left\{\sum_{m\in\mathbb{Z}^d} \left|\int_{\mathbb{T}^{d}}\mathrm{d}{k}_{*}{\tilde{\Phi}_{1}^a}(k_*)e^{{\bf i}2\pi m\cdot k_*}e^{{\bf i}r_1 \alpha_4[\omega(k_*+ \tilde\varkappa_1-\tilde\varkappa_2)-\omega(k_*+ \tilde\varkappa_1+\tilde{k}\lambda^2/|\ln|\lambda||-\tilde\varkappa_2)]}\right|^\frac43\right\}^\frac34\\
		\ = \ & \left\{\sum_{m\in\mathbb{Z}^d} \left|\int_{\mathbb{T}^{d}}\mathrm{d}{k}_{*}\frac{{\tilde{\Phi}_{1}^a}(k_*)e^{{\bf i}r_1 \alpha_4[\omega(k_*+ \tilde\varkappa_1-\tilde\varkappa_2)-\omega(k_*+ \tilde\varkappa_1+\tilde{k}\lambda^2/|\ln|\lambda||-\tilde\varkappa_2)]}}{|{\bf i}2\pi m|^{2d}}\Delta^{d}\Big(e^{{\bf i}2\pi m\cdot k_*}\Big)\right|^\frac43\right\}^\frac34\\
		\ = \ & \left\{\sum_{m\in\mathbb{Z}^d} 
		\frac{1}{|2\pi m|^{8d/3}}\left|\int_{\mathbb{T}^{d}}\mathrm{d}{k}_{*}\Delta^{d}\Big[{{\tilde{\Phi}_{1}^a}(k_*)}e^{{\bf i}r_1 \alpha_4[\omega(k_*+ \tilde\varkappa_1-\tilde\varkappa_2)-\omega(k_*+ \tilde\varkappa_1+\tilde{k}\lambda^2/|\ln|\lambda||-\tilde\varkappa_2)]}\Big]e^{{\bf i}2\pi m\cdot k_*}\right|^\frac43\right\}^\frac34.
	\end{aligned}
\end{equation}
By the mean value theorem, we write $$\omega(k_*+ \tilde\varkappa_1-\tilde\varkappa_2)-\omega(k_*+ \tilde\varkappa_1+\tilde{k}\lambda^2/|\ln\lambda|-\tilde\varkappa_2)=-\tilde{k}\lambda^2/|\ln\lambda|\int_0^1\mathrm{d}s\nabla_{\tilde{k}}\omega(k_*+ \tilde\varkappa_1+s\tilde{k}\lambda^2/|\ln\lambda|-\tilde\varkappa_2),$$ yielding

\begin{equation}
	\begin{aligned}\label{eq:crossingraph14:1:1}
		\|\mathfrak{H}_2\big\|_\frac43  
		\ \lesssim \ & \left\{\sum_{m\in\mathbb{Z}^d} 
		\frac{1}{|2\pi m|^{8d/3}}\right\}^\frac34\left|\int_{\mathbb{T}^{d}}\mathrm{d}{k}_{*}\Delta^{d}\Big[{{{\Phi}_{1}^a}(k_*)}e^{-{\bf i}r_1\lambda^2\alpha_4\tilde{k}/|\ln\lambda|\int_0^1\mathrm{d}s\nabla\omega(k_*+ \tilde\varkappa_1+s\tilde{k}\lambda^2/|\ln\lambda|-\tilde\varkappa_2)}\Big]\right|\\
		 \ \lesssim \ & \langle\ln|\lambda|\rangle^{\mathfrak{C}_{\mathfrak{H}_2}}\big\langle r_1\lambda^2\big\rangle^{2d}.
	\end{aligned}
\end{equation}
with $\mathfrak{C}_{\mathfrak{H}_2}>0$, where $\big\langle r_1\lambda^2\big\rangle^{2d}$ comes out due to the action of $\Delta^{d}$ on the oscillatory integral and the effect of $\Delta^{d}$ on the cut-off function $\Phi_1^a$ simply gives a factor of $\langle\ln|\lambda|\rangle$. 
	
	The same strategy of the homogeneous case can be redone, with an extra factor of $\big\langle r_1\lambda^2\big\rangle^{2d}$ in all estimates, and we finally obtain
	\begin{equation}
		\begin{aligned}\label{eq:crossingraph17:d}
			&\left\|\int_{\mathbb{T}^{d}}\mathrm{d}k_1
			\frac{\tilde{\Phi}_{1,l_0}(v_{l_0})|H(k_1)|}{\Big|\xi-\mathfrak{X}_{l_0}-\mathfrak{X}_{l_1}-\vartheta_*-\vartheta_{**}+2\pi\mathscr{X}-{\bf i} \lambda^2\Big|\Big|\xi-\mathfrak{X}_{l_1}-\vartheta_{**}+2\pi \mathscr{X}-{\bf i} \lambda^2\Big|}\right\|_{L^2(-\mathscr{R},\mathscr{R})}\\
			&	\lesssim  \|H\|_{L^2}\langle\ln\lambda\rangle^{{C}_{\aleph}}\lambda^{-2/q+2\bar{\epsilon}/q}\lambda^{-2/p}\mathfrak{H}_0	\lesssim  \|H\|_{L^2}\langle\ln\lambda\rangle^{{C}_{\aleph}}\lambda^{-2+2\bar{\epsilon}/q}\mathfrak{H}_0,
		\end{aligned}
	\end{equation}
	for some constant  ${C}_{\aleph}$, $\bar{\epsilon}>0$, $p,q>0$, $\frac1p+\frac1q=1$, and $\mathfrak{H}_0$ is a constant depending on the cut-off functions. 
	In this estimate, we obtain a factor of $\lambda^{-2+\bar{\epsilon}/q}$. From \eqref{eq:crossingraph4}, we have a factor of $\lambda^{-2(\frac{n}{2}-1)}$. These two factors give the total factor $\lambda^{-2(\frac{n}{2}-1)}\lambda^{-2+\bar{\epsilon}/q}=\lambda^{-n}\lambda^{2\bar{\epsilon}/q}.$ 
	By the same argument with the homogeneous case \cite{staffilani2021wave},   we gain a factor of $\lambda^{\mathfrak{C}'_{Q_{Long}}}$ with $\mathfrak{C}'_{Q_{Long}}>0$. 
	Thus the total $\lambda$ power   guarantees the convergence of the whole graph to $0$.
\end{proof}


\subsection{Sixth  type of diagram estimates}\label{Sec:SixthDiagram}
\begin{proposition}  \label{Propo:LeadingGraphs} Suppose that the corresponding graph is delayed recollisional and $d\ge 2$. There are constants $\mathfrak{C}_{Q_{Delayed}}$, $\mathfrak{C}'_{Q_{Delayed}}$, $\mathfrak{C}''_{Q_{Delayed}}>0$ such that for $1\le n\le \mathbf{N}$ and $t=\tau\lambda^{-2}>0$, if we set
	\begin{equation}
		\begin{aligned} \label{Propo:CrossingGraphs:1:0}
			{Q}_{Delayed}(\tau):=\	& 
			\lambda^{n}\left[ \sum_{\bar\sigma\in \{\pm1\}^{\mathcal{I}_n}}\int_{(\Lambda^*)^{\mathcal{I}_n}}  \mathrm{d}\bar{k}\Delta_{n,\rho}(\bar{k},\bar\sigma) \right. \mathbf{1}({\sigma_{n,1}=-1})\mathbf{1}({\sigma_{n,2}'=1})\mathscr{U}(0)\\
			&\times  \int_{(\Lambda^*)^{\mathcal{I}_n'}}\mathrm{d}\bar{k}\Delta_{n,\rho}(\bar k,\bar\sigma)\tilde{\Phi}_0(\sigma_{0,\rho_1},  k_{0,\rho_1},\sigma_{1,\rho_2},  k_{0,\rho_2})\\
			&\times \sigma_{1,\rho_1}\mathcal{M}( k_{1,\rho_1}, k_{0,\rho_1}, k_{0,\rho_1+1})\prod_{i=2}^n\Big[ \sigma_{i,\rho_i}\mathcal{M}( k_{i,\rho_i}, k_{i-1,\rho_i}, k_{i-1,\rho_i+1})\\
			&\times \Phi_{1,i}(\sigma_{i-1,\rho_i},  k_{i-1,\rho_i},\sigma_{i-1,\rho_i+1},  k_{i-1,\rho_i+1})\Big] \\
			&\left.\times\int_{(\mathbb{R}_+)^{\{0,\cdots,n\}}}\mathrm{d}\bar{s} \delta\left(t-\sum_{i=0}^ns_i\right)\prod_{i=1}^{n}e^{-s_i[\tau_{n-i}+\tau_i]}\prod_{i=1}^{n} e^{-{\bf i}t_i(s)\mathbf{X}_i}\right],
	\end{aligned}\end{equation}
	then, for any constants $1>T_*>0$ and $\mathscr{R}>0$,
	
	\begin{equation}
		\begin{aligned} \label{Propo:CrossingGraphs:1}
			&
			\Big\|\limsup_{D\to\infty} \mathscr{F}_{2/n,\mathscr{R}}\Big[\widehat{{Q}_{Delayed}\chi_{[0,T_*)}}(\mathscr{X}{\lambda^{-2}})\Big](k_{n,1},k_{n,2})\Big\|_{L^{\infty,2}(\mathbb{T}^{2d})}
			\\
			\le& e^{T_*}\frac{T_*^{1+{n}_0(n)}}{({n}_0(n))!}\lambda^{\mathfrak{C}'_{Q_{Delayed}}}\mathfrak{C}_{Q_{Delayed}}^{1+{n}_1(n)}\langle \ln n\rangle \langle\ln\lambda\rangle^{2+(2{n}_1(n)+\eth)+\mathfrak{C}''_{Q_{Delayed}}},
	\end{aligned}\end{equation}
	where, we have used the same notations as in Propositions \ref{lemma:BasicGEstimate1},   \ref{Propo:CrossingGraphs}.  If we replace $\mathscr{U}(0)$ by $\mathscr{W}_1(s_0)$ and  $\varsigma_n\int_0^{s_0}e^{-(s_0-s_0')\varsigma_n}\mathrm{d}s_0'\mathscr{W}_1(s_0')$, the same estimate holds true. 
\end{proposition}

\begin{proof}
	The proof is similar as in the homogeneous case \cite{staffilani2021wave}. We therefore only highlight the key steps and the main differences. 	As in the proof of Proposition \ref{lemma:BasicGEstimate1}, we replace $\Lambda^*$ by $\mathbb{T}^d$.  We denote by $i_1$ the index of the first degree-one vertex $v_{i_1}$, such that it does not correspond to a double-cluster, a single-cluster recollision or a cycle formed by iteratively applying   the recollisions. 
 We denote by $k_0$ the edge in $\mathfrak{E}_+(v_{i_1})$ and $k_1,k_2$ the two edges in $\mathfrak{E}_-(v_{i_1})$, in which $k_1$ is the free edge. The signs of those edges are then set to be $\sigma_{k_0},\sigma_{k_1},\sigma_{k_2}$. By   Lemma \ref{Lemma:VerticesLongCollisions}  there exists a unique vertex $v_{i_2}$, $i_1-i_2>1,$ within the cycle of $v_{i_1}$ such that $\mathbf{X}_{i_2}+\mathbf{X}_{i_1}$ is not a function of $k_1$. 
	We set  $\mathfrak{I}_{v_{i_1}}$ to be the collection of all vertices belonging to the cycle of $v_{i_1}$.   From Lemma \ref{Lemma:VerticesLongCollisions}, we deduce that $v_{i_2}$ is the only degree-zero vertex that does not belong to any recollisions in the cycle of $v_{i_1}$; while the other parts of the cycle contain only recollisions  or no recollision at all - the cycle is then just a short delayed  recollision. The real part of the  total phase of the first $i_1$ time slices is now written as $$
		\sum_{i=0}^{i_1-1}s_i\mathrm{Re}\vartheta_i \ = \ \tilde\xi\ + \ \sum_{i=0}^{i_1-1}s_i\sum_{j=i+1}^{i_1}\mathbf{X}_j,
$$
	in which the quantity $\tilde\xi$ is independent of all of the free edges of the first $i_1$ time slices. We set  $\mathfrak{D}_i=\{0\le j<i_1 \ | \ \mathrm{deg}v_j=i\} $
	for $i=0,1$ and
	$\mathfrak{D}_0' \ = \ \{i\in \mathfrak{D}_0\backslash\{i_2\} \ | \ i+1 \in	 \mathfrak{D}_0\}. $
	We compute
$$ \sum_{i=0}^{i_1-1}s_i\sum_{j=i+1}^{i_1}\mathbf{X}_j   =    \sum_{j\in \mathfrak{D}_1} \left[\mathbf{X}_j\sum_{i=0}^{j-1}s_i   +\mathbf{X}_{j-1}\sum_{i=0}^{j-2}s_i \right] +  \mathbf{X}_{i_1}\sum_{i=i_2}^{i_1-1}s_i
		  +  (\mathbf{X}_{i_2}+\mathbf{X}_{i_1})\sum_{i=0}^{i_2-1}s_i    + \sum_{j\in\mathfrak{D}_0'}\mathbf{X}_j\sum_{i=0}^{j-1}s_i, 
	$$
	and set 
$$
			\tilde\xi_0  =  \tilde\xi +  (\mathbf{X}_{i_2}+\mathbf{X}_{i_1})\sum_{i=0}^{i_2-1}s_i   +  \sum_{j\in\mathfrak{D}_0'}\mathbf{X}_j\sum_{i=0}^{j-1}s_i. 
	$$
	 It follows that  $ \tilde\xi_0$
	is independent of the free edge of $v_{i_1}$. Next, we split $\mathfrak{D}_1=\mathfrak{D}_1'\cup\mathfrak{D}_1''$, in which $\mathfrak{D}_1'$ is the set of the indices of the degree-one vertices belonging to the cycle of $v_{i_1}$ excluding $i_1$, and  $\mathfrak{D}_1''$ is the set of the indices of the degree-one vertices outside of the cycle of $v_{i_1}$. The total phase of the first $i_1$ time slices can be now re-expressed as
	$$ \sum_{i=0}^{i_1-1}s_i\mathrm{Re}\vartheta_i  =   
		  \tilde\xi_1 \ + \ \sum_{j\in \mathfrak{D}_1'} \left[\mathbf{X}_j\sum_{i=0}^{j-1}s_i \ +\mathbf{X}_{j-1}\sum_{i=0}^{j-2}s_i \right] +  \mathbf{X}_{i_1}\sum_{i=l_2}^{i_1-1}s_i,$$	where  $\tilde\xi_1$ is  independent of the free edge of $v_{i_1}$.	
	We set
	$$
			\vartheta_{i_1}^{cyc} :
			 =   \sum_{j\in \mathfrak{D}_1'} \mathbf{X}_js_{j-1}  +  \mathbf{X}_{i_1}\sum_{i=i_2}^{i_1-1}s_i,
	$$
	and 
$$
			\vartheta_{i_1}^{cy}  :=     \mathbf{X}_{i_1}\sum_{i\in\mathfrak{I}^o}s_i,
$$ where $\mathfrak{I}^o$ is  the set of indices $i$ going frmo $i_2$ to $i_1-1$ such that  $ i+1\notin \mathfrak{D}_1'.$ Next, we set $\mathfrak{I}_0=\{0\le  i< n |\mathrm{deg}v_{i+1}=0\}$, $\mathfrak{I}_1=\{0\le  i< n |\mathrm{deg}v_{i+1}=1\}$, 
	$\mathfrak{I}'=\{n+1\}\cup \{i_1-1\} \cup \mathfrak{I}_0$, and follow   the standard estimate of Proposition \ref{lemma:BasicGEstimate1} to write
	\begin{equation}\label{DelayedRecollision:A2}
		\begin{aligned} 
			& \left|-\oint_{\Gamma_{n}}\frac{\mathrm{d}\xi}{2\pi}
			\frac{{\bf i}}{\xi }
			\prod_{{ i}\in \mathfrak{I}_1' } \frac{{\bf i}}{\xi-(\mathrm{Re}\vartheta_i+2\pi\mathscr{X}-{\bf i} \lambda^2)}\int_{(\mathbb{R}_+)^{\mathfrak{I}'}}\! \mathrm{d} \bar{s} \,  \delta\left(t-\sum_{i\in \mathfrak{I}'} s_i\right) \right.\\
			&\left. \ \ \ \ \ \ \ \ \  \ \ \times e^{ -{\bf i} \vartheta_{i_1}^{cy}} e^{-{\bf i} s_{n+1} \xi}e^{-{\bf i}2\pi\mathscr{X}\sum_{i\in\mathfrak{I}^o}s_i}\prod_{i\in \mathfrak{I}'}e^{-{\bf i}s_i\vartheta_i}\right|\\ 
			\le & \left|-\oint_{\Gamma_{n}}\frac{\mathrm{d}\xi}{2\pi}
			\frac{{\bf i}}{\xi }
			\prod_{{ i}\in \mathfrak{I}_1' } \frac{{\bf i}}{\xi-(\mathrm{Re}\vartheta_i+2\pi\mathscr{X}-{\bf i} \lambda^2)}\right.\left.  \int_{(\mathbb{R}_+)^{\mathfrak{I}'}}\! \mathrm{d} \bar{s} \,  \delta\left(t-\sum_{i\in \mathfrak{I}'} s_i\right)e^{ -{\bf i} \vartheta_{i_1}^{cy}}e^{-{\bf i}2\pi\mathscr{X}\sum_{i\in\mathfrak{I}^o}s_i} e^{-{\bf i} s_{n+1}  \xi }\right|,
	\end{aligned}\end{equation}
	where $\mathfrak{I}_1'$ denotes all time slices $i$ such that $v_{i+1}$ are degree-one vertices satisfying  $i+1\ne i_1$. The same estimate was also used in the homogeneous case (see \cite{staffilani2021wave}[Proposition 91]).   Let us now explain the meaning of inequality \eqref{DelayedRecollision:A2}: we modify the  standard strategy of Proposition \ref{lemma:BasicGEstimate1} by removing $i_1-1$ from $\mathfrak{I}_1'$ and perform a special treatment for the degree-one vertex $v_{i_1}$, instead of using Lemma \ref{lemma:degree1vertex}. This is the main difference between our estimates  of delayed recollisional graphs and the basic graph estimates of Proposition \ref{lemma:BasicGEstimate1}. 	We will now explain inequality \eqref{DelayedRecollision:A2} as follows. In this estimate, that we will need to estimate all time slices belonging  to $ \mathfrak{I}_1'$ using the basic estimate of Proposition \ref{lemma:BasicGEstimate1}, from the bottom of the graph to the top, until  $v_{i_1}$ is reached. In this procedure, we will  change the form of $\vartheta_{i_1}^{cyc}$. When integrating the vertices from the bottom to the top, whenever we meet a degree one-vertex $v_j$, with $j\in \mathfrak{D}_1'$, we will  use Lemma \ref{Lemma:Kidentity} to remove $e^{{\bf i}  \mathbf{X}_js_{j-1}}$   from the total phases and put it into the fraction  $\frac{{\bf i}}{\xi-(\mathrm{Re}\vartheta_{j-1}+2\pi\mathscr{X}-{\bf i} \lambda^2)}$. However, when using  Lemma \ref{Lemma:Kidentity}, we will also remove the whole time slice $s_{j-1}$, as thus, $ \mathbf{X}_{i_1}s_{j-1}$  is also put into $\frac{{\bf i}}{\xi-(\mathrm{Re}\vartheta_{j-1}+2\pi\mathscr{X})}$. Next, using the  basic estimate of Proposition \ref{lemma:BasicGEstimate1} also implies that we apply Lemma \ref{lemma:degree1vertex} to bound  $\frac{{\bf i}}{\xi-(\mathrm{Re}\vartheta_{j-1}+2\pi\mathscr{X})}$ and then $ \mathbf{X}_{i_1}s_{j-1}$ is now removed from $\vartheta_{i_1}^{cyc}$. As a result, after the procedure of integrating all the degree-one vertices below $v_{i_1}$, we reduce $\vartheta_{i_1}^{cyc}$ to the  new quantity $	\vartheta_{i_1}^{cy}$.

	We then obtain a more refined bound concerning  the left hand side of \eqref{Propo:CrossingGraphs:1}, which takes into account the role of the full graph, namely
	\begin{equation}\label{DelayedRecollision:A3}
		\begin{aligned}
			&  \Big|\int_{(\mathbb{T}^{d})^{\mathcal{I}_n'}}\mathrm{d}\bar{k}\Delta_{n,\rho}(\bar k,\bar\sigma) \sigma_{1,\rho_1}\mathcal{M}( k_{1,\rho_1}, k_{0,\rho_1}, k_{0,\rho_1+1})\\
			&\times\prod_{i=2}^n\Big[\sigma_{i,\rho_i} \mathcal{M}( k_{i,\rho_i}, k_{i-1,\rho_i}, k_{i-1,\rho_i+1}) \tilde{\Phi}_{1,i}(\sigma_{i-1,\rho_i},  k_{i-1,\rho_i},\sigma_{i-1,\rho_i+1},  k_{i-1,\rho_i+1})\Big] \\
			&\times \oint_{\Gamma_{n}}\frac{\mathrm{d}\xi}{2\pi}
			\frac{-{\bf i}}{\xi+{\bf i}\lambda^2}
			\prod_{{ i}\in \mathfrak{I}_1'} \frac{{\bf i}}{\xi-(\mathrm{Re}\vartheta_i+2\pi\mathscr{X}-{\bf i} \lambda^2)}\int_{(\mathbb{R}_+)^{\mathfrak{I}'}}\! \mathrm{d} \bar{s} \,  \delta\left(t-\sum_{i\in \mathfrak{I}'} s_i\right)
			\\ & \qquad \left. \times e^{-{\bf i} s_{n+1}  \xi}
			e^{-{\bf i} s_i\mathrm{Re}\vartheta_i}
			e^{ -{\bf i} \vartheta_{i_1}^{cy}}e^{-{\bf i}2\pi\mathscr{X}\sum_{i\in\mathfrak{I}^o}s_i}
			\,\right.\mathbf{1}({\sigma_{n,1}=-1})\mathbf{1}({\sigma_{n,2}=1})\mathscr{U}(0)\Big| \ =: \ |\mathfrak{C}_0^{del}|, 
	\end{aligned}\end{equation}
	in which we bound $\tilde{\Phi}_0(\sigma_{0,\rho_1},  k_{0,\rho_1},\sigma_{1,\rho_2},  k_{0,\rho_2})$  by $1$. 
	For the sake of simplicity, we will denote $\tilde{\Phi}_{1,i}(\sigma_{i-1,\rho_i},  k_{i-1,\rho_i},\sigma_{i-1,\rho_i+1},  k_{i-1,\rho_i+1})=\tilde{\Phi}_{1,i}(v_i) .$  If we replace $\mathscr{U}(0)$ by $\mathscr{W}_1(s_0)$ and  $\varsigma_n\int_0^{s_0}e^{-(s_0-s_0')\varsigma_n}$ $\mathrm{d}s_0'\mathscr{W}_1(s_0')$, the same estimate holds true, as thus, we only consider the case of $\mathscr{U}(0)$. By integrating the free edge of $v_{i_1}$, using the strategy described below instead of using again  Lemma \ref{lemma:degree1vertex}, we will gain  an extra factor of $\lambda^{c}$ for some $c>0$, guaranteeing the convergence of the whole graph to zero. 
	To integrate the free edge of $v_{i_1}$, only the time slices from $i_2$ to $i_1-1$ that appear in the sum $\vartheta_{i_1}^{cy}$ are then isolated. We set
	\begin{equation}\label{DelayedRecollision:A3:1}
		\begin{aligned}
			\mathfrak{C}_1^{del} :=	&\ \int_{(\mathbb{R}_+)^{\mathfrak{I}'}}\! \mathrm{d} \bar{s} \,  \delta\left(t-\sum_{i\in \mathfrak{I}'} s_i\right) \int_{(\mathbb{T}^{d})^{\mathcal{I}_n^o}}\mathrm{d}\bar{k}\Delta^o_{n,\rho}(\bar k,\bar\sigma)\sigma_{1,\rho_1} \mathcal{M}( k_{1,\rho_1}, k_{0,\rho_1}, k_{0,\rho_1+1})\\
			&\times\prod_{i=2}^{i_1}\Big[ \sigma_{i,\rho_i}\mathcal{M}( k_{i,\rho_i}, k_{i-1,\rho_i}, k_{i-1,\rho_i+1})\\
			&\times \tilde{\Phi}_{1,i}(\sigma_{i-1,\rho_i},  k_{i-1,\rho_i},\sigma_{i-1,\rho_i+1},  k_{i-1,\rho_i+1})\Big] 
			e^{-{\bf i} s_i\mathrm{Re}\vartheta_i}
			e^{ -{\bf i} \vartheta_{i_1}^{cy}}e^{-{\bf i}2\pi\mathscr{X}\sum_{i\in\mathfrak{I}^o}s_i}\mathscr{U}(0).
	\end{aligned}\end{equation}
	This quantity is a portion of $\mathfrak{C}_0^{del}$ that involves the time slices from $0$ to $i_1-1$ as well as the  set $\mathfrak{I}^o$ of the cycle of $v_{i_1}$. The quantity ${(\mathbb{T}^{d})^{\mathcal{I}_n^o}}$ is the part of ${(\mathbb{T}^{d})^{\mathcal{I}_n'}}$ containing all the domain of all momenta $k_{i,j}$ of all  the time slices from $0$ to $i_1-1$ that involve the index set $\mathfrak{I}^o$ of the cycle of $v_{i_1}$. Similarly, $\Delta^o_{n,\rho}(\bar k,\bar\sigma)$ is the part of $\Delta_{n,\rho}(\bar k,\bar\sigma)$ involving all momenta connecting to the index set $\mathfrak{I}^o$ of the cycle of $v_{i_1}$. Following precise the same argument of the homogeneous case \cite{staffilani2021wave}[Proposition 91], if we set $\mathfrak{C}_0^{del}=\mathfrak{C}_1^{del}\mathfrak{C}_2^{del}$, then the quantity $\mathfrak{C}_2^{del}$ can be estimated using the standard method of Proposition \ref{lemma:BasicGEstimate1}. We now devote the rest of the proof to estimate $\mathfrak{C}_1^{del}$. We define
	\begin{equation}\label{DelayedRecollision:A3:2}
		\begin{aligned}
			\mathfrak{C}_3^{del} :=	&\  \int_{(\mathbb{T}^{d})^{\mathcal{I}_n^o}}\mathrm{d}\bar{k}\Delta^o_{n,\rho}(\bar k,\bar\sigma)\sigma_{1,\rho_1} \mathcal{M}( k_{1,\rho_1}, k_{0,\rho_1}, k_{0,\rho_1+1}) \prod_{i=2}^{l_1}\Big[ \sigma_{i,\rho_i}\mathcal{M}( k_{i,\rho_i}, k_{i-1,\rho_i}, k_{i-1,\rho_i+1})\\
			&\times \tilde{\Phi}_{1,i}(\sigma_{i-1,\rho_i},  k_{i-1,\rho_i},\sigma_{i-1,\rho_i+1},  k_{i-1,\rho_i+1})\Big] 
			e^{-{\bf i} s_i\mathrm{Re}\vartheta_i}
			e^{ -{\bf i} \vartheta_{l_1}^{cy}}e^{-{\bf i}2\pi\mathscr{X}\sum_{i\in\mathfrak{I}^o}s_i}\mathscr{U}(0). 
	\end{aligned}\end{equation}
and bound (see \cite{staffilani2021wave}[Proposition 91])
$$
			|\mathfrak{C}_1^{del}| 
			\lesssim 	{(\tau\lambda^{-2})^{\bar{n}}}\mathfrak{C}^{\bar{n}} \left|\int_{(\mathbb{R}_+)^{\mathfrak{I}^o}}\! \mathrm{d} \bar{s} \mathfrak{C}_3^{del}{\bf \mathbf{1}}\left(\sum_{i\in \mathfrak{I}^o} s_i\le t\right)\right|.$$
	 We then set 
	$$
			\mathfrak{C}_4^{del}
			:=   \Big|\int_{(\mathbb{R}_+)^{\mathfrak{I}^o}}\! \mathrm{d} \bar{s} \mathfrak{C}_3^{del}{\bf \mathbf{1}}\left(\sum_{i\in \mathfrak{I}^o} s_i\le t\right)\Big|.$$
Note that the graph is pairing, $\bar{n}+(|\mathfrak{I}^o|-1)=n/2$, we will  need to show that we can obtain from $\mathfrak{C}_4^{del}$ a factor of $(\lambda^{-2})^{|\mathfrak{I}^o|-1-c'}$ for some constant $0<c'$ to ensure the convergence to $0$ of the whole graph in the limit that $\lambda$ goes to $0$.

The quantity $\mathfrak{C}_4^{del}$ contains a delayed recollision in its form. 	As a result, we will explain below the structure of the recollisions and the delayed recollision formed by the cycle of $v_{i_1}$ and the  recollisions below this cycle, that is different from the homogeneous case \cite{staffilani2021wave}.
To this end, we consider a  recollision consisting of two vertices $v_j,v_{j-1}$, where $\mathrm{deg}v_j=1$, $\mathrm{deg}v_{j-1}=0$. We compute the total phase of the two time slices associated to $v_j,v_{j-1}$  $$\mathbf{X}_j\sum_{i=0}^{j-1}s_i\  + \ \mathbf{X}_{j-1}\sum_{i=0}^{j-2}s_i \ =\ \mathbf{X}_js_{j-1} \ + \ [\mathbf{X}_{j-1}+\mathbf{X}_{j}]\sum_{i=0}^{j-2}s_i.$$ The cycle contains the following product of delta functions
\begin{equation}\begin{aligned}
		&		\delta(\sigma_{j,\rho_j}k_{j,\rho_j}+\sigma_{j-1,\rho_j} k_{j-1,\rho_j}+\sigma_{j-1,\rho_j+1}k_{j-1,\rho_j+1})\\
	&\times \delta(\sigma_{j-1,\rho_{j-1}}k_{j-1,\rho_{j-1}}+\sigma_{j-2,\rho_{j-1} k_{j-2,\rho_{j-1}}}+\sigma_{j-2,\rho_{j-1}+1}k_{j-2,\rho_{j-1}+1})\end{aligned}
\end{equation}
that is either 
\begin{equation}\begin{aligned}
	& 	\delta(\sigma_{j,\rho_j}k_{j,\rho_j}+\sigma_{j-1,\rho_j} k_{j-1,\rho_j}+\sigma_{j-1,\rho_j+1}k_{j-1,\rho_j+1})\delta(\sigma_{j-1,\rho_j} k_{j-1,\rho_j}+\sigma_{j-1,\rho_{j-1}}k_{j-1,\rho_{j-1}})\\
&\times \delta(\sigma_{j-1,\rho_{j-1}}k_{j-1,\rho_{j-1}}+\sigma_{j-2,\rho_{j-1} k_{j-2,\rho_{j-1}}}+\sigma_{j-2,\rho_{j-1}+1}k_{j-2,\rho_{j-1}+1})\\
&\times\mathbf{1}(\sigma_{j,\rho_j}+\sigma_{j-2,\rho_{j-1}}=0)\mathbf{1}(\sigma_{j-1,\rho_j+1}+\sigma_{j-2,\rho_{j-1}+1}=0),
\end{aligned}	\end{equation}
and
\begin{equation}\begin{aligned}
		& 	\delta(\sigma_{j,\rho_j}k_{j,\rho_j}+\sigma_{j-1,\rho_j} k_{j-1,\rho_j}+\sigma_{j-1,\rho_j+1}k_{j-1,\rho_j+1})\delta(\sigma_{j-1,\rho_j+1}k_{j-1,\rho_j+1}+\sigma_{j-1,\rho_{j-1}}k_{j-1,\rho_{j-1}})\\
		&\times \delta(\sigma_{j-1,\rho_{j-1}}k_{j-1,\rho_{j-1}}+\sigma_{j-2,\rho_{j-1} k_{j-2,\rho_{j-1}}}+\sigma_{j-2,\rho_{j-1}+1}k_{j-2,\rho_{j-1}+1}),\\
		&\times\mathbf{1}(\sigma_{j,\rho_j}+\sigma_{j-2,\rho_{j-1}}=0)\mathbf{1}(\sigma_{j-1,\rho_j}+\sigma_{j-2,\rho_{j-1}+1}=0)
	\end{aligned}
	\end{equation}
or 
\begin{equation}\begin{aligned}
		& 	\delta(\sigma_{j,\rho_j}k_{j,\rho_j}+\sigma_{j-1,\rho_j} k_{j-1,\rho_j}+\sigma_{j-1,\rho_j+1}k_{j-1,\rho_j+1})\\
		&\times \delta(\sigma_{j,\rho_j}k_{j-1,\rho_{j-1}}+\sigma_{j-1,\rho_j}  k_{j-2,\rho_{j-1}}+\sigma_{j-1,\rho_j+1}k_{j-2,\rho_{j-1}+1}).
	\end{aligned}
\end{equation}
The product $$\mathcal{M}(k_{j,\rho_j}, k_{j-1,\rho_j},k_{j-1,\rho_j+1})\mathcal{M}(k_{j-1,\rho_{j-1}}, k_{j-2,\rho_{j-1}},k_{j-2,\rho_{j-1}+1})$$    is also combined into  $$\mbox{ either }-\sigma_{j,\rho_j}\sigma_{j-1,\rho_j}\mathrm{sign}k_{j,\rho_j}^1\mathrm{sign}k_{j-1,\rho_j}^1|\mathcal{W}(k_{j,\rho_j}, k_{j-1,\rho_j},k_{j-1,\rho_j+1})||\mathcal{W}(k_{j-1,\rho_{j-1}}, k_{j-2,\rho_{j-1}},k_{j-2,\rho_{j-1}+1})|,$$ 
$$ -\sigma_{j,\rho_j}\sigma_{j-1,\rho_j+1}\mathrm{sign}k_{j,\rho_j}^1\mathrm{sign}k_{j-1,\rho_j+1}^1|\mathcal{W}(k_{j,\rho_j}, k_{j-1,\rho_j},k_{j-1,\rho_j+1})||\mathcal{W}(k_{j-1,\rho_{j-1}}, k_{j-2,\rho_{j-1}},k_{j-2,\rho_{j-1}+1})|,$$
$$\mbox{ or }-\sigma_{j,\rho_j}^2(\mathrm{sign}k_{j,\rho_j}^1)^2|\mathcal{W}(k_{j,\rho_j}, k_{j-1,\rho_j},k_{j-1,\rho_j+1})||\mathcal{W}(k_{j-1,\rho_{j-1}}, k_{j-2,\rho_{j-1}},k_{j-2,\rho_{j-1}+1})|,$$ due to the pairings of the momenta.

As a result, if a recollision is of double-cluster type, adding it to a pairing means that we change a factor $\langle a_{k_{j,\rho_j},\sigma_{j,\rho_j}}a_{k_{j-1,\rho_{j-1}},-\sigma_{j,\rho_j}}\rangle_0$ to 

\begin{equation}
	\label{DelayedRecollision:A5b}\begin{aligned}
		&
		-\iint_{(\mathbb{T}^d)^4}\!\! \mathrm{d} k_{j-1,\rho_j+1}  \mathrm{d} k_{j-1,\rho_j}\,\mathrm{d} k_{j-2,\rho_{j-1}+1}  \mathrm{d} k_{j-2,\rho_{j-1}}\,  \\
		&\times \delta(\sigma_{j,\rho_j}k_{j-1,\rho_{j-1}}+\sigma_{j-1,\rho_j}  k_{j-2,\rho_{j-1}}+\sigma_{j-1,\rho_j+1}k_{j-2,\rho_{j-1}+1})\\ & 
		\times|\mathcal{W}(k_{j,\rho_j}, k_{j-1,\rho_j},k_{j-1,\rho_j+1})||\mathcal{W}(k_{j-1,\rho_{j-1}}, k_{j-2,\rho_{j-1}},k_{j-2,\rho_{j-1}+1})|\Phi_{1,j}(v_j) \\
		&\times e^{{\bf i}[\mathbf{X}_{j-1}+\mathbf{X}_{j}]\sum_{i=0}^{j-2}s_i} e^{{\bf i} s_{j-1} \sigma_{j,\rho_j} \omega(k_{j,\rho_j})}e^{{\bf i} s_{j-1} \sigma_{j-1,\rho_j} \omega(k_{j-1,\rho_j})} e^{{\bf i} s_{j-1} \sigma_{j-1,\rho_j+1} \omega(k_{j-1,\rho_j+1})}\\
		&\times\langle a_{k_{j-1,\rho_j},\sigma_{j-1,\rho_j}}a_{k_{j-2,\rho_{j-1}},-\sigma_{j-2,\rho_{j-1}}}\rangle_0 \langle a_{k_{j-1,\rho_j+1},\sigma_{j-1,\rho_j+1}}a_{k_{j-2,\rho_{j-1}+1},-\sigma_{j-2,\rho_{j-1}+1}}\rangle_0.
	\end{aligned}
\end{equation}
Similarly, if a recollision is of single-cluster type, adding it to an edge means that we change a factor $\langle a_{k_{j,\rho_j},\sigma_{j,\rho_j}}a_{k_{j-1,\rho_{j-1}},-\sigma_{j,\rho_j}}\rangle_0$ to   
\begin{equation}
	\label{DelayedRecollision:A5c}\begin{aligned}
		&
		\iint_{(\mathbb{T}^d)^4}\!\! \mathrm{d} k_{j-1,\rho_j+1}  \mathrm{d} k_{j-1,\rho_j}\,\mathrm{d} k_{j-2,\rho_{j-1}+1}  \mathrm{d} k_{j-2,\rho_{j-1}}\,
		\delta(\sigma_{j,\rho_j}k_{j,\rho_j}+\sigma_{j-1,\rho_j} k_{j-1,\rho_j}+\sigma_{j-1,\rho_j+1}k_{j-1,\rho_j+1})\\
		&\times \delta(\sigma_{j-1,\rho_{j-1}}k_{j-1,\rho_{j-1}}+\sigma_{j-2,\rho_{j-1} k_{j-2,\rho_{j-1}}}+\sigma_{j-2,\rho_{j-1}+1}k_{j-2,\rho_{j-1}+1})\\
		&\times \delta(\sigma_{j-1,\rho_j} k_{j-1,\rho_j}+\sigma_{j-1,\rho_{j-1}}k_{j-1,\rho_{j-1}}) \\ & 
		\times|\mathcal{W}(k_{j,\rho_j}, k_{j-1,\rho_j},k_{j-1,\rho_j+1})||\mathcal{W}(k_{j-1,\rho_{j-1}}, k_{j-2,\rho_{j-1}},k_{j-2,\rho_{j-1}+1})|\Phi_{1,j}(v_j)e^{{\bf i}[\mathbf{X}_{j-1}+\mathbf{X}_{j}]\sum_{i=0}^{j-2}s_i}  \\ & 
		\times(-\sigma_{j,\rho_j}\sigma_{j-1,\rho_j}\mathrm{sign}k_{j,\rho_j}^1\mathrm{sign}k_{j-1,\rho_j}^1)  e^{{\bf i} s_{j-1} \sigma_{j,\rho_j} \omega(k_{j,\rho_j})}e^{{\bf i} s_{j-1} \sigma_{j-1,\rho_j} \omega(k_{j-1,\rho_j})}\\
		&\times \mathbf{1}(\sigma_{j,\rho_j}+\sigma_{j-2,\rho_{j-1}}=0)\mathbf{1}(\sigma_{j-1,\rho_j+1}+\sigma_{j-2,\rho_{j-1}+1}=0) \\
		&\times \langle a_{k_{j-1,\rho_j+1},\sigma_{j-1,\rho_j+1}}a_{k_{j-2,\rho_{j-1}+1},\sigma_{j-2,\rho_{j-1}+1}}\rangle_0e^{{\bf i} s_{j-1} \sigma_{j-1,\rho_j+1} \omega(k_{j-1,\rho_j+1})}\\
		&\times \langle a_{k_{j,\rho_j},\sigma_{j,\rho_j}}a_{k_{j-2,\rho_{j-1}},-\sigma_{j,\rho_j}}\rangle_0,
	\end{aligned}
\end{equation}
and
\begin{equation}
	\label{DelayedRecollision:A5d}\begin{aligned}
	&
	\iint_{(\mathbb{T}^d)^4}\!\! \mathrm{d} k_{j-1,\rho_j+1}  \mathrm{d} k_{j-1,\rho_j}\,\mathrm{d} k_{j-2,\rho_{j-1}+1}  \mathrm{d} k_{j-2,\rho_{j-1}}\,
	\delta(\sigma_{j,\rho_j}k_{j,\rho_j}+\sigma_{j-1,\rho_j} k_{j-1,\rho_j}+\sigma_{j-1,\rho_j+1}k_{j-1,\rho_j+1}) \\
	&\times \delta(\sigma_{j-1,\rho_j+1}k_{j-1,\rho_j+1}+\sigma_{j-1,\rho_{j-1}}k_{j-1,\rho_{j-1}}) \mathbf{1}(\sigma_{j,\rho_j}+\sigma_{j-2,\rho_{j-1}}=0)\mathbf{1}(\sigma_{j-1,\rho_j}+\sigma_{j-2,\rho_{j-1}+1}=0)\\
	&\times \delta(\sigma_{j-1,\rho_{j-1}}k_{j-1,\rho_{j-1}}+\sigma_{j-2,\rho_{j-1} k_{j-2,\rho_{j-1}}}+\sigma_{j-2,\rho_{j-1}+1}k_{j-2,\rho_{j-1}+1})\\ & 
	\times|\mathcal{W}(k_{j,\rho_j}, k_{j-1,\rho_j},k_{j-1,\rho_j+1})||\mathcal{W}(k_{j-1,\rho_{j-1}}, k_{j-2,\rho_{j-1}},k_{j-2,\rho_{j-1}+1})|\Phi_{1,j}(v_j) e^{{\bf i}[\mathbf{X}_{j-1}+\mathbf{X}_{j}]\sum_{i=0}^{j-2}s_i} \\ & 
		\times(-\sigma_{j,\rho_j}\sigma_{j-1,\rho_j+1}\mathrm{sign}k_{j,\rho_j}^1\mathrm{sign}k_{j-1,\rho_j+1}^1)  e^{{\bf i} s_{j-1} \sigma_{j,\rho_j} \omega(k_{j,\rho_j})}e^{{\bf i} s_{j-1} \sigma_{j-1,\rho_j} \omega(k_{j-1,\rho_j})} \\
		&\times\mathbf{1}(\sigma_{j,\rho_j}+\sigma_{j-2,\rho_{j-1}}=0)\mathbf{1}(\sigma_{j-1,\rho_j}+\sigma_{j-2,\rho_{j-1}+1}=0)\\
		&\times \langle a_{k_{j-1,\rho_j},\sigma_{j-1,\rho_j}}a_{k_{j-2,\rho_{j-1}},\sigma_{j-2,\rho_{j-1}}}\rangle_0e^{{\bf i} s_{j-1} \sigma_{j-1,\rho_j+1} \omega(k_{j-1,\rho_j+1})} \langle a_{k_{j,\rho_j},\sigma_{j,\rho_j}}a_{k_{j-2,\rho_{j-1}},-\sigma_{j,\rho_j}}\rangle_0.
	\end{aligned}
\end{equation}

 The same estimate in the homogeneous case can be redone \cite{staffilani2021wave}, with the addition of the simple estimate $\Big|e^{{\bf i}[\mathbf{X}_{j-1}+\mathbf{X}_{j}]\sum_{i=0}^{j-2}s_i}\Big|\le1$, and we bound
	
	\begin{equation}\label{DelayedRecollision:A6:b6}
		\begin{aligned}
			\Big\|\mathfrak{C}_4^{del}\Big\|_{L^2(-\mathscr{R},\mathscr{R})}
			\lesssim\ &   	&\lambda^{-\frac2q(|\mathfrak{I}^o|-1)} \left[\int_{(\mathbb{R})^{\mathfrak{I}^o}}\! \mathrm{d} \bar{s}	{\mathbf{1}}\left( \sum_{i\in \mathfrak{I}^o} s_i\le t\right) \Big\langle \sum_{i\in \mathfrak{I}^o} s_i\Big\rangle^{-\big( \frac{p(d-1)}{16}-\big)}\right]^\frac{1}{p},
	\end{aligned}\end{equation}
	for some constants $p,q>0$ and $\frac1p+\frac1q=1.$ The same strategy of the proof of Proposition \ref{lemma:BasicGEstimate1}, yielding another estimate (see \cite{staffilani2021wave}[Proposition 91])
	\begin{equation}\label{DelayedRecollision:A4}\begin{aligned}
			&		\limsup_{D\to\infty} 
			\Big\|\mathscr{F}_{1/n,\mathscr{R}}\Big[\widehat{{Q}_{Delayed}(\tau)\chi_{[0,T_*)}(\tau)}(\mathscr{X}{\lambda^{-2}})\Big]\Big\|_{L^2} 4^n e^{T_*}\frac{T_*^{1+\bar{n}}}{(\bar{n})!}\lambda^{n-2\bar{n}-\eth'_1-\eth'_2}\\ \le\  & \times\mathfrak{C}_{Q_{1,1}}^{1+{n}_1'}\langle \ln n\rangle \langle\ln\lambda\rangle^{1+(2+\eth)({n}_1'-2
				)+\mathfrak{C}''_{Q_{Delayed}}} 
			\left[\int_{(\mathbb{R})^{\mathfrak{I}^o}}\! \mathrm{d} \bar{s}	{\bf \mathbf{1}}\left( \sum_{i\in \mathfrak{I}^o} s_i\le t\right) \Big\langle \sum_{i\in \mathfrak{I}^o} s_i\Big\rangle^{-\big(\frac{p(d-1)}{16}-\big)}\right]^\frac{1}{p},\end{aligned}
	\end{equation}
	where $\bar{n}+(|\mathfrak{I}^o|-1)=n/2$ and $n'_1$ is the number of degree-one vertices $v_i$ with $i\ne i_1$. As
		$$\int_{(\mathbb{R})^{\mathfrak{I}^o}}\! \mathrm{d} \bar{s}	{\bf \mathbf{1}}\left(\sum_{i\in \mathfrak{I}^o} s_i\le t\right) {\left\langle \sum_{i\in \mathfrak{I}^o} s_i\right\rangle^{-\big(\frac{p(d-1)}{16}-\big)}} \ \lesssim \ 
			t^{|\mathfrak{I}^o|-1-\frac{9}{16}},$$
	for $p(d-1)\ge 17$, the same argument used in the homogeneous case shows that the graphs go to $0$ in the limit that $\lambda$ goes to $0$.
\end{proof}

\subsection{Seventh  type of diagram estimates}\label{Sec:Res}

\begin{proposition}[The first estimate on   the residual terms]\label{lemma:Residual1} For $1\le n\le \mathbf{N}$ and $t=\tau\lambda^{-2}>0$, we have
	\begin{equation}
		\begin{aligned} \label{eq:Residual1}
			&	\limsup_{D\to\infty} 
			\lambda^{n}\mathbf{1}({\sigma_{n,1}=-1})\mathbf{1}({\sigma_{n,2}=1}) \\ 
			&\sum_{\substack{\bar\sigma\in \{\pm1\}^{\mathcal{I}_n},\\ \sigma_{i,\rho_i}+\sigma_{i-1,\rho_i}+\sigma_{i-1,\rho_i+1}\ne \pm3,
					\\ \sigma_{i-1,\rho_i}\sigma_{i-1,\rho_i+1}= 1}}\Big|\int_{(\Lambda^*)^{\mathcal{I}_n}}  \mathrm{d}\bar{k}\Delta_{n,\rho}(\bar{k},\bar\sigma) \mathfrak{C}^m_{n,res} 	\Phi'(\sigma_{0,\rho_1},  k_{0,\rho_1},\sigma_{1,\rho_2},  k_{0,\rho_2})\sigma_{1,\rho_1}\\
			&\ \times \mathcal{M}( k_{1,\rho_1}, k_{0,\rho_1}, k_{0,\rho_1+1})\prod_{i=2}^n\Big[\sigma_{i,\rho_i}\mathcal{M}( k_{i,\rho_i}, k_{i-1,\rho_i}, k_{i-1,\rho_i+1}) \Phi_{1,i}(\sigma_{i-1,\rho_i}, k_{i-1,\rho_i},\sigma_{i-1,\rho_i+1}, k_{i-1,\rho_i+1})\Big] \\
			&\ \  \times\int_{(\mathbb{R}_+)^{\{0,\cdots,n\}}}\mathrm{d}\bar{s} \delta\left(t-\sum_{i=0}^ns_i\right)\prod_{i=0}^{n}e^{-s_i[\varsigma_{n-i}+\tau_i]} \prod_{i=1}^{n} e^{-{\bf i}t_i(s)\mathbf{X}_i}\Big|\\
			=	& \mathcal{O}(\lambda^{\min\{(d-1)\eth'_2/3,n-(n-1)\theta_r-2\}}),
	\end{aligned}\end{equation}
	when $c_r>0$, where $c_r$   appears in the phase regulators $\tau_i$ in \eqref{Def:TauGen}. In the above expression, the constant $m$ can be either $0,2$ or $3$ and  $\mathfrak{C}^m_{n,res}$ can be either $\mathfrak{C}^0_{n,res}$, $\mathfrak{C}^2_{n,res}$ and $\mathfrak{C}^3_{n,res}$ defined in Propositions \ref{Proposition:Ampl1}, \ref{Proposition:Ampl2}, \ref{Proposition:Ampl3}.   The function $\Phi'$ is $\Phi_{1,1}$ when $\mathfrak{C}^m_{n,res}=\mathfrak{C}^0_{n,res}$. The function $\Phi'$ is $\Phi_{0,1}$  or $\Phi_{1,1}$ when $\mathfrak{C}^m_{n,res}=\mathfrak{C}^2_{n,res}$ and $\mathfrak{C}^3_{n,res}$, consistently with  Propositions \ref{Proposition:Ampl1}, \ref{Proposition:Ampl2}, \ref{Proposition:Ampl3}. The quantity
	$\mathbf{N}$ is the number of Duhamel expansions we need,   defined in Section \ref{Sec:KeyPara}.	
	Above we have used the notation
	$
		\mathbf{X}_i = \mathbf{X}(v_i)   =  \mathbf{X}(\sigma_{i,\rho_i},  k_{i,\rho_i},\sigma_{i-1,\rho_i}, k_{i-1,\rho_i},\sigma_{i-1,\rho_i+1}, k_{i-1,\rho_i+1}).
$
	and the usual  three notations $\bar{k}$, $\bar\sigma$ and $\bar{s}$ for the vectors that are composed of all of the possible $k_{i,j}$, $\sigma_{i,j}$ and $s_i$. 
\end{proposition}

\begin{proof} We divide the proof into several parts.

	{\bf Case 1: $n\ge 3$.} We estimate
	\begin{equation}
		\begin{aligned}\label{eq:Residual1:E1}
			& \lambda^{n}\left|\int_{(\mathbb{R}_+)^{\{0,\cdots,n\}}}\mathrm{d}\bar{s} \delta\left(t-\sum_{i=0}^{n-1}s_i\right)\prod_{i=1}^{n}e^{-s_i[\varsigma_{n-i}+\tau_i]} \prod_{i=1}^{n-1} e^{-{\bf i}t_i(s)\mathbf{X}_i}\right|\\
			\le & \lambda^{n}\int_{(\mathbb{R}_+)^{\{0,\cdots,n\}}}\mathrm{d}\bar{s} \delta\left(t-\sum_{i=0}^ns_i\right)\prod_{i=1}^{n}e^{-s_i\tau_i}. 
	\end{aligned}\end{equation}
	We  first observe that, in the corresponding residual graph, $\tau_i\lambda^{-2}\to \infty$ when $\lambda\to 0$ for $i=1,\cdots,n-1$. This can be seen by contradiction. We suppose that there exists $\tau_i\lambda^{-2}\to 0$ when $\lambda\to 0$, for some $i$ from $1$ to $n-1$. Thus, by Lemma \ref{Lemma:NoiseAction} 
	{and choice of $\mathbf{N}$ \eqref{Def:Para1} } all momenta of time slice $i$ are paired, that means each vertex of the time slice $i$ belongs to a cluster.
	 A  consequence of this fact is that each of the vertex of the time slice $i-1$ also belongs to a cluster, and as thus,  each of the vertices of the time slice $0$ also belongs to a cluster. As a result, the graph is not a residual diagram, leading to a contradiction. 
	For $t'\le \tau\lambda^{-2}$, we have 
$$
		\int_0^{t'}\mathrm{d}se^{-\tau_i s}  = \frac{1-e^{-\tau_i {t'}}}{\tau_i } \ \le \ \frac{1}{\tau_i }  =  \frac{1}{\mathfrak{C}_r\lambda^{\theta_r} }.
	$$
	Integrating from $s_0$ to $s_n$  gives the following bound $$
			 \lambda^{n}\left|\int_{(\mathbb{R}_+)^{\{0,\cdots,n\}}}\mathrm{d}\bar{s} \delta\left(t-\sum_{i=0}^ns_i\right)\prod_{i=1}^{n}e^{-s_i[\varsigma_{n-i}+\tau_i]} \right| \lesssim   \lambda^{n}t \lambda^{-\theta_r(n-1)}  =  \mathcal{O}(\lambda^{n-(n-1)\theta_r-2}).$$
	By our consideration $n\ge 3$, then $n-(n-1)\theta_r-2>0$. Next, we replace $\Lambda^*$   by $\mathbb{T}^d$ .   The quantity $\mathfrak{C}_{n,res}^m$ can be  bounded by the precisely the same method used to estimate the moments of $Q^1$. Finally, we get a bound $\mathcal{O}(\lambda^{n-(n-1)\theta_r-2})$ for \eqref{eq:Residual1}.
	
	{\bf Case 2:  $n= 1$.} By our choice of the parameters in Section \ref{Sec:KeyPara}, $\varsigma_0=\varsigma_1=0$, we only need to estimate   $\mathfrak{C}^0_{1,res}$ and $\mathfrak{C}^2_{1,res}$. In this case, we need to estimate  either the following quantity, associated to $\mathfrak{C}^0_{1,res}$ 
	\begin{equation}
		\begin{aligned}
			\label{eq:Residual1:E5}
			\mathfrak{A}_1\ = \ 	&\lambda\mathbf{1}({\sigma_{1,1}=-1})\mathbf{1}({\sigma_{1,2}=1})\sum_{\substack{\bar\sigma\in \{\pm1\}^{\mathcal{I}_1},\\ \sigma_{i,\rho_i}+\sigma_{i-1,\rho_i}+\sigma_{i-1,\rho_i+1}\ne \pm3,
					\\ \sigma_{i-1,\rho_i}\sigma_{i-1,\rho_i+1}= 1}}\left|\int_{(\Lambda^*)^{\mathcal{I}_1}}  \mathrm{d}\bar{k}\Delta_{1,\rho}(\bar{k},\bar\sigma)\right. \\
			&\times\mathfrak{C}^0_{1,res}\Big( a(k_{0,1},\sigma_{0,1}), a(k_{0,2},\sigma_{0,2}), a(k_{0,3},\sigma_{0,3}) \Big)(s_0)\\
			&\left.\times	\Phi_{0,1}(\sigma_{0,\rho_1},  k_{0,\rho_1},\sigma_{1,\rho_2},  k_{0,\rho_2})\sigma_{1,\rho_1}\mathcal{M}( k_{1,\rho_1}, k_{0,\rho_1}, k_{0,\rho_1+1})\int_{0}^t\mathrm{d}s_0  e^{-{\bf i}s_0\mathfrak{X}_1}\right|
		\end{aligned}
	\end{equation}
	or the following quantity, associated to $\mathfrak{C}^2_{1,res}$ 
	\begin{equation}
		\begin{aligned}
			\label{eq:Residual1:E6}
			\mathfrak{A}_2\ = \ 	&\lambda\mathbf{1}({\sigma_{1,1}=-1})\mathbf{1}({\sigma_{1,2}=1})\sum_{\substack{\bar\sigma\in \{\pm1\}^{\mathcal{I}_1},\\ \sigma_{i,\rho_i}+\sigma_{i-1,\rho_i}+\sigma_{i-1,\rho_i+1}\ne \pm3,
					\\ \sigma_{i-1,\rho_i}\sigma_{i-1,\rho_i+1}= 1}}\left|\int_{(\Lambda^*)^{\mathcal{I}_1}}  \mathrm{d}\bar{k}\Delta_{1,\rho}(\bar{k},\bar\sigma)\right. \\
			&\times \mathfrak{C}^2_{1,res}\Big(a(k_{0,1},\sigma_{0,1}) a(k_{0,2},\sigma_{0,2}) a(k_{0,3},\sigma_{0,3}) \Big)(0)\\
			&\left.\times	\Phi_{1,1}(\sigma_{0,\rho_1},  k_{0,\rho_1},\sigma_{1,\rho_2},  k_{0,\rho_2})\sigma_{1,\rho_1}\mathcal{M}( k_{1,\rho_1}, k_{0,\rho_1}, k_{0,\rho_1+1})
			\int_{0}^t\mathrm{d}s_0  e^{-{\bf i}s_0\mathbf{X}_1}\right|.
		\end{aligned}
	\end{equation} 
Both quantities are $0$ as initially all momenta are paired.
	
\
	
	{\bf Case 3:  $n= 2$.} As $\varsigma_0=\varsigma_1=\varsigma_2=0$, we only need to estimate $\mathfrak{C}^0_{2,res}$ and $\mathfrak{C}^2_{2,res}$. As the quantity associated to $\mathfrak{C}_{2,res}^0$ is $0$, we only need to estimate
	the following quantity, associated to $\mathfrak{C}^2_{2,res}$ 
	\begin{equation}
		\begin{aligned}
			\label{eq:Residual1:E8}
			\mathfrak{A}_3\ = & \ 	\lambda^{2}\mathbf{1}({\sigma_{2,1}=-1})\mathbf{1}({\sigma_{2,2}=1})\sum_{\substack{\bar\sigma\in \{\pm1\}^{\mathcal{I}_2},\\ \sigma_{i,\rho_i}+\sigma_{i-1,\rho_i}+\sigma_{i-1,\rho_i+1}\ne \pm3,
					\\ \sigma_{i-1,\rho_i}\sigma_{i-1,\rho_i+1}= 1}}\left|\int_{(\Lambda^*)^{\mathcal{I}_2}}  \mathrm{d}\bar{k}\Delta_{2,\rho}(\bar{k},\bar\sigma)\right. \\
			&\times \Phi_{0,1}(\sigma_{0,\rho_1},  k_{0,\rho_1},\sigma_{1,\rho_2},  k_{0,\rho_2})\sigma_{1,\rho_1}\mathcal{M}( k_{1,\rho_1}, k_{0,\rho_1}, k_{0,\rho_1+1}) \Big[\sigma_{2,\rho_2}\mathcal{M}( k_{2,\rho_2}, k_{1,\rho_2}, k_{1,\rho_2+1})\\
			&\times \Phi_{1,2}(\sigma_{1,\rho_2}, k_{1,\rho_2},\sigma_{1,\rho_2+1}, k_{1,\rho_2+1})\Big] 
			\mathfrak{C}^2_{2,res}\Big(a(k_{0,1},\sigma_{0,1}), a(k_{0,2},\sigma_{0,2}), a(k_{0,3},\sigma_{0,3}), a(k_{0,4},\sigma_{0,3}) \Big)(s_0)\\
			&\left.\times\int_{(\mathbb{R}_+)^{\{0,1,2\}}}\mathrm{d}\bar{s} \delta\left(t-\sum_{i=0}^2s_i\right)\prod_{i=1}^{2}e^{-s_i[\varsigma_{n-i}+\tau_i]} \prod_{i=1}^{2} e^{-{\bf i}t_i(s)\mathbf{X}_i}\right|.
		\end{aligned}
	\end{equation} 
	To evaluate $\mathfrak{A}_3$, we will again make use of the cut-off function $\Phi_{0,1}(\sigma_{0,\rho_1},  k_{0,\rho_1},\sigma_{1,\rho_2},  k_{0,\rho_2}) 
	= \Phi_{0}^b(\eth'_2,\sigma_{0,\rho_1},  k_{0,\rho_1},\sigma_{1,\rho_2},  k_{0,\rho_2}),$ which
	 leads to a factor of $\lambda^{(d-1)\eth'_2/3}$. The time integration $$\lambda^2\int_{(\mathbb{R}_+)^{\{0,1,2\}}}\mathrm{d}\bar{s} \delta\left(t-\sum_{i=0}^2s_i\right)e^{-s_1\tau_1},$$
	yields a factor of $\lambda^2\lambda^{-2-\theta_r}=\lambda^{-\theta_r}$.  Finally, a total factor of $\lambda^{(d-1)\eth'_2/3-\theta_r}$ is obtained. Since $(d-1)\eth'_2/3-\theta_r>0$, the factor  $\lambda^{(d-1)\eth'_2/3-2\theta_r}$   has a positive power of $\lambda$.

\end{proof}
\begin{proposition}[The final estimate on   the residual graphs]\label{lemma:Residual2} Denote by $ Q_{res}$ the sum of all the terms $Q^1_{res}+Q^2_{res}+Q^3_{res}+Q^4_{res}$ containing  $\mathfrak{C}^0_n$,  $\mathfrak{C}^2_n$ and   $\mathfrak{C}^3_n$, defined in   Propositions \ref{Proposition:Ampl1}, \ref{Proposition:Ampl2}, \ref{Proposition:Ampl3}
	\begin{equation}
		Q_{res} \ = \ Q^1_*+Q^2_*+Q^3_*+Q^4_*- Q^1-Q^2-Q^3-Q^4.
	\end{equation}
	
	For $t=\tau\lambda^{-2}>0$, we have, for any constants $1>T_*>0$ and $\mathscr{R}>0$,
	
	\begin{equation}
		\begin{aligned} \label{lemma:Residual2:1}
			& 
			\lim_{\lambda\to 0}\limsup_{D\to\infty} 
			\Big\|\mathscr{F}_{1/n,\mathscr{R}}\Big[\widehat{{Q}_{res}(\tau)\chi_{[0,T_*)}(\tau)}(\mathscr{X}{\lambda^{-2}})\Big]\Big\|_{L^{\infty,2}}
			= 0.
	\end{aligned}\end{equation}

\end{proposition}
\begin{proof} The proof is a combination of all of the previous estimates and is the same with  \cite{staffilani2021wave}[Proposition 96].
\end{proof}
\subsection{Ladder diagrams}
\label{Sec:SecondLadderDiagram}

\begin{proposition}[The dominance of ladder diagrams]\label{Propo:ReduceLadderGraph}
	Suppose that $t>0$ and $t=\mathcal{O}(\lambda^{-2})$. There  are universal constants $0<\lambda_0''<1$, $\mathfrak{C}^1_{ladder},\mathfrak{C}^2_{ladder}>0$ such that for $0<\lambda<\lambda_0''$, and for any constants $1>T_*>0$ and $\mathscr{R}>0$,
	
	\begin{equation}\label{Propo:ReduceLadderGraph:1}
		\begin{aligned}
			&\Big\|\limsup_{D\to\infty}\mathscr{F}_{1/n,\mathscr{R}}\Big[\widehat{Q_*\chi_{[0,T_*]}}(\mathscr{X}{\lambda^{-2}})\Big](k_{n,1},k_{n,2})\Big\|_{L^{\infty,2}(\mathbb{T}^{2d})}
			\le \  \mathfrak{C}^2_{ladder}\lambda^{\mathfrak{C}^1_{ladder}},\end{aligned}
	\end{equation}
	where  $Q_0=Q^1+Q^2+Q^3+Q^4$, $Q_*={Q}_0-\mathfrak{Q}_1-\mathfrak{Q}_2-\mathfrak{Q}_3$ and
	\begin{equation}\label{Propo:ReduceLadderGraph:2}
		\mathfrak{Q}_1\ = \ \sum_{n=0}^{\mathbf{N}-1}\sum_{S\in\mathcal{P}^1_{pair}(I_{n+2})}\mathfrak{G}_{1,n}'(S,t,k'',-1,k',1,\Gamma),
	\end{equation}
	\begin{equation}\label{Propo:ReduceLadderGraph:3}
		\mathfrak Q_{2} 
		\ = \ \sum_{n=0}^{\mathbf{N}-1}\sum_{S\in\mathcal{P}^1_{pair}(\{1,\cdots,n+2\})}\mathfrak{G}_{2,n}'(S,t,k'',-1,k',1,\Gamma),
	\end{equation}
	and
	\begin{equation}\label{Propo:ReduceLadderGraph:3}
		\mathfrak Q_{3} 
		\ = \ \sum_{n=1}^{\mathbf{N}}\sum_{S\in\mathcal{P}_{pair}^1(\{1,\cdots,n+2\})}\int_{0}^t\mathrm{d}s_0\mathfrak{G}_{3,n}'(S,s_0,t,k'',-1,k',1,\Gamma),
	\end{equation}
	in which $\mathcal{P}^1_{pair}(\{1,\cdots,n+2\})$ denotes the subset of $\mathcal{P}^0_{pair}(\{1,\cdots,n+2\})$, where we only count the terms that correspond to $\mathrm{iCL}_2$ ladder graphs. We denote  terms whose graphs are $\mathrm{iCL}_2$ ladder graphs by $\mathfrak{G}_{1,n}'$, $\mathfrak{G}_{2,n}'$, $\mathfrak{G}_{3,n}'$. In other words,  $\mathfrak{Q}_1$, $\mathfrak{Q}_2$, $\mathfrak{Q}_3$ are sum of terms coming from $Q^2,Q^3,Q^4$ that only correspond to ladder graphs.
	
\end{proposition}
\begin{proof}The proof follows directly from the previous Propositions of the section.
\end{proof}

\section{The proof of the main theorem}

In Proposition \ref{Propo:ReduceLadderGraph},  the whole Duhamel expansions have been reduced to  three ``ladder'' terms $\mathfrak Q_1$, $\mathfrak Q_2$ and $\mathfrak Q_3$.  
We first analyze  $\mathfrak{Q}_1$. Following the proof of the main theorem in \cite{staffilani2021wave}[Section 8], the ladder terms, summing from $[\mathbf{N}/4]$ to $\mathbf{N}-1$, are bounded by

\begin{equation}
	\begin{aligned} \label{FinalProof:E1}
		\sum_{n=[\mathbf{N}/4]}^{\mathbf{N}-1}  (4n)^n e^{T_*}\mathfrak{C}_{Q_{1,4}}^n n!\mathbf{N}^{-\wp	n_0(n-\mathbf{N}/2)}\frac{T_*^{1+{n}_0(n)}}{({n}_0(n))!}\mathfrak{C}_{Q1,4}^{1+{n}_1(n)}\langle \ln n\rangle \langle\ln\lambda\rangle^{1+(2+\eth){n}_1(n)},
\end{aligned}\end{equation}
where $[\mathbf{N}/4]$ denotes the integer that satisfies $[\mathbf{N}/4]\le \mathbf{N}/4<[\mathbf{N}/4]+1$ and $\mathfrak{C}_{Q1,4}$ � some universal constant. This   vanishes in the limit $\lambda\to 0$ by \eqref{Def:Para5}. 
As thus, we only need to estimate
\begin{equation}
	\begin{aligned} \label{FinalProof:E1}
		& \ \sum_{n=0}^{[\mathbf{N}/4]-1}\sum_{S\in\mathcal{P}^1_{pair}(\{1,\cdots,n+2\})}\mathfrak{G}_{1,n}'(S,t,k'',-1,k',1,\Gamma).
\end{aligned}\end{equation}

Let us consider the analytic expression of an $\mathrm{iCL}_2$ ladder graph,  with $n=2q$ vertices. This graph is formed by iteratively applying  the $\mathrm{iC}_2^r$ recollisions. Since $n=2q\le [\mathbf{N}/4]-1$, we deduce $\varsigma_{2 q-i}=0$ for all $i\in\{0,\cdots,2q\}$, due to \eqref{Def:Para3}. 
We will also  suppose that $\Phi_{1,i}=1$ since the difference $|\Phi_{1,i}-1|$   vanishes as $\lambda$ tends to $0$. A term in $\mathfrak{Q}_1$ then has the general form
\begin{equation}
	\begin{aligned}\label{FinalProof:E1}
		&(-1)^q \lambda^{2 q}
		\int_{(\Lambda^*)^{\mathcal{I}_{2q}}}\mathrm{d}\bar{k} \Delta_{2q,\rho}(\bar k,\bar\sigma) \prod_{i=1}^{2q}\Big[\sigma_{i,\rho_i}\mathcal{M}( k_{i,\rho_i}, k_{i-1,\rho_i}, k_{i-1,\rho_i+1})\Big] \\
		&\times 
	\prod_{A=\{i,j\}\in S}\left\langle  a(k_{0,i},\sigma_{0,i})a(k_{0,j},\sigma_{0,j})\right\rangle_{0}
		\int_{(\mathbb{R}_+)^{{\{0,\cdots,2q\}}}}\! \mathrm{d} \bar{s} \,
		\delta\left(t-\sum_{i=0}^{2 q}
		s_i\right) 
		\prod_{i=1}^{2 q} e^{-s_i \tau_i}
		\prod_{i=1}^{q} e^{-{\bf i} s_{2 i-1} \mathbf{X}_{2 i}}\\
		&\times\prod_{i=1}^{q} e^{-{\bf i} s_{2 i-2} [\mathbf{X}_{2 i-2}+\mathbf{X}_{2 i-1}]}e^{{-{\bf i}s_{2q}[\sigma_{n,2}\omega(k_{n,2})+\sigma_{n,1}\omega(k_{n,1})]}}.
	\end{aligned}
\end{equation}
Now, we need to sum over all possible ladders and take the limits. Let us   denote this sum by $\mathfrak{C}^\lambda_q(t)$.  To express  the explicit form of this sum, we define an input ``$q+1$-correlation function'', which represents all of the $q+1$ pairings at the bottom of the graph \begin{equation}
	\mathfrak{L}_{q+1}(k_{1}',\cdots,k_{q+1}',k_{1}'',\cdots,k_{q+1}''):=\prod_{A=\{i,j\}\in S}\left\langle  a(k_{0,i},1)a(k_{0,j},-1)\right\rangle_{0}, \end{equation}
in which $\{k_{1}',\cdots,k_{q+1}'\}$ and $\{k_{1}'',\cdots,k_{q+1}''\}$ are the momenta $\{k_{0,i}\}_{\{i,j\}\in S}$ and  $\{k_{0,j}\}_{\{i,j\}\in S}$ with a different way of indexing. In the above expression, we assume that $\sigma_{0,i}=1$ and $\sigma_{0,j}=-1$.
We  define
\begin{equation}
	\begin{aligned}\label{FinalProof:E9}
		&\mathcal{C}^\lambda_{i,q+1}(s',s,k_1',\cdots,k'_i,\cdots,k'_{q},k_{1}'',\cdots,k_i'',\cdots,k_{q}'')=   \int_{\Lambda^*}\int_{\Lambda^*}
		\mathrm{d}k'\mathrm{d}k_{q+1}'\int_{\Lambda^*}\int_{\Lambda^*}
		\mathrm{d}k''\mathrm{d}k_{q+1}''\\
		&\times\mathcal{M}(k_{i}',k',k_{q+1}')\mathcal{M}(k_{i}'',k'',k_{q+1}'')  \Big[e^{{\bf i}s(\omega(k')+\omega(k_{q+1}')-\omega(k_{i}'))-s\tau_{2i-1}} +e^{-{\bf i}s(\omega(k')+\omega(k_{q+1}')-\omega(k_{i}'))-s\tau_{2i-1}}\Big]\\
		&\times e^{-{\bf i}s(\omega(k')+\omega(k_{q+1}')+\omega(k_{i}')-\omega(k'')-\omega(k_{q+1}'')-\omega(k_{i}''))-s\tau_{2i-2}}\delta(k''+k_{q+1}''-k_{i}'')\\
		&\times\delta(k'+k_{q+1}'-k_{i}')\Big( \mathfrak{L}_{q+1}(k_1',\cdots,k',\cdots,k'_{q+1},k_1'',\cdots,k'',\cdots,k''_{q+1})\\
		&\ \ \ \ -\mathfrak{L}_{q+1}(k_1',\cdots,k'_i,\cdots,k'_{q+1},k_1'',\cdots,k'',\cdots,k''_{q+1})\mathrm{sign}(k_{i}')\mathrm{sign}(k_{q+1}')\\
		&\ \ \ \ -\mathfrak{L}_{q	+1}(k_1',\cdots,k'_i,\cdots,k',k_1'',\cdots,k'',\cdots,k''_{q+1})\mathrm{sign}(k_{i}')\mathrm{sign}(k')\Big),
\end{aligned}\end{equation}
in which $k',k''$ takes the positions of $k'_i,k_i''$ in $\mathfrak{L}_{q+1}(k_1',\cdots,k',$ $\cdots,k'_{q+1},k_1'',\cdots,k'',$ $\cdots,k''_{q+1})$ and of $k'_{q+1},k''_{q+1}$ in $\mathfrak{L}_{q+1}(k_1',\cdots,k'_i,\cdots,k',k_1'',\cdots,k''_i,\cdots,k'')$, and set 

\begin{equation}
	\begin{aligned}\label{FinalProof:E10:a}
		&\mathcal{C}^\lambda_{q+1}(s',s,k_1',\cdots,k'_i,\cdots,k'_{q},k_{1}'',\cdots,k_i'',\cdots,k_{q}'')\\
		= &\  \sum_{i=1}^{q}\mathcal{C}^\lambda_{i,q+1}(s',s,k_1',\cdots,k'_i,\cdots,k'_{q},k_{1}'',\cdots,k_i'',\cdots,k_{q}'').
\end{aligned}\end{equation}

This  procedure will be iterated all the way to  the original momenta denoted, without loss of generality, by $k_1',k_1''$, whose signs are $\sigma_1'$, $\sigma_1''$.  As a consequence, the explicit form of $\mathfrak{C}^\lambda_q(t)$ can be written as 
\begin{equation}
	\begin{aligned}\label{FinalProof:E10}
		\mathfrak{C}^\lambda_q(t)\ = \ &(-1)^q \lambda^{2 q}
		\int_{(\mathbb{R}_+)^{{\{0,\cdots,2q\}}}}\! \mathrm{d} \bar{s} \,
		\delta\left(t-\sum_{i=0}^{2 q}
		s_i\right) 
		[\mathcal{C}^\lambda_{2}(s_0,s_1)\cdots\mathcal{C}^\lambda_{q+1}(s_{2q-2},s_{2q-1})]\\
		&\times
	e^{-{{\bf i}s_{2q}[\sigma_{1}'\omega(k_1')+\sigma_{1}''\omega(k_1'')]}},
	\end{aligned}
\end{equation}
where $\mathcal{C}^\lambda_{2}(s_0,s_1),\cdots,\mathcal{C}^\lambda_{q}(s_{2q-4},s_{2q-3})$ are defined in the same manner as \eqref{FinalProof:E10:a}.

By the change of variables $s_{i}\to \lambda^{-2} s_{i}$ and $t\to \tau\lambda^{-2}$, we write
\begin{equation}
	\begin{aligned}\label{FinalProof:E10a}
		\mathfrak{C}^\lambda_q(\tau)\ = \ &
		(-1)^q \lambda^{-2q}\int_{(\mathbb{R}_+)^{{\{0,\cdots,2q\}}}}\! \mathrm{d} \bar{s} \delta\left(t-\sum_{i=0}^{2 q}
		s_i\right)
		[\mathcal{C}^\lambda_{2}(\lambda^{-2 }s_0,\lambda^{-2 }s_1)\cdots\mathcal{C}^\lambda_{q+1}(\lambda^{-2 }s_{2q-2},\lambda^{-2 }s_{2q-1})]\\
		&\times
		e^{-{{\bf i}s_{2q}[\sigma_{1}'\omega(k_1')+\sigma_{1}''\omega(k_1'')]}}.
	\end{aligned}
\end{equation}

As in Proposition \ref{lemma:BasicGEstimate1}, we can pass to the limit $D\to\infty$ and obtain

\begin{equation}
	\begin{aligned}\label{FinalProof:E10a:1}
		\mathfrak{C}^\lambda_{q,\infty}(\tau)\ = \ &
		(-1)^q \lambda^{-2q}\int_{(\mathbb{R}_+)^{{\{0,\cdots,2q\}}}}\! \mathrm{d} \bar{s} \delta\left(t-\sum_{i=0}^{2 q}
		s_i\right)
		[\mathcal{C}^\lambda_{2,\infty}(\lambda^{-2 }s_0,\lambda^{-2 }s_1)\cdots\mathcal{C}^\lambda_{q+1,\infty}(\lambda^{-2 }s_{2q-2},\lambda^{-2 }s_{2q-1})]\\
		&\times
		e^{-{{\bf i}s_{2q}[\sigma_{1}'\omega(k_1')+\sigma_{1}''\omega(k_1'')]}},
	\end{aligned}
\end{equation}
where
\begin{equation}
	\begin{aligned}\label{FinalProof:E10:a:1}
		&\mathcal{C}^\lambda_{m+1,\infty}(s',s,k_1',\cdots,k'_i,\cdots,k'_{m},k_{1}'',\cdots,k_i'',\cdots,k_{m}'')\\ 
		= &  \sum_{i=1}^{m}\mathcal{C}^\lambda_{i,m+1,\infty}(s',s,k_1',\cdots,k'_i,\cdots,k'_{m},k_{1}'',\cdots,k_i'',\cdots,k_{m}'')
\end{aligned}\end{equation}
with
\begin{equation}
	\begin{aligned}\label{FinalProof:E9:1}
		&\mathcal{C}^\lambda_{i,m+1,\infty}(s',s,k_1',\cdots,k'_i,\cdots,k'_{m},k_{1}'',\cdots,k_i'',\cdots,k_{m}'')=   \int_{\mathbb{T}^d}\int_{\mathbb{T}^d}
		\mathrm{d}k'\mathrm{d}k_{m+1}'\int_{\mathbb{T}^d}\int_{\mathbb{T}^d}
		\mathrm{d}k''\mathrm{d}k_{m+1}''\\
		&\times\mathcal{M}(k_{i}',k',k_{m+1}')\mathcal{M}(k_{i}'',k'',k_{m+1}'')  \Big[e^{{\bf i}s(\omega(k')+\omega(k_{m+1}')-\omega(k_{i}'))-s\tau_{2i-1}} +e^{-{\bf i}s(\omega(k')+\omega(k_{m+1}')-\omega(k_{i}'))-s\tau_{2i-1}}\Big]\\
		&\times e^{-{\bf i}s(\omega(k')+\omega(k_{m+1}')+\omega(k_{i}')-\omega(k'')-\omega(k_{m+1}'')-\omega(k_{i}''))-s\tau_{2i-2}}\delta(k''+k_{m+1}''-k_{i}'')\\
		&\times\delta(k'+k_{m+1}'-k_{i}')\Big( \mathfrak{L}_{m+1}(k_1',\cdots,k',\cdots,k'_{m+1},k_1'',\cdots,k'',\cdots,k''_{m+1})\\
		& \ \ \ \  -\mathfrak{L}_{m+1}(k_1',\cdots,k'_i,\cdots,k'_{q+1},k_1'',\cdots,k'',\cdots,k''_{m+1})\mathrm{sign}(k_{i}')\mathrm{sign}(k_{m+1}')\\
		& \ \ \ \  -\mathfrak{L}_{m	+1}(k_1',\cdots,k'_i,\cdots,k',k_1'',\cdots,k'',\cdots,k''_{m+1})\mathrm{sign}(k_{i}')\mathrm{sign}(k')\Big),
\end{aligned}\end{equation}
in which $k'$ takes the position of $k'_i$ in $\mathfrak{L}_{m+1}(k_1',\cdots,k',\cdots,k'_{m+1})$   and of $k'_{m+1}$ in $\mathfrak{L}_{m+1}(k_1',\cdots,$ $ k'_i,\cdots,k')$. We also have
\begin{equation}	\begin{aligned}
&	\mathfrak{L}_{m+1}^\infty(k_1',\cdots,k',\cdots,k'_{m+1},k_1'',\cdots,k'',\cdots,k''_{m+1})\\
	 \ = \ &\mathcal{C}^\lambda_{m+2,\infty}(k_1',\cdots,k',\cdots,k'_{m+1},k_1'',\cdots,k'',\cdots,k''_{m+1}),\end{aligned}
\end{equation}
and the $q+1$-correlation function of the first time slice is now defined to be
\begin{equation}
	\mathfrak{L}_{q+1}^\infty(k_{1}',\cdots,k_{q+1}'):=\prod_{A=\{i,j\}\in S}\left\langle  a(k_{0,i},1)a(k_{0,j},-1)\right\rangle_{0},\end{equation}
where $S$ is a pairing partition 
 that corresponds to the current ladder graph (see \eqref{Propo:ReduceLadderGraph:2}).
Similar as in the homogeneous case \cite{staffilani2021wave}, we  need to  take the limit $\lambda\to 0$ by iteratively applying
Lemma \ref{Lemma:Resonance1}. 
We thus set 
\begin{equation}
	\begin{aligned}
		\mathfrak{C}^\lambda_{q,\infty,a}(\tau) \ := \ &\lambda^{-2q}
\int_{(\mathbb{R}_+)^{{\{0,\cdots,2q\}}}}\! \mathrm{d} \bar{s} \delta\left(t-\sum_{i=0}^{2 q}
s_i\right)
	[\mathcal{C}^\lambda_{2,\infty}(\lambda^{-2 }s_0,\lambda^{-2 }s_1)\cdots\mathcal{C}^\lambda_{q+1,\infty}(\lambda^{-2 }s_{2q-2},\lambda^{-2 }s_{2q-1})]\\
	&\times
	e^{-{{\bf i}s_{2q}[\sigma_{1}'\omega(k_1')+\sigma_{1}''\omega(k_1'')]}},
	\end{aligned}
\end{equation}
and
\begin{equation}
	\begin{aligned}\label{FinalProof:E10a:1:1:B}
		\mathfrak{C}^\lambda_{q,\infty,a'}(\lambda^{-2}\tau')\
		:= \ &\lambda^{-2q}
		\int_{(\mathbb{R}_+)^{\{1,\cdots,2q-1\}}}\!\mathrm{d} s_{1} \cdots \mathrm{d} s_{2 q-1} 
		[\mathcal{C}^\lambda_{2,\infty}(\lambda^{-2 }s_0,\lambda^{-2 }s_1)\cdots\mathcal{C}^\lambda_{q+1,\infty}(\lambda^{-2 }s_{2q-2},\lambda^{-2 }s_{2q-1})]\\
		&\times	\delta\left(\tau'=\sum_{i=0}^{2 q-1}
		s_i\right),
	\end{aligned}
\end{equation}

Let us consider the functional $\mathscr{F}_{1/n,\mathscr{R}}$, defined in \eqref{eq:BasicGEstimate1:A}, acting on   $\widehat{\mathfrak{C}^\lambda_{q,\infty,a'}\chi_{[0,T_*)}}(\mathscr{X}{\lambda^{-2}})$, with $n=2q$
\begin{equation}
	\begin{aligned}\label{FinalProof:E10a:1:1:D}
		&	\mathscr{F}_{1/n,\mathscr{R}}[\widehat{\mathfrak{C}^\lambda_{q,\infty,a}\chi_{[0,T_*)}}(\mathscr{X}{\lambda^{-2}})]
		\\
		&	\ := \  \mathscr{F}_{1/n,\mathscr{R}}\left[\int_{\mathbb{R}}\mathrm{d}\tau \chi_{[0,T_*)}(\tau)	\int_{\mathbb{R}}\mathrm{d}\tau'\lambda^{-2} \mathfrak{C}^\lambda_{q,\infty,a'}(\lambda^{-2}\tau')\chi_{[0,\infty)}(\tau') \right.\\
		&\times\left. e^{-{{\bf i}(\tau-\tau')\lambda^{-2} [\sigma_{1}'\omega(k_1')+\sigma_{1}''\omega(k_1'')]}}\chi_{[0,\infty)}(\tau-\tau') e^{-{\bf i}2\pi\tau\lambda^{-2}\mathscr{X}}\right]
		\\
		&	\ = \   \mathscr{F}_{1/n,\mathscr{R}}\left[\int_{0}^\infty\mathrm{d}\tau''	\int_{0}^\infty\mathrm{d}\tau'\lambda^{-2} \mathfrak{C}^\lambda_{q,\infty,a'}(\lambda^{-2}\tau')e^{-{\bf i}2\pi\tau'\lambda^{-2}\mathscr{X}} e^{{-{\bf i}\tau''\lambda^{-2} [\sigma_{1}'\omega(k_1')+\sigma_{1}''\omega(k_1'')]}}\right.\\
		&\ \ \ \ \ \ \ \ \ \ \ \times\left.e^{-{\bf i}2\pi\tau''\mathscr{X}}\mathbf{1}\big(\tau'+\tau''\le T_*\big)\right]\\
		&	\ = \   \mathscr{F}_{1/n,\mathscr{R}}\left[\int_{0}^{T_*}\mathrm{d}\tau''	\int_{0}^{T_*-\tau''}\mathrm{d}\tau'\lambda^{-2} \mathfrak{C}^\lambda_{q,\infty,a'}(\lambda^{-2}\tau')\right.\\
		&\ \ \ \ \ \ \ \ \ \times\left.e^{-{\bf i}2\pi\tau'\lambda^{-2}\mathscr{X}} e^{{-{\bf i}\tau''\lambda^{-2} [\sigma_{1}'\omega(k_1')+\sigma_{1}''\omega(k_1'')]}}e^{-{\bf i}2\pi\tau''\lambda^{-2}\mathscr{X}} \right]
		\\
		&	\ = \  \mathscr{F}_{1/n,\mathscr{R}}\left[\int_{0}^{T_*}\mathrm{d}\tau''	\int_{0}^{(T_*-\tau'')\lambda^{-2}}\mathrm{d}\bar\tau \mathfrak{C}^\lambda_{q,\infty,a'}(\bar\tau)e^{-{\bf i}2\pi\bar\tau\mathscr{X}} \right.
	\\
	&\ \ \ \ \ \ \ \ \ \times\left.	e^{-{{\bf i}\tau''\lambda^{-2} [\sigma_{1}'\omega(k_1')+\sigma_{1}''\omega(k_1'')]}}e^{-{\bf i}2\pi \tau''\lambda^{-2}\mathscr{X}} \right]
		\\
		&	\ \lesssim \  \left[\int_{(-\mathscr{R},\mathscr{R})}\mathrm{d}\mathscr{X}\left|\int_{0}^{T_*}\mathrm{d}\tau''	\left|\int_{0}^{(T_*-\tau'')\lambda^{-2}}\mathrm{d}\bar\tau \mathfrak{C}^\lambda_{q,\infty,a'}(\bar\tau)e^{-{\bf i}2\pi\bar\tau\mathscr{X}}\right| \right|^{1/n}\right]^{{n}},
	\end{aligned}
\end{equation}
where we have used the change of variables $\bar\tau=\lambda^{-2}\tau'$ and $\tau''=\tau-\tau'$. Using H\"older's inequality, we bound

\begin{equation}
	\begin{aligned}\label{FinalProof:E10a:1:1c}
		&	\mathscr{F}_{1/n,\mathscr{R}}[\widehat{\mathfrak{C}^\lambda_{q,\infty,a}\chi_{[0,T_*)}}(\mathscr{X}{\lambda^{-2}})]
		\\
		&	\lesssim\ \mathfrak{C}_{final,1}^{q}|\mathscr{R}|^{\frac{q}{2}}
		\sup_{\tau'\in[0,T_*]}\left[
		\int_{(\mathbb{R}_+)^{\{1,\cdots,2q-1\}}}\!\mathrm{d} s_{1} \cdots s_{2i-1}\cdots \mathrm{d} s_{2 q-1} 
		\delta\left(\tau'=\sum_{i=1}^{q}
		s_{2i-1}\right)\right.\\
		&\times\left(\int_{(-\mathscr{R},\mathscr{R})}\mathrm{d}\mathscr{X}\left|\int_{0}^{s_1\lambda^{-2}}\mathrm{d}s\int_{0}^{s}\mathrm{d}s'\mathcal{C}^\lambda_{2,\infty}(s',s-s')e^{-{\bf i}2\pi s\mathscr{X}}\right|^2\right)^\frac12\\
		&\cdots\left.\left(\int_{(-\mathscr{R},\mathscr{R})}\mathrm{d}\mathscr{X}\left|\int_{0}^{s_{2q-1}\lambda^{-2}}\mathrm{d}s\int_{0}^{s}\mathrm{d}s'\mathcal{C}^\lambda_{q+1,\infty}(s',s-s')e^{-{\bf i}2\pi s\mathscr{X}}\right|^2\right)^\frac12\right],
	\end{aligned}
\end{equation}
for some constant $\mathfrak{C}_{final,1}>0$ and $n=2q$. We now  estimate the quantity $$\left(\int_{(-\mathscr{R},\mathscr{R})}\mathrm{d}\mathscr{X}\left|\int_{0}^{s_{2i-1}\lambda^{-2}}\mathrm{d}s\int_{0}^{s}\mathrm{d}s'\mathcal{C}^\lambda_{i+1,\infty}(s',s-s')e^{-{\bf i}2\pi s\mathscr{X}}\right|^2\right)^\frac12,$$ for $i=1,\cdots, q$. Let $\chi_{[0,\mathscr{C})}(s)$ be the cut-off function of the variable $s$ on $[0,\mathscr{C})$, for a  small enough constant $\mathscr{C}>0$, we estimate
\begin{equation}
	\begin{aligned}\label{FinalProof:E10a:1:1c:a}
		&\left(\int_{(-\mathscr{R},\mathscr{R})}\mathrm{d}\mathscr{X}\left|\int_{0}^{s_{2i-1}\lambda^{-2}}\mathrm{d}s\int_{0}^{s}\mathrm{d}s'\mathcal{C}^\lambda_{i+1,\infty}(s',s-s')e^{-{\bf i}2\pi s\mathscr{X}}\right|^2\right)^\frac12 \\
		\lesssim \ & \left(\int_{(-\mathscr{R},\mathscr{R})}\mathrm{d}\mathscr{X}\left|\int_{0}^{\mathscr{C}}\mathrm{d}s\int_{0}^{s}\mathrm{d}s'\mathcal{C}^\lambda_{i+1,\infty}(s',s-s')e^{-{\bf i}2\pi s\mathscr{X}}\right|^2\right)^\frac12\\
		&\ \ \ \  + \left(\int_{(-\mathscr{R},\mathscr{R})}\mathrm{d}\mathscr{X}\left|\int_{\mathscr{C}}^{s_{2i-1}\lambda^{-2}}\mathrm{d}s\int_{0}^{s}\mathrm{d}s'\mathcal{C}^\lambda_{i+1,\infty}(s',s-s')e^{-{\bf i}2\pi s\mathscr{X}}\right|^2\right)^\frac12 \\
		\lesssim \ & \left|\int_{0}^{\mathscr{C}}\mathrm{d}s\Big|\int_{0}^{s}\mathrm{d}s'\mathcal{C}^\lambda_{i+1,\infty}(s')\Big|^2\right|^\frac12 + \left(\int_{(-\mathscr{R},\mathscr{R})}\mathrm{d}\mathscr{X}\left|\int_{\mathscr{C}}^{s_{2i-1}\lambda^{-2}}\mathrm{d}s\int_{0}^{s}\mathrm{d}s'\mathcal{C}^\lambda_{i+1,\infty}(s',s-s')e^{-{\bf i}2\pi s\mathscr{X}}\right|^2\right)^\frac12 \\
		\lesssim \ & \left|\int_{0}^{s_{2i-1}\lambda^{-2}}\mathrm{d}s\Big|\int_{0}^{s}\mathrm{d}s'\mathcal{C}^\lambda_{i+1,\infty}(s',s-s')\Big|^\mathscr{M}\right|^\frac{1}{\mathscr{M}}\\
		& \ \ \ \ + \left(\int_{(-\mathscr{R},\mathscr{R})}\mathrm{d}\mathscr{X}\left|\int_{\mathscr{C}}^{s_{2i-1}\lambda^{-2}}\mathrm{d}s\int_{0}^{s}\mathrm{d}s'\mathcal{C}^\lambda_{i+1,\infty}(s',s-s')e^{-{\bf i}2\pi s\mathscr{X}}\right|^2\right)^\frac12,
	\end{aligned}
\end{equation}
where $\mathscr{M}$ is computed in Lemma \ref{Lemma:Resonance1}.

We define
\begin{equation}
	\begin{aligned}\label{FinalProof:E10a:1:1c:b}
		& \mathscr{B}_{0,\mathscr{R}}(\mathscr{X})\ :=\ \int_{\mathscr{C}}^{s_{2i-1}\lambda^{-2}}\mathrm{d}s\int_{0}^{s}\mathrm{d}s'\mathcal{C}^\lambda_{i+1,\infty}(s',s-s')e^{-{\bf i}2\pi s\mathscr{X}},\\
		& \mathscr{B}_{-1,\mathscr{R}}(\mathscr{X})\ :=\ \int_{\mathscr{C}}^{s_{2i-1}\lambda^{-2}}\mathrm{d}s\int_{0}^{s}\mathrm{d}s'\mathcal{C}^\lambda_{i+1,\infty}(s',s-s')e^{-{\bf i}2\pi s\mathscr{X}}[-{\bf i}2\pi s]^{-1},\\
		& \mathscr{B}_{1,\mathscr{R}}(\mathscr{X})\ :=\ \int_{\mathscr{C}}^{s_{2i-1}\lambda^{-2}}\mathrm{d}s\int_{0}^{s}\mathrm{d}s'\mathcal{C}^\lambda_{i+1,\infty}(s',s-s')e^{-{\bf i}2\pi s\mathscr{X}}[-{\bf i}2\pi s].
	\end{aligned}
\end{equation}
Observing that $$\frac{\mathrm{d}^2}{\mathrm{d}\mathscr{X}^2 }\mathscr{B}_{-1,\mathscr{R}}=\frac{\mathrm{d}}{\mathrm{d}\mathscr{X} }\mathscr{B}_{0,\mathscr{R}}=\mathscr{B}_{1,\mathscr{R}}$$ and letting	 $\iota>0$ be a small constant, $\mathscr{R}'>>\mathscr{R}$ be a sufficiently large constant, and $\mathscr{T}(\mathscr{X})$ be a smooth cut-off function satisfying $\mathscr{T}(\mathscr{X})=1$ for $-\mathscr{R}'+2\iota<\mathscr{X}<\mathscr{R}'-2\iota$, $\mathscr{T}(\mathscr{X})=0$ for $\mathscr{X}<-\mathscr{R}'+\iota$ or $\mathscr{X}>\mathscr{R}'-\iota$ and $0\le \mathscr{T}(\mathscr{X})\le 1$ elsewhere, we construct   
$$\widetilde{\mathscr{B}_{-1,\mathscr{R}}}(\mathscr{X}):=  \int_{\mathscr{C}}^{\mathscr{J}}\mathrm{d}s\mathcal{C}^\lambda_{i+1,\infty}(s)e^{-{\bf i}2\pi s\mathscr{X}}[-{\bf i}2\pi s]^{-1}\mathscr{T}(\mathscr{X}), $$ $$\frac{\mathrm{d}^2}{\mathrm{d}\mathscr{X}^2 }\widetilde{\mathscr{B}_{-1,\mathscr{R}}}=\frac{\mathrm{d}}{\mathrm{d}\mathscr{X} }\widetilde{\mathscr{B}_{0,\mathscr{R}}}=\widetilde{\mathscr{B}_{1,\mathscr{R}}},$$
for a large enough constant $\mathscr{J}>0.$
We deduce that $\widetilde{\mathscr{B}_{1,\mathscr{R}}}\in H_0^1(-\mathscr{R}',\mathscr{R}')$.
Suppose that $\mathscr{H}(\mathscr{X})$ be the solution of the  equation $\frac{\mathrm{d}^2}{\mathrm{d}\mathscr{X}^2 }\mathscr{H}(\mathscr{X})=\widetilde{\mathscr{B}_{1,\mathscr{R}}},$ in the domain $(-\mathscr{R}',\mathscr{R}')$, with the Dirichlet boundary condition $\mathscr{H}(-\mathscr{R}')=\mathscr{H}(\mathscr{R}')=0.$ Let $\phi$ be an arbitrary function in $L^2(-\mathscr{R},\mathscr{R})$ and $\bar\phi$ be a solution of the elliptic equation $\bar\phi''=\phi$ in $(-\mathscr{R},\mathscr{R})$, with the Dirichlet boundary condition $\bar\phi(-\mathscr{R})=\bar\phi(\mathscr{R})=0.$ Observing that
$$\int_{-\mathscr{R}}^{\mathscr{R}}\mathrm{d}\mathscr{X}\mathscr{H}''\bar\phi'=\int_{-\mathscr{R}}^{\mathscr{R}}\mathrm{d}\mathscr{X}\widetilde{\mathscr{B}_{1,\mathscr{R}}}\bar\phi',$$
and
$$\int_{-\mathscr{R}}^{\mathscr{R}}\mathrm{d}\mathscr{X}\mathscr{H}'\phi=\int_{-\mathscr{R}}^{\mathscr{R}}\mathrm{d}\mathscr{X}\widetilde{\mathscr{B}_{0,\mathscr{R}}}\phi,$$ we find $\|\mathscr{H}'\|_{L^2(-\mathscr{R},\mathscr{R})}=\|\widetilde{\mathscr{B}_{0,\mathscr{R}}}\|_{L^2(-\mathscr{R},\mathscr{R})}$. Since $\widetilde{\mathscr{B}_{1,\mathscr{R}}}\in H_0^1(-\mathscr{R}',\mathscr{R}')$, we obtain $
	\label{FinalProof:E10a:1:1c:d}\|\mathscr{H}'\|_{L^2(-\mathscr{R},\mathscr{R})}\le \|\mathscr{H}\|_{L^2(-\mathscr{R}',\mathscr{R}')}.$
For any $\phi_o\in H_0^1(-\mathscr{R}',\mathscr{R}')$, we observe 
$$\int_{-\mathscr{R}'}^{\mathscr{R}'}\mathrm{d}\mathscr{X}\mathscr{H}'\phi_o=\int_{-\mathscr{R'}}^{\mathscr{R}'}\mathrm{d}\mathscr{X}\widetilde{\mathscr{B}_{0,\mathscr{R}}}\phi_o,$$
which implies
$$\left|\int_{-\mathscr{R}'}^{\mathscr{R}'}\mathrm{d}\mathscr{X}\mathscr{H}\phi_o'\right|=\left|\int_{-\mathscr{R}'}^{\mathscr{R}'}\mathrm{d}\mathscr{X}\widetilde{\mathscr{B}_{-1,\mathscr{R}}}\phi_o'\right|\le \|\widetilde{\mathscr{B}_{-1,\mathscr{R}}}\|_{L^2(-\mathscr{R}',\mathscr{R}')}\|\phi_o\|_{H_0^1(-\mathscr{R}',\mathscr{R}')}.$$
We thus deduce  \begin{equation}\begin{aligned}
	\label{FinalProof:E10a:1:1c:e} &\|\mathscr{H}\|_{H^{-1}(-\mathscr{R}',\mathscr{R}')}=\|\mathscr{H}\|_{L^2(-\mathscr{R}',\mathscr{R}')} \le \|\widetilde{\mathscr{B}_{-1,\mathscr{R}}}\|_{L^2(-\mathscr{R}',\mathscr{R}')} \\ \lesssim & \left|\int_{\mathscr{C}}^{\mathscr{J}}\mathrm{d}s \Big|\frac{1}{s}\int_{0}^{s}\mathrm{d}s'\mathcal{C}^\lambda_{i+1,\infty}(s',s-s')\Big|^2\right|^\frac{1}{2}\\ \lesssim &\left|\int_{\mathscr{C}}^{\mathscr{J}}\mathrm{d}s \Big|\frac{1}{s}\Big|^\frac{2\mathscr{M}}{\mathscr{M}-2}\right|^\frac{\mathscr{M}-2}{2\mathscr{M}} \left|\int_{\mathscr{C}}^{\mathscr{J}}\mathrm{d}s \Big|\int_{0}^{s}\mathrm{d}s'\mathcal{C}^\lambda_{i+1,\infty}(s',s-s')\Big|^\mathscr{M}\right|^\frac{1}{\mathscr{M}} \\ \lesssim &\ \left|\int_{\mathscr{C}}^{\mathscr{J}}\mathrm{d}s \Big|\int_{0}^{s}\mathrm{d}s'\mathcal{C}^\lambda_{i+1,\infty}(s',s-s')\Big|^\mathscr{M}\right|^\frac{1}{\mathscr{M}} \\ \lesssim & \left|\int_{0}^{s_{2i-1}\lambda^{-2}}\mathrm{d}s\Big|\int_{0}^{s}\mathrm{d}s'\mathcal{C}^\lambda_{i+1,\infty}(s',s-s')\Big|^\mathscr{M}\right|^\frac{1}{\mathscr{M}} ,\end{aligned}\end{equation}
which leads to 
$$\|\widetilde{\mathscr{B}_{0,\mathscr{R}}}\|_{L^2(-\mathscr{R},\mathscr{R})}\ = \ \|\mathscr{H}'\|_{L^2(-\mathscr{R},\mathscr{R})}\ \lesssim \ \left|\int_{0}^{s_{2i-1}\lambda^{-2}}\mathrm{d}s\Big|\int_{0}^{s}\mathrm{d}s'\mathcal{C}^\lambda_{i+1,\infty}(s',s-s')\Big|^\mathscr{M}\right|^\frac{1}{\mathscr{M}}.$$
Thus, we have
\begin{equation}\begin{aligned}
		\label{FinalProof:E10a:1:1c:e:1}  
		&	\left(\int_{(-\mathscr{R},\mathscr{R})}\mathrm{d}\mathscr{X}\left|\int_{\mathscr{C}}^{\mathscr{J}}\mathrm{d}s\int_{0}^{s}\mathrm{d}s'\mathcal{C}^\lambda_{i+1,\infty}(s',s-s')e^{-{\bf i}2\pi s\mathscr{X}}\right|^2\right)^\frac12 \\
	\lesssim \ & \left|\int_{0}^{s_{2i-1}\lambda^{-2}}\mathrm{d}s\Big|\int_{0}^{s}\mathrm{d}s'\mathcal{C}^\lambda_{i+1,\infty}(s',s-s')\Big|^\mathscr{M}\right|^\frac{1}{\mathscr{M}},\end{aligned}\end{equation}
which holds for all $\mathscr{J}>0$. As a result, we can set  $\mathscr{J}\to \lambda^{-2}s_{2i-1}$ while \eqref{FinalProof:E10a:1:1c:f} still holds true, yielding
\begin{equation}
	\begin{aligned}\label{FinalProof:E10a:1:1c:f}
	&	\left(\int_{(-\mathscr{R},\mathscr{R})}\mathrm{d}\mathscr{X}\left|\int_{0}^{s_{2i-1}\lambda^{-2}}\mathrm{d}s\int_{0}^{s}\mathrm{d}s'\mathcal{C}^\lambda_{i+1,\infty}(s',s-s')e^{-{\bf i}2\pi s\mathscr{X}}\right|^2\right)^\frac12\\
		\lesssim \ & \left|\int_{0}^{s_{2i-1}\lambda^{-2}}\mathrm{d}s\Big|\int_{0}^{s}\mathrm{d}s'\mathcal{C}^\lambda_{i+1,\infty}(s',s-s')\Big|^\mathscr{M}\right|^\frac{1}{\mathscr{M}},
	\end{aligned}
\end{equation}
which implies

\begin{equation}
	\begin{aligned}\label{FinalProof:E10a:1:1c:g}
		&	\|\mathscr{F}_{1/q,\mathscr{R}}[\widehat{ \mathfrak{C}^\lambda_{q,\infty,a}\chi_{[0,T_*)}}(\mathscr{X}{\lambda^{-2}})]\|_{L^2(\mathbb{T}^{2d})}
		\\
		\lesssim\ &\mathfrak{C}_{final,1}^{q}|\mathscr{R}|^{\frac{q}{2}}
		\sup_{\tau'\in[0,T_*]}\left[
		\int_{(\mathbb{R}_+)^{\{1,\cdots,q\}}}\!\mathrm{d} s_{1} \cdots \mathrm{d} s_{2 q-1} 
		\lambda^{-2q}\right.\\
		&\times\Big\|\int_{0}^{\lambda^{-2 }s_1}\mathrm{d}s_1'\mathcal{C}^\lambda_{2,\infty}(s'_1,\lambda^{-2 }s_1-s'_1)\cdots\int_{0}^{\lambda^{-2 }s_{2q-1}}\mathrm{d}s_{2q-1}'\mathcal{C}^\lambda_{q+1,\infty}(s_{2q-1}',\lambda^{-2 }s_{2q-1}-s_{2q-1}')\Big\|_{L^{\infty,2}(\mathbb{T}^{2d})}^\mathscr{M}(\tau')\\
		&\ \ \ \ \ \ \left. \times
		\mathbf{1}\left(\sum_{i=1}^{q} s_{2i-1} \le \tau'  \right)\right]^\frac{1}{\mathscr{M}},
	\end{aligned}
\end{equation}
We define 
\begin{equation}
	\begin{aligned}\label{FinalProof:E10a:1:3}
		\mathfrak{C}^\lambda_{q,\infty,b}(\tau)\ := \ & 
		\int_{0}^{\lambda^{-2 }s_1}\mathrm{d}s_1'\mathcal{C}^\lambda_{2,\infty}(s'_1,\lambda^{-2 }s_1-s'_1)\cdots\int_{0}^{\lambda^{-2 }s_{2q-1}}\mathrm{d}s_{2q-1}'\mathcal{C}^\lambda_{q+1,\infty}(s_{2q-1}',\lambda^{-2 }s_{2q-1}-s_{2q-1}'),
	\end{aligned}
\end{equation}
and
\begin{equation}
	\begin{aligned}\label{FinalProof:E10a:1:4}
		\mathfrak{C}^\lambda_{q,\infty,c}(\tau)\ := \ & \lambda^{-2q}
		\int_{(\mathbb{R}_+)^{\{1,\cdots,q\}}}\!\mathrm{d} s_{1} \cdots \mathrm{d} s_{2 q-1} 
		\Big\|\mathfrak{C}^\lambda_{q,\infty,b}\Big\|_{L^{\infty,2}(\mathbb{T}^{2d})}^\mathscr{M}
		\mathbf{1}\left(\sum_{i=1}^{q} s_{2i-1} \le \tau \right).
	\end{aligned}
\end{equation}
As discussed above, the operator $\mathfrak{C}^\lambda_{q,\infty,b}(\tau)$   is an iterative application of the $\mathrm{iC}_2^r$ recollisions. Thus,  $\mathfrak{C}^\lambda_{q,\infty,b}(\tau)$  can be iterated by applying ladder operators of the form (see also \eqref{DelayedRecollision:A5b}-\eqref{DelayedRecollision:A5c}-\eqref{DelayedRecollision:A5d})
\begin{equation}\label{FinalProof:E10a:1:5}
	\begin{aligned} 
		& 
		\mathscr{Q}^{ladder}_{s',s,\sigma_0,\sigma_1,\sigma_2}[\mathscr F_0,\mathscr F_1,\mathscr F_2](k_0,k_0')\\ := \ &\iint_{(\mathbb{T}^d)^4}\!\! \mathrm{d} k_1  \mathrm{d} k_2\mathrm{d} k_1'  \mathrm{d} k_2' \,
		\delta(\sigma_0k_0+\sigma_1k_1+\sigma_2k_2) \delta(\sigma_0k_0'+\sigma_1k_1'+\sigma_2k_2')\\
		&\times [e^{{\bf i}s\sigma_0 \omega(k_0)+{\bf i}s\sigma_1 \omega(k_1)+{\bf i} s \sigma_2 \omega(k_2)}+e^{-{\bf i}s\sigma_0 \omega(k_0)-{\bf i}s\sigma_1 \omega(k_1)-{\bf i} s \sigma_2 \omega(k_2)}]\\
	&\times 	e^{-{\bf i}s' (\omega(k_0)-\omega(k_0'))-{\bf i}s' (\omega(k_1)-\omega(k_1'))-{\bf i} s' (\omega(k_2)-\omega(k_2))}\\
		&\times  \sqrt{|\sin(2\pi k_0^1)||\sin(2\pi k_1^1)||\sin(2\pi k_2^1)|}\sqrt{|\sin(2\pi k_0'^1)||\sin(2\pi k_1'^1)||\sin(2\pi k_2'^1)|}\\
		&\times\mathscr{F}_0(k_0,k_0')\mathscr {F}_1(k_1,k_1')\mathscr F_2(k_2,k_2') ,
	\end{aligned}
\end{equation}
where $\mathscr F_0,\mathscr F_1,\mathscr F_2$ can be $1, \langle a_{k_0,1}a_{k_0',-1}\rangle_0$  or $\mathscr{Q}^{ladder}_{s''',s'',\sigma_0',\sigma_1',\sigma_2'}$, where the parameters $s'',s''',\sigma_0',\sigma_1',\sigma_2'$ come from the previous iterations. We notice that $k_0-k_0'=\mathcal{O}(\epsilon)$, $k_1-k_1'=\mathcal{O}(\epsilon)$, $k_2-k_2'=\mathcal{O}(\epsilon)$ by Assumption (B) on the initial condition.  The operators \eqref{FinalProof:E10a:1:5} are classified into two main types. 

\smallskip

{\bf Type I:} If the time slice $s$ corresponds to a single-cluster recollision, the operator has the form
\begin{equation}\label{FinalProof:single-cluster}
	\begin{aligned} 
		& 
		\mathscr{Q}^{ladder}_{s',s,\sigma_0,\sigma_1,\sigma_2}[\mathscr F_0,1,\mathscr F_2](k_0,k_0'), \qquad \mbox{ or } \qquad  \mathscr{Q}^{ladder}_{s',s,\sigma_0,\sigma_1,\sigma_2}[\mathscr F_0,\mathscr F_1,1](k_0,k_0')
	\end{aligned}
\end{equation}
where  the function $\mathscr F_0$ is either $\langle a_{k_0,\sigma_0}a_{k_0',-\sigma_0}\rangle_0$ or $1$, and the last function is the product of several operators of the type $\mathscr{Q}^{ladder}_{s''',s'',\sigma_0',\sigma_1',\sigma_2'}$ and $\langle a_{k_0,\sigma_0}a_{k_0',-\sigma_0}\rangle_0$ or $1$, which  come from the previous time slices. The last function is assumed  to be  $\mathscr F_2(k)$ has two explicit representations. 
\begin{itemize}
	\item[(i)] The first representation of $\mathscr F_2(k,k')$ (or $\mathscr F_1(k,k')$) has the form
	
	\begin{equation}\label{FinalProof:product1}
		\begin{aligned} 
			\mathscr P_1(k,k') \ =	\ & 
			\mathscr	P_1^0(k,k') \prod_{i=1}^{m} \mathscr{Q}^{ladder}_{s'^i,s^i,\sigma_0^i,\sigma_1^i,\sigma_2^i}\Big[1,1,\mathscr P_1^i\Big](k,k'),
		\end{aligned}
	\end{equation}
	where $\mathscr	P_1^0(k,k')$ is $\langle a_{k,\sigma}a_{k',-\sigma}\rangle_0$, with $m\in\mathbb{N}$, $m\ge 0$ and   the functions $\mathscr	P_1^i(k,k')$  are either  $\langle a_{k,\sigma}a_{k',-\sigma}\rangle_0$ or ladder operators obtained from the previous iterations, also of the types \eqref{FinalProof:product1}-\eqref{FinalProof:product2}, with $i=1,\cdots, m$. 
	\item[(ii)] The second representation of $\mathscr F_2(k)$ (or $\mathscr F_1(k,k')$) has the form
	\begin{equation}\label{FinalProof:product2}
		\begin{aligned} 
			\mathscr P_2(k,k') \ =	\ & 
			\mathscr{Q}^{ladder}_{s'^i,s^1,\sigma_0^1,\sigma_1^l,\sigma_2^l}\Big[1,\mathscr P_1^0,\mathscr P_2^0\Big](k,k')\prod_{i=1}^{m} \mathscr{Q}^{ladder}_{s'^i,s^i,\sigma_0^i,\sigma_1^i,\sigma_2^i}\Big[1,1,\mathscr P_2^i\Big](k,k'),
		\end{aligned}
	\end{equation}
\end{itemize}
where $\mathscr{Q}^{ladder}_{s'^i,s^1,\sigma_0^1,\sigma_1^i,\sigma_2^i}[1,\mathscr P_1^0,\mathscr P_2^0]$ corresponds to a double-cluster recollision, which we will describe  below and  $m\in\mathbb{N}$, $m\ge 0$. The functions  $\mathscr	P_1^0(k,k')$, $\mathscr	P_2^i(k,k')$ are $\langle a_{k,\sigma}a_{k',-\sigma}\rangle_0$ or ladder operators obtained from the previous iterations, also of the types \eqref{FinalProof:product1}-\eqref{FinalProof:product2}, $i=0,\cdots, m$.  

\smallskip
{\bf Type II:} If the time slice $s$ corresponds to a double-cluster recollision, the operator has the form
\begin{equation}\label{FinalProof:double-cluster}
	\begin{aligned} 
		& 
		\mathscr{Q}^{ladder}_{s',s,\sigma_0,\sigma_1,\sigma_2}[1,\mathscr F_1,\mathscr F_2](k_0,k_0'),
	\end{aligned}
\end{equation}
in which the function $\mathscr F_0$ has to be $1$ and the two functions $\mathscr F_1,\mathscr F_2$ also have the forms \eqref{FinalProof:product1} and \eqref{FinalProof:product2}.

The above strategy is applied iteratively from the top to the bottom of the graph to obtain an explicit presentation  of $\mathfrak{C}^\lambda_{q,\infty,c}$.  We will also need to obtain  the bound for $\mathfrak{C}^\lambda_{q,\infty,c}$ (defined in \eqref{FinalProof:E10a:1:4}), that can be done by  interatively applying Lemma \ref{Lemma:Resonance1} to all of the ladder operator  from the top to the bottom of the graph. The procedure is described as follow.

{\bf Strategy  (A).} If we encounter a ladder operator of the type \eqref{FinalProof:double-cluster},   which is not included in a product of the type \eqref{FinalProof:product2} with $m\ge1$, our general strategy is to apply \eqref{Lemma:Resonance1:2:bis1} for $F_1 =\mathscr F_1$ and $F_2 =\mathscr F_2$, with the notice that   $k_0-k_0'=\mathcal{O}(\epsilon)$, $k_1-k_1'=\mathcal{O}(\epsilon)$, $k_2-k_2'=\mathcal{O}(\epsilon)$ for all of the momenta presented in \eqref{FinalProof:E10a:1:5} by Assumption (B) on the initial condition. Note that $\epsilon$ is defined by \eqref{Epsilon}.

{\bf Strategy (B).}  If we encounter a ladder operator of the type \eqref{FinalProof:single-cluster}, which is not included in a product of the type \eqref{FinalProof:product1} with $m\ge1$, and if  $\mathscr F_2$ is the function that is different from $1$, our general strategy is to apply \eqref{Lemma:Resonance1:2:bis2} for $F_2=\mathscr F_2$, with the notice that   $k_0-k_0'=\mathcal{O}(\epsilon)$, $k_1-k_1'=\mathcal{O}(\epsilon)$, $k_2-k_2'=\mathcal{O}(\epsilon)$ for all of the momenta appearing in \eqref{FinalProof:E10a:1:5}.
\smallskip

{\bf Strategy	(C).} After applying (A) and (B), we will encounter products of the types \eqref{FinalProof:product1} and \eqref{FinalProof:product2}, whose $L^2$-norms must be controlled.  Our strategies will be discussed below.

\smallskip
{\it Strategy	(C.1).} 	 First, we provide a treatment for the product \eqref{FinalProof:product1}. We bound

\begin{equation}\label{FinalProof:E10a:1:6:3}
	\begin{aligned} 
		\|	\mathscr P_1(k,k') \|_{L^2(\mathbb{T}^{2d})} \ \le	\ & 
		\|\langle a_{k_0,\sigma_0}a_{k_0',-\sigma_0}\rangle_0 \|_{L^{\infty,2}(\mathbb{T}^d)}\prod_{l=1}^{m}\Big\| \mathscr{Q}^{ladder}_{s^l,\sigma_0^l,\sigma_1^l,\sigma_2^l}\Big[1,1,\mathscr P_2^l\Big]\Big\|_{L^{\infty}(\mathbb{T}^{2d})}.
	\end{aligned}
\end{equation}
Next, we apply \eqref{Lemma:Resonance1:9} to bound $\Big\| \mathscr{Q}^{ladder}_{s^l,\sigma_0^l,\sigma_1^l,\sigma_2^l}\Big[1,1,\mathscr P_2^l\Big](k,k')\Big\|_{L^{\infty}(\mathbb{T}^{2d})}$ by $\|\mathscr P_2^l\|_{L^{\infty,2}(\mathbb{T}^{2d})}$, so that in the next iteration, (A), (B) or (C) can be reapplied.

\smallskip
{\it Strategy	(C.2).}  Now, we provide a treatment for the product \eqref{FinalProof:product2}. We bound

\begin{equation}\label{FinalProof:E10a:1:6:4}
	\begin{aligned} 
		\|	\mathscr P_2(k,k') \|_{L^{\infty,2}(\mathbb{T}^{2d})} \ \le	\ & 
		\Big\| \mathscr{Q}^{ladder}_{s^1,\sigma_0^1,\sigma_1^l,\sigma_2^l}\Big[1,\mathscr P_1^0,\mathscr P_2^0\Big](k,k')  \Big\|_{L^{\infty,2}(\mathbb{T}^{2d})}\\
		&\times\prod_{i=1}^{m}\Big\| \mathscr{Q}^{ladder}_{s'^i,s^i,\sigma_0^i,\sigma_1^i,\sigma_2^i}\Big[1,1,\mathscr P_2^l\Big](k,k') )\Big\|_{L^{\infty}(\mathbb{T}^{2d})}.
	\end{aligned}
\end{equation}
Next, we apply \eqref{Lemma:Resonance1:9} to bound $\Big\| \mathscr{Q}^{ladder}_{s'^i,s^i,\sigma_0^i,\sigma_1^i,\sigma_2^i}\Big[1,1,\mathscr P_2^l\Big](k,k') \Big\|_{L^{\infty}(\mathbb{T}^{2d})}$ by $\|\mathscr P_2^l\|_{L^{\infty,2}(\mathbb{T}^{2d})}$, so that in the next iteration,  (A), (B) or (C) can be reapplied. We  bound $	\Big\| \mathscr{Q}^{ladder}_{s'^1,s^1,\sigma_0^1,\sigma_1^l,\sigma_2^l}$ $\Big[1,\mathscr P_1^0,$ $\mathscr P_2^0\Big](k,k')  \Big\|_{L^{\infty,2}(\mathbb{T}^{2d})}$ using strategy (B).

At the end of the procedure, after taking into account all of the possible graph combination of the ladder operators, by Assumption (B), we obtain the  bound
\begin{equation}
	\begin{aligned}\label{FinalProof:E10b:1}
		\|\mathscr{F}_{1/{q},\mathscr{R}}[\widehat{\mathfrak{C}^\lambda_{q,\infty,a}\chi_{[0,T_*)}}(\mathscr{X}{\lambda^{-2}})]\|_{L^{\infty,2}(\mathbb{T}^{2d})} \ \le \ & |\mathscr{R}|^{\frac{q}{2}}|\mathfrak{C}_{final,2}|^qT_*^{q+1},
	\end{aligned}
\end{equation}
for some constants $\mathfrak{C}_{final,2}>0$ independent of $\lambda$.

The  same argument also gives
\begin{equation}
	\begin{aligned}\label{FinalProof:E10b:1:a}
	&\	\left\|\mathscr{F}_{1/{\mathbf{N}},\mathscr{R}}\left[\sum_{q=\mathbf{N}_0}^{\lfloor\mathbf{N}/2\rfloor}\widehat{\mathfrak{C}^\lambda_{q,\infty,a}\chi_{[0,T_*)}}(\mathscr{X}{\lambda^{-2}})\right]\right\|_{L^{\infty,2}(\mathbb{T}^{2d})}^{\frac{1}{{\mathbf{N}}}} \\
 \ \lesssim \	&\int_{-\mathscr{R}}^\mathscr{R}\mathrm{d}\mathscr{X}\sum_{q=\mathbf{N}_0}^{\lfloor\mathbf{N}/2\rfloor}\left\|\widehat{\mathfrak{C}^\lambda_{q,\infty,a}\chi_{[0,T_*)}}(\mathscr{X}{\lambda^{-2}})\right\|_{L^{\infty,2}(\mathbb{T}^{2d})} ^{1/{\mathbf{N}}}
 \\
 \ \lesssim \	&\sum_{q=\mathbf{N}_0}^{\lfloor\mathbf{N}/2\rfloor}\left[\int_{-\mathscr{R}}^\mathscr{R}\mathrm{d}\mathscr{X}\left\|\widehat{\mathfrak{C}^\lambda_{q,\infty,a}\chi_{[0,T_*)}}(\mathscr{X}{\lambda^{-2}})\right\|_{L^{\infty,2}(\mathbb{T}^{2d})} ^{1/q}\right]^\frac{q}{{\mathbf{N}}}\mathscr{R}^{\frac{{\mathbf{N}}-q}{{\mathbf{N}}}},
	\end{aligned}
\end{equation}
yielding
\begin{equation}
	\begin{aligned}\label{FinalProof:E10b:1:b}
&	\	\mathscr{F}_{1/{\mathbf{N}},\mathscr{R}}\left[\sum_{q=\mathbf{N}_0}^{\lfloor\mathbf{N}/2\rfloor}\|\widehat{\mathfrak{C}^\lambda_{q,\infty,a}\chi_{[0,T_*)}}(\mathscr{X}{\lambda^{-2}})\|_{L^{\infty,2}(\mathbb{T}^{2d})} \right]\\ \lesssim \ & \sum_{q=\mathbf{N}_0}^{\lfloor\mathbf{N}/2\rfloor}|\mathscr{R}|^{\frac{q}{2}}|\mathfrak{C}_{final,2}|^qT_*^{q}\mathscr{R}^{{{\mathbf{N}}-q}}\\ \ \lesssim \ & |\mathscr{R}|^{\mathbf{N}}\left[\sum_{q=\mathbf{N}_0}^{\lfloor\mathbf{N}/2\rfloor}|\mathscr{R}|^{-\frac{q}{2}}\right]|\mathfrak{C}_{final,2}|^{\mathbf{N}_0}T_*^{\mathbf{N}_0}\\
	\ \lesssim \ & |\mathscr{R}|^{\mathbf{N}}|\mathfrak{C}_{final,2}|^{\mathbf{N}_0}T_*^{\mathbf{N}_0}.
	\end{aligned}
\end{equation}
The above inequality leads to
\begin{equation}
	\begin{aligned}\label{FinalProof:E10b:1:c}
	&\	\left[\frac{1}{2\mathscr{R}}\int_{-\mathscr{R}}^\mathscr{R}\mathrm{d}\mathscr{X}\left(\sum_{q=\mathbf{N}_0}^{\lfloor\mathbf{N}/2\rfloor}\|\widehat{\mathfrak{C}^\lambda_{q,\infty,a}\chi_{[0,T_*)}}(\mathscr{X}{\lambda^{-2}})\|_{L^{\infty,2}(\mathbb{T}^{2d})} \right)^{\frac{1}{\mathbf{N}}}\right]^{\mathbf{N}}\ \lesssim \  |\mathfrak{C}_{final,2}|^{\mathbf{N}_0}T_*^{\mathbf{N}_0}.
	\end{aligned}
\end{equation}
Using the inequality
\begin{equation}
	\begin{aligned}\label{FinalProof:E10b:1:d}
	&\	\frac{1}{2\mathscr{R}}\int_{-\mathscr{R}}^\mathscr{R}\mathrm{d}\mathscr{X}\log\left(\sum_{q=\mathbf{N}_0}^{\lfloor\mathbf{N}/2\rfloor}\|\widehat{\mathfrak{C}^\lambda_{q,\infty,a}\chi_{[0,T_*)}}(\mathscr{X}{\lambda^{-2}})\|_{L^{\infty,2}(\mathbb{T}^{2d})} \right)\\
	 \le \ & \log\left[\frac{1}{2\mathscr{R}}\int_{-\mathscr{R}}^\mathscr{R}\mathrm{d}\mathscr{X}\left(\sum_{q=\mathbf{N}_0}^{\lfloor\mathbf{N}/2\rfloor}\|\widehat{\mathfrak{C}^\lambda_{q,\infty,a}\chi_{[0,T_*)}}(\mathscr{X}{\lambda^{-2}})\|_{L^{\infty,2}(\mathbb{T}^{2d})} \right)^{\frac{1}{\mathbf{N}}}\right]^{\mathbf{N}},
	\end{aligned}
\end{equation}
which, in combination with \eqref{FinalProof:E10b:1:c}, implies
\begin{equation}
	\begin{aligned}\label{FinalProof:E10b:1:e}
		\frac{1}{2\mathscr{R}}\int_{-\mathscr{R}}^\mathscr{R}\mathrm{d}\mathscr{X}\log\left(\sum_{q=\mathbf{N}_0}^{\infty}\|\widehat{\mathfrak{C}^\lambda_{q,\infty,a}\chi_{[0,T_*)}}(\mathscr{X}{\lambda^{-2}})\|_{L^{\infty,2}(\mathbb{T}^{2d})} \right)\
	 \le \ & \log\left[|\mathfrak{C}_{final,2}|^{\mathbf{N}_0}T_*^{\mathbf{N}_0}\right].
	\end{aligned}
\end{equation}

We then deduced that 
\begin{equation}
	\begin{aligned}\label{FinalProof:E10b:1:f}
		\exp\left[\frac{1}{2\mathscr{R}}\int_{-\mathscr{R}}^\mathscr{R}\mathrm{d}\mathscr{X}\log\left(\sum_{q=\mathbf{N}_0}^{\infty}\|\widehat{\mathfrak{C}^\lambda_{q,\infty,a}\chi_{[0,T_*)}}(\mathscr{X}{\lambda^{-2}})\|_{L^{\infty,2}(\mathbb{T}^{2d})} \right)\right]\
	 \le \ & |\mathfrak{C}_{final,2}|^{\mathbf{N}_0}T_*^{\mathbf{N}_0}\to 0,
	\end{aligned}
\end{equation}
when $\mathbf{N}_0\to \infty,$ leading to
\begin{equation}
\label{FinalProof:E10b:1:g}
\mathscr{W}\left(\lim_{\mathbf{N}_0\to \infty} \sum_{q=\mathbf{N}_0}^{\infty}\|\widehat{\mathfrak{C}^\lambda_{q,\infty,a}\chi_{[0,T_*)}}(\mathscr{X}{\lambda^{-2}})\|_{L^{\infty,2}(\mathbb{T}^{2d})} >\theta\right)= 0,
\end{equation}
for any $\theta>0$.

We now define the operators

\begin{equation}
	\begin{aligned}\label{FinalProof:E10b:1:A}
		&				\widehat{	\diamondsuit_{\ell,q}[\mathfrak{C}^\lambda_{q,\infty,a}]	}(
		\lambda^{-2}\mathscr{X})	\ := \ \int_{0}^{T_*}\mathrm{d}\tau''e^{-{\bf i}2\pi \tau''\lambda^{-2}\mathscr{X}} e^{{-{\bf i}\tau''\lambda^{-2}[\sigma_{1}'\omega(k_1')+\sigma_{1}''\omega(k_1'')]}}\\
		&\times\left[
		\int_{(\mathbb{R}_+)^{\{1,\cdots,q\}}}\!\mathrm{d} s_{1} \cdots \mathrm{d} s_{2 q-1} 
		\delta\left(\sum_{i=1}^{q} s_{2i-1} \le( T_*- \tau'') \lambda^{-2} \right)\right.\\
		&\times\mathscr{A}_\ell\left[\int_{0}^{s_1\lambda^{-2}}\mathrm{d}s\int_{0}^{s}\mathrm{d}s'\mathcal{C}^\lambda_{2,\infty}(s',s-s')e^{-{\bf i}2\pi s\mathscr{X}}\right](\mathscr{X})\\
		&\left.\cdots\mathscr{A}_\ell\left[\int_{0}^{s_{2q-1}\lambda^{-2}}\mathrm{d}s\int_{0}^{s}\mathrm{d}s'\mathcal{C}^\lambda_{q+1,\infty}(s',s-s')e^{-{\bf i}2\pi s\mathscr{X}}\right](\mathscr{X})\right],\\
		&	\diamondsuit_{\ell}^\lambda[f] \ :=\ 	 	\sum^{[([\mathbf{N}/4]-1)/2]}_{q=0}\frac{1}{q!}\diamondsuit_{\ell,q}[\mathfrak{C}^\lambda_{q,\infty,a}],
	\end{aligned}
\end{equation}
where $\mathscr{A}_\ell$ is the averaging operator
\begin{equation}
	\begin{aligned}\label{FinalProof:E10b:1:B}
		\mathscr{A}_\ell[g](\mathscr{X})	\ := \  &\frac{1}{2\ell}\int_{[\mathscr{X}-\ell,\mathscr{X}+\ell]}\mathrm{d}\mathscr{X}' g(\mathscr{X}'),
	\end{aligned}
\end{equation}
for any integrable function $g$ on $\mathbb{R}$. 
It then follows 

\begin{equation}
	\begin{aligned}\label{FinalProof:E10b:1:C}
		&	\Big\|{	\mathscr{F}_{2/\mathbf{N},\mathscr{R}}[\widehat{\diamondsuit_{\ell,q}[\mathfrak{C}^\lambda_{q,\infty,a}]	}}(
		\lambda^{-2}\mathscr{X})]	\Big\|_{L^{\infty,2}(\mathbb{T}^{2d})}\\
		\ \lesssim \ 	& \Big\|\int_{0}^{T_*}\mathrm{d}\tau''e^{-{\bf i}2\pi \tau''\lambda^{-2}\mathscr{X}} e^{-{{\bf i}\tau''\lambda^{-2}[\sigma_{1}'\omega(k_1')+\sigma_{1}''\omega(k_1'')]}}\Big\|_{L^\infty(\mathbb{R}\times \mathbb{T}^{2d})}\\
		&\times\sup_{\tau''\in[0,T_*]} \Big\|\mathscr{F}_{2/\mathbf{N},\mathscr{R}}\Big[	\int_{0}^{(T_*-\tau'')\lambda^{-2}}\mathrm{d}\bar\tau \mathfrak{C}^\lambda_{q,\infty,a'}(\bar\tau)e^{-{\bf i}2\pi\bar\tau\mathscr{X}}\Big]\Big\|_{L^{\infty,2}(\mathbb{T}^{2d})}\\
		\ \lesssim \  &   (T_*)^{q+1}|\mathfrak{C}_{final,2}|^q.
	\end{aligned} 
\end{equation}
Moreover, we also have 	\begin{equation}
	\begin{aligned}\label{FinalProof:E10b:1:D}
		&\lim_{\ell\to 0}\Big\|	\mathscr{F}_{2/\mathbf{N},\mathscr{R}}\Big[	\widehat{	\diamondsuit_{\ell,q}[\mathfrak{C}^\lambda_{q,\infty,a}]	}	 -   &\int_{0}^{T_*}\mathrm{d}\tau''e^{-{\bf i}2\pi \tau''\lambda^{-2}\mathscr{X}}e^{-{{\bf i}\tau''\lambda^{-2} [\sigma_{1}'\omega(k_1')+\sigma_{1}''\omega(k_1'')]}} 	\int_{0}^{(T_*-\tau'')\lambda^{-2}}\mathrm{d}\bar\tau\\
	&\times	 \mathfrak{C}^\lambda_{q,\infty,a'}(\bar\tau)e^{-{\bf i}2\pi\bar\tau\mathscr{X}}\Big]\Big\|_{L^{\infty,2}(\mathbb{T}^{2d})} =  0,
	\end{aligned}
\end{equation}
which, by Chebychev's inequality, implies 	\begin{equation}
	\begin{aligned}\label{FinalProof:E10b:1:E}
	&	\lim_{\ell\to 0}	\mathscr{W}\Big(\Big|\widehat{	\diamondsuit_{\ell,q}[\mathfrak{C}^\lambda_{q,\infty,a}]	}	\ - \  &\int_{0}^{T_*}\mathrm{d}\tau''e^{-{\bf i}2\pi \tau''\lambda^{-2}\mathscr{X}}e^{-{{\bf i}\tau''\lambda^{-2} [\sigma_{1}'\omega(k_1')+\sigma_{1}''\omega(k_1'')]}} \\
		&	\int_{0}^{(T_*-\tau'')\lambda^{-2}}\mathrm{d}\bar\tau \mathfrak{C}^\lambda_{q,\infty,a'}(\bar\tau)e^{-{\bf i}2\pi\bar\tau\mathscr{X}}\Big|>\theta\Big) =0,
	\end{aligned}
\end{equation}
for all $\theta>0$. 

Observing  that
$$\lim_{\lambda\to0}	\mathscr{A}_\ell\left[\int_{0}^{s_{2i-1}\lambda^{-2}}\mathrm{d}s\int_{0}^{s}\mathrm{d}s'\mathcal{C}^\lambda_{i+1,\infty}(s',s-s')e^{-{\bf i}2\pi s\mathscr{X}}\right](\lambda^2\mathscr{X})$$
$$=	\mathscr{A}_\ell\left[\int_{0}^{\infty}\mathrm{d}s\int_{0}^{s}\mathrm{d}s'\mathcal{C}_{i+1,\infty}(s',s-s')e^{-{\bf i}2\pi s\mathscr{X}}\right](0),$$
in $L^2((-\mathscr{R},\mathscr{R});L^{\infty,2}(\mathbb{T}^{2d}))$, for $i=1,\cdots,q+1$, we can pass to the limit $\lambda\to\infty$   and obtain a limit for $\mathfrak{C}_q$ defined below in \eqref{FinalProof:E10a:2}, for all test function $\varphi\in C_c^\infty(\mathbb{R})$ and using the change of variables $k'+k''=2k^o$, $k'-k''=-2\epsilon \mho'$,

\begin{equation}
	\begin{aligned}\label{FinalProof:E10a:1:1:1}
		\		&\lim_{\lambda\to 0}\Big[	 \int_{\mathbb{R}} \mathrm{d}\mathscr{X}\widehat{	\diamondsuit_{\ell,q}[\mathfrak{C}^\lambda_{q,\infty,a}]	}(
		\lambda^{-2}\mathscr{X})\varphi(\lambda^{-2}\mathscr{X})\lambda^{-2} \\
		&- \int_{\mathbb{R}} \mathrm{d}\mathscr{X}\Big[\int_{0}^{T_*}\mathrm{d}\tau e^{-{\bf i}2\pi \tau\mathscr{X}} e^{-{{\bf i}\tau\lambda^{-2} [\sigma_{1}'\omega(k_1')+\sigma_{1}''\omega(k_1'')]}} \Big]	\mathfrak{C}_{q}\varphi(\mathscr{X})
		\Big]\\
		\
		=\	&\lim_{\lambda\to 0}\Big[	  \int_{\mathbb{R}} \mathrm{d}\mathscr{X}\varphi(\mathscr{X})\Big[\int_{0}^{T_*}\mathrm{d}\tau e^{-{\bf i}2\pi \tau\mathscr{X}} e^{-{{\bf i}\tau\lambda^{-2} [\sigma_{1}'\omega(k_1')+\sigma_{1}''\omega(k_1'')]}}  \Big]\\
		&\times	\mathscr{A}_\ell\left[\int_{0}^{s_1\lambda^{-2}}\mathrm{d}s\mathcal{C}^\lambda_{2,\infty}(s)e^{-{\bf i}2\pi s\mathscr{X}}\right](\lambda^2\mathscr{X}) \cdots\mathscr{A}_\ell\left[\int_{0}^{s_{2q-1}\lambda^{-2}}\mathrm{d}s\mathcal{C}^\lambda_{q+1,\infty}(s)e^{-{\bf i}2\pi s\mathscr{X}}\right](\lambda^2	\mathscr{X})\\
		&- \int_{\mathbb{R}} \mathrm{d}\mathscr{X}\Big[\int_{0}^{T_*}\mathrm{d}\tau e^{-{\bf i}2\pi \tau\mathscr{X}}e^{-{{\bf i}\tau\lambda^{-2} [\sigma_{1}'\omega(k_1')+\sigma_{1}''\omega(k_1'')]}} \Big]	\mathfrak{C}_{q} \varphi(\mathscr{X})\Big]
		\\
		=\	&	\Big[  \int_{\mathbb{R}} \mathrm{d}\mathscr{X}\varphi(\mathscr{X})\Big[\int_{0}^{T_*}\mathrm{d}\tau e^{-{\bf i}2\pi \tau\mathscr{X}}e^{{-{\bf i}\tau\lambda^{-2}[\sigma_{1}'\omega(k_1')+\sigma_{1}''\omega(k_1'')]}} \Big]\\
		&\times	\mathscr{A}_\ell\left[\int_{0}^{\infty}\mathrm{d}s\mathrm{d}s'\mathcal{C}^\lambda_{2,\infty}(s',s-s')e^{-{\bf i}2\pi s\mathscr{X}}\right](0)
		 \cdots\mathscr{A}_\ell\left[\int_{0}^{\infty}\mathrm{d}s\mathrm{d}s'\mathcal{C}^\lambda_{q+1,\infty}(s',s-s')e^{-{\bf i}2\pi s\mathscr{X}}\right](0)\\	&- \int_{\mathbb{R}} \mathrm{d}\mathscr{X}\Big[\int_{0}^{T_*}\mathrm{d}\tau e^{-{\bf i}2\pi \tau\mathscr{X}} e^{{-{\bf i}\tau}\mho_1'\cdot\nabla\omega(k_1^o)}  \Big]	\mathfrak{C}_{q} \varphi(\mathscr{X})\Big]\ = \ 0,
	\end{aligned}
\end{equation}
with
$\sigma_1'=1, \sigma_1''=-1$, the limit is in the ${L^\infty(\mathbb{R}^{d};L^2( \mathbb{T}^{d}))}$ norm, and
\begin{equation}
\begin{aligned}\label{FinalProof:E10a:2}
\mathfrak{C}_{q}(\tau,k_1^o,\mho_1')\ = \ &
\int_0^\tau \mathrm{d}se^{-{\bf i}s\nabla\omega(k_1^o)\mho_1'}\mathcal{C}_{2}\cdots\mathcal{C}_{q+1}(\tau-s,k_1^o,\mho_1'),
\end{aligned}
\end{equation}
where
\begin{equation}
\begin{aligned}\label{FinalProof:E10:a:1}
&\mathcal{C}_{m+1}(s,k_1^o,\cdots,k^o_i,\cdots,k^o_{m},\mho_1^o,\cdots,\mho_i^o,\cdots,\mho_m^o)\\
= &  \sum_{i=1}^{m}\mathcal{C}_{i,m+1}(s,k_1^o,\cdots,k^o_i,\cdots,k^o_{m},\mho_1^o,\cdots,\mho_i^o,\cdots,\mho_m^o)
\end{aligned}\end{equation}
\begin{equation}
\begin{aligned}\label{FinalProof:E9:1}
&\mathcal{C}_{i,m+1}(s,k_1^o,\cdots,k^o_i,\cdots,k^o_{m},\mho_1^o,\cdots,\mho_i^o,\cdots,\mho_m^o)\\
= &\   \frac{\pi}{2^{d-2}}\int_0^s\mathrm{d}s'\iint_{(\mathbb{T}^d)^2}
\mathrm{d}k^o\mathrm{d}k_{m+1}^o\iint_{(\mathbb{R}^d)^2}
\mathrm{d}\mho^o\mathrm{d}\mho_{m+1}^o\\
&\times|\mathcal{M}(k_{i}^o,k^o,k_{m+1}^o)|^2\delta_{\ell}\Big(\omega(k^o)+\omega(k_{m+1}^o)-\omega(k_{i}^o)\Big)\delta(k^o+k_{m+1}^o-k_{i}^o)\delta(\mho^o+\mho_{m+1}^o-\mho_{i}^o)\\
&\times e^{-{\bf i}s'[\nabla\omega(k^o)\mho'+\nabla\omega(k^o_{m+1})\mho^o_{m+1}+\nabla\omega(k^o_i)\mho'_i]}\Big( \mathfrak{L}^{m+1}(s',k_1^o,\cdots,k^o,\cdots,k^o_{m+1},\mho_1^o,\cdots,\mho^o,\cdots,\mho^o_{m+1})\\
&-\mathfrak{L}^{m+1}(s',k_1^o,\cdots,k^o_i,\cdots,k^o_{m+1},\mho_1^o,\cdots,\mho^o,\cdots,\mho^o_{m+1})\mathrm{sign}(k_{i}^o)\mathrm{sign}(k_{m+1}^o)\\
& -\mathfrak{L}^{m+1}(s',k_1^o,\cdots,k^o_i,\cdots,k^o,\mho_1^o,\cdots,\mho^o,\cdots,\mho^o_{m+1})\mathrm{sign}(k_{i}^o)\mathrm{sign}(k^o)\Big).
\end{aligned}\end{equation}
In the above expression, $k^o$ takes the position of $k^o_i$ in $\mathfrak{L}^{m+1}(k_1^o,\cdots,k^o,\cdots,k^o_{m+1},\mho_1^o,$ $\cdots,\mho^o,$ $\cdots,\mho_{m+1}^o)$ and of $k^o_{m+1}$ in $\mathfrak{L}^{m+1}(k_1^o,\cdots,$ $ k^o_i,\cdots,k^o,\mho_1^o,\cdots,\mho^o,\cdots,\mho^o_{m+1})$. Moreover,
\begin{equation}\begin{aligned}
& \mathfrak{L}^{m+1}(s,k_1^o,\cdots,k^o_i,\cdots,k^o_{m+1},\mho_1^o,\cdots,\mho_i^o,\cdots,\mho_{m+1}^o)\\ \ = \ & \mathcal{C}_{m+2}(s,k_1^o,\cdots,k^o_i,\cdots,k^o_{m+1},\mho_1^o,\cdots,\mho_i^o,\cdots,\mho_{m+1}^o),\end{aligned}
\end{equation}
and
\begin{equation}
\mathfrak{L}^{q+1}(s,k_1^o,\cdots,k^o_i,\cdots,k^o_{q+1},\mho_1^o,\cdots,\mho_i^o,\cdots,\mho_{q+1}^o):=\prod_{\{i,j\}\in S} {f}^o_\ell(\mho_{0,i}^o,k_{0,i}^o).\end{equation}
The above expression is a consequence of the limits $$\lim_{\lambda\to 0} e^{-{\bf i}s' (\omega(k')-\omega(k''))-{\bf i}s' (\omega(k_{m+1})-\omega(k_{m+1}'))-{\bf i} s' (\omega(k_i')-\omega(k_i''))}$$ $$ = e^{-{\bf i}s'[\nabla\omega(k^o)\mho^o+\nabla\omega(k^o_{m+1})\mho^o_{m+1}+\nabla\omega(k^o_i)\mho^o_i]},$$ 
$$\lim_{\lambda\to 0}\mathcal{M}(k_{i}',k',k_{m+1}')\mathcal{M}(k_{i}'',k'',k_{m+1}'')=|\mathcal{M}(k_{i}^o,k^o,k_{m+1}^o)|^2,$$
with  $$k_{m+1}-k_{m+1}'=2\epsilon \mho'_{m+1},\ \  k_i'-k_i''=2\epsilon\mho'_i, \ \  k_{m+1}+k_{m+1}'=2k_{m+1}^o,\ \ k_i'+k_i''=2k_i^o.$$
The product of the delta functions  is also replaced by $$\delta_{\ell}\Big(\omega_\infty(k^o)+\omega(k_{m+1}^o)-\omega(k_{i}^o)\Big)\delta(k^o+k_{m+1}^o-k_{i}^o)\delta(\mho^o+\mho_{m+1}^o-\mho_{i}^o).$$

It is straightforward from \eqref{FinalProof:E10b:1} that
$$\|\mathcal{C}_{2}\cdots\mathcal{C}_{q+1}(\mho',k^o)\|_{{L^\infty(\mathbb{R}^{d};L^2( \mathbb{T}^{d}))}} \le \mathcal{C}_o^q,$$
for some universal constant $\mathcal{C}_o>0$, we then find
\begin{equation}\label{FinalProof:E11}\begin{aligned}
		&	\sum_{q=0}^\infty \left\| \mathscr{F}_{2/\mathbf{N},\mathscr{R}}\Big[\int_{\mathbb{R}} \mathrm{d}\mathscr{X}\Big[\int_{0}^{T_*}\mathrm{d}\tau e^{-{\bf i}2\pi \tau\mathscr{X}}\mathfrak{C}_{q}(\mho',k^o) 	\Big]\right\|_{L^\infty(\mathbb{R}^{d};L^2( \mathbb{T}^{d}))}\lesssim \sum_{q=0}^\infty T_*(T_* \mathcal{C}_o)^q,\end{aligned}
\end{equation}
which converges for $T_*<\frac{1}{\mathcal{C}_o}$.

Now, let us estimate $\mathfrak Q_{2} $, which we write as \begin{equation}\label{FinalProof:E16}
	\mathfrak Q_{2} 
	\ = \ \sum_{n=[\mathbf{N}/4]}^{\mathbf{N}-1}\sum_{S\in\mathcal{P}^1_{pair}(\{1,\cdots,n+2\})}\mathfrak{G}_{2,n}'(S,t,k'',-1,k',1,\Gamma).
\end{equation}
The same estimates used for $\mathfrak{Q}_1$ can be redone for $\mathfrak{Q}_2$, yielding a sum of the type
\begin{equation}\label{FinalProof:E17}
	\ \lim_{\lambda\to0}\sum_{2q=[\mathbf{N}/4]}^{\mathbf{N}}\varsigma' (\mathcal{C}_oT_*)^q,
\end{equation}
where $\varsigma'$ is defined in \eqref{Def:Para3}, 
which tends to $0$ as $\mathbf{N}$ tends to infinity and $T_*$ sufficiently small.
The last quantities $\mathfrak Q_3$ also goes to $0$ as $\Phi_{0,i}$ goes to $0$ in the limit of $\lambda\to 0$.

As a result, we obtain the limit
\begin{equation}\label{FinalProof:E18}
	\lim_{\lambda\to0}\diamondsuit_{\ell}^\lambda[f]	\ = \ 	f^\infty_\ell(\mho',k^o,\tau)  \ = \    \sum_{q=0}^\infty \mathfrak{C}_{q},
\end{equation}
and the convergence \eqref{TheoremMain:1}-\eqref{TheoremMain:2} by a standard Chebyshev's inequality. This series is  a solution of
\eqref{Eq:WT1:a}.

\bibliographystyle{alpha}

\bibliography{WaveTurbulence}

\end{document}